\documentclass[10pt]{article}
\usepackage[utf8]{inputenc}
\usepackage{tikz-cd}
\usepackage{tikz}
\usepackage{graphicx}
\usepackage{amsfonts}
\usepackage{url}
\usepackage{amsmath,amsthm,amssymb,enumitem}
\usepackage{mathtools}
\usepackage{geometry}
\usepackage{mathrsfs} 
\mathtoolsset{showonlyrefs}
\usepackage{stmaryrd}
\usetikzlibrary{shapes.geometric}

\usepackage{xcolor,palatino}
\definecolor{ultramarine}{RGB}{0,32,96}
\colorlet{mygreen}{green!20!gray}
\colorlet{myultramarine}{ultramarine!20!gray}

\usepackage[colorlinks=true,citecolor=mygreen, linkcolor=mygreen, urlcolor=myultramarine,breaklinks, naturalnames, pagebackref]{hyperref}

\usepackage{float}
\usepackage{comment}

\usepackage{marginnote}

\makeatletter
\newsavebox\myboxA
\newsavebox\myboxB
\newlength\mylenA

\DeclareMathOperator*{\dprime}{^{\hskip-0.5mm\prime \prime}}
\DeclareMathOperator*{\dprimeind}{^{\hskip-1.5mm\prime \prime}}

\DeclareMathOperator*{\tprime}{^{\hskip-1.5mm\prime \prime \prime}}

\newcommand{\upperset}[2]{%
\underset{%
        \text{\raisebox{0.5ex}{\smash{\fontsize{5}{5}$#1$}}}
              }{#2}%
                    }

\numberwithin{equation}{section}
\numberwithin{figure}{section}
\allowdisplaybreaks[4]

\DeclareFontFamily{U}{BOONDOX-calo}{\skewchar\font=45 }
\DeclareFontShape{U}{BOONDOX-calo}{m}{n}{
  <-> s*[1.05] BOONDOX-r-calo}{}
\DeclareFontShape{U}{BOONDOX-calo}{b}{n}{
  <-> s*[1.05] BOONDOX-b-calo}{}
\DeclareMathAlphabet{\mathcalboondox}{U}{BOONDOX-calo}{m}{n}
\SetMathAlphabet{\mathcalboondox}{bold}{U}{BOONDOX-calo}{b}{n}
\DeclareMathAlphabet{\mathbcalboondox}{U}{BOONDOX-calo}{b}{n}

\title{Morphisms of pre-Calabi-Yau categories and morphisms of cyclic $A_{\infty}$-categories}
\author{Marion Boucrot}
\date{}

\newcommand{\MA}{\mathcal{A}}
\newcommand{\MB}{\mathcal{B}}
\newcommand{\MO}{\mathcal{O}}

\newcommand{\ZZ}{{\mathbb{Z}}}
\newcommand{\NN}{{\mathbb{N}}}
\newcommand{\kk}{\Bbbk}
\newcommand{\Hom}{\operatorname{Hom}}
\newcommand{\shom}{\operatorname{hom}}
\newcommand{\Homgr}{\mathcal{H}om}

\newcommand{\llg}{\operatorname{lg}}
\newcommand{\nec}{\operatorname{nec}}
\newcommand{\inn }{\operatorname{inn}}
\newcommand{\out}{\operatorname{out}}
\newcommand{\llt}{\operatorname{lt}}
\newcommand{\rrt}{\operatorname{rt}}
\newcommand{\Multi}{\operatorname{Multi}}

\newcommand{\id}{\operatorname{id}}

\newcommand{\multinec}{\operatorname{multinec}}
\newcommand{\pre}{\operatorname{pre}}
\newcommand{\pCY}{\operatorname{pre-CY}_d}
\newcommand{\SpCY}{\operatorname{Spre-CY}_d}
\newcommand{\NpCY}{\operatorname{Npre-CY}_d}

\newcommand{\Ahat}{\operatorname{\widehat{A_{\infty}}_d}}
\newcommand{\cyc}{\operatorname{\operatorname{cyc}\widehat{A_{\infty}}_d}}
\newcommand{\Scyc}{\operatorname{\operatorname{Scyc}\widehat{A_{\infty}}_d}}
\makeatletter
\newcommand*{\relrelbarsep}{.386ex}
\newcommand*{\relrelbar}{%
  \mathrel{%
    \mathpalette\@relrelbar\relrelbarsep
  }%
}
\newcommand*{\@relrelbar}[2]{%
  \raise#2\hbox to 0pt{$\m@th#1\relbar$\hss}%
  \lower#2\hbox{$\m@th#1\relbar$}%
}
\providecommand*{\rightrightarrowsfill@}{%
  \arrowfill@\relrelbar\relrelbar\rightrightarrows
}
\providecommand*{\leftleftarrowsfill@}{%
  \arrowfill@\leftleftarrows\relrelbar\relrelbar
}
\providecommand*{\xrightrightarrows}[2][]{%
  \ext@arrow 0359\rightrightarrowsfill@{#1}{#2}%
}
\providecommand*{\xleftleftarrows}[2][]{%
  \ext@arrow 3095\leftleftarrowsfill@{#1}{#2}%
}
\newcommand*\xoverline[2][0.75]{%
    \sbox{\myboxA}{$\m@th#2$}%
    \setbox\myboxB\null
    \ht\myboxB=\ht\myboxA%
    \dp\myboxB=\dp\myboxA%
    \wd\myboxB=#1\wd\myboxA
    \sbox\myboxB{$\m@th\overline{\copy\myboxB}$}
    \setlength\mylenA{\the\wd\myboxA}
    \addtolength\mylenA{-\the\wd\myboxB}%
    \ifdim\wd\myboxB<\wd\myboxA%
       \rlap{\hskip 0.5\mylenA\usebox\myboxB}{\usebox\myboxA}%
    \else
        \hskip -0.5\mylenA\rlap{\usebox\myboxA}{\hskip 0.5\mylenA\usebox\myboxB}%
    \fi}
\makeatother

\newcommand*{\doubarl}[1]{\xoverline{\xoverline{#1}}} 
\newcommand*{\doubar}[1]{\bar{\bar{#1}}} 

\newcommand{\doublefleche}
{ (0.05,-0.04) -- (0.18,-0.04)
(0.05,0.04) -- (0.18,0.04)
(0,0) -- (0.07,-0.07) 
(0,0) -- (0.07,0.07)}

\newcommand{\doubleflechescindeeleft}
{(0.07,-0.07-0.1) -- (0.22,-0.07-0.1)
(0,-0.1) -- (0.22,-0.1)
(0,-0.1) -- (0.1,-0.1-0.1) }

\newcommand{\doubleflechescindeeright}
{(0,0.1) -- (0.22,0.1)
(0.07,0.07+0.1) -- (0.22,0.07+0.1)
(0,0.1) -- (0.1,0.1+0.1)
}
\newcommand{\fleche}{
[<-,> = stealth] (0,0)--(0.5,0)
}
\newcommand{\flechelong}{
[<-, > = stealth] (0,0)--(1,0)
}

\newtheorem{definition}{Definition}[section]
\newtheorem{definition-proposition}[definition]{Definition-Proposition}
\newtheorem{remark}[definition]{Remark}

\newtheorem{proposition}[definition]{Proposition}
\newtheorem{example}[definition]{Example}
\newtheorem{corollary}[definition]{Corollary}
\newtheorem{lemma}[definition]{Lemma}
\newtheorem{theoreme}[]{Theorem}

\begin{document}

\maketitle
\hrulefill
\begin{abstract} 
In this article we prove that there exists a relation between $d$-pre-Calabi-Yau morphisms introduced by M. Kontsevich, A. Takeda and Y. Vlassopoulos and cyclic $A_{\infty}$-morphisms, extending a result proved by D. Fern\'andez and E. Herscovich. This leads to a functor between the category of $d$-pre-Calabi-Yau structures and the partial category of $A_{\infty}$-categories of the form $\MA\oplus\MA^*[d-1]$ with $\MA$ a graded quiver and whose morphisms are the data of an $A_{\infty}$-structure on $\MA\oplus\MB^*[d-1]$ together with $A_{\infty}$-morphisms $\MA[1]\oplus\MB^*[d]\rightarrow \MA[1]\oplus\MA^*[d]$ and $\MA[1]\oplus\MB^*[d]\rightarrow \MB[1]\oplus\MB^*[d]$.
\end{abstract}



\hrulefill

\tableofcontents

\section{Introduction}
Pre-Calabi-Yau structures were introduced by M. Kontsevich and Y. Vlassopoulos in the last de\-cade.
These structures have also appeared
under different names, such as $V_{\infty}$-algebras in \cite{tz}, $A_{\infty}$-algebras with boundary in
\cite{seidel}, or weak Calabi-Yau structures in \cite{kontsevich} for example. These references show that pre-Calabi-Yau structures
play an important role in homological algebra, symplectic geometry, string topology, noncommutative geometry and even in Topological Quantum Field Theory. 

A pre-Calabi-Yau structure is a Maurer-Cartan element of the necklace bracket introduced in \cite{ktv}. This bracket is given as the commutator of a necklace product and is a noncommutative analogue to the Schouten bracket on polyvector fields. This allows to regard pre-Calabi-Yau algebras as formal noncommutative Poisson structures. Therefore, one should expect a relation between pre-Calabi-Yau algebras and double Poisson algebras, which are natural candidates for Poisson structures for noncommutative differential geometry developed in \cite{c-beg} and \cite{vdb}. Actually, W.-K. Yeung show in \cite{yeung} that double Poisson structures on dg categories are examples of pre-Calabi-Yau categories. Moreover, it is proved in \cite{ik} that nongraded double Poisson algebras are in correspondence with a restricted class of pre-Calabi-Yau algebras.
More precisely, they established a one-to-one correspondence between the class of nongraded double Poisson algebra and the class of pre-Calabi-Yau algebras of the form $A\oplus A^*[d-1]$ whose only nonzero multiplications are the usual product and $m_3$ sending $A\otimes A^*\otimes A$ to $A$ and $A^*\otimes A\otimes A^*$ to $A^*$. 
D. Fern\'andez and E. Herscovich extend in \cite{fh} this correspondence to the differential graded setting and show that this correspondence satisfy a functorial property by studying the link between morphisms of double Poisson dg algebras and cyclic $A_{\infty}$-morphisms. 
Actually, they show that every morphism between double Poisson dg algebras $A$ and $B$ induces a cyclic $A_{\infty}$-structure on $A\oplus B^*[d-1]$ and strict cyclic $A_{\infty}$-morphisms 
 \begin{equation}
        \begin{tikzcd}[ampersand replacement=\&, column sep=small]
          \& A[1]\oplus B^*[d] \ar[dl, ""] \ar[dr, ""']  \\
            A[1]\oplus A^*[d] \&\& B[1]\oplus B^*[d] 
        \end{tikzcd}
\end{equation}
where the cyclic $A_{\infty}$-structures on $A\oplus A^*[d-1]$ and $B\oplus B^*[d-1]$ are the one corresponding to the double Poisson dg algebras $A$ and $B$ respectively.

There is also a strong relation between pre-Calabi-Yau algebras and $A_{\infty}$-algebras. Actually, M. Kontsevich, A. Takeda and Y. Vlassopoulos show in \cite{ktv} that for $d\in\ZZ$, a $d$-pre-Calabi-Yau structure on a finite dimensional vector space $A$ is equivalent to a cyclic $A_{\infty}$-structure on $A\oplus A^*[d-1]$ such that $A\hookrightarrow A\oplus A^*[d-1]$ is an $A_{\infty}$-subalgebra.
The definition of pre-Calabi-Yau morphisms first appeared in \cite{ktv} and then in \cite{lv}, in the properadic setting.
A natural question is then about the link between pre-Calabi-Yau morphisms and $A_{\infty}$-morphisms of the corresponding boundary construction.

In this paper, we study the relation between $A_{\infty}$-morphisms and pre-Calabi-Yau morphisms in a larger generality than in \cite{fh} with the same kind of construction, namely, we prove the following theorem: 

\begin{theoreme}
\label{thm:main-article-1}
     There exists a functor $\mathcal{P} : \pCY\rightarrow \Ahat$ which sends a $d$-pre-Calabi-Yau category $(\MA,s_{d+1}M_{\MA})$ to the $A_{\infty}$-category $(\MA\oplus\MA^*[d-1],sm_{\MA\oplus\MA^*})$ given in \cite{ktv} and a $d$-pre-Calabi-Yau morphism $(\Phi_0,s_{d+1}\Phi) : (\MA,s_{d+1}M_{\MA})\rightarrow (\MB,s_{d+1}M_{\MB})$ to an $A_{\infty}$-structure $sm_{\MA\oplus\MB^*}$ together with $A_{\infty}$-morphisms 
 \begin{equation}
\begin{tikzcd}
&(\MA \oplus \MB^*[d-1],sm_{\MA\oplus\MB^*}) \arrow[swap,"s\varphi_{\MA}"]{dl} \arrow[swap,"s\varphi_{\MB}"]{dr}\\
(\MA \oplus \MA^*[d-1],sm_{\MA\oplus\MA^*})&& (\MB \oplus \MB^*[d-1], sm_{\MB\oplus\MB^*}).
\end{tikzcd}
\end{equation}
\end{theoreme}

The aim of this work is to have a better understanding of pre-Calabi-Yau morphisms as it is an intricate notion. The relation with notions of $A_{\infty}$-structures and $A_{\infty}$-morphisms give an other point of view and the possibility of using well-known results and constructions in the $A_{\infty}$-case to give an intuition in the pre-Calabi-Yau case. For example, this can be relevant to investigate the notion of homotopy of pre-Calabi-Yau morphisms which is not known for the moment.

Let us briefly present the contents of the article.
In Section \ref{not-conv}, we fix the notations and conventions we use in this paper and in Section \ref{A-cat}, we recall the notions related to $A_{\infty}$-categories.
Section \ref{pcy-cat} is devoted to present the notions of discs and diagrams, which are a rewriting of the graphical calculus presented in \cite{ktv}, as well as the notion of pre-Calabi-Yau structures.
We also recall the link between pre-Calabi-Yau structures and $A_{\infty}$-structures in the case of a $\Hom$-finite graded quiver.
We incidentally show that the necklace product for a graded quiver $\MA$ is in fact equivalent to the usual Gerstenhaber circle product on $\MA \oplus \MA^{*}[d-1]$ (see Proposition \ref{proposition:j-bracket}), which does not seem to have been observed in the literature so far.  
In Section \ref{pcy-morphisms}, we recall the definitions of pre-Calabi-Yau morphisms and of the category $pCY_d$ given in \cite{ktv} and \cite{lv}. 

Section \ref{main section} is the core of the article. 
In Subsection \ref{mixed necklace}, we define a graded Lie algebra whose Maurer-Cartan elements induce $A_{\infty}$-structures on $A\oplus B^*[d-1]$.
In Subsection \ref{strict case}, we prove Theorem \ref{thm:main-article-1} in the case of strict $d$-pre-Calabi-Yau morphisms. In that case, the resulting $A_{\infty}$-morphisms are strict and cyclic and the $A_{\infty}$-structure on $\MA\oplus\MB^*[d-1]$ is cyclic. Moreover, we prove in Corollary \ref{coro:strict-case} that those data define a functor between partial categories.
In Subsection \ref{general case}, we prove Theorem \ref{thm:main-article-1} in the general case and show that with an additional assumption on the pre-Calabi-Yau morphism $\MA\rightarrow\MB$, the $A_{\infty}$-structure on $\MA\oplus\MB^*[d-1]$ is almost cyclic. We prove in Corollary \ref{corollary:functors} that this defines a functor between partial categories.

\textbf{Acknowledgements.} This work is part of a PhD thesis directed by Estanislao Herscovich and Hossein Abbaspour. The author thank them for the useful discussions and advices. The author also thank Johan Leray for a useful discussion about pre-Calabi-Yau morphisms. 

This work is supported by the French National Research Agency in the framework of the ``France 2030" program (ANR-15-IDEX-0002) and by the LabEx PERSYVAL-Lab (ANR-11-LABX-0025-01).

\section{Notations and conventions} \label{not-conv}
In what follows $\kk$ will be a field of characteristic different from $2$ and $3$ and to simplify we will denote $\otimes$ for $\otimes_{\kk}$. We will denote by $\NN = \{0,1,2,\dots \}$ the set of natural numbers and we define $\NN^*=\NN\setminus\{0\}$.
For $i,j\in\NN$, we define the interval of integers $\llbracket i,j\rrbracket=\{n\in\NN | i\leq n\leq j\}$. 

Recall that if we have a (cohomologically) graded vector space $V=\oplus_{i\in\ZZ}V^i$, we define for $n\in\ZZ$ the graded vector space $V[n]$ given by $V[n]^i=V^{n+i}$ for $i\in\ZZ$ and the map $s_{V,n} : V\rightarrow V[n]$ whose underlying set theoretic map is the identity. 
Moreover, if $f:V\rightarrow W$ is a morphism of graded vector spaces, we define the map $f[n] : V[n]\rightarrow W[n]$ sending an element $s_{V,n}(v)$ to $s_{W,n}(f(v))$ for all $v\in V$.
We will denote $s_{V,n}$ simply by $s_n$ when there is no possible confusion, and $s_1$ just by $s$.

We now recall the Koszul sign rules, that are the ones we use to determine the signs appearing in this paper. If $V,W$ are graded vector spaces, we have a map
$\tau_{V,W} : V\otimes W\rightarrow W\otimes V$ defined as
\begin{equation}
    \tau_{V,W}(v\otimes w)=(-1)^{|w||v|}w\otimes v
\end{equation}
where $v\in V$ is a homogeneous element of degree $|v|$ and $w\in W$ is a homogeneous element of degree $|w|$. 
More generally, given graded vector spaces $V_1,\dots,V_n$ and $\sigma\in\mathcalboondox{S}_n$, 
we have a map
\[
\tau^{\sigma}_{V_1,\dots,V_n} : V_1\otimes\dots\otimes V_n \rightarrow V_{\sigma^{-1}(1)}\otimes\dots\otimes V_{\sigma^{-1}(n)}
\]
defined as
\begin{equation}
    \tau^{\sigma}_{V_1,\dots,V_n}(v_1\otimes\dots\otimes v_n)=(-1)^{\epsilon}(v_{\sigma^{-1}(1)}\otimes\dots\otimes v_{\sigma^{-1}(n)})
\end{equation}
with
\[
\epsilon=\sum\limits_{\substack{i>j\\ \sigma^{-1}(i)<\sigma^{-1}(j)}} |v_{\sigma^{-1}(i)}||v_{\sigma^{-1}(j)}|
\]
where $v_i\in V_i$ is a homogeneous element of degree $|v_i|$ for $i\in\llbracket 1,n \rrbracket$.

Throughout this paper, when we consider an element $v$ of degree $|v|$ in a graded vector space $V$, we mean a homogeneous element $v$ of $V$.
Also, we will denote by $\id$ the identity map of every space of morphisms, without specifying it.
All the products in this paper will be products in the category of graded vector spaces. Given graded vector spaces $(V_i)_{i\in I}$, we thus have \[
\prod\limits_{i\in I}V_i=\bigoplus\limits_{n\in\ZZ}\prod\limits_{i\in I}V_{i}^{n}\]
where the second product is the usual product of vector spaces.

Given graded vector spaces $V,W$ we will denote by $\Hom_{\kk}(V,W)$ the vector space of $\kk$-linear maps $f : V \rightarrow W$ and by $\shom_{\kk}^d(V,W)$ the vector space of homogeneous $\kk$-linear maps $f : V \rightarrow W$ of degree d, \textit{i.e.} $f(v)\in W^{n+d}$ for all $v\in V^n$.
We assemble them in the graded vector space $\Homgr_{\kk}(V,W) =\bigoplus_{d\in \ZZ} \shom_{\kk}^d(V,W)\subseteq \Hom_{\kk}(V,W)$. We define the \textbf{\textcolor{ultramarine}{graded dual}} of a graded vector space $V=\bigoplus_{n\in\ZZ}V^n$ as the graded vector space $V^{*}=\Homgr_{\kk}(V,\kk)$.
Moreover, given graded vector spaces $V$, $V'$, $W$, $W'$ and homogeneous elements $f\in \Homgr_{\kk}(V,V')$ and $g\in\Homgr_{\kk}(W,W')$, we have that 
\begin{equation}
(f\otimes g)(v\otimes w)=(-1)^{|g||v|}f(v)\otimes g(w)
\end{equation}
for homogeneous elements $v\in V$ and $w\in W$.
Recall that given graded vector spaces $V_1,\dots,V_n$ and $d\in\ZZ$ we have a homogeneous linear isomorphism of degree $0$
\begin{equation}
    \label{eq:iso-shift-tensor-prod}
    H_j : (\bigotimes\limits_{i=1}^n V_i)[d]\rightarrow V_1\otimes\dots\otimes V_{j-1}\otimes V_j[d]\otimes V_{j+1}\otimes\dots\otimes V_n
\end{equation}
sending an element $s_d(v_1\otimes\dots \otimes v_n)$ to $(-1)^{d(|v_1|+\dots+|v_{j-1}|)}v_1\otimes\dots\otimes v_{j-1}\otimes s_dv_j\otimes v_{j+1}\otimes\dots\otimes v_n$.
Moreover, given graded vector spaces $V$ and $W$ and an integer $d\in\ZZ$, we have homogeneous linear isomorphisms of degree $0$
\begin{equation}
    \label{eq:iso-shift-output}
    \Homgr_{\kk}(V,W)[d]\rightarrow \Homgr_{\kk}(V,W[d])
\end{equation}
sending $s_df\in \Homgr_{\kk}(V,W)[d]$ to the map sending $v\in V$ to $s_d(f(v))$
and
\begin{equation}
    \label{eq:iso-shift-input}
    \Homgr_{\kk}(V,W)[d]\rightarrow\Homgr_{\kk}(V[-d],W)
\end{equation}
sending $s_df\in \Homgr_{\kk}(V,W)[d]$ to the map sending $s_{-d}v\in V[-d]$ to $(-1)^{d|f|}f(v)$.

Recall that a graded quiver $\MA$ consists of a set of objects $\MO$ together with graded vector spaces ${}_y\MA_x$ for every $x,y\in\MO$.
A dg quiver $\MA$ is a graded quiver such that ${}_y\MA_x$ is a dg vector space for every $x,y\in\MO$.
Given a quiver $\MA$, its \textbf{\textcolor{ultramarine}{enveloping graded quiver}} is the graded quiver $\MA^{e}=\MA^{op}\otimes \MA$ whose set of objects is $\MO\times \MO$ and whose space of morphisms from an object $(x,y)$ to an object $(x',y')$ is 
defined as the graded vector space ${}_{(x',y')}(\MA^{op}\otimes \MA)_{(x,y)}={}_{x}\MA_{x'}\otimes {}_{y'}\MA_{y}$.
A graded quiver $\MA$ with set of objects $\MO$ is said to be \textbf{\textcolor{ultramarine}{$\Hom$-finite}} if ${}_{y}\MA_x$ is finite dimensional for every $x,y\in\MO$. Given a graded quiver $\MA$ with set of objects $\MO$ its \textbf{\textcolor{ultramarine}{graded dual quiver}} is the quiver $\MA^{*}$ whose set of objects is $\MO$ and for $x,y\in \MO$, the space of morphisms from $x$ to $y$ is defined as ${}_{y}\MA^{*}_{x}=({}_{x}\MA_{y})^{*}$. 
Given graded quivers $\MA$ and $\MB$ with respective sets of objects $\MO_{\MA}$ and $\MO_{\MB}$ , a \textbf{\textcolor{ultramarine}{morphism of graded quivers}} $(\Phi_0,\Phi) : \MA \rightarrow \MB$ is the data of a map $\Phi_0 : \MO_{\MA}\rightarrow \MO_{\MB}$ between the sets of objects together with a collection $\Phi=(\Phi^{x,y})_{x,y\in\MO_{\MA}}$ of morphisms of graded vector spaces $\Phi^{x,y} : {}_x\MA_y\rightarrow {}_{\Phi_0(x)}\MB_{\Phi_0(y)}$ for every $x,y\in\MO_{\MA}$.

We will denote \[
\bar{\MO}=\bigsqcup_{n\in\NN^*}\MO^n
\]and more generally, we will denote by $\doubar{\MO}$ the set formed by all finite tuples of elements of $\bar{\MO}$, \textit{i.e.} 
\begin{equation}
\begin{split}
\bar{\bar{\MO}}=\bigsqcup_{n\in\NN^*} \bar{\MO}^n=\bigsqcup_{n\in\NN^*}\bigsqcup_{(p_1,\dots,p_n)\in\mathcal{T}_n}\MO^{p_1}\times\dots\times \MO^{p_n}
\end{split}
\end{equation}  
where $\mathcal{T}_n=\NN^n$ for $n>1$ and $\mathcal{T}_1=\NN^*$.
Given $\bar{x}=(x_1,\dots,x_n)\in\bar{\MO}$ we define its \textbf{\textcolor{ultramarine}{length}} as $\llg(\bar{x})=n$, its \textbf{\textcolor{ultramarine}{left term}} as $\llt(\bar{x})=x_1$ and its \textbf{\textcolor{ultramarine}{right term}} as $\rrt(\bar{x})=x_n$. For $i\in\llbracket 1,n\rrbracket$, we define $\bar{x}_{\leq i}=(x_1,\dots,x_i)$, $\bar{x}_{\geq i}=(x_i,\dots,x_n)$ and for $j>i$, $\bar{x}_{\llbracket i,j\rrbracket}=(x_i,x_{i+1},\dots,x_j)$. One can similarly define $\bar{x}_{<i}$ and $\bar{x}_{>i}$.
Moreover, given $\doubar{x}=(\bar{x}^1,\dots,\bar{x}^n)\in\doubar{\MO}$ we define its \textbf{\textcolor{ultramarine}{length}} as $\llg(\doubar{x})=n$, its \textbf{\textcolor{ultramarine}{left term}} as $\llt(\doubar{x})=\bar{x}^1$ and its \textbf{\textcolor{ultramarine}{right term}} as $\rrt(\doubar{x})=\bar{x}^n$. 
For $\bar{x}=(x_1,\dots,x_n)\in\bar{\MO}$, we will denote 
\[
\MA^{\otimes \bar{x}}={}_{x_1}\MA_{x_2}\otimes {}_{x_2}\MA_{x_3}\otimes\dots\otimes {}_{x_{\llg(\bar{x})-1}}\MA_{x_{\llg(\bar{x})}}
\]
and we will often denote an element of $\MA^{\otimes\bar{x}}$ as $a_1,a_2,\dots,a_{\llg(\bar{x})-1}$ instead of $a_1\otimes a_2\otimes \dots \otimes a_{\llg(\bar{x})-1}$ for $a_i\in {}_{x_i}\MA_{x_{i+1}}$, $i\in\llbracket 1,\llg(\bar{x})-1\rrbracket$. Moreover, given a tuple $\doubar{x}=(\bar{x}^1,\dots,\bar{x}^n)\in\doubar{\MO}$ we will denote $\MA^{\otimes \doubar{x}}=\MA^{\otimes \bar{x}^1}\otimes \MA^{\otimes \bar{x}^2}\otimes\dots\otimes \MA^{\otimes \bar{x}^n}$.
Given tuples $\bar{x}=(x_1,\dots,x_n),\bar{y}=(y_1,\dots,y_m)\in\bar{\MO}$, we define their concatenation as $\bar{x}\sqcup\bar{y}=(x_1,\dots,x_n,y_1,\dots,y_m)$.
We also define the inverse of a tuple $\bar{x}=(x_1,\dots,x_n)\in\bar{\MO}$ as $\bar{x}^{-1}=(x_n,x_{n-1},\dots,x_1)$.
If $\sigma\in\mathcalboondox{S}_n$ and $\bar{x}=(x_1,\dots,x_n)\in\MO^n$, we define $\bar{x}\cdot\sigma=(x_{\sigma(1)},x_{\sigma(2)},\dots,x_{\sigma(n)})$.
Moreover, given $\sigma\in\mathcalboondox{S}_n$ and $\doubar{x}=(\bar{x}^1,\dots,\bar{x}^n)\in\bar{\MO}^n$, we define $\doubar{x}\cdot\sigma=(\bar{x}^{\sigma(1)},\bar{x}^{\sigma(2)},\dots,\bar{x}^{\sigma(n)})$.
We denote by $C_n$ the subgroup of $\mathcalboondox{S}_n$ generated by the cycle $\sigma=(1 2\dots n)$ which sends $i\in \llbracket 1,n-1\rrbracket $ to $i+1$ and $n$ to $1$.
\section{\texorpdfstring{$A_{\infty}$-categories}{A-infinity-categories}} 
\label{A-cat}
In this section, we recall the notion of (cyclic) $A_{\infty}$-categories and (cyclic) $A_{\infty}$-morphisms as well as the definition of the natural bilinear form associated to a graded quiver.
We also introduce a bilinear form on categories of the form $\MA\oplus\MB^*$ where $\MA$ and $\MB$ are graded quivers related by a morphism $\MA\rightarrow\MB$.

\begin{definition}
\label{def:hochschild-graded-vs}
    Given graded quivers $\MA$ and $\MB$ with respective sets of objects $\MO_{\MA}$ and $\MO_{\MB}$ as well as a map $\Phi_0 : \MO_{\MA}\rightarrow \MO_{\MB}$, we define the graded quiver
    \[
    C_{\Phi_0}(\MA,\MB)=\prod\limits_{p\geq 1}\prod\limits_{\bar{x}\in\MO_{\MA}^p}C^{\bar{x}}(\MA,\MB)=\prod\limits_{p\geq 1}\prod\limits_{\bar{x}\in\MO_{\MA}^p}\Homgr_{\kk}(\MA[1]^{\otimes \bar{x}}, {}_{\Phi_0(\llt(\bar{x}))}\MB_{\Phi_0(\rrt(\bar{x}))})
    \]
    We will denote $C_{\id}(\MA,\MA)$ simply by $C(\MA)$.
    A homogeneous element $sf^{x_1,\dots,x_n}\in C_{\Phi_0}^{x_1,\dots,x_n}(\MA,\MB)$ will be represented by a disc with $n-1$ inputs with the objects $x_1,\dots,x_n$ between them and  one output (see Figure \ref{fig:element-of-C(A,B)}).
\begin{figure}[H]
    \centering
\begin{tikzpicture}[line cap=round,line join=round,x=1.0cm,y=1.0cm]
\clip(-7.5,-0.8) rectangle (10.896080730271247,0.8);
   \draw [line width=0.5pt] (0.,0.) circle (0.5cm);
    \shadedraw [->,>=stealth](0.5,0)--(0.9,0);
\draw (0,0.25)node[anchor=north]{$\mathbf{f}$};
\draw[rotate =180][<-,>=stealth] (0.5,0)--(0.9,0);
\draw[rotate =120][<-,>=stealth] (0.5,0)--(0.9,0);
\draw[rotate =150][<-,>=stealth] (0.5,0)--(0.9,0);
\draw[rotate =-120][<-,>=stealth] (0.5,0)--(0.9,0);
\draw (-0.1,0.9) node[anchor=north]{$\scriptstyle{x_1}$};
\draw (-0.55,0.7) node[anchor=north]{$\scriptstyle{x_2}$};
\draw (-0.7,0.37) node[anchor=north]{$\scriptstyle{x_3}$};
\draw (-0.1,-0.5) node[anchor=north]{$\scriptstyle{x_{n}}$};
\draw[fill=black] (-0.7,-0.2) circle (0.3pt);
\draw[fill=black] (-0.65,-0.37) circle (0.3pt);
\draw[fill=black] (-0.55,-0.5) circle (0.3pt);
\end{tikzpicture}
\caption{The representation of a homogeneous element in $C(\MA,\MB)$}
    \label{fig:element-of-C(A,B)}
\end{figure}
    \noindent
    To simplify, we regroup the incoming arrows as in Figure \ref{fig:element-of-C(A,B)-simplified}.
    
      \begin{figure}[H]
    \centering
\begin{tikzpicture}[line cap=round,line join=round,x=1.0cm,y=1.0cm]
\clip(-7.5,-0.6) rectangle (10.896080730271247,0.6);
   \draw [line width=0.5pt] (0.,0.) circle (0.5cm);
    \shadedraw [rotate=180, shift={(0.5cm,0cm)}] \doublefleche;
    \shadedraw [->,>=stealth](0.5,0)--(0.9,0);
\draw (0,0.25)node[anchor=north]{$\mathbf{f}$};
\end{tikzpicture}
\caption{The simplified representation of a homogeneous element in $C(\MA,\MB)$}
    \label{fig:element-of-C(A,B)-simplified}
\end{figure}
\end{definition}

\begin{definition}
    The \textbf{\textcolor{ultramarine}{type}} of a disc representing a map $F^{\bar{x}} : \MA[1]^{\otimes \bar{x}}\rightarrow {}_{\Phi_0(\llt(\bar{x}))}\MB_{\Phi_0(\rrt(\bar{x}))}$ is the tuple $\bar{x}\in\bar{\MO}$.
\end{definition}

\begin{definition}
    Let $\MA$ be a graded quiver.
    By the isomorphism \eqref{eq:iso-shift-output}, an element $s\mathbf{F}\in C(\MA)[1]$ induces maps 
    \begin{equation}
        \label{eq:shift-disc}
        \MA[1]^{\otimes\bar{x}}\rightarrow {}_{\llt(\bar{x})}\MB_{\rrt(\bar{x})}[1]
    \end{equation}
sending an element $(sa_1,\dots,sa_{n-1})$ to $s(F^{\bar{x}}(sa_1,\dots,sa_{n-1}))$
    for $\bar{x}=(x_1,\dots,x_n)\in \MO^n$, $a_i\in{}_{x_i}\MA_{x_{i+1}}$, $i\in\llbracket 1,n-1 \rrbracket$.
     A homogeneous element $s\mathbf{f}\in C_{\Phi_0}(\MA,\MB)[1]$ will be represented by a disc with a bold output (see Figure \ref{fig:element-of-C(A,B)[1]}).
     
      \begin{figure}[H]
    \centering
\begin{tikzpicture}[line cap=round,line join=round,x=1.0cm,y=1.0cm]
\clip(-7.5,-0.6) rectangle (10.896080730271247,0.6);
   \draw [line width=0.5pt] (0.,0.) circle (0.5cm);
    \shadedraw [rotate=180, shift={(0.5cm,0cm)}] \doublefleche;
    \shadedraw[line width=1.1pt] [->,>=stealth](0.5,0)--(0.9,0);
\draw (0,0.25)node[anchor=north]{$\mathbf{f}$};
\end{tikzpicture}
\caption{The representation of a homogeneous element in $C(\MA,\MB)[1]$}
    \label{fig:element-of-C(A,B)[1]}
\end{figure}
\end{definition}

\begin{definition}
\label{def:gerstenhaber-prod}
Given a graded quiver $\MA$ with set of objects $\MO$, the \textbf{\textcolor{ultramarine}{Gerstenhaber product}} of elements $s\mathbf{f},s\mathbf{g}\in C(\MA)[1]$ is defined as the element $s\mathbf{f}\upperset{G}{\circ}s\mathbf{g}\in C(\MA)[1]$ given by
\begin{equation}
    \label{eq:gerstenhaber-product}
   (s\mathbf{f}\upperset{G}{\circ}s\mathbf{g})^{\bar{x}}=\sum\limits_{1\leq i<j\leq \llg(\bar{x})} sf^{\bar{x}_{\leq i}\sqcup \bar{x}_{\geq j}}\upperset{G,i,j}{\circ}sg^{\bar{x}_{\llbracket i,j\rrbracket}}
\end{equation}
for $\bar{x}\in\bar{\MO}$
where 
\begin{equation}
    \begin{split}
    &\big(sf^{\bar{x}_{\leq i}\sqcup \bar{x}_{\geq j}}\upperset{G,i,j}{\circ}sg^{\bar{x}_{\llbracket i,j\rrbracket}}\big)(sa_1,\dots,sa_{n-1})\\&=(-1)^{\epsilon_i} sf^{\bar{x}_{\leq i}\sqcup \bar{x}_{\geq j}}(sa_1,\dots,sa_{i-1},sg^{\bar{x}_{\llbracket i,j\rrbracket}}(sa_i,\dots,sa_{j-1}),sa_{j},\dots,sa_{n-1})
    \end{split}
\end{equation}
for $a_i\in{}_{x_{i}}\MA_{x_{i+1}}$, with $\epsilon_i=(|\mathbf{g}|+1)\sum\limits_{r=1}^{i-1}|sa_r|$.

The map $(s\mathbf{f}\upperset{G}{\circ}s\mathbf{g})^{\bar{x}}$ is by definition the sum of the maps associated to diagrams of type $\bar{x}$ and of the form

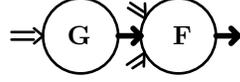
\begin{tikzpicture}[line cap=round,line join=round,x=1.0cm,y=1.0cm]
\clip(-6,-1) rectangle (4.302897460633269,1);
 \draw [line width=0.5pt] (0.,0.) circle (0.5cm);
    \shadedraw [rotate=180, shift={(0.5cm,0cm)}] \doublefleche;
    \draw [line width=0.5pt] (1.5,0.) circle (0.5cm);
    \shadedraw[shift={(1cm,0cm)},rotate=180] \doubleflechescindeeleft;
    \shadedraw[shift={(1cm,0cm)},rotate=180] \doubleflechescindeeright;
    \shadedraw[ shift={(1cm,0cm)},rotate=180] \fleche;
    \draw [line width=1.1pt,rotate around={0:(1.5,0)}] [->,, >= stealth, >= stealth](2,0) -- (2.4,0);
\draw (1.5,0.25) node[anchor=north] {$\mathbf{f}$};
\draw (0,0.25) node[anchor=north] {$\mathbf{g}$};
\end{tikzpicture}

for $\Bar{x}\in\Bar{\MO}$.
\end{definition}

The following Definition-Proposition was first given in \cite{gerstenhaber}, where the reader can find a proof of the statement. 

\begin{definition-proposition}
\label{def-prop:gerstenhaber-bracket}
Given a graded quiver $\MA$ with set of objects $\MO$, the \textbf{\textcolor{ultramarine}{Gerstenhaber bracket}} is the graded Lie bracket $[-,-]_G$ on $C(\MA)[1]$ defined for two elements $s\textbf{f},s\textbf{g}\in C(\MA)[1]$
as $[s\textbf{f},s\textbf{g}]_G\in C(\MA)[1]$ with
\begin{equation}
    \label{eq:gerstenhaber-bracket}
    [s\textbf{f},s\textbf{g}]_G^{\bar{x}}=(s\textbf{f}\upperset{G}{\circ}s\textbf{g})^{\bar{x}}-(-1)^{(|\textbf{f}|+1)(|\textbf{g}|+1)}(s\textbf{g}\upperset{G}{\circ}s\textbf{f})^{\bar{x}}
\end{equation}
for $\bar{x}\in\bar{\MO}$.
\end{definition-proposition}

\begin{definition}
\label{def:A-inf-alg}
An \textbf{\textcolor{ultramarine}{$A_{\infty}$-structure}} on a graded quiver $\MA$ with set of objects $\MO$ is an element of degree $1$ $sm_{\MA}\in C(\MA)[1]$ satisfying the Stasheff identities introduced by J. Stasheff in \cite{stasheff} and given by
\begin{equation}
    \label{eq:stasheff-identities}
   \tag{$\operatorname{SI}^{\bar{x}}$}
    \sum\limits_{1\leq i <j \leq n}sm_{\MA}^{\bar{x}_{\leq i}\sqcup\bar{x}_{\geq j}}\circ(\id^{\otimes \bar{x}_{\leq i}}\otimes sm_{\MA}^{\bar{x}_{\llbracket i,j\rrbracket}}\otimes \id^{\otimes \bar{x}_{\geq j}})=0
\end{equation}
for each $n\in\NN^*$ and $\bar{x}\in\MO^n$. 
This is tantamount to $[sm_{\MA},sm_{\MA}]_{G}=0$ where $[-,-]_G$ denotes the Gerstenhaber bracket.
In terms of diagrams, this reads $\sum\mathcal{E}(\mathcalboondox{D})=0$ where the sum is over all the filled diagrams $\mathcalboondox{D}$ of type $\bar{x}$ and of the form 

\begin{tikzpicture}[line cap=round,line join=round,x=1.0cm,y=1.0cm]
\clip(-6,-0.6) rectangle (4.302897460633269,0.6);
 \draw [line width=0.5pt] (0.,0.) circle (0.5cm);
    \shadedraw [rotate=180, shift={(0.5cm,0cm)}] \doublefleche;
    \draw [line width=0.5pt] (1.5,0.) circle (0.5cm);
    \shadedraw[shift={(1cm,0cm)},rotate=180] \doubleflechescindeeleft;
    \shadedraw[shift={(1cm,0cm)},rotate=180] \doubleflechescindeeright;
    \shadedraw[ shift={(1cm,0cm)},rotate=180] \fleche;
    \draw [line width=1.1pt,rotate around={0:(1.5,0)}] [->,, >= stealth, >= stealth](2,0) -- (2.4,0);
\draw (1.5,0.2) node[anchor=north] {$m_{\MA}$};
\draw (0,0.2) node[anchor=north] {$m_{\MA}$};
\end{tikzpicture}

 A graded quiver endowed with an $A_{\infty}$-structure is called an \textbf{\textcolor{ultramarine}{$A_{\infty}$-category}}. 
\end{definition}

\begin{example}
\label{example:dg-cat-pCY}
    If $\MA$ is a dg category with differential $d_{\MA}$ and product $\mu$, it carries a natural $A_{\infty}$-structure $sm_{\MA}\in C(\MA)[1]$ with $sm_{\MA}^{x,y}=d_{{}_{x}\MA_{y}}[1]$, $sm_{\MA}^{x,y,z}(sa,sb)=(-1)^{|a|}s\mu(a,b)$ for $a\in{}_{x}\MA_{y}$, $b\in{}_{y}\MA_{z}$ and $sm_{\MA}^{\bar{x}}=0$ for $\bar{x}\in\MO^n$ with $n>3$.
\end{example}

We now recall the notion of (nondegenerate) bilinear form on a category, which will be useful for the definition of cyclic $A_{\infty}$-category.

\begin{definition}
A \textbf{\textcolor{ultramarine}{bilinear form of degree d}} on a graded quiver $\MA$ is a collection $\Gamma=({}_{y}\Gamma_{x})_{x,y\in\MO}$ of homogeneous $\kk$-linear maps ${}_{y}\Gamma_x : {}_{y}\MA_{x}[1]\otimes {}_{x}\MA_{y}[1] \rightarrow \kk$ of degree $d+2$.
\end{definition}
\begin{definition}
    A bilinear form $\Gamma$ on a graded quiver $\MA$ is \textbf{\textcolor{ultramarine}{nondegenerate}} if the induced map 
    \[{}_y\MA_x[1]\rightarrow ({}_{x}\MA_{y}[1])^*\] sending an element $sa\in {}_y\MA_x[1]$ to the map sending $sb\in {}_{x}\MA_{y}[1]$ to ${}_y\Gamma_x(sa,sb)$ is an isomorphism.
\end{definition}

\begin{example}
\label{example:natural-mixed-bilinear-form}
Consider two graded quivers $\MA$ and $\MB$ and a morphism $(\Phi_0,\Phi) : \MA\rightarrow \MB$ of graded quivers. 
Define the bilinear form $\Gamma^{\Phi} : (\MA[1]\oplus \MB^{*}[d])^{\otimes 2} \rightarrow \kk$ of degree $d+2$ by 
\begin{equation}
    {}_y\Gamma^{\Phi}_x(tf,sa)=-(-1)^{|sa||tf|}{}_x\Gamma^{\Phi}_y(sa,tf)=(-1)^{|tf|+1}(f\circ \Phi^{x,y})(a)
\end{equation} 
for $f\in {}_{\Phi_0(y)}\MB^{*}_{\Phi_0(x)}$, $a\in{}_{x}\MA_{y}$, where t stands for the shift morphism $\MB^{*}\rightarrow\MB^{*}[d]$ and 
\begin{equation}
{}_y\Gamma^{\Phi}_x(tf,tg)={}_y\Gamma^{\Phi}_x(sa,sb)=0    
\end{equation} for
$f\in{}_{\Phi_0(y)}\MB^{*}_{\Phi_0(x)}$, $g\in{}_{\Phi_0(x)}\MB^{*}_{\Phi_0(y)}$, $a\in{}_{y}\MA_{x}$ and $b\in{}_{x}\MA_{y}$.
This bilinear form $\Gamma^{\Phi}$ will be called the \textbf{\textcolor{ultramarine}{$\Phi$-mixed bilinear form}}.
\end{example}

\begin{example}
\label{example:natural-bilinear-form}
If $\MB=\MA$, the bilinear form $\Gamma^{\id}$ of the previous example is called the \textbf{\textcolor{ultramarine}{natural bilinear form}} on $\MA$ and will be denoted $\Gamma^{\MA}$.
\end{example}

\begin{remark}
The natural bilinear form on a $\Hom$-finite graded quiver $\MA$ is nondegenerate, whereas the $\Phi$-mixed bilinear form is not in general. 
\end{remark}

\begin{definition}
An $A_{\infty}$-structure $sm_{\MA}\in C(\MA)[1]$ on a graded quiver $\MA$ is \textbf{\textcolor{ultramarine}{almost cyclic}} with respect to a homogeneous bilinear form $\Gamma :\MA[1]^{\otimes 2}\rightarrow \kk$ if the following holds:
\begin{equation}
\begin{split}
    \label{eq:cyclicity-pcy}
    &{}_{x_{1}}\Gamma_{x_n}(sm_{\MA}^{\bar{x}}(sa_1,\dots,sa_{n-1}),sa_n)\\&=(-1)^{|sa_n|(\sum\limits_{i=1}^{n-1}|sa_i|)}{}_{x_{n}}\Gamma_{x_{n-1}}(sm_{\MA}^{\bar{x}\cdot\sigma}(sa_n,sa_1,\dots,sa_{n-2}),sa_{n-1})
    \end{split}
\end{equation}
for each $n\in\NN^*$, $\bar{x}=(x_1,\dots,x_{n})\in \MO^n$, $\sigma=(1 2\dots n)$ with $a_i\in{}_{x_i}\MA_{x_{i+1}}$ for $i\in\llbracket 1, n-1\rrbracket$ and $a_n\in{}_{x_n}\MA_{x_1}$.
An \textbf{\textcolor{ultramarine}{almost cyclic}} $A_{\infty}$-category is an $A_{\infty}$-category whose $A_{\infty}$-structure is almost cyclic with respect to a fixed homogeneous bilinear form.  
 An almost cyclic $A_{\infty}$-category with respect to a nondegenerate bilinear form is called a \textbf{\textcolor{ultramarine}{cyclic}} $A_{\infty}$-category.
\end{definition}

We end this section with the definitions of (cyclic) $A_{\infty}$-morphisms. The following definition is due to M. Sugawara (see \cite{sugawara}).

\begin{definition}
\label{def:A-inf-morphism}
An $A_{\infty}$\textbf{\textcolor{ultramarine}{-morphism}} between $A_{\infty}$-categories $(\MA,sm_{\MA})$ and $(\MB,sm_{\MB})$ with respective sets of objects $\MO_{\MA}$ and $\MO_{\MB}$ is a map $F_0 : \MO_{\MA}\rightarrow \MO_{\MB}$ together with a degree $0$ element $s\mathbf{f}=(sf^{\bar{x}})_{\bar{x}\in\bar{\MO}_{\MA}}$ of $C_{F_0}(\MA,\MB)[1]$ that satisfies
\begin{equation}
\label{eq:stasheff-morphisms}
\tag{$\operatorname{MI}^{\bar{x}}$}
\begin{split}
    \sum\limits_{1\leq i<j\leq \llg(\bar{x})}sf^{\bar{x}_{\leq i}\sqcup\bar{x}_{\geq j}}&\circ(\id^{\otimes \bar{x}_{\leq i}}\otimes sm_{\MA}^{\bar{x}_ {\llbracket i,j\rrbracket}}\otimes \id^{\otimes \bar{x}_{\geq j}})\\&=\sum\limits_{1\leq i_1<\dots<i_n\leq \llg(\bar{x})}sm_{\MB}^{\scriptsize{\xoverline{F_0(x)}}}(sf^{\bar{x}_{\leq i_1}}\otimes sf^{\bar{x}_{\llbracket i_1,i_2\rrbracket}}\otimes\dots\otimes sf^{\bar{x}_{\llbracket i_n,\llg(\bar{x})\rrbracket}})
    \end{split}
\end{equation}
for every $\bar{x}\in\bar{\MO}_{\MA}$. This is tantamount to 
\[
s\mathbf{f}\upperset{G}{\circ} sm_{\MA}=sm_{\MB}\upperset{M}{\circ} s\mathbf{f}
\]
where the composition $(sm_{\MB}\upperset{M}{\circ}s\mathbf{f})^{\bar{x}}$ is the right hand side of the identity \eqref{eq:stasheff-morphisms}. 
Note that given $\bar{x}\in\MO_{\MA}$ we have that $s\mathbf{f}\upperset{G}{\circ} sm_{\MA}=\sum\mathcal{E}(\mathcalboondox{D})$ and $sm_{\MB}\upperset{M}{\circ} s\mathbf{f}=\sum\mathcal{E}(\mathcalboondox{D'})$ where the sums are over all the filled diagrams $\mathcalboondox{D}$ and $\mathcalboondox{D'}$ of type $\bar{x}$ and of the form 

\begin{minipage}{21cm}
    \begin{tikzpicture}[line cap=round,line join=round,x=1.0cm,y=1.0cm]
\clip(-3.5,-2) rectangle (4,1);
    \draw [line width=0.5pt] (0.,0.) circle (0.5cm);
    \shadedraw [rotate=180, shift={(0.5cm,0cm)}] \doublefleche;
    \draw [line width=0.5pt] (1.5,0.) circle (0.5cm);
    \shadedraw[shift={(1cm,0cm)},rotate=180] \doubleflechescindeeleft;
    \shadedraw[shift={(1cm,0cm)},rotate=180] \doubleflechescindeeright;
    \shadedraw[ shift={(1cm,0cm)},rotate=180] \fleche;
    \draw [line width=1.1pt,rotate around={0:(1.5,0)}] [->,, >= stealth, >= stealth](2,0) -- (2.4,0);
    \draw (0,0.2) node[anchor=north] {$m_{\MA}$};
\draw (1.5,0.25) node[anchor=north] {$\mathbf{f}$};
\draw (3.5,0.25) node[anchor=north] {and};
\end{tikzpicture}
\begin{tikzpicture}[line cap=round,line join=round,x=1.0cm,y=1.0cm]
\clip(-2,-2) rectangle (4.30289746063327,1.5894004481482595);
   \draw [line width=0.5pt] (0.,0.) circle (0.5cm);
    \draw [line width=1.1pt,rotate=0] [->, >= stealth, >= stealth] (0.5,0) -- (0.9,0);
    \draw [rotate=-90] [<-, >= stealth, >= stealth](0.5,0) -- (0.9,0);
    \draw [rotate=180] [<-, >= stealth, >= stealth](0.5,0) -- (0.9,0);
    \draw [rotate=135] [<-, >= stealth, >= stealth](0.5,0) -- (0.9,0);
    \draw [line width=0.5pt] (0.,-1.2) circle (0.3cm);
    \draw[rotate around={-90:(0,-1.2)},shift={(0.3,-1.2)}]\doublefleche;
    \draw [line width=0.5pt] (-1.2,0) circle (0.3cm);
     \draw[rotate around={180:(-1.2,0)},shift={(-0.9,0)}]\doublefleche;
    \draw [line width=0.5pt] (-0.84,0.84) circle (0.3cm);
    \draw[rotate around={135:(-0.84,0.84)},shift={(-0.54,0.84)}]\doublefleche;
\begin{scriptsize}
\draw [fill=black] (-0.45,-0.4) circle (0.3pt);
\draw [fill=black] (-0.3,-0.5) circle (0.3pt);
\draw [fill=black] (-0.5,-0.25) circle (0.3pt);
\end{scriptsize}
  \draw (0,0.2) node[anchor=north] {$m_{\MB}$};
\draw (-1.2,0.25) node[anchor=north] {$\mathbf{f}$};
\draw (0,-0.95) node[anchor=north] {$\mathbf{f}$};
\draw (-0.84,1.09) node[anchor=north] {$\mathbf{f}$};
\end{tikzpicture}

\end{minipage}

\noindent respectively.
\end{definition}

The following definition was introduced in \cite{kajiura} by H. Kajiura in the case of cyclic $A_{\infty}$-categories.
\begin{definition}
\label{def:cyclic-morphism}
An $A_{\infty}$-morphism $(F_0,s\mathbf{f})$ between almost cyclic $A_{\infty}$-categories $(\MA,sm_{\MA})$ and $(\MB,sm_{\MB})$ with respect to bilinear forms $\gamma$ and $\Gamma$ is \textbf{\textcolor{ultramarine}{cyclic}} if 
\begin{equation}
\label{eq:cyclicity-morphism-1}
        {}_{F_0(x)}\Gamma_{F_0(y)}(f^{x,y}(sa),f^{y,x}(sb))={}_{x}\gamma_{y}(sa,sb)
\end{equation}
for $x,y\in\MO$, $a\in{}_{x}\MA_{y}$ and $b\in{}_{y}\MA_{x}$
and for $n\geq 3$
\begin{equation}
\label{eq:cyclicity-morphism-2}
        \sum\limits_{\substack{\bar{x}\in Z\\ \bar{y}\in Z'}}{}_{\llt(\bar{z})}\Gamma_{\rrt(\bar{z})}(F^{\bar{x}}(sa_1,\dots,sa_i),F^{\bar{y}}(sa_{i+1},\dots,sa_n))=0
\end{equation}
for $\bar{z}\in\bar{\MO}$, $(sa_1,\dots,sa_i)\in \MA[1]^{\otimes\bar{x}}$ and $(sa_{i+1},\dots,sa_n)\in\MA[1]^{\otimes\bar{y}}$ where 
\begin{equation}
    \begin{split}
        Z=\{\bar{x}\in\bar{\MO}\;|\; \llt(\bar{x})=\llt(\bar{z}), \rrt(\bar{x})=\rrt(\bar{z})\}\;\text{and}\;
        Z'=\{\bar{y}\in\bar{\MO}\;|\; \llt(\bar{y})=\rrt(\bar{z}), \rrt(\bar{y})=\llt(\bar{z})\}.
    \end{split}
\end{equation}
\end{definition}
\section{Pre-Calabi-Yau categories} \label{pcy-cat}
In this section, we present the diagrammatic calculus and recall the definition of $d$-pre-Calabi-Yau structures, $d\in\ZZ$, appearing in \cite{ktv} and \cite{yeung} as well as their relation with $A_{\infty}$-structures when the graded quiver considered is $\Hom$-finite.
\subsection{Diagrammatic calculus}
In this subsection, we define discs and diagrams and we explain how to evaluate and compose them.
In order to give the definition of a $d$-pre-Calabi-Yau structure on a graded quiver $\MA$, we consider the following graded vector space, called the higher Hochschild complex in \cite{ktv} and the multidual in \cite{yeung}.
\begin{definition}
\label{def:multidual}
Given a graded quiver $\MA$ with set of objects $\MO$, we define the graded vector space 
\[
\Multi^{\bullet}_d(\MA)=\prod\limits_{n\in\NN^*}\Multi^{n}_d(\MA) = \prod\limits_{n\in\NN^*}\prod\limits_{\doubar{x}\in\bar{\MO}^n} \Multi^{\doubar{x}}_d(\MA)\]
where $\Multi^{\doubar{x}}_d(\MA)$ is the graded vector space consisting of sums of homogeneous $\kk$-linear maps of the form
\begin{equation}
\begin{split}
    \label{eq:element-multidual-Bar}
    F^{\doubar{x}}=F^{\bar{x}^1,\dots,\bar{x}^n} &: \MA[1]^{\otimes \bar{x}^1}\otimes\MA[1]^{\otimes\bar{x}^2}\otimes \dots \otimes \MA[1]^{\otimes \bar{x}^n} \\&\hskip2cm\rightarrow  {}_{\llt(\bar{x}^{1})}\MA_{\rrt(\bar{x}^2)}[-d]\otimes{}_{\llt(\bar{x}^{2})}\MA_{\rrt(\bar{x}^{3})}[-d]\otimes\dots\otimes {}_{\llt(\bar{x}^n)}\MA_{\rrt(\bar{x}^1)}[-d]
\end{split}
\end{equation}
for $\doubar{x}=(\bar{x}^1,\dots,\bar{x}^n)\in\bar{\MO}^n$.

The action of $\sigma=(\sigma_n)_{n\in\NN^*}\in\prod_{n\in\NN^*} C_n$ on an element $\mathbf{F} = (F^{\doubar{x}})_{\doubar{x} \in \doubar{\MO}}\in\Multi^{\bullet}_d(\MA)$ is the element $\sigma\cdot\mathbf{F}\in\Multi^{\bullet}_d(\MA)$ given by 
\[(\sigma\cdot\mathbf{F})^{\doubar{x}}=\tau^{\sigma_n}_{{}_{\llt(\bar{x}^{1})}\MA_{\rrt(\bar{x}^2)}[-d],{}_{\llt(\bar{x}^{2})}\MA_{\rrt(\bar{x}^{3})}[-d],\dots, {}_{\llt(\bar{x}^n)}\MA_{\rrt(\bar{x}^1)}[-d]}\circ F^{\doubar{x}\cdot\sigma_n} \circ \tau^{\sigma_n}_{\MA[1]^{\otimes \bar{x}^1},\MA[1]^{\otimes\bar{x}^2}, \dots ,\MA[1]^{\otimes \bar{x}^n}}
\] 
for $\doubar{x}=(\bar{x}^1,\dots,\bar{x}^n)\in\bar{\MO}^n$.
We will denote by $\Multi^{\bullet}_d(\MA)^{C_{\llg(\bullet)}}$ the space of elements of $\Multi^{\bullet}_d(\MA)$ that are invariant under the action of $\prod_{n\in\NN^*} C_n$.
\end{definition}

The aim of this section is to represent elements of (a shifted version of) the graded vector space $\Multi^{\bullet}_d(\MA)^{C_{\llg(\bullet)}}$ by discs.

\begin{definition}
    A \textbf{\textcolor{ultramarine}{disc}} $D$ is a circle with two distinguished sets of points $I$ and $O$ such that $|O|\geq 1$, $|I|\geq 0$ if $|O|>1$ and $|I|>0$ if $|O|=1$. The points in $I$ (resp. $O$) are called \textbf{\textcolor{ultramarine}{incoming}} (resp. \textbf{\textcolor{ultramarine}{outgoing}}).
    An incoming (resp. outgoing) point will be pictured as an incoming (resp. outgoing) arrow (see Figure \ref{fig:disc}).
    The \textbf{\textcolor{ultramarine}{size}} of the disc $D$ is the number of outgoing arrows, and it will be denoted by $|D|$. 
\end{definition}

\begin{figure}[H]
    \centering
   \begin{tikzpicture}
     \draw [line width=0.5pt] (0.,0.) circle (0.5cm);
     \draw [rotate=120] [->,> = stealth] (0.5,0) -- (0.9,0);
     \draw [rotate=-120] [->,> = stealth] (0.5,0) -- (0.9,0);
     \draw [rotate=0] [->,> = stealth] (0.5,0) -- (0.9,0);
     \draw [rotate=30] [<-,> = stealth] (0.5,0) -- (0.9,0);
     \draw [rotate=60] [<-,> = stealth] (0.5,0) -- (0.9,0);
     \draw [rotate=90] [<-,> = stealth] (0.5,0) -- (0.9,0);
     \draw [rotate=-40] [<-,> = stealth] (0.5,0) -- (0.9,0);
     \draw [rotate=-80] [<-,> = stealth] (0.5,0) -- (0.9,0);
     \draw [rotate=-160] [<-,> = stealth] (0.5,0) -- (0.9,0);
     \draw [rotate=-200] [<-,> = stealth] (0.5,0) -- (0.9,0);
\end{tikzpicture}
    \caption{A disc of size 3.}
    \label{fig:disc}
\end{figure}

\begin{definition} 
\label{definition:decorated-disc}
Given a graded quiver $\MA$ with set of objects $\MO$, a \textbf{\textcolor{ultramarine}{decorated disc of size $n$}} is a disc of size $n$ together with a clockwise labeling of the outgoing arrows from $1$ to $n$ and a distinguished object of $\MO$ between any couple of consecutive arrows. 
Given an arrow $\alpha$ of the disc $D$, we will write it $\alpha = {}_{y}\alpha_{x}$ with $x, y \in \MO$ to indicate that $x$ clockwise precedes the arrow $\alpha$ in $D$ and $\alpha$ clockwise precedes $y$ in $D$. 
For instance, the outgoing arrow labeled by $3$ in Figure \ref{fig:decorated-disc} might be denoted by ${}_{x_{3}^{1}}\alpha_{x_{1}^{3}}$. 

The \textbf{\textcolor{ultramarine}{type}} of the decorated disc is the tuple of the form $(\bar{x}^1,\dots,\bar{x}^n)$ where  $\bar{x}^i$ is the tuple formed by objects of $\MO$, read in counterclockwise order, between the outgoing arrows $i-1$ and $i$, with the convention that the arrow $0$ is the arrow $n$ (see Figure \ref{fig:decorated-disc}). 
\end{definition}

\begin{figure}[H]
    \centering
\begin{small}
 \begin{tikzpicture}
     \draw [line width=0.5pt] (0.,0.) circle (0.5cm);
     \draw [rotate=120] [->,> = stealth] (0.5,0) -- (0.9,0);
     \draw [rotate=-120] [->,> = stealth] (0.5,0) -- (0.9,0);
     \draw [rotate=0] [->,> = stealth] (0.5,0) -- (0.9,0);
     \draw [rotate=30] [<-,> = stealth] (0.5,0) -- (0.9,0);
     \draw [rotate=60] [<-,> = stealth] (0.5,0) -- (0.9,0);
     \draw [rotate=90] [<-,> = stealth] (0.5,0) -- (0.9,0);
     \draw [rotate=-40] [<-,> = stealth] (0.5,0) -- (0.9,0);
     \draw [rotate=-80] [<-,> = stealth] (0.5,0) -- (0.9,0);
     \draw [rotate=-160] [<-,> = stealth] (0.5,0) -- (0.9,0);
     \draw [rotate=-200] [<-,> = stealth] (0.5,0) -- (0.9,0);
\draw (-0.65,1.2) node[anchor=north west]{$\scriptstyle{1}$};
\draw (0.9,0.25) node[anchor=north west]{$\scriptstyle{2}$};
\draw (-0.65,-0.75) node[anchor=north west]{$\scriptstyle{3}$};
\draw (-0.9,-0.2) node[anchor=north west] {$\scriptstyle{x_3^1}$};
\draw (-1.05,0.3) node[anchor=north west] {$\scriptstyle{x_2^1}$};
\draw (-0.95,0.8143778265370197) node[anchor=north west] {$\scriptstyle{x_1^1}$};
\draw (-0.05,1.06853467879074) node[anchor=north west] {$\scriptstyle{x_3^2}$};
\draw (0.25,0.8325318874122855) node[anchor=north west] {$\scriptstyle{x_2^2}$};
\draw (-0.45,1.05) node[anchor=north west] {$\scriptstyle{x_4^2}$};
\draw (0.5,0.53) node[anchor=north west] {$\scriptstyle{x_1^2}$};
\draw (0.4,-0.02070897372520403) node[anchor=north west] {$\scriptstyle{x_3^3}$};
\draw (0.1,-0.4) node[anchor=north west] {$\scriptstyle{x_2^3}$};
\draw (-0.4,-0.5) node[anchor=north west] {$\scriptstyle{x_1^3}$};
\end{tikzpicture}
\end{small}
        \caption{A decorated disc of type $\doubar{x}=(\bar{x}^1,\bar{x}^2,\bar{x}^3)$ where $\bar{x}^1=(x_1^1,x_2^1,x_3^1)$, $\bar{x}^2=(x_1^2,x_2^2,x_3^2,x_4^2)$ and $\bar{x}^3=(x_1^3,x_2^3,x_3^3)$.}
    \label{fig:decorated-disc}
\end{figure}

\begin{definition}
    Given a graded quiver $\MA$ with set of objects $\MO$ and a tuple $\doubar{x}\in\doubar{\MO}$ we associate to a map $F^{\doubar{x}}\in\Multi^{\doubar{x}}_d(\MA)$ the unique decorated disc of type $\doubar{x}$.
\end{definition}

\begin{definition}
    A \textbf{\textcolor{ultramarine}{marked}} disc is a decorated disc with a bold arrow (see Figure \ref{fig:marked-disc}).
\end{definition}

\begin{figure}[H]
    \centering
    \begin{small}
    \begin{tikzpicture}
     \draw [line width=0.5pt] (0.,0.) circle (0.5cm);
     \draw [rotate=120] [->,> = stealth] (0.5,0) -- (0.9,0);
     \draw [line width=1.1pt, rotate=-120] [->,> = stealth] (0.5,0) -- (0.9,0);
     \draw [rotate=0] [->,> = stealth] (0.5,0) -- (0.9,0);
     \draw [rotate=30] [<-,> = stealth] (0.5,0) -- (0.9,0);
     \draw [rotate=60] [<-,> = stealth] (0.5,0) -- (0.9,0);
     \draw [rotate=90] [<-,> = stealth] (0.5,0) -- (0.9,0);
     \draw [rotate=-40] [<-,> = stealth] (0.5,0) -- (0.9,0);
     \draw [rotate=-80] [<-,> = stealth] (0.5,0) -- (0.9,0);
     \draw [rotate=-160] [<-,> = stealth] (0.5,0) -- (0.9,0);
     \draw [rotate=-200] [<-,> = stealth] (0.5,0) -- (0.9,0);
\draw (-0.65,1.2) node[anchor=north west]{$\scriptstyle{1}$};
\draw (0.9,0.25) node[anchor=north west]{$\scriptstyle{2}$};
\draw (-0.65,-0.75) node[anchor=north west]{$\scriptstyle{3}$};
\draw (-0.9,-0.2) node[anchor=north west] {$\scriptstyle{x_3^1}$};
\draw (-1.05,0.3) node[anchor=north west] {$\scriptstyle{x_2^1}$};
\draw (-0.95,0.8143778265370197) node[anchor=north west] {$\scriptstyle{x_1^1}$};
\draw (-0.05,1.06853467879074) node[anchor=north west] {$\scriptstyle{x_3^2}$};
\draw (0.25,0.8325318874122855) node[anchor=north west] {$\scriptstyle{x_2^2}$};
\draw (-0.45,1.05) node[anchor=north west] {$\scriptstyle{x_4^2}$};
\draw (0.5,0.53) node[anchor=north west] {$\scriptstyle{x_1^2}$};
\draw (0.4,-0.02070897372520403) node[anchor=north west] {$\scriptstyle{x_3^3}$};
\draw (0.1,-0.4) node[anchor=north west] {$\scriptstyle{x_2^3}$};
\draw (-0.4,-0.5) node[anchor=north west] {$\scriptstyle{x_1^3}$};
\end{tikzpicture}
\end{small}
    \caption{A marked disc.}
    \label{fig:marked-disc}
\end{figure}

\begin{definition}
    Given a graded quiver $\MA$ and an element $s_{d+1}\mathbf{F}\in\Multi_d^{\bullet}(\MA)[d+1]$, a disc \textbf{\textcolor{ultramarine}{filled}} with $\mathbf{F}$ is a marked disc with the letter $\mathbf{F}$ in his center (see Figure \ref{fig:marked-disc}). Such a disc is called a \textbf{\textcolor{ultramarine}{filled}} disc.
\end{definition}

\begin{figure}[H]
    \centering
    \begin{small}
    \begin{tikzpicture}
     \draw [line width=0.5pt] (0.,0.) circle (0.5cm);
     \draw [rotate=120] [->,> = stealth] (0.5,0) -- (0.9,0);
     \draw [line width=1.1pt, rotate=-120] [->,> = stealth] (0.5,0) -- (0.9,0);
     \draw [rotate=0] [->,> = stealth] (0.5,0) -- (0.9,0);
     \draw [rotate=30] [<-,> = stealth] (0.5,0) -- (0.9,0);
     \draw [rotate=60] [<-,> = stealth] (0.5,0) -- (0.9,0);
     \draw [rotate=90] [<-,> = stealth] (0.5,0) -- (0.9,0);
     \draw [rotate=-40] [<-,> = stealth] (0.5,0) -- (0.9,0);
     \draw [rotate=-80] [<-,> = stealth] (0.5,0) -- (0.9,0);
     \draw [rotate=-160] [<-,> = stealth] (0.5,0) -- (0.9,0);
     \draw [rotate=-200] [<-,> = stealth] (0.5,0) -- (0.9,0);
\draw (-0.65,1.2) node[anchor=north west]{$\scriptstyle{1}$};
\draw (0.9,0.25) node[anchor=north west]{$\scriptstyle{2}$};
\draw (-0.65,-0.75) node[anchor=north west]{$\scriptstyle{3}$};
\draw (-0.9,-0.2) node[anchor=north west] {$\scriptstyle{x_3^1}$};
\draw (-1.05,0.3) node[anchor=north west] {$\scriptstyle{x_2^1}$};
\draw (-0.95,0.8143778265370197) node[anchor=north west] {$\scriptstyle{x_1^1}$};
\draw (-0.05,1.06853467879074) node[anchor=north west] {$\scriptstyle{x_3^2}$};
\draw (0.25,0.8325318874122855) node[anchor=north west] {$\scriptstyle{x_2^2}$};
\draw (-0.45,1.05) node[anchor=north west] {$\scriptstyle{x_4^2}$};
\draw (0.5,0.53) node[anchor=north west] {$\scriptstyle{x_1^2}$};
\draw (0.4,-0.02070897372520403) node[anchor=north west] {$\scriptstyle{x_3^3}$};
\draw (0.1,-0.4) node[anchor=north west] {$\scriptstyle{x_2^3}$};
\draw (-0.4,-0.5) node[anchor=north west] {$\scriptstyle{x_1^3}$};
\draw (0,0.25) node[anchor=north]{$\mathbf{F}$};
\end{tikzpicture}
\end{small}
    \caption{A filled disc.}
    \label{fig:filled-disc}
\end{figure}

To simplify, we will omit the objects when drawing a marked disc and we will draw a large incoming arrow instead of several consecutive incoming arrows (see Figure \ref{fig:simplified-disc}). Moreover, we also omit the label of the outgoing arrows and define the last outgoing arrow as the bold outgoing arrow if it is outgoing and as the outgoing arrow preceding clockwise the bold arrow if it is incoming. This convention comes from the fact that we will always apply the shift either on the last tensor factor of the result of applying a map on elements or on one of the first elements in the rest of the thesis.

\begin{figure}[H]
\centering
\begin{tikzpicture}
     \draw [line width=0.5pt] (0.,0.) circle (0.5cm);
     \draw [line width=1.1pt,rotate=-120] [->, >= stealth] (0.5,0) -- (0.9,0);
     \draw [rotate=0] [->, >= stealth] (0.5,0) -- (0.9,0);
     \draw [rotate=120] [->, >= stealth] (0.5,0) -- (0.9,0);
     \shadedraw[rotate=180,shift={(0.5,0)}] \doublefleche;
    \shadedraw[rotate=60,shift={(0.5,0)}] \doublefleche;
    \shadedraw[rotate=-60,shift={(0.5,0)}] \doublefleche;
\end{tikzpicture}
    \caption{A marked disc.}
    \label{fig:simplified-disc}
\end{figure}

\begin{definition}
    Consider a graded quiver $\MA$ with set of objects $\MO$ and $s_{d+1}\mathbf{F}$ of $\Multi_d^{\bullet}(\MA)[d+1]$. 
    Given $\doubar{x} = (\bar{x}^{1}, \dots, \bar{x}^{n}) \in \bar{\MO}^{n}$ and 
    \[(a,b) \in (\{ i\} \times \llbracket 1 , \llg(\bar{x}^{1}) + \dots + \llg(\bar{x}^{n})\rrbracket) \sqcup (\{ o\} \times \llbracket 1 , n\rrbracket)\],
    by the isomorphisms \eqref{eq:iso-shift-tensor-prod} and \eqref{eq:iso-shift-input}, $s_{d+1}F^{\doubar{x}}$ induces a morphism  of the form 
     \begin{small}
    \begin{equation}
    \label{eq:eval-1}
     \begin{split}
     &\MA[1]^{\otimes \bar{x}^1}\otimes \dots\otimes \MA[1]^{\otimes \bar{x}^{j-1}}\otimes \MA[1]^{\otimes \bar{x}^j_{\leq b'}}\otimes {}_{x^j_{b'}}\MA_{x^j_{b'+1}}[-d]\otimes \MA[1]^{\otimes \bar{x}^j_{>b'}} \otimes \MA[1]^{\otimes \bar{x}^n} 
    \\& \hspace{8cm}\rightarrow  {}_{\llt(\bar{x}^{1})}\MA_{\rrt(\bar{x}^2)}[-d]\otimes\dots\otimes {}_{\llt(\bar{x}^n)}\MA_{\rrt(\bar{x}^1)}[-d]
    \end{split}
     \end{equation}
       \end{small}
       given by 
       \begin{equation}
           (-1)^{(d+1)|\mathbf{F}|}F^{\doubar{x}}\circ (\id^{\otimes (\llg(\bar{x}^1)+\dots+\llg(\bar{x}^{j-1})+b'-j)}\otimes s_{-d-1}\otimes \id^{\otimes (\llg(\bar{x}^{j})-b-1'+\llg(\bar{x}^{j+1})+\dots+\llg(\bar{x}^n)-n+j)})
       \end{equation}
    if $a = i$ and $b=\llg(\bar{x}^1)+\dots+\llg(\bar{x}^j)+b'$ with $j\in\llbracket 1,n\rrbracket$, $b'\in \llbracket 1,\llg(\bar{x}^j)-1\rrbracket$. We associate to it the marked disc filled with $\mathbf{F}$ and decorated as the one of $F^{\doubar{x}}$ whose $b'$-th incoming arrow between outgoing arrows $j-1$ and $j$ is a bold arrow (see Figure \ref{fig:shifted-input}).
    
    \begin{figure}[H]
    \centering
\begin{tikzpicture}
     \draw [line width=0.5pt] (0.,0.) circle (0.5cm);
     \draw [rotate=0] [->, >= stealth] (0.5,0) -- (0.9,0);
     \draw [rotate=-90] [->, >= stealth] (0.5,0) -- (0.9,0);
     \draw [rotate=90] [->, >= stealth] (0.5,0) -- (0.9,0);
     \draw [rotate=180] [->, >= stealth] (0.5,0) -- (0.9,0);
     \shadedraw[rotate=45,shift={(0.5,0)}] \doublefleche;
    \shadedraw[rotate=-45,shift={(0.5,0)}] \doublefleche;
     \shadedraw[rotate=-135,shift={(0.5,0)}] \doubleflechescindeeleft;
     \shadedraw[rotate=-135,shift={(0.5,0)}] \doubleflechescindeeright;
     \shadedraw[line width=1.1pt, rotate=-135,shift={(0.5,0)}] \fleche;
\begin{scriptsize}
\draw [fill=black] (-0.45,0.4) circle (0.3pt);
\draw [fill=black] (-0.25,0.55) circle (0.3pt);
\draw [fill=black] (-0.55,0.2) circle (0.3pt);
\end{scriptsize}
\draw (0,0.25) node[anchor=north]{$\mathbf{F}$};
\end{tikzpicture}
    \caption{A disc representing a map of the form \eqref{eq:eval-1}.}
    \label{fig:shifted-input}
\end{figure}
    
    Moreover, by the isomorphisms \eqref{eq:iso-shift-tensor-prod} and \eqref{eq:iso-shift-output} $s_{d+1}F^{\doubar{x}}$ induces a morphism of the form
    \begin{small}
    \begin{equation}
    \label{eq:eval-2}
   \begin{split}   
     &\MA[1]^{\otimes \bar{x}^1}\otimes\MA[1]^{\otimes\bar{x}^2}\otimes \dots \otimes \MA[1]^{\otimes \bar{x}^n} 
    \\
    &\hspace{1.7cm}\rightarrow  {}_{\llt(\bar{x}^{1})}\MA_{\rrt(\bar{x}^2)}[-d]\otimes\dots\otimes{}_{\llt(\bar{x}^{b-1})}\MA_{\rrt(\bar{x}^{b})}[-d]\otimes {}_{\llt(\bar{x}^{b})}\MA_{\rrt(\bar{x}^{b+1})}[1]\otimes\dots\otimes {}_{\llt(\bar{x}^n)}\MA_{\rrt(\bar{x}^1)}[-d]
    \end{split}
     \end{equation}
       \end{small}
       given by $(\id^{\otimes (b-1)} \otimes s_{d+1} \otimes \id^{\otimes (n-b)}) \circ F^{\doubar{x}}$
    if $a = o$. We associate to it the marked disc filled with $\mathbf{F}$ and decorated as the one of $F^{\doubar{x}}$ where the $b$-th outgoing arrow is a bold arrow (see Figure \ref{fig:shifted-output}).

\begin{figure}[H]
    \centering
\begin{tikzpicture}
     \draw [line width=0.5pt] (0.,0.) circle (0.5cm);
     \draw [line width=1.1pt,rotate=0] [->, >= stealth] (0.5,0) -- (0.9,0);
     \draw [rotate=-90] [->, >= stealth] (0.5,0) -- (0.9,0);
     \draw [rotate=90] [->, >= stealth] (0.5,0) -- (0.9,0);
     \draw [rotate=180] [->, >= stealth] (0.5,0) -- (0.9,0);
     \shadedraw[rotate=45,shift={(0.5,0)}] \doublefleche;
    \shadedraw[rotate=-45,shift={(0.5,0)}] \doublefleche;
    \shadedraw[rotate=-135,shift={(0.5,0)}] \doublefleche;
\begin{scriptsize}
\draw [fill=black] (-0.45,0.4) circle (0.3pt);
\draw [fill=black] (-0.25,0.55) circle (0.3pt);
\draw [fill=black] (-0.55,0.2) circle (0.3pt);
\end{scriptsize}
\draw (0,0.25) node[anchor=north]{$\mathbf{F}$};
\end{tikzpicture}
    \caption{A disc representing a map of the form \eqref{eq:eval-2}.}
    \label{fig:shifted-output}
\end{figure}

Given a disc $D$ of type $\doubar{x}$ filled with $\mathbf{F}$, we denote by $\mathcal{E}(D)$ the map \eqref{eq:eval-1} multiplied by $(-1)^{(|\mathbf{F}|+d+1)(d+1)}$ if the bold arrow is incoming and the map \eqref{eq:eval-2} if the bold arrow is outgoing.
\end{definition}

\begin{definition}
    Consider a graded quiver $\MA$ and an element $s_{d+1}\mathbf{F}\in\Multi_d^{\bullet}(\MA)[d+1]$. Given a marked disc $D$, we denote by $\mathcal{E}(D,s_{d+1}\mathbf{F})$ the map $\mathcal{E}(D')$ where $D'$ is the disc $D$ filled with $\mathbf{F}$.
\end{definition}

\begin{example}
    Given the filled disc $D$ in Figure \ref{fig:filled-disc} the evaluation of the map $\mathcal{E}(D)$    
    at elements $(\bar{sa}^1,\bar{sa}^2,\bar{sa}^3)$ is obtained by first compute 
    \[F^{\doubar{x}}(\bar{sa}^1,\bar{sa}^2,\bar{sa}^3)={}_{\llt(\bar{x}^1)}F^{-d}_{\rrt(\bar{x}^2)}\otimes {}_{\llt(\bar{x}^2)}F^{-d}_{\rrt(\bar{x}^3)}\otimes {}_{\llt(\bar{x}^3)}F^{-d}_{\rrt(\bar{x}^1)}\] and then apply the shift $s_{d+1}$ to the third tensor factor of the result since the bold arrow is the third outgoing arrow of the disc. The result is then 
    \[
    (-1)^{\epsilon}{}_{\llt(\bar{x}^1)}F^{-d}_{\rrt(\bar{x}^2)}\otimes {}_{\llt(\bar{x}^2)}F^{-d}_{\rrt(\bar{x}^3)}\otimes {}_{\llt(\bar{x}^3)}F^1_{\rrt(\bar{x}^1)}
    \]
    with $\epsilon=(d+1)(|{}_{\llt(\bar{x}^1)}F^{-d}_{\rrt(\bar{x}^2)}|+|{}_{\llt(\bar{x}^2)}F^{-d}_{\rrt(\bar{x}^3)}|)$ and where ${}_{\llt(\Bar{x}^3)}F^1_{\rrt(\Bar{x}^1)}=s_{d+1}{}_{\llt(\Bar{x}^3)}F^{-d}_{\rrt(\Bar{x}^1)}$. The indices $\llt(\bar{x}^i)$ and $\rrt(\bar{x}^{i+1})$ of $F$ indicate that the element is in a shift of ${}_{\llt(\bar{x}^i)}\MA_{\rrt(\bar{x}^{i+1})}$ whereas the super script $-d$ or $1$ indicate the shift.
\end{example}

When working with elements of $\Multi_d^{\bullet}(\MA)^{C_{\bullet}}[d+1]$, we will usually sum over all the possible outgoing positions of the bold arrow in order to obtain cyclically invariant elements.

\begin{definition}
A \textbf{\textcolor{ultramarine}{diagram}} $(D,R)$ is the data of a collection $D$ of marked discs $D_1,\dots,D_n$, each of which having sets of arrows $A_i$ for $i \in \llbracket 1 , n \rrbracket$, together with a subset $R \subseteq (\sqcup _{i=1}^n A_i)^2$ satisfying that
\begin{enumerate}[label=(D.\arabic*)]
\item if $(\alpha,\beta)\in R\cap (A_i\times A_j)$, then $i\neq j$ and $\alpha$ is an incoming arrow and $\beta$ is an outgoing arrow;
\item  if $({}_{y'}\alpha_{x'},{}_{y}\beta_{x}) \in R$, then $x' = y$ and $x = y'$; 
\end{enumerate}
where we use the notation introduced in Definition \ref{definition:decorated-disc}. 
We represent a pair of arrows $(\alpha,\beta) \in R \cap (A_i\times A_j)$ by connecting the outgoing arrow $\beta$ of $D_{j}$ with the incoming arrow $\alpha$ of $D_{i}$ (see Figure \ref{fig:composed-diag}), and we will say that the disc $D_{i}$ 
\textbf{\textcolor{ultramarine}{shares an arrow}} with $D_{j}$ or that $\alpha$ and $\beta$ are \textbf{\textcolor{ultramarine}{connected}}. 

An \textbf{\textcolor{ultramarine}{incoming (resp. outgoing) arrow}} 
of the diagram $(D,R)$ is an arrow $\alpha$ of one of the discs $D_i$ such that there is no arrow $\beta$ satisfying that $(\alpha,\beta) \in R$ (resp. $(\beta,\alpha) \in R$). 
An outgoing arrow of a disc $D_i$ is \textbf{\textcolor{ultramarine}{free}} if it is an outgoing arrow of $(D,R)$.
A diagram $(D,R)$ has a distinguished object of $\MO$ between any couple of consecutive arrows of $(D,R)$ given by the decoration of the discs $D_{1}, \dots, D_{n}$.
Roughly speaking, a diagram is a finite collection of decorated discs, each of which shares at most an arrow with any other one (see Figure \ref{fig:diagram} for an example).
\end{definition}

\begin{figure}[H]
    \centering
 \begin{tikzpicture}
     \draw [line width=0.5pt] (0.,0.) circle (0.5cm);
     \draw [rotate=-120] [->, >= stealth] (0.5,0) -- (0.9,0);
     \draw [rotate=120] [->, >= stealth] (0.5,0) -- (0.9,0);
     \shadedraw[rotate=-60,shift={(0.5,0)}]\doublefleche;
    \shadedraw[rotate=-180,shift={(0.5,0)}] \doubleflechescindeeleft;
     \shadedraw[rotate=-180,shift={(0.5,0)}] \doubleflechescindeeright;
     \shadedraw[line width=1.1pt, rotate=-180,shift={(0.5,0)}] \fleche;
     \draw [line width=0.5pt] (-1.5,0.) circle (0.5cm);
     \draw[rotate around={-120:(-1.5,0)}][->, >= stealth] (-1,0)--(-0.6,0);
     \draw[rotate around={120:(-1.5,0)}][->, >= stealth] (-1,0)--(-0.6,0);
      \draw[rotate around={60:(-1.5,0)},shift={(-1,0)}]\doublefleche;
      \draw[rotate around={-60:(-1.5,0)},shift={(-1,0)}]\doublefleche;
      \draw[rotate around={180:(-1.5,0)},shift={(-1,0)}]\doublefleche;
     \shadedraw[rotate=60,shift={(0.5,0)}] \doubleflechescindeeleft;
     \shadedraw[rotate=60,shift={(0.5,0)}] \doubleflechescindeeright;
     \shadedraw[line width=1.1pt, rotate=60,shift={(0.5,0)}] \fleche;
     \draw [line width=0.5pt] (0.75,1.3) circle (0.5cm);
     \draw[rotate around={0:(0.75,1.3)}][->, >= stealth](1.25,1.3)--(1.65,1.3);
     \draw[rotate around={120:(0.75,1.3)}][->, >= stealth](1.25,1.3)--(1.65,1.3);
     \draw[rotate around={60:(0.75,1.3)},shift={(1.25,1.3)}]\doublefleche;
     \draw[rotate around={-60:(0.75,1.3)},shift={(1.25,1.3)}]\doublefleche;
     \draw[rotate around={180:(0.75,1.3)},shift={(1.25,1.3)}]\doublefleche;
     \draw [line width=0.5pt] (1.5,0.) circle (0.5cm);
     \shadedraw[rotate around={-180:(1.5,0)},shift={(2,0)}] \doubleflechescindeeleft;
     \shadedraw[rotate around={-180:(1.5,0)},shift={(2,0)}] \doubleflechescindeeright;
     \shadedraw[line width=1.1pt, rotate around={-180:(1.5,0)},shift={(2,0)}] \fleche;
     \shadedraw[rotate around={0:(1.5,0)},shift={(2,0)}] \doublefleche;
     \draw[line width=1.1pt, rotate around={90:(1.5,0)}][->, >= stealth] (2,0)--(2.4,0);
     \draw[rotate around={-90:(1.5,0)}][->, >= stealth](2,0)--(2.4,0);
\end{tikzpicture}
 \caption{A diagram.}
    \label{fig:diagram}
\end{figure}

\begin{definition}
    We say that a diagram $(D=\{D_1,\dots,D_n\},R)$ is \textbf{\textcolor{ultramarine}{admissible}} if the following conditions are satisfied 
    \begin{enumerate}[label=(A.\arabic*)]
        \item\label{item:a1} $\forall i\in \llbracket 1,n\rrbracket, \exists\; \alpha_i\in D_i \; \text{such that}\; (\alpha_i,\beta_j)\in R$ for some $j\neq i$, $\beta_j\in D_j$; 
        \item for $(x,y)\in R$, either $x$ is a bold arrow or $y$ is;
        \item there is no family of arrows $\{x_1,\dots,x_k\}$ such that $(x_i,x_{i+1})\in R$ for all $i\in \llbracket 1,k \rrbracket$ and $x_{k}=x_1$. 
    \end{enumerate} 

    Roughly speaking, a diagram $(D=\{D_1,\dots,D_n\},R)$ is admissible if $D$ is connected, with no cycles and if every couple of arrows in $R$ contains a unique bold arrow.
    Note that there is precisely one disc in $D = \{ D_{1}, \dots, D_{n}\}$ whose bold arrow is also an incoming or outgoing arrow of the diagram $(D,R)$, \textit{i.e.} an admissible diagram $(D,R)$ always has either an incoming or outgoing bold arrow. 
\end{definition}

\begin{example}
The diagram in Figure \ref{fig:diagram} is admissible whereas the following is not admissible because of the cycle.
\begin{figure}[H]
    \centering
\begin{tikzpicture}[line cap=round,line join=round,x=1.0cm,y=1.0cm]
\clip(-3,-1) rectangle (2.5,2.2);
     \draw [line width=0.5pt] (0.,0.) circle (0.5cm);
     \draw [rotate=-120] [->, >= stealth] (0.5,0) -- (0.9,0);
     \draw [rotate=120] [->, >= stealth] (0.5,0) -- (0.9,0);
     \shadedraw[rotate=-60,shift={(0.5,0)}]\doublefleche;
    \shadedraw[rotate=-180,shift={(0.5,0)}] \doubleflechescindeeleft;
     \shadedraw[rotate=-180,shift={(0.5,0)}] \doubleflechescindeeright;
     \shadedraw[line width=1.1pt, rotate=-180,shift={(0.5,0)}] \fleche;
     \draw [line width=0.5pt] (-1.5,0.) circle (0.5cm);
     \draw[rotate around={-120:(-1.5,0)}][->, >= stealth] (-1,0)--(-0.6,0);
     \draw[rotate around={120:(-1.5,0)}][->, >= stealth] (-1,0)--(-0.6,0);
      \draw[rotate around={60:(-1.5,0)},shift={(-1,0)}]\doublefleche;
      \draw[rotate around={-60:(-1.5,0)},shift={(-1,0)}]\doublefleche;
      \draw[rotate around={180:(-1.5,0)},shift={(-1,0)}]\doublefleche;
     \shadedraw[rotate=60,shift={(0.5,0)}] \doubleflechescindeeleft;
     \shadedraw[rotate=60,shift={(0.5,0)}] \doubleflechescindeeright;
     \shadedraw[line width=1.1pt, rotate=60,shift={(0.5,0)}] \fleche;
     \draw [line width=0.5pt] (0.75,1.3) circle (0.5cm);
     \draw[line width=1.1pt, rotate around={0:(0.75,1.3)}][->, >= stealth](1.25,1.3)--(1.65,1.3);
     \draw[rotate around={120:(0.75,1.3)}][->, >= stealth](1.25,1.3)--(1.65,1.3);
     \draw[rotate around={60:(0.75,1.3)},shift={(1.25,1.3)}]\doublefleche;
     \draw[rotate around={-60:(0.75,1.3)},shift={(1.25,1.3)}]\doubleflechescindeeleft;
     \draw[rotate around={-60:(0.75,1.3)},shift={(1.25,1.3)}]\doubleflechescindeeright;
     \draw[rotate around={180:(0.75,1.3)},shift={(1.25,1.3)}]\doublefleche;
     \draw [line width=0.5pt] (1.5,0.) circle (0.5cm);
     \shadedraw[rotate around={-180:(1.5,0)},shift={(2,0)}] \doubleflechescindeeleft;
     \shadedraw[rotate around={-180:(1.5,0)},shift={(2,0)}] \doubleflechescindeeright;
     \shadedraw[line width=1.1pt, rotate around={-180:(1.5,0)},shift={(2,0)}] \fleche;
     \shadedraw[rotate around={60:(1.5,0)},shift={(2,0)}] \doublefleche;
     \shadedraw[rotate around={-60:(1.5,0)},shift={(2,0)}] \doublefleche;
     \draw[line width =1.1pt, rotate around={120:(1.5,0)}][->, >= stealth] (2,0)--(2.5,0);
     \draw[rotate around={-120:(1.5,0)}][->, >= stealth](2,0)--(2.4,0);
      \draw[rotate around={0:(1.5,0)}][->, >= stealth](2,0)--(2.4,0);
\end{tikzpicture}
    \caption{A diagram which is not admissible.}
    \label{fig:composed-diag}
\end{figure}
\end{example}

\begin{definition}
    The \textbf{\textcolor{ultramarine}{size}} of an admissible diagram $(D=\{D_1,\dots,D_n\},R)$ is denoted by $|D|$ and given by 
    $|D|=\sum_{i=1}^n|D_i|-n+1$. 
    We will relabel the outgoing arrows of $(D,R)$ in clockwise direction from $1$ to $|D|$, such that the outgoing arrow labeled by $|D|$ is precisely the bold arrow of $(D,R)$ if the latter arrow is outgoing, and it is the outgoing arrow preceding the bold arrow of $(D,R)$ clockwise if the bold arrow of $(D,R)$ is incoming. 
\end{definition}

\begin{example}
    Consider the following three filled discs $D_1$, $D_2$ and $D_3$:

\begin{minipage}{21cm}
    \begin{tikzpicture}[line cap=round,line join=round,x=1.0cm,y=1.0cm]
\clip(-3,-1) rectangle (2,1);
     \draw [line width=0.5pt] (0.,0.) circle (0.5cm);
     \draw [line width=1.1pt,rotate=90] [->, >= stealth] (0.5,0) -- (0.9,0);
     \draw [rotate=-90] [->, >= stealth] (0.5,0) -- (0.9,0);
     \draw [rotate=0] [->, >= stealth] (0.5,0) -- (0.9,0);
     \draw [rotate=180] [->, >= stealth] (0.5,0) -- (0.9,0);
     \shadedraw[rotate=45,shift={(0.5,0)}] \doublefleche;
    \shadedraw[rotate=-45,shift={(0.5,0)}] \doublefleche;
    \shadedraw[rotate=-135,shift={(0.5,0)}] \doublefleche;
    \shadedraw[rotate=135,shift={(0.5,0)}] \doublefleche;
\draw (0,0.25) node[anchor=north]{$D_1$};
\end{tikzpicture}
\begin{tikzpicture}[line cap=round,line join=round,x=1.0cm,y=1.0cm]
\clip(-2,-1) rectangle (2,1);
     \draw [line width=0.5pt] (0.,0.) circle (0.5cm);
     \draw [line width=1.1pt,rotate=90] [->, >= stealth] (0.5,0) -- (0.9,0);
     \draw [rotate=-90] [->, >= stealth] (0.5,0) -- (0.9,0);
     \draw [rotate=0] [->, >= stealth] (0.5,0) -- (0.9,0);
     \draw [rotate=180] [->, >= stealth] (0.5,0) -- (0.9,0);
     \shadedraw[rotate=45,shift={(0.5,0)}] \doublefleche;
    \shadedraw[rotate=-45,shift={(0.5,0)}] \doublefleche;
    \shadedraw[rotate=-135,shift={(0.5,0)}] \doublefleche;
    \shadedraw[rotate=135,shift={(0.5,0)}] \doublefleche;
\draw (0,0.25) node[anchor=north]{$D_2$};
\end{tikzpicture}
\begin{tikzpicture}[line cap=round,line join=round,x=1.0cm,y=1.0cm]
\clip(-2,-1) rectangle (2,1);
     \draw [line width=0.5pt] (0.,0.) circle (0.5cm);
     \draw [line width=1.1pt,rotate=120] [->, >= stealth] (0.5,0) -- (0.9,0);
     \draw [rotate=-120] [->, >= stealth] (0.5,0) -- (0.9,0);
     \draw [rotate=0] [->, >= stealth] (0.5,0) -- (0.9,0);
     \shadedraw[rotate=60,shift={(0.5,0)}] \doublefleche;
    \shadedraw[rotate=-60,shift={(0.5,0)}] \doublefleche;
    \shadedraw[rotate=180,shift={(0.5,0)}] \doublefleche;
\draw (0,0.25) node[anchor=north]{$D_3$};
\end{tikzpicture}
\end{minipage}

    We connect the bold arrows of $D_1$ and $D_3$ to incoming arrows of $D_2$ and end with the following admissible diagram: 

\begin{tikzpicture}[line cap=round,line join=round,x=1.0cm,y=1.0cm]
\clip(-7,-2) rectangle (2,2);
     \draw [line width=0.5pt] (0.,0.) circle (0.5cm);
     \draw [line width=1.1pt,rotate=90] [->, >= stealth] (0.5,0) -- (0.9,0);
     \draw [rotate=-90] [->, >= stealth] (0.5,0) -- (0.9,0);
     \draw [rotate=0] [->, >= stealth] (0.5,0) -- (0.9,0);
     \draw [rotate=180] [->, >= stealth] (0.5,0) -- (0.9,0);
     \shadedraw[rotate=45,shift={(0.5,0)}] \doublefleche;
    \shadedraw[rotate=-135,shift={(0.5,0)}] \doublefleche;
    \shadedraw[rotate=135,shift={(0.5,0)}] \doubleflechescindeeleft;
    \shadedraw[rotate=135,shift={(0.5,0)}] \doubleflechescindeeright;
    \shadedraw[line width=1.1pt,rotate=135,shift={(0.5,0)}] \fleche;
    \draw [line width=0.5pt] (-1.06,1.06) circle (0.5cm);
    \draw[rotate around={45:(-1.06,1.06)}][->, >= stealth](-0.56,1.06)--(-0.16,1.06);
    \draw[rotate around={-135:(-1.06,1.06)}][->, >= stealth](-0.56,1.06)--(-0.16,1.06);
    \draw[rotate around={135:(-1.06,1.06)}][->, >= stealth](-0.56,1.06)--(-0.16,1.06);
    \shadedraw[rotate around={0:(-1.06,1.06)},shift={(-0.56,1.06)}]\doublefleche;
     \shadedraw[rotate around={90:(-1.06,1.06)},shift={(-0.56,1.06)}]\doublefleche;
      \shadedraw[rotate around={-90:(-1.06,1.06)},shift={(-0.56,1.06)}]\doublefleche;
       \shadedraw[rotate around={180:(-1.06,1.06)},shift={(-0.56,1.06)}]\doublefleche;
    \shadedraw[rotate=-45,shift={(0.5,0)}] \doubleflechescindeeleft;
    \shadedraw[rotate=-45,shift={(0.5,0)}] \doubleflechescindeeright;
    \shadedraw[line width=1.1pt,rotate=-45,shift={(0.5,0)}] \fleche;
    \draw [line width=0.5pt] (1.06,-1.06) circle (0.5cm);
    \draw[rotate around={0:(1.06,-1.06)}][->, >= stealth](1.56,-1.06)--(1.96,-1.06);
    \draw[rotate around={-120:(1.06,-1.06)}][->, >= stealth](1.56,-1.06)--(1.96,-1.06);
    \shadedraw[rotate around={60:(1.06,-1.06)},shift={(1.56,-1.06)}]\doublefleche;
    \shadedraw[rotate around={-60:(1.06,-1.06)},shift={(1.56,-1.06)}]\doublefleche;
    \shadedraw[rotate around={180:(1.06,-1.06)},shift={(1.56,-1.06)}]\doublefleche;
\draw (0,0.25) node[anchor=north]{$D_2$};
\draw (-1.05,1.32) node[anchor=north]{$D_1$};
\draw (1.1,-0.8) node[anchor=north]{$D_3$};
\end{tikzpicture}
    
    The remaining bold arrow of $D_2$ is now the last outgoing arrow of the diagram so the labelling of the outgoing arrows of the diagram is as follows:

\begin{tikzpicture}[line cap=round,line join=round,x=1.0cm,y=1.0cm]
\clip(-7,-2) rectangle (2.5,2);
     \draw [line width=0.5pt] (0.,0.) circle (0.5cm);
     \draw [line width=1.1pt,rotate=90] [->, >= stealth] (0.5,0) -- (0.9,0);
     \draw [rotate=-90] [->, >= stealth] (0.5,0) -- (0.9,0);
     \draw [rotate=0] [->, >= stealth] (0.5,0) -- (0.9,0);
     \draw [rotate=180] [->, >= stealth] (0.5,0) -- (0.9,0);
     \shadedraw[rotate=45,shift={(0.5,0)}] \doublefleche;
    \shadedraw[rotate=-135,shift={(0.5,0)}] \doublefleche;
    \shadedraw[rotate=135,shift={(0.5,0)}] \doubleflechescindeeleft;
    \shadedraw[rotate=135,shift={(0.5,0)}] \doubleflechescindeeright;
    \shadedraw[line width=1.1pt,rotate=135,shift={(0.5,0)}] \fleche;
    \draw [line width=0.5pt] (-1.06,1.06) circle (0.5cm);
    \draw[rotate around={45:(-1.06,1.06)}][->, >= stealth](-0.56,1.06)--(-0.16,1.06);
    \draw[rotate around={-135:(-1.06,1.06)}][->, >= stealth](-0.56,1.06)--(-0.16,1.06);
    \draw[rotate around={135:(-1.06,1.06)}][->, >= stealth](-0.56,1.06)--(-0.16,1.06);
    \shadedraw[rotate around={0:(-1.06,1.06)},shift={(-0.56,1.06)}]\doublefleche;
     \shadedraw[rotate around={90:(-1.06,1.06)},shift={(-0.56,1.06)}]\doublefleche;
      \shadedraw[rotate around={-90:(-1.06,1.06)},shift={(-0.56,1.06)}]\doublefleche;
       \shadedraw[rotate around={180:(-1.06,1.06)},shift={(-0.56,1.06)}]\doublefleche;
    \shadedraw[rotate=-45,shift={(0.5,0)}] \doubleflechescindeeleft;
    \shadedraw[rotate=-45,shift={(0.5,0)}] \doubleflechescindeeright;
    \shadedraw[line width=1.1pt,rotate=-45,shift={(0.5,0)}] \fleche;
    \draw [line width=0.5pt] (1.06,-1.06) circle (0.5cm);
    \draw[rotate around={0:(1.06,-1.06)}][->, >= stealth](1.56,-1.06)--(1.96,-1.06);
    \draw[rotate around={-120:(1.06,-1.06)}][->, >= stealth](1.56,-1.06)--(1.96,-1.06);
    \shadedraw[rotate around={60:(1.06,-1.06)},shift={(1.56,-1.06)}]\doublefleche;
    \shadedraw[rotate around={-60:(1.06,-1.06)},shift={(1.56,-1.06)}]\doublefleche;
    \shadedraw[rotate around={180:(1.06,-1.06)},shift={(1.56,-1.06)}]\doublefleche;
\draw (0,0.25) node[anchor=north]{$D_2$};
\draw (-1.05,1.32) node[anchor=north]{$D_1$};
\draw (1.1,-0.8) node[anchor=north]{$D_3$};
\draw (1,0.25) node[anchor=north]{$\scriptstyle{1}$};
\draw (2,-1) node[anchor=north]{$\scriptstyle{2}$};
\draw (0.45,-1.5) node[anchor=north]{$\scriptstyle{3}$};
\draw (0,-0.8) node[anchor=north]{$\scriptstyle{4}$};
\draw (-1,0.25) node[anchor=north]{$\scriptstyle{5}$};
\draw (-1.8,0.8) node[anchor=north]{$\scriptstyle{6}$};
\draw (-1.8,1.8) node[anchor=north]{$\scriptstyle{7}$};
\draw (-0.25,1.8) node[anchor=north]{$\scriptstyle{8}$};
\draw (0,1.25) node[anchor=north]{$\scriptstyle{9}$};
\end{tikzpicture}

\end{example}

\begin{definition}
The \textbf{\textcolor{ultramarine}{type}} of an admissible diagram $(D=\{D_1,\dots,D_n\},R)$ is the tuple \[
\doubar{x}=(\bar{x}^1,\dots,\bar{x}^m) \in \bar{\MO}^m\] where $m=|D|$ and $\bar{x}^i \in \bar{\MO}$ is the tuple composed of the objects of $(D,R)$, placed in counterclockwise order, that we can read between the outgoing arrows $i-1$ and $i$ of $(D,R)$. 
\end{definition}

\begin{remark}
    Given a diagram $(D,R)$ of type $\doubar{x}=(\bar{x}^1,\dots,\bar{x}^n) \in \bar{\MO}^n$, there is an order on the incoming arrows given by \[
    ({}_{x_1^1}\alpha_{x_2^1},\dots,{}_{x_{\llg(\Bar{x}^1)-1}^1}\alpha_{x_{\llg(\Bar{x}^1)}^1},\dots,{}_{x_1^n}\alpha_{x_2^n},\dots,{}_{x_{\llg(\Bar{x}^n)-1}^n}\alpha_{x_{\llg(\Bar{x}^n)}^n})
    \] 
    This allows us to define an order on the discs composing a diagram recursively: the first disc is the one carrying the first incoming arrow of $(D,R)$ and the following is the disc carrying the first incoming arrow of $(D,R)$ that is not one of a preceding disc.
\end{remark}

\begin{definition}
    A \textbf{\textcolor{ultramarine}{source}} (resp. \textbf{\textcolor{ultramarine}{sink}}) of an admissible diagram is a disc which shares none of its incoming (resp. outgoing) arrows with another one.
\end{definition}

\begin{remark}
    An admissible diagram has at least one source and one sink.
\end{remark}

\begin{definition}
    Given an admissible diagram $(D=\{D_1,\dots,D_n\},R)$, one can define its \textbf{\textcolor{ultramarine}{application order}} which takes the form of a tuple $App_D$ of names of discs that we read from right to left. The process starts with the first disc of $(D,R)$ and is as follows:

\begin{enumerate}[label=(App.\arabic*)]
        \item \label{item:app-1} if none of its incoming arrows is connected to an arrow of a disc that is not in $App_D$, add it as the first element of $App_D$ and restart the process with the following disc of $(D,R)$ that is not in $App_D$;
        \item \label{item:app-2} if some of its incoming arrows are shared with some discs not in $App_D$, consider the first disc among those and restart the process with it.
\end{enumerate}
\end{definition}

\begin{example}
    \label{example:app-order}
    We consider the following admissible diagram: 
    
    \begin{tikzpicture}[line cap=round,line join=round,x=1.0cm,y=1.0cm]
\clip(-7,-1) rectangle (2.5,1);
     \draw [line width=0.5pt] (0.,0.) circle (0.5cm);
    \shadedraw [rotate=90, shift={(0.5cm,0cm)}] \doublefleche;
    \shadedraw [rotate=-90, shift={(0.5cm,0cm)}] \doublefleche;
    \draw [line width=0.5pt] (1.5,0.) circle (0.5cm);
    \shadedraw[shift={(-1cm,0cm)}] \doubleflechescindeeleft;
    \shadedraw[shift={(-1cm,0cm)}] \doubleflechescindeeright;
    \shadedraw[line width=1.1pt,shift={(-1cm,0cm)}] \fleche;
    \shadedraw[shift={(1cm,0cm)},rotate=180] \doubleflechescindeeleft;
    \shadedraw[shift={(1cm,0cm)},rotate=180] \doubleflechescindeeright;
    \shadedraw[line width=1.1pt,shift={(1cm,0cm)},rotate=180] \fleche;
    \shadedraw[rotate around={0:(1.5,0)}, shift={(2cm,0cm)}] \doublefleche;
    \draw [rotate around={90:(1.5,0)}] [->, >= stealth](2,0) -- (2.4,0);
    \draw [rotate around={-90:(1.5,0)}] [->, >= stealth](2,0) -- (2.4,0);
    \draw [line width=0.5pt] (-1.5,0.) circle (0.5cm);
    \draw [line width=1.1pt, rotate around={90:(-1.5,0)}] [->, >= stealth](-1,0) -- (-0.6,0);
    \draw [rotate around={-90:(-1.5,0)}] [->, >= stealth](-1,0) -- (-0.6,0);
    \shadedraw [rotate around={180:(-1.5,0)}, shift={(-1cm,0cm)}] \doublefleche;
\draw(0,0.25)node[anchor=north]{$D_1$};
\draw(-1.5,0.25)node[anchor=north]{$D_2$};
\draw(1.5,0.25)node[anchor=north]{$D_3$};
\end{tikzpicture}

The first disc of $(D,R)$ is $D_2$. However, it shares one of its incoming arrows with $D_1$. We thus consider this latter disc. None of its incoming arrows is connected to an arrow of another disc so we set $App_D=(D_1)$. The first disc of $(D,R)$ which is not in $App_D$ is $D_2$. It shares one of its incoming arrows with $D_1$ which is in $App_D$, so we simply add $D_2$ at the beginning of $App_D$. We do the same with $D_3$ and end with $App_D=(D_3,D_2,D_1)$. 
\end{example}

We now explain how to determine the element of $\Multi_d^{\bullet}(\MA)[d+1]$ associated to an admissible diagram.

\begin{definition}
    Given an admissible diagram $(D = ( D_{1}, \dots, D_{n} ) ,R)$ with application order given by $App_d=(D_1,\dots,D_n)$ of type $\doubar{x}$ and a tuple 
$(s_{d+1}\mathbf{F}_{1},\dots, s_{d+1}\mathbf{F}_{n})$ with 
$s_{d+1}\mathbf{F}_{i} \in \Multi^{\bullet}_d(\MA)^{C_{\llg(\bullet)}}[d+1]$ for all $i \in \llbracket 1, n \rrbracket$, we define a map
\[
\mathcal{E}((D,R),s_{d+1}\mathbf{F}_1,\dots,s_{d+1}\mathbf{F}_n) \in \Multi^{\doubar{x}}_d(\MA)[d+1]
\]
as follows. 

We place each element $sa^{i}_{j} \in {}_{x^i_{j}}\MA_{x^i_{j+1}}[1]$ in the incoming arrow ${}_{x^i_{j}}\alpha_{x^i_{j+1}}$ of $(D,R)$ in counterclockwise order beginning at the bold arrow. 
        This will create a sign, as follows : if an element of degree $\ell$ pass through an element of degree $\ell'$ to go to its place, we add a sign $(-1)^{\ell\ell'}$.
Then, we apply the following process beginning with $D_n$:

\begin{enumerate}[label=(Ev.\arabic*)]  
    \item \label{item:ev-1} We transpose all the elements corresponding to incoming arrows of $D_s$ with the elements that do not correspond to incoming arrows of $D_s$ and that are before them 
    on $(D,R)$. We add the corresponding Koszul sign for those transpositions. 
    Moreover, if the first tuple of objects $\Bar{sa}^u$ on $(D,R)$ corresponds to incoming arrows of $D_s$ preceding an incoming arrow $\alpha$ connected with an outgoing arrow of another disc, we transpose the tuple of objects $\Bar{sa}^v$ between $\alpha$ and the last outgoing arrow of $D_s$ on $(D,R)$ with all the others corresponding to incoming arrows of $D_s$, again adding the corresponding sign.
    \item \label{item:ev-2} 
    For each incoming arrow of $D_s$ connected with an outgoing arrow of a disc $D_i$ whose corresponding element is $b_i^j$ between tuples of objects $\Bar{sa}^u$ and $\Bar{sa}^v$ on $D_s$ we transpose $\Bar{sa}^u$ with $\Bar{sa}^v$ and both with $b_i^j$ which gives a sign
     \[
        (-1)^{|b_i^j|(|\bar{sa}^u|+|\bar{sa}^v|)+|\bar{sa}^u||\bar{sa}^v|}.
        \] 
    \item \label{item:ev-3} We evaluate $\mathcal{E}(D_{s},s_{d+1}\mathbf{F}_{s})$ at the elements corresponding to the incoming arrows of $D_{s}$, to obtain an element of the form 
    \begin{itemize}
        \item  $(-1)^{(d+1)\hskip-2mm\sum\limits_{i=1}^{|D_s|-1}\hskip-2mm|s_{-d}b_s^i|}s_{-d}b_s^{1} \otimes \dots \otimes s_{-d}b_s^{|D_{s}|-1} \otimes sb_s^{|D_{s}|} \in {}_{y_{1}'}\MA_{y_{1}}[-d] \otimes \dots \otimes {}_{y'_{|D_{s}|}}\MA_{y_{|D_{s}|}}[1]$ if the bold arrow of $D_s$ is outgoing;
        \item  $(-1)^{(d+1)(|\mathbf{F}_s|+d+1+|\Bar{sc}|)}s_{-d}b_s^{1} \otimes \dots \otimes s_{-d}b_s^{|D_{1}|} \in {}_{y_{1}'}\MA_{y_{1}}[-d] \otimes \dots \otimes {}_{y'_{|D_{s}|}}\MA_{y_{|D_{s}|}}[-d]$
        if the bold arrow of $D_s$ is the incoming arrow just after the tuple $\Bar{sc}$ on $D_s$.
    \end{itemize}    
    \item \label{item:ev-4} We add a Koszul sign coming from transposing each element $b_s^i$ corresponding to a free outgoing arrow of $D_s$
    with the elements corresponding to incoming arrows that are on discs after $D_s$ in the application order.
    We also add a Koszul sign coming from transposing each element $b_s^i$ corresponding to a free outgoing arrow of $D_s$ with the elements $b_s^j$, $j>i$, corresponding to non-free outgoing arrows of $D_s$.
    \item \label{item:ev-5} We transpose each element $b_s^i$ corresponding to an outgoing arrow of $D_s$ connected to an incoming arrow $\alpha$ of another disc $D_r$ with all the elements corresponding to incoming arrows of discs after $D_s$ in the application order and that are before the tuple corresponding to incoming arrows preceding $\alpha$ on $D_r$. 
    We add the corresponding Koszul sign.
    \item \label{item:ev-6} Repeat the process with the following element of the application order. 
\end{enumerate}  
    
When we have applied the previous process $n$ times, we finally reorder the tensor factors obtained in steps \ref{item:ev-3} according to the labeling of the outgoing arrows of $(D,R)$, and add the respective Koszul sign. 
\end{definition}

\begin{example}
    We will apply the previous process with the admissible diagram of Example \ref{example:app-order}.
    We first place all the elements in their corresponding incoming arrows. We will regroup the elements corresponding to consecutive arrows in tuples, cutting the tuples when the arrows belong to different discs (see Figure \ref{fig:example-eval}). We end with a sign $(-1)^{\epsilon}$ with 
    \begin{small}
        \begin{equation}
            \epsilon=|\bar{sa}^1|(|\Bar{sa}^2|+|\Bar{sa}^3|)+|\Bar{sa}^2||\Bar{sa}^3|+|\bar{sa}^5|(|\Bar{sa}^6|+|\Bar{sa}^7|)+|\Bar{sa}^6||\Bar{sa}^7|.
        \end{equation}
    \end{small}

    \begin{figure}[H]
        \centering
     \begin{tikzpicture}[line cap=round,line join=round,x=1.0cm,y=1.0cm]
\clip(-7,-1.3) rectangle (4,1.3);
     \draw [line width=0.5pt] (0.,0.) circle (0.5cm);
    \shadedraw [rotate=90, shift={(0.5cm,0cm)}] \doublefleche;
    \shadedraw [rotate=-90, shift={(0.5cm,0cm)}] \doublefleche;
    \draw [line width=0.5pt] (1.5,0.) circle (0.5cm);
    \shadedraw[shift={(-1cm,0cm)}] \doubleflechescindeeleft;
    \shadedraw[shift={(-1cm,0cm)}] \doubleflechescindeeright;
    \shadedraw[line width=1.1pt,shift={(-1cm,0cm)}] \fleche;
    \shadedraw[shift={(1cm,0cm)},rotate=180] \doubleflechescindeeleft;
    \shadedraw[shift={(1cm,0cm)},rotate=180] \doubleflechescindeeright;
    \shadedraw[line width=1.1pt,shift={(1cm,0cm)},rotate=180] \fleche;
    \shadedraw[rotate around={0:(1.5,0)}, shift={(2cm,0cm)}] \doublefleche;
    \draw [rotate around={90:(1.5,0)}] [->, >= stealth](2,0) -- (2.4,0);
    \draw [rotate around={-90:(1.5,0)}] [->, >= stealth](2,0) -- (2.4,0);
    \draw [line width=0.5pt] (-1.5,0.) circle (0.5cm);
    \draw [line width=1.1pt, rotate around={90:(-1.5,0)}] [->, >= stealth](-1,0) -- (-0.6,0);
    \draw [rotate around={-90:(-1.5,0)}] [->, >= stealth](-1,0) -- (-0.6,0);
    \shadedraw [rotate around={180:(-1.5,0)}, shift={(-1cm,0cm)}] \doublefleche;
\draw(0,0.25)node[anchor=north]{$D_1$};
\draw(-1.5,0.25)node[anchor=north]{$D_2$};
\draw(1.5,0.25)node[anchor=north]{$D_3$};
\draw(-0.75,0.75)node[anchor=north]{$\Bar{sa}^1$};
\draw(0,1.3)node[anchor=north]{$\Bar{sa}^2$};
\draw(0.75,0.75)node[anchor=north]{$\Bar{sa}^3$};
\draw(2.5,0.35)node[anchor=north]{$\Bar{sa}^4$};
\draw(-0.75,-0.25)node[anchor=north]{$\Bar{sa}^7$};
\draw(0,-0.66)node[anchor=north]{$\Bar{sa}^6$};
\draw(0.75,-0.25)node[anchor=north]{$\Bar{sa}^5$};
\draw(-2.5,0.35)node[anchor=north]{$\Bar{sa}^8$};
\end{tikzpicture}
\caption{Place of the elements.}
    \label{fig:example-eval}
    \end{figure}
    
    We begin with $D_1$. After the step \ref{item:ev-1}, the sign is $(-1)^{\epsilon+\epsilon_1}$ with \begin{small}
        \begin{equation}
            \epsilon_1=|\Bar{sa}^2||\Bar{sa}^1|+|\Bar{sa}^6|(|\Bar{sa}^1|+|\Bar{sa}^3|+|\Bar{sa}^4|+|\Bar{sa}^5|)
        \end{equation}
    \end{small} and the elements are in the following order: 
    \[(\Bar{sa}^2,\Bar{sa}^6,\Bar{sa}^1,\Bar{sa}^3,\Bar{sa}^4,\Bar{sa}^5,\Bar{sa}^7,\Bar{sa}^8).
    \]
    Since none of the incoming arrows of $D_1$ is connected with the outgoing arrow of another disc, step \ref{item:ev-2} does nothing.
    In step \ref{item:ev-3}, we compute \[\mathcal{E}(D_1,s_{d+1}\mathbf{F}_1)(\Bar{sa}^2,\Bar{sa}^6)=(-1)^{\epsilon_1'}s_{-d}b_1^1\otimes s_{d+1}b_1^2\] with $\epsilon_1'=(d+1)|s_{-d}b_1^1|$  so the total sign is $(-1)^{\epsilon+\epsilon_1+\epsilon_1'}$.
    Since each of the outgoing arrows of $D_1$ is connected with the incoming arrow of another disc, step \ref{item:ev-4} does nothing.
    After step \ref{item:ev-5}, we end with the sign $(-1)^{\epsilon+\epsilon_1+\epsilon_1'+\epsilon_1\dprimeind}$ with $\epsilon_1\dprimeind=|s_{-d}b_1^1||\Bar{sa}^1|+|s_{d+1}b_1^2|(|\Bar{sa}^1|+|\Bar{sa}^3|+|\Bar{sa}^4|+|\Bar{sa}^5|)$ and the elements are now in the following order: 
    \[
    (\Bar{sa}^1,s_{-d}b_1^1,\Bar{sa}^3,\Bar{sa}^4,\Bar{sa}^5,s_{d+1}b_1^2,\Bar{sa}^7,\Bar{sa}^8).
    \]

    We now apply the same process to $D_2$. After the step \ref{item:ev-1}, the sign is $(-1)^{\epsilon+\epsilon_1+\epsilon_1'+\epsilon_1\dprimeind+\epsilon_2}$ with $\epsilon_2=(|s_{d+1}b_1^2|+|\Bar{sa}^7|+|\Bar{sa}^8|)(|s_{-d}b_1^1|+|\Bar{sa}^3|+|\Bar{sa}^4|+|\Bar{sa}^5|)$ and the elements are now in the following order: 
    \[
    (\Bar{sa}^1,s_{d+1}b_1^2,\Bar{sa}^7,\Bar{sa}^8,s_{-d}b_1^1,\Bar{sa}^3,\Bar{sa}^4,\Bar{sa}^5).
    \]
    Since the incoming arrow of $D_2$ between $\Bar{sa}^1$ and $\Bar{sa}^7$ is connected to an outgoing arrow of $D_1$ but the bold arrow of $D_2$ is outgoing, we end step \ref{item:ev-2} with the sign $(-1)^{\epsilon+\epsilon_1\epsilon_1'+\epsilon_1\dprimeind+\epsilon_2+\epsilon_2'}$ where \[\epsilon_2'=|s_{d+1}b_1^2|(|\Bar{sa}^1|+|\Bar{sa}^7|)+|\Bar{sa}^1||\Bar{sa}^7|\]
  and the order 
   \[
    (\Bar{sa}^7,s_{d+1}b_1^2,\Bar{sa}^1,\Bar{sa}^8,s_{-d}b_1^1,\Bar{sa}^3,\Bar{sa}^4,\Bar{sa}^5).
    \]
    In step \ref{item:ev-3}, we compute \[\mathcal{E}(D_2,s_{d+1}\mathbf{F}_2)(\Bar{sa}^7\otimes s_{d+1}b^2_1\otimes \Bar{sa}^1,\Bar{sa}^8)=(-1)^{\epsilon_2\dprimeind}s_{-d}b_2^1\otimes s_{d+1}b_2^2\] with $\epsilon_2\dprimeind=(d+1)|s_{-d}b_2^1|$ so the total sign is $(-1)^{\epsilon+\epsilon_1+\epsilon_1'+\epsilon_1\dprimeind+\epsilon_2+\epsilon_2'+\epsilon_2\dprimeind}$.
    After the step \ref{item:ev-4}, we multiply the latter sign by $(-1)^{\epsilon_2\tprime}$ where $\epsilon_2\tprime=(|s_{-d}b_2^1|+|s_{d+1}b_2^2|)(|s_{-d}b_1^1|+|\Bar{sa}^3|+|\Bar{sa}^4|+|\Bar{sa}^5|)$ and end with the elements in the following order: 
     \[
    (s_{-d}b_1^1,\Bar{sa}^3,\Bar{sa}^4,\Bar{sa}^5,s_{-d}b_2^1, s_{d+1}b_2^2).
    \]
    Since none of the outgoing arrows of $D_2$ is connected with the incoming arrow of another disc, step \ref{item:ev-5} does nothing.

    Then, we apply the process to $D_3$. 
    After the step \ref{item:ev-1}, the sign is \[(-1)^{\epsilon+\epsilon_1+\epsilon_1'+\epsilon_1\dprimeind+\epsilon_2+\epsilon_2'+\epsilon_2\dprimeind+\epsilon_2\tprime+\epsilon_3}\] with $\epsilon_3=|\Bar{sa}^5|(|s_{-d}b_1^1|+|\Bar{sa}^3|+|\Bar{sa}^4|)$ and the elements are now in the following order: 
     \[
    (\Bar{sa}^5,s_{-d}b_1^1,\Bar{sa}^3,\Bar{sa}^4,s_{-d}b_2^1, s_{d+1}b_2^2).
    \]
    Since the incoming arrow of $D_2$ between $\Bar{sa}^3$ and $\Bar{sa}^5$ is connected to an outgoing arrow of $D_1$, we end step \ref{item:ev-2} with the sign $(-1)^{\epsilon+\epsilon_1+\epsilon_1'+\epsilon_1\dprimeind+\epsilon_2+\epsilon_2'+\epsilon_2\dprimeind+\epsilon_2\tprime+\epsilon_3+\epsilon_3'}$ where 
    \begin{small}
        \begin{equation}
            \epsilon_3'=|s_{-d}b_2^1|(|\Bar{sa}^3|+|\Bar{sa}^5|)+|\Bar{sa}^3||\Bar{sa}^5|
        \end{equation}
    \end{small}
  and the order 
  \[
    (\Bar{sa}^3,s_{-d}b_1^1,\Bar{sa}^5,\Bar{sa}^4,s_{-d}b_2^1, s_{d+1}b_2^2).
    \]
    In step \ref{item:ev-3}, we compute \[\mathcal{E}(D_3,s_{d+1}\mathbf{F}_3)(\Bar{sa}^3\otimes s_{-d}b^1_1\otimes \Bar{sa}^5,\Bar{sa}^4)=(-1)^{\epsilon_3\dprimeind}s_{-d}b_2^1\otimes s_{-d}b_3^2 \] with $\epsilon_3\dprimeind=(d+1)(|\mathbf{F}_3|+d+1+|\Bar{sa}^3|)$  so the total sign is $(-1)^{\epsilon+\epsilon_1+\epsilon_1'+\epsilon_1\dprimeind+\epsilon_2+\epsilon_2'+\epsilon_2\dprimeind+\epsilon_2\tprime+\epsilon_3+\epsilon_3'+\epsilon_3\dprimeind}$.
    Since there is no more elements corresponding to incoming arrows of $(D,R)$ and since none of the outgoing arrows of $D_2$ is connected with the incoming arrow of another disc, steps \ref{item:ev-4} and \ref{item:ev-5} do nothing.
    
    We have applied the process $3$ times and ended with the sign \[(-1)^{\epsilon+\epsilon_1+\epsilon_1'+\epsilon_1\dprimeind+\epsilon_2+\epsilon_2'+\epsilon_2\dprimeind+\epsilon_2\tprime+\epsilon_3+\epsilon_3'+\epsilon_3\dprimeind}\] and the element 
    \[
    s_{-d}b_2^1\otimes s_{-d}b_3^2\otimes s_{-d}b_2^1\otimes s_{d+1}b_2^2.
    \]
According to the labeling of the outgoing arrows of $(D,R)$, we have the correct result.
\end{example}

\begin{definition}
    A \textbf{\textcolor{ultramarine}{filled diagram}} is an admissible diagram $(D=\{D_1,\dots,D_n\},R)$ together with elements $s_{d+1}\mathbf{F}_{i}\in\Multi^{\bullet}_d(\MA)^{C_{\llg(\bullet)}}[d+1]$ for $i\in\llbracket 1,n\rrbracket$.
    
    To each filled diagram $\mathcalboondox{D} = \{(D=\{D_1,\dots,D_n\},R),(s_{d+1}\mathbf{F}_{i})_{i\in\llbracket 1,n\rrbracket}\}$ with application order given by $App_D=(D_1,\dots,D_n)$ we associate the element 
    \[
    \mathcal{E}(\mathcalboondox{D})=\mathcal{E}((D,R),s_{d+1}\mathbf{F}_1,\dots,s_{d+1}\mathbf{F}_n)\in\Multi^{\doubar{x}}_d(\MA)[d+1] 
    \]
     where $\doubar{x}$ is the type of $(D,R)$. 
We depict a filled diagram as an admissible diagram replacing the names of the discs by the corresponding maps (see Figure \ref{fig:filled-diag}).

\begin{figure}[H]
    \centering
\begin{tikzpicture}
     \draw [line width=0.5pt] (0.,0.) circle (0.5cm);
    \draw [rotate=90] [line width=1.1pt,->, >= stealth] (0.5,0) -- (0.9,0);
    \draw [rotate=-90] [->, >= stealth](0.5,0) -- (0.9,0);
    \draw [rotate=180] [->, >= stealth](0.5,0) -- (0.9,0);
    \shadedraw [rotate=45, shift={(0.5cm,0cm)}] \doublefleche;
    \shadedraw [rotate=-45, shift={(0.5cm,0cm)}] \doublefleche;
    \draw [line width=0.5pt] (1.5,0.) circle (0.5cm);
    \shadedraw[shift={(1cm,0cm)},rotate=180] \doubleflechescindeeleft;
    \shadedraw[shift={(1cm,0cm)},rotate=180] \doubleflechescindeeright;
    \shadedraw[line width=1.1pt,shift={(1cm,0cm)},rotate=180] \fleche;
    \shadedraw[rotate around={-60:(1.5,0)}, shift={(2cm,0cm)}] \doublefleche;
    \draw [->, >= stealth](2,0) -- (2.4,0);
    \draw [rotate around={120:(1.5,0)}] [->, >= stealth](2,0) -- (2.4,0);
    \draw [rotate around={-120:(1.5,0)}] [->, >= stealth](2,0) -- (2.4,0);
\begin{scriptsize}
\draw [fill=black] (-0.45,0.4) circle (0.3pt);
\draw [fill=black] (-0.25,0.55) circle (0.3pt);
\draw [fill=black] (-0.55,0.2) circle (0.3pt);
\draw [rotate=90][fill=black] (-0.45,0.4) circle (0.3pt);
\draw [rotate=90][fill=black] (-0.25,0.55) circle (0.3pt);
\draw [rotate=90][fill=black] (-0.55,0.2) circle (0.3pt);
\draw [fill=black] (1.85,0.5) circle (0.3pt);
\draw [fill=black] (2,0.35) circle (0.3pt);
\draw [fill=black] (1.65,0.6) circle (0.3pt);
\end{scriptsize}
\draw (0,0.25) node[anchor=north]{$\mathbf{F}$};
\draw (1.5,0.25) node[anchor=north]{$\mathbf{G}$};
\end{tikzpicture}
\caption{A filled diagram.}
    \label{fig:filled-diag}
\end{figure}
\end{definition}

\begin{remark}
    In the rest of the thesis, we will omit to put in bold the arrows connecting two discs. We will often omit the bold arrow of the diagrams, meaning that we sum over all its possible positions. Moreover, when dealing with diagrams consisting of discs filled with different letters, by putting in bold some outgoing arrow of a disc filled with a given letter, we mean sum over all the possibilities to put in bold an outgoing arrow of discs filled with this letter. 
\end{remark}

\subsection{The necklace graded Lie algebra}
In order to recall what a pre-Calabi-Yau structure on a graded quiver $\MA$ is, one first defines a graded Lie algebra, called the \textbf{\textcolor{ultramarine}{necklace graded Lie algebra}} and appearing in \cite{ktv}. 
As a graded vector space, this graded Lie algebra is $\Multi^{\bullet}_d(\MA)^{C_{\llg(\bullet)}}[d+1]$.
To define a graded Lie bracket on this space, we first define a new operation as follows.

\begin{definition}
    Consider a graded quiver $\MA$ with set of objects $\MO$ as well as tuples of elements of $\bar{\MO}$ given by $\bar{\bar{x}}=(\bar{x}^1,...,\bar{x}^n),\bar{\bar{y}}=(\bar{y}^{1},...,\bar{y}^m)\in\bar{\bar{\MO}}$ such that $\rrt(\bar{y}^1)=x_{j}^v$ and $\llt(\bar{y}^{m})=x_{j-1}^v$ for some $v\in\llbracket 1,n \rrbracket$ and $j\in \llbracket 2, \llg(\bar{x}^v) \rrbracket$.
    The \textbf{\textcolor{ultramarine}{inner necklace composition at $v$,$j$}} of elements $s_{d+1}F^{\doubar{x}}\in\Multi^{\doubar{x}}_d(\MA)[d+1]$ and $s_{d+1}G^{\doubar{y}}\in\Multi^{\doubar{y}}_d(\MA)[d+1]$ is given by 
    \[
    s_{d+1}F^{\doubar{x}} \upperset{\substack{\nec,v,j\\ \inn}}{\circ} s_{d+1}G^{\doubar{y}}=\mathcal{E}(\mathcalboondox{D})\in \Multi^{{\bar{\bar{x}}\upperset{v,j,\inn}{\sqcup}\bar{\bar{y}}}}_d(\MA)[d+1]
    \]
with
\begin{equation}
\bar{\bar{x}}\upperset{v,j,\inn}{\sqcup}\bar{\bar{y}}=(\bar{x}^1,\dots,\bar{x}^{v-1},\bar{y}^1\sqcup\bar{x}^{v}_{> j},\bar{y}^2,\dots,\bar{y}^{m-1},\bar{x}^{v}_{<j-1}\sqcup\bar{y}^m,\bar{x}^{v+1},\dots,\bar{x}^n)
\end{equation}
and where $ \mathcalboondox{D}$ is the filled diagram of type ${\bar{\bar{x}}\upperset{v,j,\inn}{\sqcup}\bar{\bar{y}}}$ given by

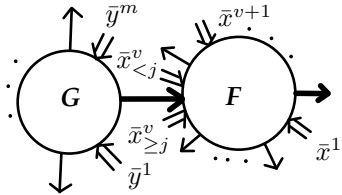
\begin{tikzpicture}[line cap=round,line join=round,x=1.0cm,y=1.0cm]
\clip(-6,-1) rectangle (6,1);
    \draw [line width=0.5pt] (0.,0.) circle (0.5cm);
    \draw [rotate=90] [->,> = stealth] (0.5,0) -- (0.9,0);
    \draw [rotate=-90] [->,> = stealth] (0.5,0) -- (0.9,0);
    \shadedraw [rotate=45, shift={(0.5cm,0cm)}] \doublefleche;
    \shadedraw [rotate=-45, shift={(0.5cm,0cm)}] \doublefleche;
    \draw [line width=0.5pt] (2,0.) circle (0.5cm);
    \shadedraw[shift={(1.5cm,0cm)},rotate=180] \doubleflechescindeeleft;
    \shadedraw[shift={(1.5cm,0cm)},rotate=180] \doubleflechescindeeright;
    \shadedraw[ shift={(1.5cm,0cm)},rotate=180] \flechelong;
    \shadedraw[shift={(2.5cm,0cm)}]\doublefleche;
    \draw [rotate around={-45:(2,0)}] [->,> = stealth](2.5,0) -- (2.9,0);
    \draw [line width=1.1pt,rotate around={45:(2,0)}] [->,> = stealth](2.5,0) -- (2.9,0);
    \draw [rotate around={-135:(2,0)}] [->,> = stealth](2.5,0) -- (2.9,0);
    \draw [rotate around={135:(2,0)}] [->,> = stealth](2.5,0) -- (2.9,0);
\begin{scriptsize}
\draw [fill=black] (2,0.6) circle (0.3pt);
\draw [fill=black] (2.2,0.55) circle (0.3pt);
\draw [fill=black] (1.8,0.55) circle (0.3pt);
\draw [fill=black] (2,-0.6) circle (0.3pt);
\draw [fill=black] (2.2,-0.55) circle (0.3pt);
\draw [fill=black] (1.8,-0.55) circle (0.3pt);
\draw [fill=black] (-0.6,0) circle (0.3pt);
\draw [fill=black] (-0.55,0.2) circle (0.3pt);
\draw [fill=black] (-0.55,-0.2) circle (0.3pt);
\end{scriptsize}
\draw (0,0.25) node[anchor=north] {$\mathbf{G}$};
\draw (2,0.25) node[anchor=north] {$\mathbf{F}$};
\draw (2.7,0) node[anchor=north] {$\scriptstyle{\bar{x}^1}$};
\draw (1.1,0.7) node[anchor=north ] {$\scriptstyle{\bar{x}^v_{<j-1}}$};
\draw (1.1,-0.05) node[anchor=north ] {$\scriptstyle{\bar{x}^v_{>j}}$};
\draw (0.35,-0.35) node[anchor=north west] {$\scriptstyle{\bar{y}^1}$};
\draw (0.25,0.9) node[anchor=north west] {$\scriptstyle{\bar{y}^m}$};
\end{tikzpicture}
\end{definition}

\begin{definition}
   Consider a graded quiver $\MA$ with set of objects $\MO$ as well as tuples of elements of $\bar{\MO}$ given by $\bar{\bar{x}}=(\bar{x}^1,...,\bar{x}^n),\bar{\bar{y}}=(\bar{y}^{1},...,\bar{y}^m)\in\bar{\bar{\MO}}$ such that $\llt(\bar{y}^{v})=x_{j-1}^1$ and $\rrt(\bar{y}^{v+1})=x_{j}^1$ for some $v\in\llbracket 1,m-1\rrbracket$ and $j\in\llbracket 2,\llg(\bar{x}^1)\rrbracket$.
    We define the \textbf{\textcolor{ultramarine}{outer necklace composition at $v$,$j$}} of two elements $s_{d+1}F^{\doubar{x}}\in \Multi^{\doubar{x}}_d(\MA)[d+1]$ 
 and $s_{d+1}G^{\doubar{y}}\in \Multi^{\doubar{y}}_d(\MA)[d+1]$ as
 \[
 s_{d+1}F^{\doubar{x}}\upperset{\substack{\nec, v,j\\ \out}}{\circ } s_{d+1}G^{\doubar{y}}=\mathcal{E}(\mathcalboondox{D})\in\Multi_d^{\bar{\bar{x}}\upperset{v,j,\out}{\sqcup}\bar{\bar{y}}}(\MA)[d+1]
 \]
with
\begin{equation}
\bar{\bar{x}}\upperset{v,j,\out}{\sqcup}\bar{\bar{y}}=(\bar{y}^1,\dots,\bar{y}^{v-1},\bar{x}^1_{< j-1}\sqcup\bar{y}^v,\bar{x}^2,\dots,\bar{x}^n,\bar{y}^{v+1}\sqcup\bar{x}^1_{> j},\bar{y}^{v+2},\dots,\bar{y}^m)
\end{equation}
and where $\mathcalboondox{D}$ is the filled diagram of type $\bar{\bar{x}}\upperset{v,j,\out}{\sqcup}\bar{\bar{y}}$ given by

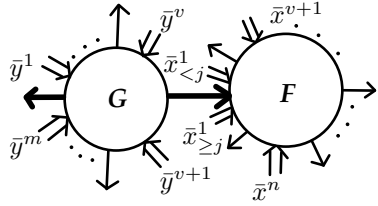
\begin{tikzpicture}[line cap=round,line join=round,x=1.0cm,y=1.0cm]
\clip(-6,-1) rectangle (6,1);
     \draw [line width=0.5pt] (0.,0.) circle (0.5cm);
    \draw [rotate=90] [->,> = stealth] (0.5,0) -- (0.9,0);
    \draw [rotate=-90] [->,> = stealth] (0.5,0) -- (0.9,0);
    \draw [line width=1.1pt,rotate=180] [->,> = stealth](0.5,0) -- (0.9,0);
    \shadedraw [rotate=-45, shift={(0.5cm,0cm)}] \doublefleche;
    \shadedraw [rotate=45, shift={(0.5cm,0cm)}] \doublefleche;
    \draw [line width=0.5pt] (2,0.) circle (0.5cm);
    \shadedraw[shift={(1.5cm,0cm)},rotate=180] \doubleflechescindeeleft;
    \shadedraw[shift={(1.5cm,0cm)},rotate=180] \doubleflechescindeeright;
    \shadedraw[shift={(1.5cm,0cm)},rotate=180] \flechelong;
    \shadedraw[rotate around={-60:(2,0)}, shift={(2.5cm,0cm)}] \doublefleche;
    \draw [->,> = stealth](2.5,0) -- (2.9,0);
    \draw [rotate around={120:(2,0)}] [->,> = stealth](2.5,0) -- (2.9,0);
    \draw [rotate around={-120:(2,0)}] [->,> = stealth](2.5,0) -- (2.9,0);
\begin{scriptsize}
\draw [fill=black] (2.3,0.5) circle (0.3pt);
\draw [fill=black] (2.1,0.57) circle (0.3pt);
\draw [fill=black] (2.45,0.38) circle (0.3pt);
\draw [rotate=75][fill=black] (0.3,0.5) circle (0.3pt);
\draw [rotate=75][fill=black] (0.1,0.57) circle (0.3pt);
\draw [rotate=75][fill=black] (0.45,0.38) circle (0.3pt);
\draw [rotate=165][fill=black] (0.3,0.5) circle (0.3pt);
\draw [rotate=165][fill=black] (0.1,0.57) circle (0.3pt);
\draw [rotate=165][fill=black] (0.45,0.38) circle (0.3pt);
\end{scriptsize}
\draw (0,0.25) node[anchor=north] {$\mathbf{G}$};
\draw (2,0.25) node[anchor=north] {$\mathbf{F}$};
\draw (2.4,-0.5) node[anchor=north] {$\scriptstyle{\bar{x}^n}$};
\draw (1.2,0.78) node[anchor=north] {$\scriptstyle{\bar{x}^1_{<j-1}}$};
\draw (1.2,-0.1) node[anchor=north] {$\scriptstyle{\bar{x}^1_{>j}}$};
\draw (0.7,-0.4) node[anchor=north] {$\scriptstyle{\bar{y}^{v+1}}$};
\draw (0.6,0.9) node[anchor=north] {$\scriptstyle{\bar{y}^v}$};
\end{tikzpicture}
\end{definition}

\begin{definition}
    Given a graded quiver $\MA$ with set of objects $\MO$, the \textbf{\textcolor{ultramarine}{necklace product}} of elements $s_{d+1}\mathbf{F},s_{d+1}\mathbf{G}\in\Multi^{\bullet}_d(\MA)^{C_{\llg(\bullet)}}[d+1]$ is the element 
    \[
        s_{d+1}\mathbf{F} \upperset{\nec}{\circ} s_{d+1}\mathbf{G}\in\Multi^{\bullet}_d(\MA)[d+1] 
    \]
    given by
\begin{equation}
    (s_{d+1}\mathbf{F} \upperset{\nec}{\circ} s_{d+1}\mathbf{G})^{\bar{\bar{z}}} 
    = 
    \sum\limits_{(\bar{\bar{x}}, 
    \bar{\bar{y}}, 
    v, j)\in\mathcal{I}_{\inn}^{\doubar{z}}}
     s_{d+1}F^{\bar{\bar{x}}}\displaystyle\upperset{\substack{\nec,v,j\\ \inn}}{\circ}s_{d+1}G^{\bar{\bar{y}}}  + \sum\limits_{(\bar{\bar{x}}, 
    \bar{\bar{y}}, 
    v, j)\in\mathcal{I}_{\out}^{\doubar{z}}}
     s_{d+1}F^{\bar{\bar{x}}}\displaystyle\upperset{\substack{\nec,v,j\\ \out}}{\circ}s_{d+1}G^{\bar{\bar{y}}}
\end{equation}
for all $\bar{\bar{z}}\in \doubar{\MO}$, 
    where 
\begin{equation}
    \label{eq:I-inn-out}
    \begin{split}
        \mathcal{I}_{\inn}^{\doubar{z}}&=\{(\bar{\bar{x}}, 
    \bar{\bar{y}}, 
    v, j)\in  \doubar{\MO}\times \doubar{\MO}\times\llbracket 1,\llg(\bar{\bar{x}}) \rrbracket \times \llbracket 2,\llg(\bar{x}^v) \rrbracket |
    \bar{\bar{x}} \hskip -2mm\upperset{v,j,\inn}{\sqcup} \hskip -2mm \bar{\bar{y}}=\bar{\bar{z}}\}\\
        \mathcal{I}_{\out}^{\doubar{z}}&=\{(\bar{\bar{x}}, 
    \bar{\bar{y}}, 
    v, j)\in   \doubar{\MO}\times \doubar{\MO}\times\llbracket 1,\llg(\bar{\bar{y}})-1
     \rrbracket \times
    \llbracket 2,\llg(\bar{x}^1) \rrbracket |
    \bar{\bar{x}} \hskip -2mm\upperset{v,j,\out}{\sqcup} \hskip -2mm \bar{\bar{y}}=\bar{\bar{z}}\}.
    \end{split}
\end{equation} 
\end{definition}

\begin{lemma}
    Consider two elements $s_{d+1}\mathbf{F},s_{d+1}\mathbf{G}\in\Multi^{\bullet}_d(\MA)^{C_{\llg(\bullet)}}[d+1]$. Then, we have that $s_{d+1}\mathbf{F} \upperset{\nec}{\circ} s_{d+1}\mathbf{G}\in\Multi^{\bullet}_d(\MA)^{C_{\llg(\bullet)}}[d+1]$.
\end{lemma}
\begin{proof}
    Consider $\doubar{z}=(\bar{z}^1,\dots,\bar{z}^p)\in\bar{\MO}^p$ and denote $\sigma_i=(1 2\dots i)$ for every $i\in\llbracket 1,n\rrbracket$. We have to show that 
    \begin{equation}
    \begin{split}
        &(s_{d+1}\mathbf{F} \upperset{\nec}{\circ} s_{d+1}\mathbf{G})^{\doubar{z}}
        \\&=\tau^{\sigma_p}_{{}_{\llt(\bar{z}^{1})}\MA_{\rrt(\bar{z}^2)}[-d],{}_{\llt(\bar{z}^{2})}\MA_{\rrt(\bar{z}^{3})}[-d],\dots, {}_{\llt(\bar{z}^p)}\MA_{\rrt(\bar{z}^1)}[1]}\circ (s_{d+1}\mathbf{F} \upperset{\nec}{\circ} s_{d+1}\mathbf{G})^{\doubar{z}\cdot\sigma_p} \circ \tau^{\sigma_p}_{\MA[1]^{\otimes \bar{z}^1},\dots,\MA[1]^{\otimes \bar{z}^p}}
        \end{split}
    \end{equation}

First, note that we have a bijection $\mathcal{I}_{\inn}^{\doubar{z}}\cup \mathcal{I}_{\out}^{\doubar{z}}\rightarrow \mathcal{I}_{\inn}^{\doubar{z}\cdot\sigma_p}\cup \mathcal{I}_{\out}^{\doubar{z}\cdot\sigma_p}$ sending $(\bar{\bar{x}}, 
    \bar{\bar{y}}, 
    v, j)\in\mathcal{I}_{\inn}^{\doubar{z}}$ to $(\bar{\bar{x}}\cdot\sigma_n, 
    \bar{\bar{y}}, v+1, j)\in\mathcal{I}_{\inn}^{\doubar{z}\cdot\sigma_p}$ if $v<n$ and to $(\bar{\bar{x}}\cdot\sigma_n, 
    \bar{\bar{y}}\cdot\sigma_m, 1, j)\in\mathcal{I}_{\out}^{\doubar{z}\cdot\sigma_p}$ if $v=n$ and sending $(\bar{\bar{x}}, 
    \bar{\bar{y}}, 
    v, j)\in\mathcal{I}_{\out}^{\doubar{z}}$ to $(\bar{\bar{x}}, 
    \bar{\bar{y}}\cdot\sigma_m, v+1, j)\in\mathcal{I}_{\out}^{\doubar{z}\cdot\sigma_p}$ if $v<m-1$ and to $(\bar{\bar{x}}, 
    \bar{\bar{y}}\cdot\sigma_m,  1, j)\in\mathcal{I}_{\inn}^{\doubar{z}\cdot\sigma_p}$ if $v=m-1$.

Moreover, given $(\bar{\bar{x}}\cdot\sigma_n, 
    \bar{\bar{y}}, v+1, j)\in\mathcal{I}_{\inn}^{\doubar{z}\cdot\sigma_p}$ with $v<n$, we have
\allowdisplaybreaks
\begin{align*}
    &\tau^{\sigma_p}_{{}_{\llt(\bar{z}^{1})}\MA_{\rrt(\bar{z}^2)}[-d],\dots, {}_{\llt(\bar{z}^p)}\MA_{\rrt(\bar{z}^1)}[1]}\circ (s_{d+1}F^{\bar{\bar{x}}\cdot\sigma_n}\displaystyle\upperset{\substack{\nec,v+1,j\\ \inn}}{\circ}s_{d+1}G^{\bar{\bar{y}}}) \circ \tau^{\sigma_p}_{\MA[1]^{\otimes \bar{z}^1},\dots,\MA[1]^{\otimes \bar{z}^p}}
    \\&=
    \tau^{\sigma_n}_{{}_{\llt(\bar{x}^{1})}\MA_{\rrt(\bar{x}^2)}[-d],\dots, {}_{\llt(\bar{x}^n)}\MA_{\rrt(\bar{x}^1)}[1]}\circ s_{d+1}F^{\bar{\bar{x}}\cdot\sigma_n}\circ\tau^{\sigma_n}_{\MA[1]^{\otimes \bar{x}^1},\dots,\MA[1]^{\otimes \bar{x}^{n}}}\displaystyle\upperset{\substack{\nec,v,j\\ \inn}}{\circ}s_{d+1}G^{\bar{\bar{y}}}
    \\&=s_{d+1}F^{\bar{\bar{x}}}\displaystyle\upperset{\substack{\nec,v,j\\ \inn}}{\circ}s_{d+1}G^{\bar{\bar{y}}}
\end{align*}    
using that $s_{d+1}\mathbf{F}\in\Multi^{\bullet}_d(\MA)^{C_{\llg(\bullet)}}[d+1]$. Similarly, given $(\doubar{x}\cdot\sigma_n,\doubar{y}\cdot\sigma_m,1,j)$ we have that 
\allowdisplaybreaks
\begin{align*}
    &\tau^{\sigma_p}_{{}_{\llt(\bar{z}^{1})}\MA_{\rrt(\bar{z}^2)}[-d],\dots, {}_{\llt(\bar{z}^p)}\MA_{\rrt(\bar{z}^1)}[1]}\circ (s_{d+1}F^{\bar{\bar{x}}\cdot\sigma_n}\displaystyle\upperset{\substack{\nec,1,j\\ \out}}{\circ}s_{d+1}G^{\bar{\bar{y}}\cdot\sigma_m}) \circ \tau^{\sigma_p}_{\MA[1]^{\otimes \bar{z}^1},\dots,\MA[1]^{\otimes \bar{z}^p}}
    \\&=
    (\tau^{\sigma_n}_{{}_{\llt(\bar{x}^{1})}\MA_{\rrt(\bar{x}^2)}[-d],\dots, {}_{\llt(\bar{x}^n)}\MA_{\rrt(\bar{x}^1)}[1]}\circ s_{d+1}F^{\bar{\bar{x}}\cdot\sigma_n}\circ\tau^{\sigma_n}_{\MA[1]^{\otimes \bar{x}^1},\dots,\MA[1]^{\otimes \bar{x}^{n}}})
    \\&\hspace{3cm}\displaystyle\upperset{\substack{\nec,n,j\\ \inn}}{\circ} (\tau^{\sigma_m}_{{}_{\llt(\bar{y}^{1})}\MA_{\rrt(\bar{y}^2)}[-d],\dots, {}_{\llt(\bar{y}^m)}\MA_{\rrt(\bar{y}^1)}[1]}\circ s_{d+1}G^{\bar{\bar{y}}\cdot\sigma_m}\circ\tau^{\sigma_m}_{\MA[1]^{\otimes \bar{y}^1},\dots,\MA[1]^{\otimes \Bar{y}^{m}}}) 
    \\&=s_{d+1}F^{\bar{\bar{x}}}\displaystyle\upperset{\substack{\nec,n,j\\ \inn}}{\circ}s_{d+1}G^{\bar{\bar{y}}}
\end{align*}
using that $s_{d+1}\mathbf{F},s_{d+1}\mathbf{G}\in\Multi^{\bullet}_d(\MA)^{C_{\llg(\bullet)}}[d+1]$. It is easy to check that we also have
\begin{equation}
    \begin{split}
        &\tau^{\sigma_p}_{{}_{\llt(\bar{z}^{1})}\MA_{\rrt(\bar{z}^2)}[-d],\dots, {}_{\llt(\bar{z}^p)}\MA_{\rrt(\bar{z}^1)}[1]}\circ (s_{d+1}F^{\bar{\bar{x}}}\displaystyle\upperset{\substack{\nec,v+1,j\\ \out}}{\circ}s_{d+1}G^{\bar{\bar{y}}\cdot\sigma_m}) \circ \tau^{\sigma_p}_{\MA[1]^{\otimes \bar{z}^1},\dots,\MA[1]^{\otimes \bar{z}^p}}
    \\&=s_{d+1}F^{\bar{\bar{x}}}\displaystyle\upperset{\substack{\nec,v,j\\ \out}}{\circ}s_{d+1}G^{\bar{\bar{y}}}
    \end{split}
\end{equation}
and
\begin{equation}
    \begin{split}
         &\tau^{\sigma_p}_{{}_{\llt(\bar{z}^{1})}\MA_{\rrt(\bar{z}^2)}[-d],\dots, {}_{\llt(\bar{z}^p)}\MA_{\rrt(\bar{z}^1)}[1]}\circ (s_{d+1}F^{\bar{\bar{x}}}\displaystyle\upperset{\substack{\nec,1,j\\ \inn}}{\circ}s_{d+1}G^{\bar{\bar{y}}\cdot\sigma_m}) \circ \tau^{\sigma_p}_{\MA[1]^{\otimes \bar{z}^1},\dots,\MA[1]^{\otimes \bar{z}^p}}
    \\&=s_{d+1}F^{\bar{\bar{x}}}\displaystyle\upperset{\substack{\nec,m-1,j\\ \out}}{\circ}s_{d+1}G^{\bar{\bar{y}}}
    \end{split}
\end{equation}
\end{proof}

\begin{definition}

\label{def:necklace-bracket}

    Given a graded quiver $\MA$ with set of objects $\MO$, the \textbf{\textcolor{ultramarine}{necklace bracket}} of two elements $s_{d+1}\mathbf{F},s_{d+1}\mathbf{G}\in\Multi^{\bullet}_d(\MA)^{C_{\llg(\bullet)}}[d+1]$ is defined as the element
    \[
        [s_{d+1}\mathbf{F}, s_{d+1}\mathbf{G}]_{\nec}\in\Multi^{\bullet}_d(\MA)^{C_{\llg(\bullet)}}[d+1] 
    \]
    where
    \[
        [s_{d+1}\mathbf{F},s_{d+1}\mathbf{G}]_{\nec}^{\doubar{z}}= (s_{d+1}\mathbf{F} \upperset{\nec}{\circ} s_{d+1}\mathbf{G})^{\bar{\bar{z}}}-(-1)^{(|\mathbf{F}|+d  +1)(|\mathbf{G}|+d+1)} (s_{d+1}\mathbf{G} \upperset{ \nec}{\circ} s_{d+1}\mathbf{F})^{\doubar{z}}
    \]
    for every $\doubar{z}\in\doubar{\MO}$.
\end{definition}

\begin{lemma}
\label{lemma:j}
Let $\MA$ be a graded quiver with set of objects $\MO$.
Then, we have a homogeneous linear map of degree $0$ 
\begin{equation}
    \label{eq:map-j} 
    \mathcalboondox{j} : \Multi^{\bullet}_d(\MA)[d+1] \rightarrow C(\MA\oplus\MA^*[d-1])[1]
\end{equation}
defined by $\mathcalboondox{j}(s_{d+1}\phi)=s\psi$ where
\begin{equation}
    \label{eq:map-j-x-A*} 
    \begin{split}
    &(\pi_{\MA}\circ \psi)^{\doubar{x}}(\bar{sa}^n,tf^{n-1},\bar{sa}^{n-2},tf^{n-2},...,\bar{sa}^{2},tf^{2},\bar{sa}^1)\\&\hspace{5cm}=(-1)^{\epsilon}\big( \bigotimes\limits_{i=1}^{n-1}(f^{i}\circ s_d)\otimes s_d\big)\big(\phi^{\doubar{x}}(\bar{sa}^1,\bar{sa}^{2},...,\bar{sa}^n)\big)
    \end{split}
\end{equation}
with
\begin{small}
\begin{equation}
    \begin{split}
        \epsilon=\sum\limits_{i=1}^{n-1}|tf^i|\sum\limits_{j=i+1}^n|\bar{sa}^j|+|\phi|\sum\limits_{i=1}^{n-1}|tf^i|+d(n-1)+\sum\limits_{1\leq i < j\leq n}|\bar{sa}^i||\bar{sa}^j|+\sum\limits_{1\leq i < j\leq n-1}|tf^i||tf^j|
    \end{split}
\end{equation}
\end{small}
and 
\begin{equation}
    \label{eq:map-j-x-A} 
    \begin{split}
    (\pi_{\MA^*[d-1]}\circ \psi)^{\doubar{x}}(\bar{sa}^n,tf^{n-1},&\bar{sa}^{n-1},tf^{n-2},...,\bar{sa}^{2},tf^{1},\bar{sa}^1)(s_{1-d}b)\\&=(-1)^{\delta} \big(\bigotimes\limits_{i=1}^{n-1}(f^{i}\circ s_d)\big)\big(\phi^{\doubar{x}'}(\bar{sa}^1\otimes sb \otimes\bar{sa}^n,\bar{sa}^{2},...,\bar{sa}^{n-1})\big)
    \end{split}
\end{equation}
for any $\doubar{x}=(\bar{x}^1,\dots,\bar{x}^n)$, $\doubar{x}'=(\bar{x}^1\sqcup\bar{x}^n,\bar{x}^2,\dots,\bar{x}^{n-1})$ and elements  $\bar{sa}^i\in\MA[1]^{\otimes\bar{x}^i}$ for every $i\in\llbracket 1,n\rrbracket$, $sb\in{}_{\rrt(\bar{x}^1)}\MA_{\llt(\bar{x}^n)}[1]$ and $tf^i\in{}_{\rrt(\bar{x}^{i+1})}\MA^{*}_{\llt(\bar{x}^{i})}[d]$ for every $i\in\llbracket 1,n-1\rrbracket$
with
\begin{small}
\begin{equation}
    \begin{split}
        \delta=\epsilon
        +d(|\phi|+\sum\limits_{i=1}^n|\bar{sa}^i|)+(|\bar{sa}^n|+|sb|)\sum\limits_{i=2}^{n-1}|\bar{sa}^i|+|\bar{sa}^n||sb|
    \end{split}
\end{equation}
\end{small}
and where we have denoted by $\pi_{\MA}$ (resp. $\pi_{\MA^*[d-1]}$) the canonical projection $\MA\oplus\MA^*[d-1]\rightarrow \MA$ (resp. $\MA\oplus\MA^*[d-1]\rightarrow \MA^*[d-1]$).
Moreover, $\mathcalboondox{j}$ is injective.
\end{lemma}
\begin{proof}
    It is clear that $\mathcalboondox{j}$ is a homogeneous linear map of degree $0$. We prove that $\mathcalboondox{j}$ is injective. Take $s_{d+1}\phi\in\Multi^{\bullet}_d(\MA)[d+1]$ such that $s\psi=\mathcalboondox{j}(s_{d+1}\phi)$ is the zero map, \textit{i.e.} 
    \begin{equation}
        \psi^{\doubar{x}}(\bar{sa}^n,tf^{n-1},\bar{sa}^{n-2},tf^{n-2},...,\bar{sa}^{2},tf^{2},\bar{sa}^1)=0
    \end{equation}
    for every $\doubar{x}=(\bar{x}^1,\dots,\bar{x}^n)$, $\bar{sa}^i\in\MA[1]^{\otimes\bar{x}^i}$ for $i\in\llbracket 1,n\rrbracket$ and $tf^i\in{}_{\rrt(\bar{x}^{i+1})}\MA^{*}_{\llt(\bar{x}^{i})}[d]$ for $i\in\llbracket 1,n-1\rrbracket$.
    In particular, 
    \begin{equation}
        (\pi_{\MA}\circ\psi^{\doubar{x}})(\bar{sa}^n,tf^{n-1},\bar{sa}^{n-2},tf^{n-2},...,\bar{sa}^{2},tf^{2},\bar{sa}^1)=0
    \end{equation}
    which gives
    \begin{equation}
        \bigotimes\limits_{i=1}^{n-1}\big((f^{i}\circ s_d)\otimes s_d\big)\big(\phi^{\doubar{x}}(\bar{sa}^1,\bar{sa}^{2},...,\bar{sa}^n)\big)=0.
    \end{equation}
    If $s_{d+1}\phi$ is nonzero, there exists $\bar{sa}^i\in\MA[1]^{\otimes\bar{x}^i}$ for $i\in\llbracket 1,n\rrbracket$ such that $\phi^{\doubar{x}}(\bar{sa}^1,\dots,\bar{sa}^n)=\phi_1\otimes\dots\otimes\phi_n$ where $\phi_i\neq 0$ for every $i\in\llbracket 1,n\rrbracket$. Therefore, there exists $tf^i\in{}_{\rrt(\bar{x}^{i+1})}\MA^{*}_{\llt(\bar{x}^{i})}[d]$ for $i\in\llbracket 1,n-1\rrbracket$ such that 
    \begin{equation}
        \bigotimes\limits_{i=1}^{n-1}\big((f^{i}\circ s_d)\otimes s_d\big)\big(\phi^{\doubar{x}}(\bar{sa}^1,\bar{sa}^{2},...,\bar{sa}^n)\big)\neq 0
    \end{equation}
    This leads to a contradiction. Then, $s_{d+1}\phi$ is zero and $\mathcalboondox{j}$ is injective.
\end{proof}

We have the following relation between the necklace product and the usual Gerstenhaber circle product, which does not seem to have been observed in the literature so far.  

\begin{proposition}
\label{proposition:j-bracket}
Let $s_{d+1}\mathbf{F}, s_{d+1}\mathbf{G}$ be elements in $\Multi^{\bullet}_d(\MA)^{C_{\llg(\bullet)}}[d+1]$.
Then, we have
\begin{equation}
\label{eq:j-bracket-1}
    \mathcalboondox{j}_{\bar{\bar{x}}\upperset{v,j,\inn}{\sqcup}\bar{\bar{y}}}(s_{d+1}F^{\bar{\bar{x}}}\displaystyle\upperset{\substack{\nec,v,j\\ \inn}}{\circ}s_{d+1}G^{\bar{\bar{y}}})
    =\mathcalboondox{j}_{\bar{\bar{x}}}(s_{d+1}F^{\bar{\bar{x}}})\upperset{G,p,q}{\circ}\pi_{\MA}(\mathcalboondox{j}_{\bar{\bar{y}}}(s_{d+1}G^{\bar{\bar{y}}}))
\end{equation}
for $\bar{\bar{x}}, \bar{\bar{y}} \in \bar{\bar{\MO}}$, $v\in \llbracket 1,\llg(\bar{\bar{x}})\rrbracket$ and $j\in \llbracket 2,\llg(\bar{x}^v)\rrbracket$
and
\begin{equation}
\label{eq:j-bracket-2}
    \mathcalboondox{j}_{\bar{\bar{x}}\upperset{v,j,\out}{\sqcup}\bar{\bar{y}}}(s_{d+1}F^{\bar{\bar{x}}}\displaystyle\upperset{\substack{\nec,v,j\\ \out}}{\circ}s_{d+1}G^{\bar{\bar{y}}})
    =-(-1)^{(|\mathbf{F}|+d+1)(|\mathbf{G}|+d+1)}\mathcalboondox{j}_{\bar{\bar{y}}}(s_{d+1}G^{\bar{\bar{y}}})\upperset{G,p,q}{\circ}\pi_{\MA^*}(\mathcalboondox{j}_{\bar{\bar{x}}}(s_{d+1}F^{\bar{\bar{x}}}))
\end{equation}
for $v\in \llbracket 1,\llg(\bar{\bar{y}})-1\rrbracket$ and $j\in \llbracket 2,\llg(\bar{x}^1)\rrbracket$, where $p=\llg(\bar{x}^1)+...+\llg(\bar{x}^v)-j$ and $q=p+N(\doubar{y})$.
\end{proposition}
\begin{proof}
We first show the identity \eqref{eq:j-bracket-1}.
Given $\doubar{x}=(\bar{x}^1,\dots,\bar{x}^n),\doubar{y}=(\bar{y}^{1},\dots,\bar{y}^m)\in \doubar{\MO}$, both the compositions \[s_{d+1}F^{\bar{\bar{x}}}\displaystyle\upperset{\substack{\nec,v,j\\ \inn}}{\circ}s_{d+1}G^{\bar{\bar{y}}}\] and \[\mathcalboondox{j}_{\bar{\bar{x}}}(s_{d+1}F^{\bar{\bar{x}}})\upperset{G,p,q}{\circ}\pi_{\MA}(\mathcalboondox{j}_{\bar{\bar{y}}}(s_{d+1}G^{\bar{\bar{y}}}))\] are zero if 
there are no $v\in\llbracket 1,n\rrbracket$ and $j\in \llbracket 2,\llg(\bar{x}^v) \rrbracket$ such that
$\llt(\bar{y}^m)=x_j^v$ and $\rrt(\bar{y}^{1})=x_{j+1}^v$ with the convention that $\bar{x}^{n+1}=\bar{x}^1$. 
We will thus assume that there exist such $v\in\llbracket 1,n \rrbracket$ and $j\in \llbracket 2,\llg(\bar{x}^v) \rrbracket$. To simplify the expressions, we will denote the result of applying $F^{\doubar{x}}$ on any argument as a tensor product ${}_{\llt(\bar{x}^1)}F_{\rrt(\bar{x}^2)}^{-d}\otimes\dots\otimes{}_{\llt(\bar{x}^{n-1})}F_{\rrt(\bar{x}^n)}^{-d}\otimes {}_{\llt(\bar{x}^n)}F_{\rrt(\bar{x}^1)}^{-d}$, where we omit those arguments. We also write $s_{d+1}F^{\bar{\bar{x}}}\displaystyle\upperset{\substack{\nec,v,j\\ \inn}}{\circ}s_{d+1}G^{\bar{\bar{y}}}=s_{d+1}H_1$.
We have that 
\[
\mathcalboondox{j}_{\bar{\bar{x}}\upperset{v,j,\inn}{\sqcup}\bar{\bar{y}}}(s_{d+1}F^{\bar{\bar{x}}}\displaystyle\upperset{\substack{\nec,v,j\\ \inn}}{\circ}s_{d+1}G^{\bar{\bar{y}}})=s\psi^{\bar{\bar{x}}\upperset{v,j,\inn}{\sqcup}\bar{\bar{y}}}
\]
where 
\allowdisplaybreaks
\begin{small}
\begin{align*}
&(\pi_{\MA}\circ \psi)^{\bar{\bar{x}}\upperset{v,j,\inn}{\sqcup}\bar{\bar{y}}}\\&\hspace{0.7cm}(\bar{sa}^n,tf^{n-1},\dots,\bar{sa}^{v+1},tf^{v},\bar{sa}_{< j-1}^{v},\bar{sb}^m,tg^{m-1},\dots,\bar{sb}^{2},tg^{1},\bar{sb}^{1},\bar{sa}_{\geq j}^{v},tf^{v-1},\dots,\bar{sa}^{2},tf^{1},\bar{sa}^1)
        \\&=(-1)^{\epsilon}
         ((f^{1}\circ s_d)\otimes\dots\otimes (f^{v-1}\circ s_d)\otimes \bigotimes\limits_{i=1}^{m-1}(g^{i}\circ s_d)\otimes (f^{v}\circ s_d)\otimes \dots\otimes (f^{n-1}\circ s_d)\otimes \id)
        \\&\phantom{=}s_dH_1(\bar{sa}^1,\bar{sa}^2,\dots,\bar{sa}^{v-1},\bar{sa}^v_{\geq j},\bar{sb}^1,\dots,\bar{sb}^{m-1},\bar{sb}^{m},\bar{sa}_{<j-1}^{v},\dots,\bar{sa}^{n-1},\bar{sa}^n)
        \\
        &=(-1)^{\epsilon+\epsilon'}((f^{1}\circ s_d)\otimes\dots\otimes (f^{v-1}\circ s_d)\otimes \bigotimes\limits_{i=1}^{m-1}(g^{i}\circ s_d)\otimes (f^{v}\circ s_d)\otimes \dots\otimes (f^{n-1}\circ s_d)\otimes \id)
        \\&\phantom{=} s_{d}({}_{\llt(\bar{x}^1)}F_{\rrt(\bar{x}^2)}^{-d},\dots {}_{\llt(\bar{x}^{v-1})}F_{\rrt(\bar{x}^v)}^{-d},{}_{\llt(\bar{y}^{1})}G_{\rrt(\bar{y}^{2})}^{-d},...,{}_{\llt(\bar{y}^{m-1})}G_{\rrt(\bar{y}^m)}^{-d},{}_{\llt(\bar{x}^v)}F_{\rrt(\bar{x}^{v+1})}^{-d}\dots{}_{\llt(\bar{x}^{n})}F_{\rrt(\bar{x}^1)}^{-d})
\end{align*}
\end{small}
\hskip -0.6mm for elements $\bar{sa}^i\in\MA[1]^{\otimes\bar{x}^i}$, $\bar{sb}^i\in\MA[1]^{\otimes\bar{y}^i}$, $tf^i\in{}_{\rrt(\bar{x}^{i+1})}\MA^{*}_{\llt(\bar{x}^{i})}[d]$ and $tg^i\in{}_{\rrt(\bar{y}^{i+1})}\MA^{*}_{\llt(\bar{y}^{i})}[d]$,
with
\begin{small}
\allowdisplaybreaks
\begin{align*}
        \epsilon&=\sum\limits_{i=v}^{n-1}|tf^i|\sum\limits_{k=i+1}^n|\bar{sa}^k|+\sum\limits_{i=1}^{m-1}|tg^i|(\sum\limits_{k=v+1}^n|\bar{sa}^k|+|\bar{sa}^{v}_{<j-1}|+\sum\limits_{k=i+1}^{m}|\bar{sb}^k|)
        \\&\phantom{=}+\sum\limits_{i=1}^{v-1}|tf^i|(\sum\limits_{k=v+1}^n|\bar{sa}^k|+|\bar{sa}^{v}_{\geq j}|+\sum\limits_{k=1}^{m}|\bar{sb}^k|+|\bar{sa}^v_{< j-1}|+\sum\limits_{k=i+1}^{v-1}|\bar{sa}^k|)
        +d(n+m)
        \\
        \stepcounter{equation}\tag{\theequation}\label{fbn}
        &\phantom{=}+(|\textbf{F}|+|\textbf{G}|+d+1)(\sum\limits_{i=1}^{n-1}|tf^i|+\sum\limits_{i=1}^{m-1}|tg^i|)
        +\sum\limits_{i=1}^{n-1}|tf^i|\sum\limits_{i=1}^{m-1}|tg^i|
        +\hskip-2mm\sum\limits_{1\leq i<k\leq n-1}\hskip-2mm|tf^i||tf^k|\\&\phantom{=}+\hskip-2mm\sum\limits_{1\leq i<k\leq m-1}\hskip-2mm|tg^i||tg^k|
        +\hskip-2mm\sum\limits_{1\leq i<k\leq n}\hskip-2mm|\bar{sa}^i||\bar{sa}^k|+|\bar{sa}^v_{\geq j}||\bar{sa}^v_{<j-1}|+\hskip-2mm\sum\limits_{1\leq i<k\leq m}\hskip-2mm|\bar{sb}^i||\bar{sb}^k|+\sum\limits_{i=1}^{n}|\bar{sa}^i|\sum\limits_{i=1}^{m}|\bar{sb}^i|,
        \\
        \epsilon'&=(|\textbf{G}|+d+1)(\sum\limits_{i=1}^{v-1}|\bar{sa}^i|+|\bar{sa}^{v}_{\geq j}|)
        +|{}_{\llt(\Bar{y}^m)}G^1_{\rrt(\Bar{y}^1)}|(|\Bar{sa}^v_{\geq j}|+|\Bar{sa}^v_{<j-1}|)+|\Bar{sa}^v_{\geq j}||\Bar{sa}^v_{<j-1}|\\&+\phantom{=}\sum\limits_{i=1}^{m-1}|{}_{\llt(\bar{y}^i)}G_{\rrt(\bar{y}^{i+1})}^{-d}|(|\bar{sa}^v_{<j-1}|+\sum\limits_{i=v+1}^n|\bar{sa}^i|+\sum\limits_{i=v}^n|{}_{\llt(\bar{x}^i)}F_{\rrt(\bar{x}^{i+1})}^{-d}|+|{}_{\llt(\bar{y}^m)}G_{\rrt(\bar{y}^1)}^1|+d+1).
    \end{align*}    
\end{small}
On the other hand, we have that 
\[
\pi_{\MA[1]}(\mathcalboondox{j}_{\bar{\bar{x}}}(s_{d+1}F^{\bar{\bar{x}}}))\upperset{G,p,q}{\circ}\pi_{\MA[1]}(\mathcalboondox{j}_{\bar{\bar{y}}}(s_{d+1}G^{\bar{\bar{y}}}))=s\varphi^{\bar{\bar{x}}\upperset{v,j,\inn}{\sqcup}\bar{\bar{y}}}
\]
where
\allowdisplaybreaks
\begin{small}
\begin{align*}
        &(\pi_{\MA}\circ\varphi)^{\bar{\bar{x}}\upperset{v,j,\inn}{\sqcup}\bar{\bar{y}}}\\&\hspace{1cm}(\bar{sa}^n,tf^{n-1},\dots,\bar{sa}^{v+1},tf^{v},\bar{sa}_{< j-1}^{v},\bar{sb}^m,tg^{m-1},\dots,\bar{sb}^{2},tg^{1},\bar{sb}^{1},\bar{sa}_{\geq j}^{v},tf^{v-1},\dots,\bar{sa}^{2},tf^{1},\bar{sa}^1)
        \\&=(-1)^{\delta}\pi_{\MA}(\mathcalboondox{j}_{\bar{\bar{x}}}(s_{d+1}F^{\bar{\bar{x}}}))\\&\hspace{2cm}
        (\bar{sa}^n,tf^{n-1},\dots,\bar{sa}^{v+1},tf^{v},\bar{sa}_{<j-1}^{v},s\big((\bigotimes\limits_{i=1}^{m-1}(g^i\circ s_d)\otimes \id)s_{d}G^{\doubar{y}}(\bar{sb}^1,\dots,\bar{sb}^{m})\big),\\&\hspace{11cm}\bar{sa}_{\geq j}^{v},tf^{v-1},\dots,\bar{sa}^{2},tf^{1},\bar{sa}^1)
        \\&=(-1)^{\delta+\delta'} ((f^{1}\circ s_d)\otimes\dots\otimes (f^{v-1}\circ s_d)\otimes (f^{v}\circ s_d)\otimes \dots\otimes (f^{n-1}\circ s_d)\otimes \id)
        \\&\hspace{5cm}s_{d}({}_{\llt(\bar{x}^1)}F_{\rrt(\bar{x}^2)}^{-d},\dots {}_{\llt(\bar{x}^{v-1})}F_{\rrt(\bar{x}^v)}^{-d},\lambda_G,{}_{\llt(\bar{x}^v)}F_{\rrt(\bar{x}^{v+1})}^{-d}\dots{}_{\llt(\bar{x}^{n})}F_{\rrt(\bar{x}^1)}^{-d})
\end{align*}
\end{small}
where $\lambda_G=\big(\bigotimes\limits_{i=1}^{m-1}(g^i\circ s_d)\big)({}_{\llt(\bar{y}^{1})}G_{\rrt(\bar{y}^{2})}^{-d},...,{}_{\llt(\bar{y}^{m-1})}G_{\rrt(\bar{y}^m)}^{-d})\in \kk$ and 
\begin{small}
\allowdisplaybreaks
\begin{align*}
        \delta&=(|\textbf{G}|+d+1)(\sum\limits_{i=v+1}^{n}|\bar{sa}^i|+|\bar{sa}^{v}_{<j-1}|+\sum\limits_{i=v}^{n-1}|tf^i|)+\sum\limits_{i=1}^{m-1}|tg^i|\sum\limits_{k=i+1}^m|\bar{sb}^k|+d(m-1)
        \\&\phantom{=}+|\textbf{G}|\sum\limits_{i=1}^{m-1}|tg^i|+\hskip-2mm\sum\limits_{1\leq i<k\leq m-1}\hskip-2mm|tg^i||tg^k|+\hskip-2mm\sum\limits_{1\leq i<k\leq m}\hskip-2mm|\bar{sb}^i||\bar{sb}^k|
        \\
        \delta'&=\hskip-2mm\sum\limits_{1\leq i<k\leq n-1}\hskip-2mm|tf^i||tf^k|+\hskip-2mm\sum\limits_{1\leq i<k\leq n}\hskip-2mm|\bar{sa}^i||\bar{sa}^k|+\sum\limits_{\substack{i=1\\i\neq v}}^n|\bar{sa}^i||{}_{\llt(\bar{y}^m)}G^1_{\rrt(\bar{y}^1)}|
        \\&\phantom{=}+\sum\limits_{i=1}^{v-1}|tf^i|(\sum\limits_{k=i+1}^{n}|\bar{sa}^k|+|{}_{\llt(\bar{y}^m)}G^1_{\rrt(\bar{y}^1)}|)
        +\sum\limits_{i=v}^{n-1}|tf^i|\sum\limits_{k=i+1}^{n}|\bar{sa}^k|+d(n-1)
        \\&\phantom{=}+|\lambda_G|(\sum\limits_{i=1}^{v-1}|\bar{sa}^i|+|\bar{sa}^v_{\geq j}|+|{}_{\llt(\bar{y}^m)}G^1_{\rrt(\bar{y}^1)}|+\sum\limits_{i=1}^{v-1}|tf^i|+\sum\limits_{i=v}^n|{}_{\llt(\bar{x}^{i})}F_{\rrt(\bar{x}^{i+1})}^{-d}|)
        \\&\phantom{=}+|\mathbf{F}|\sum\limits_{i=1}^{n-1}|tf^i|
        +(d+1)\sum\limits_{i=1}^{m-1}|{}_{\llt(\bar{y}^i)}G_{\rrt(\bar{y}^{i+1})}^{-d}|+\sum\limits_{i=1}^{m-1}|tg^i|.
\end{align*}
\end{small}
Therefore, we have
\begin{small}
\allowdisplaybreaks
\begin{align*}
         &(\pi_{\MA}\circ\varphi)^{\bar{\bar{x}}\upperset{v,j,\inn}{\sqcup}\bar{\bar{y}}}
        \\&\hspace{1cm}(\bar{sa}^n,tf^{n-1},\dots,\bar{sa}^{v+1},tf^{v},\bar{sa}_{<j-1}^{v},\bar{sb}^m,tg^{m-1},\dots,\bar{sb}^{2},tg^{1},\bar{sb}^{1},\bar{sa}_{\geq j}^{v},tf^{v-1},\dots,\bar{sa}^{2},tf^{1},\bar{sa}^1)
        \\&=(-1)^{\delta+\delta'+\delta \dprime}
        ((f^{1}\circ s_d)\otimes\dots\otimes (f^{v-1}\circ s_d)\otimes \bigotimes\limits_{i=1}^{m-1}(g^{i}\circ s_d)\otimes (f^{v}\circ s_d)\otimes \dots\otimes (f^{n-1}\circ s_d)\otimes \id)
        \\&s_{d}({}_{\llt(\bar{x}^1)}F_{\rrt(\bar{x}^2)}^{-d},\dots {}_{\llt(\bar{x}^{v-1})}F_{\rrt(\bar{x}^v)}^{-d},{}_{\llt(\bar{y}^{1})}G_{\rrt(\bar{y}^{2})}^{-d},...,{}_{\llt(\bar{y}^{m-1})}G_{\rrt(\bar{y}^m)}^{-d},{}_{\llt(\bar{x}^v)}F_{\rrt(\bar{x}^{v+1})}^{-d}\dots{}_{\llt(\bar{x}^{n})}F_{\rrt(\bar{x}^1)}^{-d})
    \end{align*}    
\end{small}
with 
\begin{small}
    \begin{equation}
        \delta\dprime=\sum\limits_{i=1}^{m-1}|tg^i|(d+\sum\limits_{i=v}^{n-1}|tf^i|+\sum\limits_{i=1}^{v-1}|{}_{\llt(\bar{x}^i)}F^{-d}_{\rrt(\bar{x}^{i+1})}|).
    \end{equation}
\end{small}
One can easily check that $\epsilon + \epsilon' = \delta + \delta'+\delta\dprime$ $\operatorname{mod} 2$.

Now, write $s_{d+1}F^{\bar{\bar{x}}'}\displaystyle\upperset{\substack{\nec,v,j\\ \inn}}{\circ}s_{d+1}G^{\bar{\bar{y}}}=s_{d+1}H_2$ for $\doubar{x}'=(\bar{x}^1\sqcup\bar{x}^n,\dots,\bar{x}^{n-1})$ and $\doubar{z}=\bar{\bar{x}}\upperset{v,j,\inn}{\sqcup}\bar{\bar{y}}$.
We have that 
\begin{small}
\begin{equation}
    \begin{split}
        &(\pi_{\MA^*[d-1]}\circ \psi)^{\doubar{z}}
        \\&(\bar{sa}^n,tf^{n-1},\dots,\bar{sa}^{v+1},tf^{v},\bar{sa}_{<j-1}^{v},\bar{sb}^m,tg^{m-1},\dots,tg^{1},\bar{sb}^{1},\bar{sa}_{\geq j}^{v},tf^{v-1},\dots,\bar{sa}^{2},tf^{1},\bar{sa}^1)(s_{1-d}c)
        \\&=(-1)^{\gamma}
         \big((f^{1}\circ s_d)\otimes\dots\otimes (f^{v-1}\circ s_d)\otimes (g^{1}\circ s_d)\otimes \dots\otimes (g^{m-1} \circ s_d)\otimes (f^{v}\circ s_d)\otimes \dots\otimes (f^{n-1}\circ s_d)\big)
        \\&\phantom{=}H_2(\bar{sa}^1\otimes sc\otimes \bar{sa}^n,\bar{sa}^2,\dots,\bar{sa}^{v-1},\bar{sa}^v_{\geq j},\bar{sb}^1,\dots,\bar{sb}^{m-1},\bar{sb}^{m},\bar{sa}_{<j-1}^{v},\dots,\bar{sa}^{n-1})
        \\
        &=(-1)^{\gamma+\gamma'}\big((f^{1}\circ s_d)\otimes\dots\otimes (f^{v-1}\circ s_d)\otimes \bigotimes\limits_{i=1}^n(g^{i}\circ s_d)\otimes (f^{v}\circ s_d)\otimes \dots\otimes (f^{n-1}\circ s_d)\big)
        \\&\phantom{=} s_{d}({}_{\llt(\bar{x}^1)}F_{\rrt(\bar{x}^2)}^{-d},\dots {}_{\llt(\bar{x}^{v-1})}F_{\rrt(\bar{x}^v)}^{-d},{}_{\llt(\bar{y}^{1})}G_{\rrt(\bar{y}^{2})}^{-d},...,{}_{\llt(\bar{y}^{m-1})}G_{\rrt(\bar{y}^m)}^{-d},{}_{\llt(\bar{x}^v)}F_{\rrt(\bar{x}^{v+1})}^{-d}\dots{}_{\llt(\bar{x}^{n-1})}F_{\rrt(\bar{x}^n)}^{-d})
    \end{split}
\end{equation}
\end{small}
\hskip -0.6mm for $\bar{sa}^i\in\MA[1]^{\otimes\bar{x}^i}$, $\bar{sb}^i\in\MA[1]^{\otimes\bar{y}^i}$, $c\in {}_{\rrt(\bar{x}^1)}\MA_{\llt(\bar{x}^n)}$, $tf^i\in{}_{\rrt(\bar{x}^{i+1})}\MA^{*}_{\llt(\bar{x}^{i})}[d]$ and $tg^i\in{}_{\rrt(\bar{y}^{i+1})}\MA^{*}_{\llt(\bar{y}^{i})}[d]$,
with
\begin{small}
\allowdisplaybreaks
\begin{align*}
        \gamma&=\epsilon +d(|\mathbf{F}|+|\mathbf{G}|+d+1+\sum\limits_{i=1}^{n}|\bar{sa}^i|+\sum\limits_{i=1}^m|\bar{sb}^i|)+(|\bar{sa}^n|+|sc|)(\sum\limits_{i=2}^{n-1}|\bar{sa}^i|+\sum\limits_{i=1}^m|\bar{sb}^i|)+|\bar{sa}^n||sc|
        \\
        \gamma'&=(|\textbf{G}|+d+1)(\sum\limits_{i=1}^{v-1}|\bar{sa}^i|+|\bar{sa}^{v}_{\geq j}|+|\bar{sa}^n|+|sc|)
        \\&\phantom{=}+|{}_{\llt(\Bar{y}^m)}G^1_{\rrt(\Bar{y}^1)}|(|\Bar{sa}^v_{\geq j}|+|\Bar{sa}^v_{<j-1}|)+|\Bar{sa}^v_{\geq j}||\Bar{sa}^v_{<j-1}|
        \\&\phantom{=}+\sum\limits_{i=1}^{m-1}|{}_{\llt(\bar{y}^i)}G_{\rrt(\bar{y}^{i+1})}^{-d}|(|\bar{sa}^v_{<j-1}|+\sum\limits_{i=v+1}^{n-1}|\bar{sa}^i|+\sum\limits_{i=v}^{n}|{}_{\llt(\bar{x}^i)}F_{\rrt(\bar{x}^{i+1})}^{-d}|+|{}_{\llt(\bar{y}^m)}G_{\rrt(\bar{y}^1)}^1|+d+1).
    \end{align*}    
\end{small}
On the other hand, we have that 
\begin{small}
\allowdisplaybreaks
\begin{align*}
        &(\pi_{\MA^*[d-1]}\circ\varphi)^{\doubar{z}}\\&\hspace{1cm}(\bar{sa}^n,tf^{n-1},\dots,\bar{sa}^{v+1},tf^{v},\bar{sa}_{> j}^{v},\bar{sb}^m,tg^{m-1},\dots,tg^{1},\bar{sb}^{1},\bar{sa}_{\leq j}^{v},tf^{v-1},\dots,\bar{sa}^{2},tf^{1},\bar{sa}^1)(s_{1-d}c)
        \\&=(-1)^{\delta+\Delta'} \big((f^{1}\circ s_d)\otimes\dots\otimes (f^{v-1}\circ s_d)\otimes (f^{v}\circ s_d)\otimes \dots\otimes (f^{n-1}\circ s_d)\big)
        \\&\phantom{=}\hspace{4cm}s_{d}({}_{\llt(\bar{x}^1)}F_{\rrt(\bar{x}^2)}^{-d},\dots {}_{\llt(\bar{x}^{v-1})}F_{\rrt(\bar{x}^v)}^{-d},\lambda_G,{}_{\llt(\bar{x}^v)}F_{\rrt(\bar{x}^{v+1})}^{-d}\dots{}_{\llt(\bar{x}^{n-1})}F_{\rrt(\bar{x}^n)}^{-d})
\end{align*}
\end{small}
with 
\begin{small}
\begin{equation}
    \begin{split}
       \Delta'&=\delta'+d(|\mathbf{F}|+\sum\limits_{i=1}^n|\bar{sa}^i|+\sum\limits_{i=1}^m|\bar{sb}^i|+|\mathbf{G}|+d+1+\sum\limits_{i=1}^{m-1}|tg^i|)\\&\phantom{=}+(|\bar{sa}^n|+|sc|)(\sum\limits_{i=2}^{n-1}|\bar{sa}^i|+\sum\limits_{i=1}^m|\bar{sb}^i|+|\mathbf{G}|+d+1+\sum\limits_{i=1}^{m-1}|tg^i|)+|\bar{sa}^n||sc|+|\lambda_G||\bar{sa}^n|.
    \end{split}
\end{equation}
\end{small}
Therefore, we have
\begin{small}
\allowdisplaybreaks
\begin{align*}
         &\varphi_2^{\doubar{z}}(\bar{sa}^n,tf^{n-1},\dots,\bar{sa}^{v+1},tf^{v},\bar{sa}_{> j}^{v},\bar{sb}^m,tg^{m-1},\dots,tg^{1},\bar{sb}^{1},\bar{sa}_{\leq j}^{v},tf^{v-1},\dots,\bar{sa}^{2},tf^{1},\bar{sa}^1)(s_{1-d}c)
        \\&=(-1)^{\delta+\Delta'+\Delta \dprime}
        \big((f^{1}\circ s_d)\otimes\dots\otimes (f^{v-1}\circ s_d)\otimes \bigotimes\limits_{i=1}^{m-1}(g^{i}\circ s_d)\otimes (f^{v}\circ s_d)\otimes \dots\otimes (f^{n-1}\circ s_d)\big)
        \\&\phantom{=}s_{d}({}_{\llt(\bar{x}^1)}F_{\rrt(\bar{x}^2)}^{-d},\dots {}_{\llt(\bar{x}^{v-1})}F_{\rrt(\bar{x}^v)}^{-d},{}_{\llt(\bar{y}^{1})}G_{\rrt(\bar{y}^{2})}^{-d},...,{}_{\llt(\bar{y}^{m-1})}G_{\rrt(\bar{y}^m)}^{-d},{}_{\llt(\bar{x}^v)}F_{\rrt(\bar{x}^{v+1})}^{-d}\dots{}_{\llt(\bar{x}^{n-1})}F_{\rrt(\bar{x}^n)}^{-d})
    \end{align*}    
\end{small}
where 
\begin{small}
    \begin{equation}
        \Delta\dprime=\sum\limits_{i=1}^{m-1}|tg^i|(\sum\limits_{i=v}^{n-1}|tf^i|+\sum\limits_{i=1}^{v-1}|{}_{\llt(\bar{x}^i)}F^{-d}_{\rrt(\bar{x}^{i+1})}|).
    \end{equation}
\end{small}
It is straightforward to check that $\gamma + \gamma' = \delta + \Delta'+\Delta\dprime$ $\operatorname{mod} 2$, which proves \eqref{eq:j-bracket-1}. 

We now prove the identity \eqref{eq:j-bracket-2}.
Given $\doubar{x}=(\bar{x}^1,\dots,\bar{x}^n),\doubar{y}=(\bar{y}^{1},\dots,\bar{y}^m)\in \doubar{\MO}$, both the compositions \[s_{d+1}F^{\bar{\bar{x}}}\displaystyle\upperset{\substack{\nec,v,j\\ \out}}{\circ}s_{d+1}G^{\bar{\bar{y}}}\] and \[\mathcalboondox{j}_{\bar{\bar{y}}}(s_{d+1}G^{\bar{\bar{y}}})\upperset{G,p,q}{\circ}\pi_{\MA^*}(\mathcalboondox{j}_{\bar{\bar{x}}}(s_{d+1}F^{\bar{\bar{x}}}))\] are zero if there are no $v\in\llbracket 1,m-1\rrbracket$ and $j \in\llbracket 2,\llg(\bar{x}^1)\rrbracket$ such that $\bar{x}^1_{j-1}=\llt(\bar{y}^{v})$ and $\bar{x}^1_{j}=\rrt(\bar{y}^{v+1})$. Suppose that such $v\in\llbracket 1,m\rrbracket$ and $j \in\llbracket 2,\llg(\bar{x}^1)\rrbracket$ exist.

We write $s_{d+1}F^{\bar{\bar{x}}}\displaystyle\upperset{\substack{\nec,v,j\\ \out}}{\circ}s_{d+1}G^{\bar{\bar{y}}}=s_{d+1}H_3$.
Then, we have that 
\[
\mathcalboondox{j}_{\bar{\bar{x}}\upperset{v,j,\out}{\sqcup}\bar{\bar{y}}}(s_{d+1}F^{\bar{\bar{x}}}\displaystyle\upperset{\substack{\nec,v,j\\ \out}}{\circ}s_{d+1}G^{\bar{\bar{y}}})=s\psi_1^{\bar{\bar{x}}\upperset{v,j,\out}{\sqcup}\bar{\bar{y}}}
\]
with
\allowdisplaybreaks
\begin{small}
\begin{align*}
     &(\pi_{\MA}\circ \psi_1)^{\bar{\bar{x}}\upperset{v,j,\out}{\sqcup}\bar{\bar{y}}}
     (\bar{sb}^m,tg^{m-1},\dots,\bar{sb}^{v+1},\bar{sa}^{1}_{\geq j},tf^n,\bar{sa}^n,tf^{n-1},\dots,tf^1,\bar{sa}^1_{< j-1},\bar{sb}^{v},tg^{v-1},\dots,\bar{sb}^{1})
     \\&=(-1)^{\epsilon_1}
     ((g^{1}\circ s_d) \otimes \dots \otimes (g^{v-1}\circ s_d) \otimes \bigotimes\limits_{i=1}^n(f^i\circ s_d) \otimes (g^{v+1}\circ s_d) \otimes \dots \otimes (g^{m-1}\circ s_d)\otimes \id)
     \\& \phantom{=}s_dH_3(\bar{sb}^1,\dots,\bar{sb}^{v},\bar{sa}^1_{< j-1},\dots,\bar{sa}^n,\bar{sa}^{1}_{\geq j},\bar{sb}^{v+1},\dots,\bar{sb}^{m})
     \\&=(-1)^{\epsilon_1+\epsilon'_1}
      ((g^{1}\circ s_d) \otimes \dots \otimes (g^{v-1}\circ s_d) \otimes \bigotimes\limits_{i=1}^n(f^i\circ s_d) \otimes (g^{v+1}\circ s_d) \otimes \dots \otimes (g^{m-1}\circ s_d)\otimes \id)
     \\& \phantom{=}s_{d}({}_{\llt(\bar{y}^{1})}G_{\rrt(\bar{y}^{2})}^{-d},\dots,{}_{\llt(\bar{y}^{v-1})}G_{\rrt(\bar{y}^{v})}^{-d},{}_{\llt(\bar{x}^{1})}F_{\rrt(\bar{x}^2)}^{-d},\dots,{}_{\llt(\bar{x}^{n})}F_{\rrt(\bar{x}^{1})}^{-d},{}_{\llt(\bar{y}^{v+1})}G_{\rrt(\bar{y}^{v+2})}^{-d},\dots,{}_{\llt(\bar{y}^{m})}G_{\rrt(\bar{y}^1)}^{-d})
\end{align*}
\end{small}
with
\begin{small}
\begin{equation}
    \begin{split}
        \epsilon_1&=(|\mathbf{F}|+|\mathbf{G}|+d+1)(\sum\limits_{i=1}^{m-1}|tg^i|+\sum\limits_{i=1}^{n}|tf^i|)+d(n+m)+\sum\limits_{i=v+1}^{m-1}|tg^i|\sum\limits_{k=i+1}^m|\bar{sb}^k|+|\bar{sa}^{1}_{\geq j}||\bar{sa}^{1}_{<j-1}|
        \\&\phantom{=}+\hskip-2mm\sum\limits_{1\leq i < k \leq n}\hskip-2mm|\bar{sa}^i||\bar{sa}^k|+\sum\limits_{i=1}^{n}|tf^i|(\sum\limits_{k=i+1}^n|\bar{sa}^k|+|\bar{sa}^1_{\geq j}|+\hskip-1mm\sum\limits_{i=v+1}^m|\bar{sb}^i|)+\sum\limits_{i=1}^{v-1}|tg^i|(\sum\limits_{i=1}^n|\bar{sa}^i|+\hskip-1mm\sum\limits_{k=i+1}^m|\bar{sb}^k|)
        \\&\phantom{=}+\hskip-2mm\sum\limits_{1\leq i < k \leq m}\hskip-2mm|\bar{sb}^i||\bar{sb}^k|+\hskip-2mm\sum\limits_{1\leq i < k \leq m-1}\hskip-2mm|tg^i||tg^k|+\hskip-2mm\sum\limits_{1\leq i < k \leq n}\hskip-2mm|tf^i||tf^k|
        +\sum\limits_{i=1}^{m-1}|tg^i|\sum\limits_{i=1}^{n}|tf^i|+\sum\limits_{i=1}^{m}|\bar{sb}^i|\sum\limits_{i=1}^{n}|\bar{sa}^i|
        \\&(|\mathbf{F}|+d+1)(d+1),
        \\\epsilon'_1&=\sum\limits_{i=1}^n|\bar{sa}^i|(\sum\limits_{i=v+1}^m|\bar{sb}^i|+\sum\limits_{i=v+1}^m|{}_{\llt(\bar{y}^{i})}G_{\rrt(\bar{y}^{i+1})}^{-d}|)
        +(|\mathbf{F}|+d+1)\sum\limits_{i=1}^{v-1}|{}_{\llt(\bar{y}^{i})}G_{\rrt(\bar{y}^{i+1})}^{-d}|
        \\&\phantom{=}+(d+1)|\mathbf{F}|+d+1+|\bar{sa}^1_{\geq j}|\sum\limits_{i=2}^n|\bar{sa}^i|+|\bar{sa}^1_{<j-1}||{}_{\llt(\bar{y}^{v})}G_{\rrt(\bar{y}^{v+1})}^{1}|.
    \end{split}
\end{equation}
\end{small}
On the other hand, we have that 
\[ \pi_{\MA[1]}(\mathcalboondox{j}_{\bar{\bar{y}}}(s_{d+1}G^{\bar{\bar{y}}}))\upperset{G,p,q}{\circ}\pi_{\MA^*[d]}(\mathcalboondox{j}_{\bar{\bar{x}}}(s_{d+1}F^{\bar{\bar{x}}}))=s\varphi_1^{\bar{\bar{x}}\upperset{v,j,\out}{\sqcup}\bar{\bar{y}}}\]
where
\begin{small}
\begin{equation}
    \begin{split}
        &(\pi_{\MA}\circ \varphi_1)^{\bar{\bar{x}}\upperset{v,j,\out}{\sqcup}\bar{\bar{y}}}(\bar{sb}^m,tg^{m-1},\dots,\bar{sb}^{v+1},\bar{sa}^{1}_{\geq j},tf^n,\bar{sa}^n,tf^{n-1},\dots,\bar{sa}^1_{< j-1},\bar{sb}^{v},tg^{v-1},\dots,\bar{sb}^{1})
        \\&=(-1)^{\delta_1}\mathcalboondox{j}_{\bar{\bar{y}}}(s_{d+1}G^{\bar{\bar{y}}})
        (\bar{sb}^m,tg^{m-1},\dots,\bar{sb}^{v+1},tg,\bar{sb}^{v},tg^{v-1},\dots,\bar{sb}^{1})
    \end{split}
\end{equation}
\end{small}
where $tg=\pi_{\MA^{*}}(\mathcalboondox{j}_{\bar{\bar{x}}}(s_{d+1}F^{\bar{\bar{x}}}))(\bar{sa}^{1}_{\geq j},tf^n,\bar{sa}^n,tf^{n-1},\bar{sa}^{n-1},\dots,\bar{sa}^1_{< j-1})\in \MA^{*}[d]$
and 
\begin{small}
\begin{equation}
        \delta_1=(|\textbf{F}|+d+1)(\sum\limits_{i=v+1}^{m}|\bar{sb}^i|+\sum\limits_{i=v+1}^{m-1}|tg^i|).
        \\
\end{equation}
\end{small}
Therefore, we have that
\begin{small}
\begin{equation}
    \begin{split} &(\pi_{\MA}\circ \varphi_1)^{\bar{\bar{x}}\upperset{v,j,\out}{\sqcup}\bar{\bar{y}}}(\bar{sb}^m,tg^{m-1},\dots,\bar{sb}^{v+1},\bar{sa}^{1}_{\geq j},tf^n,\bar{sa}^n,tf^{n-1},\bar{sa}^{n-1},\dots,\bar{sa}^1_{< j-1},\bar{sb}^{v},tg^{v-1},\dots,\bar{sb}^{1})
        \\&=(-1)^{\delta_1+\delta'_1}((g^{1}\circ s_d) \otimes \dots \otimes (g^{v-1}\circ s_d) \otimes (g\circ s_d) \otimes (g^{v+1}\circ s_d) \otimes \dots \otimes (g^m\circ s_d)\otimes \id)
        \\&\phantom{=} s_{d}({}_{\llt(\bar{y}^{1})}G_{\rrt(\bar{y}^{2})}^{-d},\dots,{}_{\llt(\bar{y}^{v-1})}G_{\rrt(\bar{y}^{v})}^{-d},{}_{\llt(\bar{y}^{v})}G_{\rrt(\bar{y}^{v+1})}^{-d},\dots,{}_{\llt(\bar{y}^{m})}G_{\rrt(\bar{y}^1)}^{-d})
    \end{split}
\end{equation}
\end{small}
where 
\begin{small}
\begin{equation}
    \begin{split}
        \delta'_1&=|tg|\sum\limits_{k=v+1}^m|\bar{sb}^k|+\sum\limits_{i=1}^{m-1}|tg^i|\sum\limits_{k=i+1}^m|\bar{sb}^k|
        +\hskip-2mm\sum\limits_{1\leq i < k \leq m}\hskip-2mm|\bar{sb}^i||\bar{sb}^k|+\hskip-2mm\sum\limits_{1\leq i < k \leq m-1}\hskip-2mm|tg^i||tg^k|
        \\&\phantom{=}+\sum\limits_{i=1}^{m-1}|tg^i||tg|+|\mathbf{G}|(\sum\limits_{i=1}^{m-1}|tg^i|+|tg|)+d(m-1).
    \end{split}
\end{equation}
\end{small}
Furthermore, by definition, we have that 
\begin{equation}
    \begin{split}
        (g\circ s_d)({}_{\llt(\bar{y}^{v})}G_{\rrt(\bar{y}^{v+1})}^{-d})=(-1)^{\Delta_1}(\bigotimes\limits_{i=1}^n(f^{i}\circ s_d) )({}_{\llt(\bar{x}^1)}F_{\rrt(\bar{x}^2)}^{-d}\otimes\dots \otimes {}_{\llt(\bar{x}^{n})}F_{\rrt(\bar{x}^1)}^{-d})
    \end{split}
\end{equation}
where 
\begin{small}
\begin{equation}
    \begin{split}
        \Delta_1=&\sum\limits_{1\leq i < k \leq n}|\bar{sa}^i||\bar{sa}^k|+|\bar{sa}^1_{\geq j}||\bar{sa}^1_{< j-1}|+(|{}_{\llt(\bar{y}^v)}G^1_{\rrt(\bar{y}^{v+1})}|+|\Bar{sa}^1_{\geq j}|)\sum\limits_{i=2}^n|\bar{sa}^i|+\sum\limits_{1\leq i < k \leq n}|tf^i||tf^k|\\&\phantom{=}+dn+(|\mathbf{F}|+d+1)\sum\limits_{i=1}^{n}|tf^i|+\sum\limits_{i=1}^{n}|tf^i|(\sum\limits_{k=i+1}^n|\bar{sa}^k|+|\bar{sa}^1_{\geq j}|)+|\bar{sa}^{1}_{\geq j}||{}_{\llt(\bar{y}^{v})}G_{\rrt(\bar{y}^{v+1})}^1|.
    \end{split}
\end{equation}
\end{small}
Finally, we have that
\begin{small}
\begin{equation}
    \begin{split}
        &(\pi_{\MA}\circ \varphi_1)^{\bar{\bar{x}}\upperset{v,j,\out}{\sqcup}\bar{\bar{y}}}(\bar{sb}^m,tg^{m-1},\dots,\bar{sb}^{v+1},tg^{v+1},\bar{sa}^{1}_{\geq j},tf^n,\bar{sa}^n,tf^{n-1},\dots,\bar{sa}^1_{< j-1},\bar{sb}^{v},tg^{v-1},\dots,\bar{sb}^{1})
        \\&=(-1)^{\delta_1+\delta'_1+\Delta_1+\Delta'_1}
        \\&((g^{1}\circ s_d) \otimes \dots \otimes (g^{v-1}\circ s_d) \otimes \bigotimes\limits_{i=1}^n(f^i\circ s_d)\otimes (g^{v+1}\circ s_d) \otimes \dots \otimes (g^{m-1}\circ s_d)\otimes \id)
        \\&\phantom{=} s_{d}({}_{\llt(\bar{y}^{1})}G_{\rrt(\bar{y}^{2})}^{-d},\dots,{}_{\llt(\bar{y}^{v-1})}G_{\rrt(\bar{y}^{v})}^{-d},{}_{\llt(\bar{x}^{1})}F_{\rrt(\bar{x}^2)}^{-d},\dots,{}_{\llt(\bar{x}^{n})}F_{\rrt(\bar{x}^{1})}^{-d},{}_{\llt(\bar{y}^{v+1})}G_{\rrt(\bar{y}^{v+2})}^{-d},\dots,{}_{\llt(\bar{y}^{m-1})}G_{\rrt(\bar{y}^1)}^{-d})
    \end{split}
\end{equation}
\end{small}
where 
\begin{small}
\begin{equation}
    \begin{split}
    \Delta'_1=(|tg|+\sum\limits_{i=1}^n|tf^i|)(\sum\limits_{i=v+1}^{m-1}|tg^i|+d+\sum\limits_{i=1}^{v-1}|{}_{\llt(\bar{y}^i)}G_{\rrt(\bar{y}^{i+1})}^{-d}|).
\end{split}
\end{equation}
\end{small}
It is straightforward to check that \[\epsilon_1+\epsilon'_1+\delta_1+\delta'_1+\Delta_1+\Delta'_1 = 1+(|\mathbf{F}|+d+1)(|\mathbf{G}|+d+1)\operatorname{mod} 2.\]

We now write $s_{d+1}H_4=s_{d+1}F^{\bar{\bar{x}}}\displaystyle\upperset{\substack{\nec,v,j\\ \out}}{\circ}s_{d+1}G^{\bar{\bar{y}}}$ and $\doubar{z}'=\bar{\bar{x}}\upperset{v,j,\out}{\sqcup}\bar{\bar{y}}$.
We thus have that 
\allowdisplaybreaks
\begin{small}
\begin{align*}
     &(\pi_{\MA^*[d-1]}\circ\psi_1)^{\doubar{z}'}\\&\hspace{1cm}(\bar{sb}^m,tg^{m-1},\dots,\bar{sb}^{v+1},\bar{sa}^{1}_{\geq j},tf^n,\bar{sa}^n,tf^{n-1},\dots,tf^1,\bar{sa}^1_{< j-1},\bar{sb}^{v},tg^{v-1},\dots,\bar{sb}^{1})(s_{1-d}c)
     \\&=(-1)^{\gamma_1}
     ((g^{1}\circ s_d) \otimes \dots \otimes (g^{v-1}\circ s_d) \otimes \bigotimes\limits_{i=1}^n(f^i\circ s_d)\otimes (g^{v+1}\circ s_d) \otimes \dots \otimes (g^{m-1}\circ s_d))
     \\&\phantom{=} \hspace{3cm}H_4(\bar{sb}^1\otimes sc\otimes \bar{sb}^m,\dots,\bar{sb}^{v},\bar{sa}^1_{< j-1},\dots,\bar{sa}^n,\bar{sa}^{1}_{\geq j},\bar{sb}^{v+1},\dots,\bar{sb}^{m-1})
     \\&=(-1)^{\gamma_1+\gamma'_1}
     ((g^{1}\circ s_d) \otimes \dots \otimes (g^{v-1}\circ s_d) \otimes \bigotimes\limits_{i=1}^n(f^i\circ s_d)\otimes (g^{v+1}\circ s_d) \otimes \dots \otimes (g^{m-1}\circ s_d))
     \\&\phantom{=} ({}_{\llt(\bar{y}^{1})}G_{\rrt(\bar{y}^{2})}^{-d},\dots,{}_{\llt(\bar{y}^{v-1})}G_{\rrt(\bar{y}^{v})}^{-d},{}_{\llt(\bar{x}^{1})}F_{\rrt(\bar{x}^2)}^{-d},\dots,{}_{\llt(\bar{x}^{n})}F_{\rrt(\bar{x}^{1})}^{-d},{}_{\llt(\bar{y}^{v+1})}G_{\rrt(\bar{y}^{v+2})}^{-d},\dots,{}_{\llt(\bar{y}^{m})}G_{\rrt(\bar{y}^1)}^{-d})
\end{align*}
\end{small}
with 
\begin{small}
\begin{equation}
    \begin{split}
        \gamma_1&=\epsilon_1+d(|\mathbf{F}|+|\mathbf{G}|+\sum\limits_{i=1}^n|\bar{sa}^i|+\sum\limits_{i=1}^m|\bar{sb}^i|)+(|\bar{sb}^m|+|sc|)(\sum\limits_{i=1}^n|\bar{sa}^i|+\sum\limits_{i=2}^{m-1}|\bar{sb}^i|)+|\bar{sb}^m||sc|
        \\ \gamma'_1&=\sum\limits_{i=1}^n|\bar{sa}^i|(\sum\limits_{i=v+1}^{m-1}|\bar{sb}^i|+\sum\limits_{i=v+1}^{m-1}|{}_{\llt(\bar{y}^{i})}G_{\rrt(\bar{y}^{i+1})}^{-d}|)
        +(|\mathbf{F}|+d+1)\sum\limits_{i=1}^{v-1}|{}_{\llt(\bar{y}^{i})}G_{\rrt(\bar{y}^{i+1})}^{-d}|
        \\&+|\bar{sa}^1_{\geq j}|\sum\limits_{i=2}^n|\bar{sa}^i|+|\bar{sa}^1_{<j-1}||{}_{\llt(\bar{y}^{v})}G_{\rrt(\bar{y}^{v+1})}^{1}|+(|\mathbf{F}|+d+1)(d+1).
    \end{split}
\end{equation}
\end{small}
On the other hand, we have that 
\allowdisplaybreaks
\begin{align*}
        &(\pi_{\MA^*[d-1]}\circ\varphi_1)^{\doubar{z}'}
        \\&(\bar{sb}^m,tg^{m-1},\dots,\bar{sb}^{v+1},tg^{v+1},\bar{sa}^{1}_{\geq j},\bar{sa}^n,tf^{n-1},\bar{sa}^{n-1},\dots,\bar{sa}^1_{< j-1},\bar{sb}^{v},tg^{v-1},\dots,\bar{sb}^{1})(s_{1-d}c)
        \\&=(-1)^{\delta_1+\Gamma_1}((g^{1}\circ s_d) \otimes \dots \otimes (g^{v-1}\circ s_d) \otimes (g\circ s_d) \otimes (g^{v+1}\circ s_d) \otimes \dots \otimes (g^{m-1}\circ s_d))
        \\&\phantom{=}\hspace{3cm} s_{d}({}_{\llt(\bar{y}^{1})}G_{\rrt(\bar{y}^{2})}^{-d},\dots,{}_{\llt(\bar{y}^{v-1})}G_{\rrt(\bar{y}^{v})}^{-d},{}_{\llt(\bar{y}^{v})}G_{\rrt(\bar{y}^{v+1})}^{-d},\dots,{}_{\llt(\bar{y}^{m-1})}G_{\rrt(\bar{y}^m)}^{-d})
\end{align*}
where 
\begin{small}
\begin{equation}
    \begin{split}
        \Gamma_1&=\delta'_1+ d(|\mathbf{G}|+\sum\limits_{i=1}^m|\bar{sb}^i|)+(|\bar{sb}^m|+|sc|)(\sum\limits_{i=1}^n|\bar{sa}^i|+\sum\limits_{i=2}^{m-1}|\bar{sb}^i|)+|\bar{sb}^m||sc|.
    \end{split}
\end{equation}
\end{small}
Finally, we have that
\begin{small}
\begin{equation}
    \begin{split}
        &(\pi_{\MA^*[d-1]}\circ\varphi_1)^{\doubar{z}'}\\&(\bar{sb}^m,tg^{m-1},\dots,\bar{sb}^{v+1},tg^{v+1},\bar{sa}^{1}_{\geq j},tf^n,\bar{sa}^n,tf^{n-1},\dots,\bar{sa}^1_{< j-1},\bar{sb}^{v},tg^{v-1},\dots,\bar{sb}^{1})(s_{1-d}c)
        \\&=(-1)^{\delta_1+\Gamma_1+\Gamma'_1+\Delta_1}
        ((g^{1}\circ s_d) \otimes \dots \otimes (g^{v-1}\circ s_d) \otimes \bigotimes\limits_{i=1}^n(f^i\circ s_d)\otimes (g^{v+1}\circ s_d) \otimes \dots \otimes (g^{m-1}\circ s_d))
        \\&({}_{\llt(\bar{y}^{1})}G_{\rrt(\bar{y}^{2})}^{-d},\dots,{}_{\llt(\bar{y}^{v-1})}G_{\rrt(\bar{y}^{v})}^{-d},{}_{\llt(\bar{x}^{1})}F_{\rrt(\bar{x}^2)}^{-d},\dots,{}_{\llt(\bar{x}^{n})}F_{\rrt(\bar{x}^{1})}^{-d},{}_{\llt(\bar{y}^{v+1})}G_{\rrt(\bar{y}^{v+2})}^{-d},\dots,{}_{\llt(\bar{y}^{m})}G_{\rrt(\bar{y}^1)}^{-d})
    \end{split}
\end{equation}
\end{small}
where 
\begin{small}
\begin{equation}
    \begin{split}
    \Gamma'_1=(|tg|+\sum\limits_{i=1}^n|tf^i|)(\sum\limits_{i=v+1}^{m-1}|tg^i|+\sum\limits_{i=1}^{v-1}|{}_{\llt(\bar{y}^i)}G_{\rrt(\bar{y}^{i+1})}^{-d}|)
\end{split}
\end{equation}
\end{small}
which concludes the proof.
\end{proof}
The relation between the necklace product and the Gerstenhaber product gives us the following result.
\begin{corollary}
    \label{corollary:Lie-bracket-Bar}
    The necklace bracket $[-,-]_{\nec}$ introduced in Definition \ref{def:necklace-bracket} gives a graded Lie algebra structure on
    $\Multi^{\bullet}_d(\MA)^{C_{\llg(\bullet)}}[d+1]$.
\end{corollary}
\begin{proof}
    For $s_{d+1}\mathbf{F}, s_{d+1}\mathbf{G}\in\Multi^{\bullet}_d(\MA)^{C_{\llg(\bullet)}}[d+1]$, we have that
\begin{equation}
    \mathcalboondox{j}([s_{d+1}\textbf{F},s_{d+1}\textbf{G}]_{\nec})=[\mathcalboondox{j}(s_{d+1}\textbf{F}),\mathcalboondox{j}(s_{d+1}\textbf{G})]_{G}.
\end{equation}
Therefore, since the Gerstenhaber bracket satisfies the Jacobi identity and $\mathcalboondox{j}$ is injective we conclude that $[-,-]_{\nec}$ is a graded Lie bracket.
\end{proof}

\begin{remark}
    In terms of diagrams, given an element $s_{d+1}\mathbf{F}\in\Multi_d^{\bullet}(\MA)[d+1]$, the two elements $\pi_{\MA}\circ \mathcalboondox{j}(s_{d+1}\mathbf{F})$ and $\pi_{\MA^*}\circ \mathcalboondox{j}(s_{d+1}\mathbf{F})$ are equal to $\mathcal{E}(\mathcalboondox{D})$ and $-\mathcal{E}(\mathcalboondox{D'})$ where $\mathcalboondox{D}$ and $\mathcalboondox{D'}$ are the diagrams

    \begin{minipage}{21cm}
        \begin{tikzpicture}[line cap=round,line join=round,x=1.0cm,y=1.0cm]
\clip(-4,-1.2) rectangle (3.7,1.7);
     \draw [line width=0.5pt] (0.,0.) circle (0.5cm);
     \draw [line width=1.1pt,rotate=0] [->, >= stealth] (0.5,0) -- (0.9,0);
     \draw [rotate=-90] [->, >= stealth] (0.5,0) -- (0.9,0);
     \draw [rotate=90] [->, >= stealth] (0.5,0) -- (0.9,0);
     \draw [rotate=180] [->, >= stealth] (0.5,0) -- (0.9,0);
     \shadedraw[rotate=45,shift={(0.5,0)}] \doublefleche;
    \shadedraw[rotate=-45,shift={(0.5,0)}] \doublefleche;
    \shadedraw[rotate=-135,shift={(0.5,0)}] \doublefleche;
\begin{scriptsize}
\draw [fill=black] (-0.45,0.4) circle (0.3pt);
\draw [fill=black] (-0.25,0.55) circle (0.3pt);
\draw [fill=black] (-0.55,0.2) circle (0.3pt);
\end{scriptsize}
\draw (0,0.25) node[anchor=north]{$\mathbf{F}$};
\draw (3,0.25) node[anchor=north]{and};
\end{tikzpicture}
\begin{tikzpicture}[line cap=round,line join=round,x=1.0cm,y=1.0cm]
\clip(-2.5,-1.2) rectangle (10.586743541133627,1.7);
     \draw [line width=0.5pt] (0.,0.) circle (0.5cm);
     \draw [rotate=0] [->, >= stealth] (0.5,0) -- (0.9,0);
     \draw [rotate=-90] [->, >= stealth] (0.5,0) -- (0.9,0);
     \draw [rotate=90] [->, >= stealth] (0.5,0) -- (0.9,0);
     \draw [rotate=180] [->, >= stealth] (0.5,0) -- (0.9,0);
     \shadedraw[rotate=45,shift={(0.5,0)}] \doublefleche;
    \shadedraw[rotate=-45,shift={(0.5,0)}] \doublefleche;
     \shadedraw[rotate=-135,shift={(0.5,0)}] \doubleflechescindeeleft;
     \shadedraw[rotate=-135,shift={(0.5,0)}] \doubleflechescindeeright;
     \shadedraw[line width=1.1pt, rotate=-135,shift={(0.5,0)}] \fleche;
\begin{scriptsize}
\draw [fill=black] (-0.45,0.4) circle (0.3pt);
\draw [fill=black] (-0.25,0.55) circle (0.3pt);
\draw [fill=black] (-0.55,0.2) circle (0.3pt);
\end{scriptsize}
\draw (0,0.25) node[anchor=north]{$\mathbf{F}$};
\end{tikzpicture}
    \end{minipage}

    \noindent respectively, where the non-bold outgoing arrows are seen as incoming arrows of $\MA^*[d]$ and where the bold incoming arrow in $\mathcalboondox{D'}$ is seen as an outgoing arrow of $\MA^*[d]$.
\end{remark}
\subsection{Pre-Calabi-Yau structures}
We now recall the definition of pre-Calabi-Yau categories given in \cite{ktv}.
\begin{definition}
A \textbf{\textcolor{ultramarine}{$d$-pre-Calabi-Yau structure}} on a graded quiver $\MA$ is an element
\[
s_{d+1}M_{\MA}\in\Multi^{\bullet}_d(\MA)^{C_{\llg(\bullet)}}[d+1]
\]
of degree $1$, solving the Maurer-Cartan equation 
\begin{equation}
    \label{eq:MC}
    [s_{d+1}M_{\MA},s_{d+1}M_{\MA}]_{\nec}=0.
\end{equation}
Note that, since $s_{d+1}M_{\MA}$ has degree $1$, this is tantamount to requiring that $s_{d+1}M_{\MA}\upperset{\nec}{\circ} s_{d+1}M_{\MA}=0$. A\\\vskip-3.8mm\noindent graded quiver endowed with a $d$-pre-Calabi-Yau structure is called a \textbf{\textcolor{ultramarine}{$d$-pre-Calabi-Yau category}}. 
\end{definition}

\begin{example}
    An $A_{\infty}$-category $(\MA,sm_{\MA})$ is a $d$-pre-Calabi-Yau category for every $d$. Indeed, it suffices to consider $s_{d+1}M_{\MA}^{\bar{x}}=sm_{\MA}^{\bar{x}}$ for $\bar{x}\in\bar{\MO}$ and $s_{d+1}M_{\MA}^{\doubar{x}}=0$ for $\doubar{x}\in\bar{\MO}^n$, $n>1$.
\end{example}
The following example given in \cite{yeung} relates double Poisson structure on dg categories and $d$-pre-Calabi-Yau categories.
\begin{example}
    Given a dg category $\MA$, a shifted double Poisson structure 
    \[
    \{\{-,-\}\}_{(x,y),(z,t)} : {}_{x}\MA_{y}[1]\otimes {}_{z}\MA_{t}[1]\rightarrow \MA[-d]\otimes \MA[-d]
    \]
    induces a $d$-pre-Calabi-Yau structure $s_{d+1}M_{\MA}$ on $\MA$ whose only nonzero components are $s_{d+1}M_{\MA}^{x,y}$ and $s_{d+1}M_{\MA}^{x,y,z}$ which are induced by the differential and the product of $\MA$ respectively (see Example \ref{example:dg-cat-pCY}) and  $s_{d+1}M_{\MA}^{(x,y),(z,t)}=\{\{-,-\}\}_{(x,y),(z,t)}$.
\end{example}

We now recall the following result of \cite{ktv} (Proposition 4.26) which states the link beteween $d$-pre-Calabi-Yau structures on a $\Hom$-finite graded quiver $\MA$ and cyclic $A_{\infty}$-structures on the graded quiver $\MA\oplus\MA^*[d-1]$.

\begin{proposition}
\label{prop:equiv-pCY-Ainf}
A $d$-pre-Calabi-Yau structure on a graded quiver $\MA$ induces a cyclic $A_{\infty}$-structure on $\MA\oplus \MA^{*}[d-1]$ that restricts to $\MA$ and that is almost cyclic with respect to $\Gamma$. Moreover, if the graded quiver $\MA$ is $\Hom$-finite, then the data of a $d$-pre-Calabi-Yau structure on $\MA$ is equivalent to the data of a cyclic $A_{\infty}$-structure on $\MA\oplus\MA^*[d-1]$ that restricts to $\MA$.
\end{proposition}
\begin{proof}
Consider $s_{d+1}M_{\MA} \in \Multi^{\bullet}_d(\MA)^{C_{\llg(\bullet)}}[d+1]$ of degree $1$ and define $sm_{\MA\oplus\MA^*}=\mathcalboondox{j}(s_{d+1}M_{\MA})$.
By Proposition \ref{proposition:j-bracket}, $s_{d+1}M_{\MA}$ defines a $d$-pre-Calabi-Yau structure if and only if $sm_{\MA\oplus\MA^*}$
defines an $A_{\infty}$-structure on $\MA\oplus \MA^{*}[d-1]$. Moreover, it is straightforward to show that this $A_{\infty}$-structure is almost cyclic with respect to $\Gamma$ and that $(sm_{\MA\oplus\MA^*})^{\bar{x}}_{|\MA[1]^{\otimes \bar{x}}}\subseteq \MA[1]$ for every $\bar{x}\in\bar{\MO}$.

If $\MA$ is $\Hom$-finite, the bijectivity of $\mathcalboondox{j}$ tells us that the collection of maps $m_{\MA\oplus\MA^*}$ are in correspondence with maps of the form $M_{\MA}$, 
which shows the equivalence. 
\end{proof}

\subsection{Pre-Calabi-Yau morphisms} \label{pcy-morphisms}
Following \cite{ktv} (cf. \cite{lv}) we give the following definition of $d$-pre-Calabi-Yau morphisms.
\begin{definition}
Given graded quivers $\MA$ and $\MB$ with respective sets of objects $\MO_{\MA}$ and $\MO_{\MB}$ and a map $F_0 : \MO_{\MA}\rightarrow \MO_{\MB}$, we define the graded vector space 
\begin{small}
    \begin{equation}
    \begin{split}
&\Multi^{\bullet}_d(\MA,\MB)\\&=\prod_{n\in\NN^*}\prod_{\doubar{x}\in\bar{\MO}_{\MA}^n}\Homgr_{\kk}\bigg(\bigotimes\limits_{i=1}^n\MA[1]^{\otimes \bar{x}^i},\bigotimes\limits_{i=1}^{n-1}{}_{F_0(\llt(\bar{x}^i))}\MB_{F_0(\rrt(\bar{x}^{i+1}))}[-d]\otimes{}_{F_0(\llt(\bar{x}^n))}\MB_{F_0(\rrt(\bar{x}^1))}[-d])\bigg).
    \end{split}
    \end{equation}
\end{small}
The action of $\sigma=(\sigma_n)_{n\in\NN^*}\in\prod_{n\in\NN^*} C_n$ on an element $\mathbf{F} = (F^{\doubar{x}})_{\doubar{x} \in \doubar{\MO}_{\MA}}\in\Multi^{\bullet}_d(\MA,\MB)$ is the element $\sigma\cdot\mathbf{F}\in\Multi^{\bullet}_d(\MA,\MB)$ given by 
\begin{equation}
    \begin{split}
(\sigma\cdot\mathbf{F})^{\doubar{x}}=\tau^{\sigma}_{{}_{F_0(\llt(\bar{x}^{1}))}\MB_{F_0(\rrt(\bar{x}^2))}[-d],\dots, {}_{F_0(\llt(\bar{x}^n))}\MB_{F_0(\rrt(\bar{x}^1))}[-d]}\circ F^{\doubar{x}\cdot\sigma} \circ \tau^{\sigma}_{\MA[1]^{\otimes \bar{x}^1},\MA[1]^{\otimes\bar{x}^2}, \dots ,\MA[1]^{\otimes \bar{x}^n}}
    \end{split}
\end{equation}
for $\doubar{x}=(\bar{x}^1,\dots,\bar{x}^n)\in\bar{\MO}_{\MA}^n$.
We will denote by $\Multi^{\bullet}_d(\MA,\MB)^{C_{\llg(\bullet)}}$ the space of elements of $\Multi^{\bullet}_d(\MA,\MB)$ that are invariant under the action of $\prod_{n\in\NN^*} C_n$.
\end{definition}

\begin{definition}
    Given $d$-pre-Calabi-Yau categories $(\MA,s_{d+1}M_{\MA})$ and $(\MB,s_{d+1}M_{\MB})$ as well as a map $F_0:\MO_{\MA} \rightarrow\MO_{\MB}$ between their respective sets of objects and $s_{d+1}\mathbf{F}\in \Multi^{\bullet}_{d,F_0}(\MA,\MB)[d+1]$ of degree $0$, the \textbf{\textcolor{ultramarine}{multinecklace composition of $s_{d+1}M_{\MA}$ and $s_{d+1}\mathbf{
F}$}} is the element 
    \[s_{d+1}\mathbf{F}\upperset{\multinec}{\circ} s_{d+1}M_{\MA}\in \Multi^{\bullet}_{d,F_0}(\MA,\MB)[d+1]
    \]
    given by 
    \[(s_{d+1}\mathbf{F}\upperset{\multinec}{\circ}s_{d+1}M_{\MA})^{\doubar{x}}=\sum\mathcal{E}(\mathcalboondox{D})\] for $\doubar{x}\in\MO_{\MA}$, where the sum is over all the filled diagrams $\mathcalboondox{D}$ of type $\doubar{x}$ of the form
\begin{equation}
\begin{tikzpicture}[line cap=round,line join=round,x=1.0cm,y=1.0cm]
\clip(-7.4,-1.2) rectangle (10.586743541133627,1.7);
 \draw [line width=0.5pt] (0.,0.) circle (0.5cm);
     \shadedraw[rotate=30,shift={(0.5cm,0cm)}] \doublefleche;
     \shadedraw[rotate=150,shift={(0.5cm,0cm)}] \doublefleche;
     \draw [line width=0.5pt] (0,1.3) circle (0.3cm);
     \shadedraw [shift={(0cm,1cm)},rotate=-90] \doubleflechescindeeleft;
     \shadedraw [shift={(0cm,1cm)},rotate=-90] \doubleflechescindeeright;
     \shadedraw [shift={(0cm,1cm)},rotate=-90] \fleche;
     \draw [->,> = stealth] (0.3,1.3)--(0.6,1.3);
     \draw [->,> = stealth] (-0.3,1.3)--(-0.6,1.3);
     \draw [line width=0.5pt] (1.12,-0.65) circle (0.3cm);
     \shadedraw[shift={(0.86cm,-0.5cm)},rotate=150] \doubleflechescindeeleft;
     \shadedraw[shift={(0.86cm,-0.5cm)},rotate=150] \doubleflechescindeeright;
     \shadedraw[shift={(0.86cm,-0.5cm)},rotate=150] \fleche;
     \draw [rotate around ={60:(1.12,-0.65)}] [->,> = stealth] (1.43,-0.65)--(1.73,-0.65);
     \draw [rotate around ={-120:(1.12,-0.65)}] [->,> = stealth] (1.43,-0.65)--(1.73,-0.65);
     \draw [line width=0.5pt] (-1.12,-0.65) circle (0.3cm);
     \shadedraw[shift={(-0.86cm,-0.5cm)},rotate=30] \doubleflechescindeeleft;
      \shadedraw[shift={(-0.86cm,-0.5cm)},rotate=30] \doubleflechescindeeright;
       \shadedraw[shift={(-0.86cm,-0.5cm)},rotate=30] \fleche;
      \draw [rotate around ={-60:(-1.12,-0.65)}] [->,> = stealth] (-1.43,-0.65)--(-1.73,-0.65);
     \draw [rotate around ={120:(-1.12,-0.65)}] [->,> = stealth] (-1.43,-0.65)--(-1.73,-0.65);
\begin{scriptsize}
\draw [fill=black] (0,1.7) circle (0.3pt);
\draw [fill=black] (0.2,1.65) circle (0.3pt);
\draw [fill=black] (-0.2,1.65) circle (0.3pt);
\draw [fill=black] (0,-0.6) circle (0.3pt);
\draw [fill=black] (0.2,-0.55) circle (0.3pt);
\draw [fill=black] (-0.2,-0.55) circle (0.3pt);
\draw [fill=black] (1.45,-0.85) circle (0.3pt);
\draw [fill=black] (1.5,-0.67) circle (0.3pt);
\draw [fill=black] (1.33,-0.97) circle (0.3pt);
\draw [fill=black] (-1.45,-0.85) circle (0.3pt);
\draw [fill=black] (-1.5,-0.67) circle (0.3pt);
\draw [fill=black] (-1.33,-0.97) circle (0.3pt);
\end{scriptsize}
\draw (0,0.25)node[anchor=north]{$M_{\MA}$};
\draw (0,1.55)node[anchor=north]{$\mathbf{F}$};
\draw (-1.12,-0.4)node[anchor=north]{$\mathbf{F}$};
\draw (1.12,-0.4)node[anchor=north]{$\mathbf{F}$};
\end{tikzpicture}
\label{fig:morph-1}
\end{equation}
\noindent
\end{definition}

\begin{definition}
 Given $d$-pre-Calabi-Yau categories $(\MA,s_{d+1}M_{\MA})$ and $(\MB,s_{d+1}M_{\MB})$ as well as a map $F_0:\MO_{\MA} \rightarrow\MO_{\MB}$ between their respective sets of objects and $s_{d+1}\mathbf{F}\in \Multi^{\bullet}_{d,F_0}(\MA,\MB)[d+1]$ of degree $0$, the \textbf{\textcolor{ultramarine}{pre composition of $s_{d+1}\mathbf{F}$ and $s_{d+1}M_{\MB}$}} is the element \[s_{d+1}M_{\MB}\upperset{\pre}{\circ}s_{d+1}\mathbf{F}\in \Multi^{\bullet}_{d,F_0}(\MA,\MB)[d+1]\] given by \[(s_{d+1}M_{\MB}\upperset{\pre}{\circ}s_{d+1}\mathbf{F})^{\doubar{x}}=\sum\mathcal{E}(\mathcal{D'})\] for $\doubar{x}\in\MO_{\MA}$, where the sum is over all the filled diagrams $\mathcalboondox{D'}$ of type $\doubar{x}$ of the form

\begin{equation}
\begin{tikzpicture}[line cap=round,line join=round,x=1.0cm,y=1.0cm]
\clip(-4.6,-1) rectangle (4.579407206898253,1.5);
     \draw [line width=0.5pt] (0.,0.) circle (0.5cm);
     \draw [rotate=90] [->,> = stealth] (0.5,0)--(0.9,0);
     \draw [rotate=-30] [->,> = stealth] (0.5,0)--(0.9,0);
     \draw [rotate=-150] [->,> = stealth] (0.5,0)--(0.9,0);
     \draw [rotate=0] [<-,> = stealth] (0.5,0)--(0.9,0);
     \draw [line width=0.5pt] (1.15,0.) circle (0.25cm);
     \draw[rotate around={60:(1.15,0)}] [->,> = stealth] (1.4,0)--(1.7,0);
      \draw[rotate around={-60:(1.15,0)}] [->,> = stealth] (1.4,0)--(1.7,0);
      \shadedraw[rotate around={120:(1.15,0)}, shift={(1.4cm,0cm)}] \doublefleche;
       \shadedraw[shift={(1.4cm,0cm)}] \doublefleche;
     \draw [rotate=60] [<-,> = stealth] (0.5,0)--(0.9,0);
       \draw [rotate=180] [<-,> = stealth] (0.5,0)--(0.9,0);
      \draw [line width=0.5pt] (-1.15,0.) circle (0.25cm);
     \draw[rotate around={-60:(-1.15,0)}] [->,> = stealth] (-1.4,0)--(-1.7,0);
      \draw[rotate around={60:(-1.15,0)}] [->,> = stealth] (-1.4,0)--(-1.7,0);
      \shadedraw[rotate around={-60:(-1.15,0)}, shift={(-0.9cm,0cm)}] \doublefleche;
       \shadedraw[rotate around={180:(-1.15,0)},shift={(-0.9cm,0cm)}] \doublefleche;
        \draw [rotate=120] [<-,> = stealth] (0.5,0)--(0.9,0);
        \draw [line width=0.5pt] (0.575,1) circle (0.25cm);
        \draw[rotate around={120:(0.575,1)}] [->,> = stealth] (0.825,1)--(1.125,1);
      \draw[rotate around={0:(0.575,1)}] [->,> = stealth] (0.825,1)--(1.125,1);
      \shadedraw[rotate around={60:(0.575,1)}, shift={(0.825cm,1cm)}] \doublefleche;
       \shadedraw[rotate around={180:(0.575,1)},shift={(0.825cm,1cm)}] \doublefleche;
        \draw [line width=0.5pt] (-0.575,1) circle (0.25cm);
      \draw[rotate around={-120:(-0.575,1)}] [->,> = stealth] (-0.825,1)--(-1.125,1);
      \draw[rotate around={0:(-0.575,1)}] [->,> = stealth] (-0.825,1)--(-1.125,1);
      \shadedraw[rotate around={-120:(-0.575,1)}, shift={(-0.325cm,1cm)}] \doublefleche;
       \shadedraw[rotate around={120:(-0.575,1)},shift={(-0.325cm,1cm)}] \doublefleche;
\begin{scriptsize}
\draw [fill=black] (0.65,0.7) circle (0.3pt);
\draw [fill=black] (0.79,0.75) circle (0.3pt);
\draw [fill=black] (0.88,0.85) circle (0.3pt);
\draw [fill=black] (-0.25,1) circle (0.3pt);
\draw [fill=black] (-0.29,0.85) circle (0.3pt);
\draw [fill=black] (-0.29,1.15) circle (0.3pt);
\draw [fill=black] (-1,0.3) circle (0.3pt);
\draw [fill=black] (-0.9,0.2) circle (0.3pt);
\draw [fill=black] (-1.15,0.3) circle (0.3pt);
\draw [fill=black] (1,-0.3) circle (0.3pt);
\draw [fill=black] (1.15,-0.3) circle (0.3pt);
\draw [fill=black] (0.9,-0.2) circle (0.3pt);
\draw [fill=black] (0,-0.6) circle (0.3pt);
\draw [fill=black] (0.2,-0.55) circle (0.3pt);
\draw [fill=black] (-0.2,-0.55) circle (0.3pt);
\draw [rotate=120][fill=black] (0,-0.6) circle (0.3pt);
\draw [rotate=120][fill=black] (0.15,-0.55) circle (0.3pt);
\draw [rotate=120][fill=black] (-0.15,-0.55) circle (0.3pt);
\draw [rotate=240][fill=black] (0,-0.6) circle (0.3pt);
\draw [rotate=240][fill=black] (0.15,-0.55) circle (0.3pt);
\draw [rotate=240][fill=black] (-0.15,-0.55) circle (0.3pt);
\end{scriptsize}
\draw (0,0.25)node[anchor=north]{$M_{\MB}$};
\draw (-1.15,0.25)node[anchor=north]{$\mathbf{F}$};
\draw (1.15,0.25)node[anchor=north]{$\mathbf{F}$};
\draw (0.575,1.25)node[anchor=north]{$\mathbf{F}$};
\draw (-0.575,1.25)node[anchor=north]{$\mathbf{F}$};
\end{tikzpicture}
\label{fig:morph-2}
\end{equation}
\noindent
\end{definition}

\begin{definition}
\label{def:pcY-morphism}
Given $d$-pre-Calabi-Yau categories $(\MA,s_{d+1}M_{\MA})$ and $(\MB,s_{d+1}M_{\MB})$ with respective sets of objects $\MO_{\MA}$ and $\MO_{\MB}$ a \textbf{\textcolor{ultramarine}{$d$-pre-Calabi-Yau morphism}} is a pair $(F_0,s_{d+1}\mathbf{F})$ where $F_0$ is a map $\MO_{\MA}\rightarrow \MO_{\MB}$ and $s_{d+1}\mathbf{F}\in \Multi^{\bullet}_{d,F_0}(\MA,\MB)^{C_{\llg(\bullet)}}[d+1]$
is of degree $0$ satisfying the following equation
\begin{equation}
\label{eq:morphism-pCY}
\tag{$\operatorname{MI}^{\doubar{x}}$}
(s_{d+1}\mathbf{F}\upperset{\multinec}{\circ} s_{d+1}M_{\MA})^{\doubar{x}}= (s_{d+1}M_{\MB}\upperset{\pre}{\circ} s_{d+1}\mathbf{F})^{\doubar{x}}
\end{equation}
for all $\doubar{x}\in\doubar{\MO}_{\MA}$.
Note that the left member and right member of the previous identity belong to
\[
\Homgr_{\kk}\bigg(\bigotimes\limits_{i=1}^{n}\MA[1]^{\otimes\bar{x}^{i}}, \bigotimes\limits_{i=1}^{n-1}{}_{F_0(\llt(\bar{x}^{i}))}\MB_{F_0(\rrt(\bar{x}^{i+1}))}[-d]\otimes{}_{F_0(\llt(\bar{x}^{n}))}\MB_{F_0(\rrt(\bar{x}^{1}))}[1]\bigg).
\]
\end{definition}

We now recall how to compose $d$-pre-Calabi-Yau morphisms.
\begin{definition}
\label{def:cp-pCY}
Let $(\MA,s_{d+1}M_{\MA})$, $(\MB,s_{d+1}M_{\MB})$ and $(\mathcal{C},s_{d+1}M_{\mathcal{C}})$ be $d$-pre-Calabi-Yau categories and consider $(F_0,s_{d+1}\mathbf{F}) :(\MA,s_{d+1}M_{\MA})\rightarrow (\MB,s_{d+1}M_{\MB})$ and $(G_0,s_{d+1}\mathbf{G}) :(\MB,s_{d+1}M_{\MB})\rightarrow (\mathcal{C},s_{d+1}M_{\mathcal{C}})$ two $d$-pre-Calabi-Yau morphisms.
The \textbf{\textcolor{ultramarine}{composition of $s_{d+1}\mathbf{F}$ and $s_{d+1}\mathbf{G}$}} is the pair \[(G_0\circ F_0,s_{d+1}\mathbf{G}\circ s_{d+1}\mathbf{F})\] where
\[s_{d+1}\mathbf{G}\circ s_{d+1}\mathbf{F}\in \Multi^{\bullet}_{d,G_0\circ F_0}(\MA,\mathcal{C})^{C_{\llg(\bullet)}}[d+1]\] is of degree $0$ and is given by $(s_{d+1}\mathbf{G}\circ s_{d+1}\mathbf{F})^{\doubar{x}}=\sum\mathcal{E}(\mathcalboondox{D})$ where the sum is over all filled diagrams $\mathcalboondox{D}$ of type $\doubar{x}\in\doubar{\MO}_{\MA}$ of the form
\begin{equation}
\begin{tikzpicture}[line cap=round,line join=round,x=1.0cm,y=1.0cm]
\clip(-7.5,-3) rectangle (11.248895821273946,1);
     \draw [line width=0.5pt] (0.,0.) circle (0.5cm);
     \draw [rotate=30] [->,> = stealth] (0.5,0)--(0.9,0);
     \draw [rotate=-90] [->,> = stealth] (0.5,0)--(0.9,0);
     \draw [rotate=150] [->,> = stealth] (0.5,0)--(0.9,0);
     \draw [rotate=0] [<-,> = stealth] (0.5,0)--(0.9,0);
     \draw [line width=0.5pt] (1.2,0.) circle (0.3cm);
     \draw [->,> = stealth] (1.5,0)--(1.9,0);
     \shadedraw[rotate around={90:(1.2,0)},shift={(1.5cm,0cm)}] \doublefleche;
     \shadedraw[rotate around={-90:(1.2,0)},shift={(1.5cm,0cm)}] \doublefleche;
      \draw [line width=0.5pt] (2.4,0.) circle (0.5cm);
      \draw [rotate around={-90:(2.4,0)}] [<-,> = stealth] (2.9,0)--(3.3,0);
      \draw [rotate around={-45:(2.4,0)}] [->,> = stealth] (2.9,0)--(3.3,0);
    \draw [rotate around={135:(2.4,0)}] [->,> = stealth] (2.9,0)--(3.3,0);
     \draw [line width=0.5pt] (2.4,-1.2) circle (0.3cm);
     \shadedraw[rotate around={-90:(2.4,-1.2)},shift={(2.7cm,-1.2cm)}] \doublefleche;
     \draw [rotate=-60] [<-,> = stealth] (0.5,0)--(0.9,0);
      \draw [line width=0.5pt] (0.6,-1.03) circle (0.3cm);
     \shadedraw[rotate around={-60:(0.6,-1.03)},shift={(0.9cm,-1.03cm)}] \doublefleche;
     \draw [line width=0.5pt] (-1.2,0.) circle (0.3cm);
     \draw [<-,> = stealth] (-0.5,0)--(-0.9,0);
     \shadedraw[rotate around={180:(-1.2,0)},shift={(-0.9cm,0cm)}] \doublefleche;
     \draw [rotate=-120] [<-,> = stealth] (0.5,0)--(0.9,0);
    \draw [line width=0.5pt] (-0.6,-1.03) circle (0.3cm);
     \shadedraw[rotate around={-120:(-0.6,-1.03)},shift={(-0.3cm,-1.03cm)}] \doublefleche;
     \shadedraw[rotate around={-240:(-0.6,-1.03)},shift={(-0.3cm,-1.03cm)}] \doublefleche;
     \draw[rotate around={180:(-0.6,-1.03)}][->,> = stealth] (-0.3,-1.03)--(0.1,-1.03);
     \draw[rotate around={-60:(-0.6,-1.03)}][->,> = stealth] (-0.3,-1.03)--(0.1,-1.03);
      \draw [line width=0.5pt] (-1.8,-1.03) circle (0.5cm);
      \draw [rotate around={120:(-1.8,-1.03)}] [<-,> = stealth] (-1.3,-1.03)--(-0.9,-1.03);
       \draw [rotate around={-60:(-1.8,-1.03)}] [->,> = stealth] (-1.3,-1.03)--(-0.9,-1.03);
        \draw [rotate around={180:(-1.8,-1.03)}] [->,> = stealth] (-1.3,-1.03)--(-0.9,-1.03);
     \draw [line width=0.5pt] (-2.4,0.) circle (0.3cm);
     \shadedraw[rotate around={120:(-2.4,0)},shift={(-2.1cm,0cm)}] \doublefleche;
     \draw [line width=0.5pt] (0,-2.06) circle (0.5cm);
      \draw [rotate around={-60:(0,-2.06)}] [->,> = stealth] (0.5,-2.06)--(0.9,-2.06);
\begin{scriptsize}
\draw [fill=black] (0,0.6) circle (0.3pt);
\draw [fill=black] (0.2,0.55) circle (0.3pt);
\draw [fill=black] (-0.2,0.55) circle (0.3pt);
\draw [fill=black] (0.5,-0.3) circle (0.3pt);
\draw [fill=black] (0.4,-0.4) circle (0.3pt);
\draw [fill=black] (0.55,-0.18) circle (0.3pt);
\draw [fill=black] (-0.5,-0.3) circle (0.3pt);
\draw [fill=black] (-0.4,-0.4) circle (0.3pt);
\draw [fill=black] (-0.55,-0.18) circle (0.3pt);
\draw [fill=black] (1.98,-0.4) circle (0.3pt);
\draw [fill=black] (1.88,-0.26) circle (0.3pt);
\draw [fill=black] (2.12,-0.5) circle (0.3pt);
\draw [fill=black] (2.8,0.4) circle (0.3pt);
\draw [fill=black] (2.95,0.2) circle (0.3pt);
\draw [fill=black] (2.6,0.55) circle (0.3pt);
\draw [fill=black] (-0.2,-1.03) circle (0.3pt);
\draw [fill=black] (-0.25,-0.9) circle (0.3pt);
\draw [fill=black] (-0.25,-1.16) circle (0.3pt);
\draw [fill=black] (0.3,-1.55) circle (0.3pt);
\draw [fill=black] (0.5,-1.7) circle (0.3pt);
\draw [fill=black] (0.6,-1.9) circle (0.3pt);
\draw [rotate around={180:(0,-2.06)}][fill=black] (0.6,-1.9) circle (0.3pt);
\draw [rotate around={180:(0,-2.06)}][fill=black] (0.5,-1.7) circle (0.3pt);
\draw [rotate around={180:(0,-2.06)}][fill=black] (0.3,-1.55) circle (0.3pt);
\draw [fill=black] (-1.3,-0.7) circle (0.3pt);
\draw [fill=black] (-1.45,-0.55) circle (0.3pt);
\draw [fill=black] (-1.67,-0.45) circle (0.3pt);
\draw [rotate around={180:(-1.8,-1.03)}][fill=black] (-1.3,-0.7) circle (0.3pt);
\draw [rotate around={180:(-1.8,-1.03)}][fill=black] (-1.45,-0.55) circle (0.3pt); circle (0.3pt);
\draw [rotate around={180:(-1.8,-1.03)}][fill=black] (-1.67,-0.45) circle (0.3pt);
\end{scriptsize}
\draw (0,0.25)node[anchor=north]{$\mathbf{G}$};
\draw (0,-1.8)node[anchor=north]{$\mathbf{G}$};
\draw (2.4,0.25)node[anchor=north]{$\mathbf{G}$};
\draw (-1.8,-0.78)node[anchor=north]{$\mathbf{G}$};
\draw (1.2,0.25)node[anchor=north]{$\mathbf{F}$};
\draw (-1.2,0.25)node[anchor=north]{$\mathbf{F}$};
\draw (2.4,-0.95)node[anchor=north]{$\mathbf{F}$};
\draw (-0.6,-0.78)node[anchor=north]{$\mathbf{F}$};
\draw (-2.4,0.25)node[anchor=north]{$\mathbf{F}$};
\draw (0.6,-0.78)node[anchor=north]{$\mathbf{F}$};
\end{tikzpicture}
    \label{fig:composition}
\end{equation}
More precisely, the sum is over all the diagrams $\mathcalboondox{D}$ of type $\doubar{x}$ where the discs are filled with either $\mathbf{F}$ or $\mathbf{G}$ such that each outgoing arrow of a disc filled with $\mathbf{F}$ is connected to an incoming arrow of a disc filled with $\mathbf{G}$ and each incoming arrow of a disc filled with $\mathbf{G}$ is connected to an outgoing arrow of a disc filled with $\mathbf{F}$. In particular, the incoming (resp. outgoing) arrows of the diagram $\mathcalboondox{D}$ are incoming (resp. outgoing) arrows of a disc filled with $\mathbf{F}$ (resp. $\mathbf{G}$).
\end{definition}

\begin{proposition}
\label{prop:pCY-category}
The data of $d$-pre-Calabi-Yau categories together with $d$-pre-Calabi-Yau morphisms and their composition define a category, denoted by $\pCY$. Given a graded quiver $\MA$ with set of objects $\MO$, the identity morphism $(\MA,s_{d+1}M_{\MA})\rightarrow (\MA,s_{d+1}M_{\MA})$ is given by $\operatorname{
Id}^{x,y}=\id_{{}_{x}\MA_{y}}$ for $x,y\in\MO$ and $\operatorname{Id}^
{\bar{x}^1,\dots,\bar{x}^n}=0$ for $(\bar{x}^1,\dots,\bar{x}^n)\in\bar{\MO}^n$ such that $n\neq 1$ or $n=1$ and $\llg(\bar{x}^1)>2$.
\end{proposition}
\begin{proof}
We have to show that the composition is associative and that the composition of any two $d$-pre-Calabi-Yau morphisms is a $d$-pre-Calabi-Yau morphism. 
We first consider $d$-pre-Calabi-Yau morphisms $(F_0,s_{d+1}\mathbf{F}) :(\MA,s_{d+1}M_{\MA})\rightarrow (\MB,s_{d+1}M_{\MB})$, $(G_0,s_{d+1}\mathbf{G}) :(\MB,s_{d+1}M_{\MB})\rightarrow (\mathcal{C},s_{d+1}M_{\mathcal{C}})$ and $(H_0,s_{d+1}\mathbf{H}) :(\mathcal{C},s_{d+1}M_{\mathcal{C}})\rightarrow (\mathcal{D},s_{d+1}M_{\mathcal{D}})$. It is clear that $H_0\circ(G_0\circ F_0)=(H_0\circ G_0)\circ F_0$. By definition, we have 
\[\big(s_{d+1}\mathbf{H}\circ(s_{d+1}\mathbf{G}\circ s_{d+1}\mathbf{F})\big)^{\doubar{x}}=\sum\mathcal{E}(\mathcalboondox{D})\] where the sum is over all the filled diagrams $\mathcalboondox{D}$ of type $\doubar{x}$ and of the form

\begin{tikzpicture}[line cap=round,line join=round,x=1.0cm,y=1.0cm]
\clip(-7.5,-3) rectangle (11.248895821273946,1);
     \draw [line width=0.5pt] (0.,0.) circle (0.5cm);
     \draw [rotate=30] [->,> = stealth] (0.5,0)--(0.9,0);
     \draw [rotate=-90] [->,> = stealth] (0.5,0)--(0.9,0);
     \draw [rotate=150] [->,> = stealth] (0.5,0)--(0.9,0);
     \draw [rotate=0] [<-,> = stealth] (0.5,0)--(0.9,0);
     \draw [line width=0.5pt] (1.2,0.) circle (0.3cm);
     \draw [->,> = stealth] (1.5,0)--(1.9,0);
     \shadedraw[rotate around={90:(1.2,0)},shift={(1.5cm,0cm)}] \doublefleche;
     \shadedraw[rotate around={-90:(1.2,0)},shift={(1.5cm,0cm)}] \doublefleche;
      \draw [line width=0.5pt] (2.4,0.) circle (0.5cm);
      \draw [rotate around={-90:(2.4,0)}] [<-,> = stealth] (2.9,0)--(3.3,0);
      \draw [rotate around={-45:(2.4,0)}] [->,> = stealth] (2.9,0)--(3.3,0);
    \draw [rotate around={135:(2.4,0)}] [->,> = stealth] (2.9,0)--(3.3,0);
     \draw [line width=0.5pt] (2.4,-1.2) circle (0.3cm);
     \shadedraw[rotate around={-90:(2.4,-1.2)},shift={(2.7cm,-1.2cm)}] \doublefleche;
     \draw [rotate=-60] [<-,> = stealth] (0.5,0)--(0.9,0);
      \draw [line width=0.5pt] (0.6,-1.03) circle (0.3cm);
     \shadedraw[rotate around={-60:(0.6,-1.03)},shift={(0.9cm,-1.03cm)}] \doublefleche;
     \draw [line width=0.5pt] (-1.2,0.) circle (0.3cm);
     \draw [<-,> = stealth] (-0.5,0)--(-0.9,0);
     \shadedraw[rotate around={180:(-1.2,0)},shift={(-0.9cm,0cm)}] \doublefleche;
     \draw [rotate=-120] [<-,> = stealth] (0.5,0)--(0.9,0);
    \draw [line width=0.5pt] (-0.6,-1.03) circle (0.3cm);
     \shadedraw[rotate around={-120:(-0.6,-1.03)},shift={(-0.3cm,-1.03cm)}] \doublefleche;
     \shadedraw[rotate around={-240:(-0.6,-1.03)},shift={(-0.3cm,-1.03cm)}] \doublefleche;
     \draw[rotate around={180:(-0.6,-1.03)}][->,> = stealth] (-0.3,-1.03)--(0.1,-1.03);
     \draw[rotate around={-60:(-0.6,-1.03)}][->,> = stealth] (-0.3,-1.03)--(0.1,-1.03);
      \draw [line width=0.5pt] (-1.8,-1.03) circle (0.5cm);
      \draw [rotate around={120:(-1.8,-1.03)}] [<-,> = stealth] (-1.3,-1.03)--(-0.9,-1.03);
       \draw [rotate around={-60:(-1.8,-1.03)}] [->,> = stealth] (-1.3,-1.03)--(-0.9,-1.03);
        \draw [rotate around={180:(-1.8,-1.03)}] [->,> = stealth] (-1.3,-1.03)--(-0.9,-1.03);
     \draw [line width=0.5pt] (-2.4,0.) circle (0.3cm);
     \shadedraw[rotate around={120:(-2.4,0)},shift={(-2.1cm,0cm)}] \doublefleche;
     \draw [line width=0.5pt] (0,-2.06) circle (0.5cm);
      \draw [rotate around={-60:(0,-2.06)}] [->,> = stealth] (0.5,-2.06)--(0.9,-2.06);
\begin{scriptsize}
\draw [fill=black] (0,0.6) circle (0.3pt);
\draw [fill=black] (0.2,0.55) circle (0.3pt);
\draw [fill=black] (-0.2,0.55) circle (0.3pt);
\draw [fill=black] (0.5,-0.3) circle (0.3pt);
\draw [fill=black] (0.4,-0.4) circle (0.3pt);
\draw [fill=black] (0.55,-0.18) circle (0.3pt);
\draw [fill=black] (-0.5,-0.3) circle (0.3pt);
\draw [fill=black] (-0.4,-0.4) circle (0.3pt);
\draw [fill=black] (-0.55,-0.18) circle (0.3pt);
\draw [fill=black] (1.98,-0.4) circle (0.3pt);
\draw [fill=black] (1.88,-0.26) circle (0.3pt);
\draw [fill=black] (2.12,-0.5) circle (0.3pt);
\draw [fill=black] (2.8,0.4) circle (0.3pt);
\draw [fill=black] (2.95,0.2) circle (0.3pt);
\draw [fill=black] (2.6,0.55) circle (0.3pt);
\draw [fill=black] (-0.2,-1.03) circle (0.3pt);
\draw [fill=black] (-0.25,-0.9) circle (0.3pt);
\draw [fill=black] (-0.25,-1.16) circle (0.3pt);
\draw [fill=black] (0.3,-1.55) circle (0.3pt);
\draw [fill=black] (0.5,-1.7) circle (0.3pt);
\draw [fill=black] (0.6,-1.9) circle (0.3pt);
\draw [rotate around={180:(0,-2.06)}][fill=black] (0.6,-1.9) circle (0.3pt);
\draw [rotate around={180:(0,-2.06)}][fill=black] (0.5,-1.7) circle (0.3pt);
\draw [rotate around={180:(0,-2.06)}][fill=black] (0.3,-1.55) circle (0.3pt);
\draw [fill=black] (-1.3,-0.7) circle (0.3pt);
\draw [fill=black] (-1.45,-0.55) circle (0.3pt);
\draw [fill=black] (-1.67,-0.45) circle (0.3pt);
\draw [rotate around={180:(-1.8,-1.03)}][fill=black] (-1.3,-0.7) circle (0.3pt);
\draw [rotate around={180:(-1.8,-1.03)}][fill=black] (-1.45,-0.55) circle (0.3pt); circle (0.3pt);
\draw [rotate around={180:(-1.8,-1.03)}][fill=black] (-1.67,-0.45) circle (0.3pt);
\end{scriptsize}
\draw (0,0.2)node[anchor=north]{$\mathbf{H}$};
\draw (0,-1.85)node[anchor=north]{$\mathbf{H}$};
\draw (2.4,0.2)node[anchor=north]{$\mathbf{H}$};
\draw (-1.8,-0.83)node[anchor=north]{$\mathbf{H}$};
\draw (1.2,0.2)node[anchor=north]{$\scriptscriptstyle{\mathbf{G}\circ\mathbf{F}}$};
\draw (-1.2,0.2)node[anchor=north]{$\scriptscriptstyle{\mathbf{G}\circ\mathbf{F}}$};
\draw (2.4,-1)node[anchor=north]{$\scriptscriptstyle{\mathbf{G}\circ\mathbf{F}}$};
\draw (-0.6,-0.83)node[anchor=north]{$\scriptscriptstyle{\mathbf{G}\circ\mathbf{F}}$};
\draw (-2.4,0.2)node[anchor=north]{$\scriptscriptstyle{\mathbf{G}\circ\mathbf{F}}$};
\draw (0.6,-0.83)node[anchor=north]{$\scriptscriptstyle{\mathbf{G}\circ\mathbf{F}}$};
\end{tikzpicture}
\noindent where a disc filled with $\mathbf{G}\circ \mathbf{F}$ of a given type is the sum of the diagrams of the same type and of the form \eqref{fig:composition}. 
On the other hand, we have that \[\big((s_{d+1}\mathbf{H}\circ s_{d+1}\mathbf{G})\circ s_{d+1}\mathbf{F})\big)^{\doubar{x}}=\sum\mathcal{E}(\mathcalboondox{D'})\] where the sum is over all the filled diagrams $\mathcalboondox{D'}$ of type $\doubar{x}$ and of the form

\begin{tikzpicture}[line cap=round,line join=round,x=1.0cm,y=1.0cm]
\clip(-7,-5.5) rectangle (11.402841460158584,1.5);
     \draw [line width=0.5pt] (0.,0.) circle (0.5cm);
     \draw [rotate=30] [->,> = stealth] (0.5,0)--(0.9,0);
     \draw [rotate=-90] [->,> = stealth] (0.5,0)--(0.9,0);
     \draw [rotate=150] [->,> = stealth] (0.5,0)--(0.9,0);
     \draw [rotate=0] [<-,> = stealth] (0.5,0)--(1.2,0);
     \draw [line width=0.5pt] (1.5,0.) circle (0.3cm);
     \draw [->,> = stealth] (1.8,0)--(2.5,0);
     \draw[rotate around={60:(1.5,0)}] [<-,> = stealth] (1.8,0)--(2.1,0);
     \draw [line width=0.5pt] (1.5+0.425,0.74) circle (0.25cm);
      \shadedraw[rotate around={60:(1.5+0.425,0.74)}, shift={(2.175cm,0.74cm)}] \doublefleche;
     \draw[rotate around={120:(1.5,0)}] [<-,> = stealth] (1.8,0)--(2.1,0);
     \draw [line width=0.5pt] (1.5-0.425,0.74) circle (0.25cm);
      \shadedraw[rotate around={120:(1.5-0.425,0.74)}, shift={(1.325cm,0.74cm)}] \doublefleche;
      \draw [line width=0.5pt] (3,0.) circle (0.5cm);
      \draw [rotate around={-90:(3,0)}] [<-,> = stealth] (3.5,0)--(3.9,0);
      \draw [rotate around={-45:(3,0)}] [->,> = stealth] (3.5,0)--(3.9,0);
    \draw [rotate around={135:(3,0)}] [->,> = stealth] (3.5,0)--(3.9,0);
     \draw [line width=0.5pt] (3,-1.2) circle (0.3cm);
     \shadedraw[rotate around={-45:(3,-1.2)}] [<-,> = stealth] (3.3,-1.2)--(3.6,-1.2);
    \draw [line width=0.5pt] (3+0.6,-1.2-0.6) circle (0.25cm);
    \shadedraw[rotate around={-45:(3+0.6,-1.2-0.6)},shift={(3.85cm,-1.8cm)}] \doublefleche;
     \shadedraw[rotate around={-135:(3,-1.2)}] [<-,> = stealth] (3.3,-1.2)--(3.6,-1.2);
     \draw [line width=0.5pt] (3-0.6,-1.2-0.6) circle (0.25cm);
     \shadedraw[rotate around={-135:(3-0.6,-1.2-0.6)},shift={(2.65cm,-1.8cm)}] \doublefleche;
     \draw [rotate=-60] [<-,> = stealth] (0.5,0)--(0.9,0);
      \draw [line width=0.5pt] (0.6,-1.03) circle (0.3cm);
     \draw[rotate around={-60:(0.6,-1.03)}] [<-,> = stealth] (0.9,-1.03)--(1.2,-1.03);
      \draw [line width=0.5pt] (0.6+0.425,-1.03-0.74) circle (0.25cm);
     \shadedraw[rotate around={-60:(0.6+0.425,-1.03-0.74)},shift={(1.275cm,-1.77cm)}] \doublefleche;
     \draw[rotate around={-60:(0.6,-1.03)}] [<-,> = stealth] (0.9,-1.03)--(1.2,-1.03);
     \draw [line width=0.5pt] (-1.2,0.) circle (0.3cm);
     \draw [<-,> = stealth] (-0.5,0)--(-0.9,0);
     \shadedraw[rotate around={120:(-1.2,0)}] [<-,> = stealth] (-0.9,0)--(-0.6,0);
      \draw [line width=0.5pt] (-1.2-0.425,0.74) circle (0.25cm);
      \shadedraw[rotate around={120:(-1.2-0.425,0.74)},shift={(-1.375cm,0.74cm)}] \doublefleche;
     \shadedraw[rotate around={-120:(-1.2,0)}] [<-,> = stealth] (-0.9,0)--(-0.6,0);
     \draw [line width=0.5pt] (-1.2-0.425,-0.74) circle (0.25cm);
     \shadedraw[rotate around={-120:(-1.2-0.425,-0.74)},shift={(-1.375cm,-0.74cm)}] \doublefleche;
     \draw [rotate=-120] [<-,> = stealth] (0.5,0)--(0.9,0);
     \draw [line width=0.5pt] (-0.6,-1.03) circle (0.3cm);
     \draw[rotate around={-120:(-0.6,-1.03)}][<-,> = stealth] (-0.3,-1.03)--(0,-1.03);
     \draw [line width=0.5pt] (-0.6-0.425,-1.03-0.74) circle (0.25cm);
     \shadedraw[rotate around={-30:(-0.6-0.425,-1.03-0.74)},shift={(-0.775cm,-1.77cm)}] \doublefleche;
     \draw[rotate around={-120:(-0.6-0.425,-1.03-0.74)}][->,> = stealth] (-0.775,-1.77)--(-0.475,-1.77);
     \draw [line width=0.5pt] (-0.6-0.425-0.425,-1.03-0.74-0.74) circle (0.3cm); 
     \draw[rotate around={-120:(-0.6-0.425-0.425,-1.03-0.74-0.74)}][->,> = stealth](-0.6-0.425-0.425+0.3,-1.03-0.74-0.74)--(-0.6-0.425-0.425+0.7,-1.03-0.74-0.74) ;
     \draw [line width=0.5pt] (-0.6-0.425-0.425-0.6,-1.03-0.74-0.74-1.04) circle (0.5cm);
     \draw [rotate around={105:(-0.6-0.425-0.425-0.6,-1.03-0.74-0.74-1.04)}] [->,> = stealth] (-0.6-0.425-0.425-0.6+0.5,-1.03-0.74-0.74-1.04) -- (-0.6-0.425-0.425-0.6+0.9,-1.03-0.74-0.74-1.04); 
      \draw [rotate around={285:(-0.6-0.425-0.425-0.6,-1.03-0.74-0.74-1.04)}] [->,> = stealth] (-0.6-0.425-0.425-0.6+0.5,-1.03-0.74-0.74-1.04) -- (-0.6-0.425-0.425-0.6+0.9,-1.03-0.74-0.74-1.04); 
       \draw [rotate around={-30:(-0.6-0.425-0.425-0.6,-1.03-0.74-0.74-1.04)}] [<-,> = stealth] (-0.6-0.425-0.425-0.6+0.5,-1.03-0.74-0.74-1.04) -- (-0.6-0.425-0.425-0.6+0.9,-1.03-0.74-0.74-1.04); 
     \draw [line width=0.5pt] (-0.6-0.425-0.425-0.6+1.04,-1.03-0.74-0.74-1.04-0.6) circle (0.3cm);
     \draw[rotate around={30:(-0.6-0.425-0.425-0.6+1.04,-1.03-0.74-0.74-1.04-0.6)}] [<-,> = stealth] (-0.6-0.425-0.425-0.6+1.04+0.3,-1.03-0.74-0.74-1.04-0.6)-- (-0.6-0.425-0.425-0.6+1.04+0.6,-1.03-0.74-0.74-1.04-0.6);
     \draw [line width=0.5pt] (-0.6-0.425-0.425-0.6+1.04+0.74,-1.03-0.74-0.74-1.04-0.6+0.425) circle (0.25cm);
     \shadedraw[rotate around ={30:(-0.6-0.425-0.425-0.6+1.04+0.74,-1.03-0.74-0.74-1.04-0.6+0.425)},shift={(-0.6-0.425-0.425-0.6+1.04+0.74+0.25,-1.03-0.74-0.74-1.04-0.6+0.425)}] \doublefleche;
     \draw[rotate around={-90:(-0.6-0.425-0.425-0.6+1.04,-1.03-0.74-0.74-1.04-0.6)}] [<-,> = stealth] (-0.6-0.425-0.425-0.6+1.04+0.3,-1.03-0.74-0.74-1.04-0.6)-- (-0.6-0.425-0.425-0.6+1.04+0.6,-1.03-0.74-0.74-1.04-0.6);
      \draw [line width=0.5pt] (-0.6-0.425-0.425-0.6+1.04,-1.03-0.74-0.74-1.04-0.6-0.85) circle (0.25cm);
     \shadedraw[rotate around ={-90:(-0.6-0.425-0.425-0.6+1.04,-1.03-0.74-0.74-1.04-0.6-0.85)},shift={(-0.6-0.425-0.425-0.6+1.04+0.25,-1.03-0.74-0.74-1.04-0.6-0.85)}] \doublefleche;
\begin{scriptsize}
\draw [fill=black] (0,0.6) circle (0.3pt);
\draw [fill=black] (0.2,0.55) circle (0.3pt);
\draw [fill=black] (-0.2,0.55) circle (0.3pt);
\draw [fill=black] (0.5,-0.3) circle (0.3pt);
\draw [fill=black] (0.4,-0.4) circle (0.3pt);
\draw [fill=black] (0.55,-0.18) circle (0.3pt);
\draw [fill=black] (-0.5,-0.3) circle (0.3pt);
\draw [fill=black] (-0.4,-0.4) circle (0.3pt);
\draw [fill=black] (-0.55,-0.18) circle (0.3pt);
\draw [fill=black] (-1.6,0) circle (0.3pt);
\draw [fill=black] (-1.55,0.2) circle (0.3pt);
\draw [fill=black] (-1.55,-0.2) circle (0.3pt);
\draw [fill=black] (1.5,-0.4) circle (0.3pt);
\draw [fill=black] (1.35,-0.35) circle (0.3pt);
\draw [fill=black] (1.65,-0.35) circle (0.3pt);
\draw [fill=black] (1.5,0.4) circle (0.3pt);
\draw [fill=black] (1.4,0.38) circle (0.3pt);
\draw [fill=black] (1.6,0.38) circle (0.3pt);
\draw [fill=black] (2.58,-0.4) circle (0.3pt);
\draw [fill=black] (2.48,-0.26) circle (0.3pt);
\draw [fill=black] (2.72,-0.5) circle (0.3pt);
\draw [fill=black] (3.4,0.4) circle (0.3pt);
\draw [fill=black] (3.55,0.2) circle (0.3pt);
\draw [fill=black] (3.2,0.55) circle (0.3pt);
\draw [fill=black] (3,-1.6) circle (0.3pt);
\draw [fill=black] (2.85,-1.55) circle (0.3pt);
\draw [fill=black] (3.15,-1.55) circle (0.3pt);
\draw [fill=black] (3.4,0.4) circle (0.3pt);
\draw [fill=black] (3.55,0.2) circle (0.3pt);
\draw [fill=black] (3.2,0.55) circle (0.3pt);
\draw [fill=black] (0.95,-0.85) circle (0.3pt);
\draw [fill=black] (0.8,-0.7) circle (0.3pt);
\draw [fill=black] (0.97,-1.05) circle (0.3pt);
\draw [rotate around={180:(0.6,-1.03)}][fill=black] (0.95,-0.85) circle (0.3pt);
\draw [rotate around={180:(0.6,-1.03)}][fill=black]  (0.8,-0.7) circle (0.3pt);
\draw [rotate around={180:(0.6,-1.03)}][fill=black] (0.97,-1.05) circle (0.3pt);
\draw [fill=black] (-0.95,-0.85) circle (0.3pt);
\draw [fill=black] (-0.8,-0.7) circle (0.3pt);
\draw [fill=black] (-0.97,-1.05) circle (0.3pt);
\draw [rotate around={180:(-0.6,-1.03)}][fill=black] (-0.95,-0.85) circle (0.3pt);
\draw [rotate around={180:(-0.6,-1.03)}][fill=black]  (-0.8,-0.7) circle (0.3pt);
\draw [rotate around={180:(-0.6,-1.03)}][fill=black] (-0.97,-1.05) circle (0.3pt);
\draw [fill=black] (-0.95-0.37,-0.85-0.74) circle (0.3pt);
\draw [fill=black] (-0.8-0.37,-0.7-0.74) circle (0.3pt);
\draw [fill=black] (-0.97-0.38,-1.05-0.74) circle (0.3pt);
\draw [fill=black] (-0.95-0.425-0.425,-0.85-0.74-0.74) circle (0.3pt);
\draw [fill=black] (-0.8-0.425-0.425,-0.7-0.74-0.74) circle (0.3pt);
\draw [fill=black] (-0.97-0.425-0.425,-1.05-0.74-0.74) circle (0.3pt);
\draw[rotate around={180:(-0.6-0.425-0.425,-1.03-0.74-0.74)}] [fill=black] (-0.95-0.425-0.425,-0.85-0.74-0.74) circle (0.3pt);
\draw[rotate around={180:(-0.6-0.425-0.425,-1.03-0.74-0.74)}] [fill=black] (-0.8-0.425-0.425,-0.7-0.74-0.74) circle (0.3pt);
\draw[rotate around={180:(-0.6-0.425-0.425,-1.03-0.74-0.74)}] [fill=black] (-0.97-0.425-0.425,-1.05-0.74-0.74) circle (0.3pt);
\draw [fill=black] (-0.6-0.425-0.425-0.6+0.5,-1.03-0.74-0.74-1.04+0.3) circle (0.3pt);
\draw [fill=black] (-0.6-0.425-0.425-0.6+0.58,-1.03-0.74-0.74-1.04+0.15) circle (0.3pt);
\draw [fill=black] (-0.6-0.425-0.425-0.6+0.6,-1.03-0.74-0.74-1.04) circle (0.3pt);
\draw [rotate around={180:(-0.6-0.425-0.425-0.6,-1.03-0.74-0.74-1.04)}][fill=black] (-0.6-0.425-0.425-0.6+0.4,-1.03-0.74-0.74-1.04+0.45) circle (0.3pt);
\draw[rotate around={180:(-0.6-0.425-0.425-0.6,-1.03-0.74-0.74-1.04)}] [fill=black] (-0.6-0.425-0.425-0.6+0.58,-1.03-0.74-0.74-1.04+0.23) circle (0.3pt);
\draw [rotate around={180:(-0.6-0.425-0.425-0.6,-1.03-0.74-0.74-1.04)}][fill=black] (-0.6-0.425-0.425-0.6+0.6,-1.03-0.74-0.74-1.04-0.1) circle (0.3pt);
\draw [fill=black] (-0.6-0.425-0.425-0.6+1.04+0.26,-1.03-0.74-0.74-1.04-0.6-0.3) circle (0.3pt);
\draw [fill=black] (-0.6-0.425-0.425-0.6+1.04+0.35,-1.03-0.74-0.74-1.04-0.6-0.2) circle (0.3pt);
\draw [fill=black] (-0.6-0.425-0.425-0.6+1.04+0.4,-1.03-0.74-0.74-1.04-0.6-0.05) circle (0.3pt);
\end{scriptsize}
\draw (0,0.25)node[anchor=north]{$\mathbf{H}$};
\draw (-0.6-0.425-0.425-0.6,-1.03-0.74-0.74-1.04+0.25)node[anchor=north]{$\mathbf{H}$};
\draw (3,0.25)node[anchor=north]{$\mathbf{H}$};
\draw (1.5,0.25)node[anchor=north]{$\mathbf{G}$};
\draw (-1.2,0.25)node[anchor=north]{$\mathbf{G}$};
\draw (3,-0.95)node[anchor=north]{$\mathbf{G}$};
\draw (-0.6,-0.78)node[anchor=north]{$\mathbf{G}$};
\draw (0.6,-0.78)node[anchor=north]{$\mathbf{G}$};
\draw (-0.6-0.425-0.425-0.6+1.04,-1.03-0.74-0.74-1.04-0.6+0.25)node[anchor=north]{$\mathbf{G}$};
\draw (-0.6-0.425-0.425,-1.03-0.74-0.74+0.25)node[anchor=north]{$\mathbf{G}$};
\draw (-0.6-0.425-0.425-0.6+1.04,-1.03-0.74-0.74-1.04-0.6-0.85+0.25)node[anchor=north]{$\mathbf{F}$};
\draw(-0.6-0.425-0.425-0.6+1.04+0.74,-1.03-0.74-0.74-1.04-0.6+0.425+0.25)node[anchor=north]{$\mathbf{F}$};
\draw (-0.6-0.425,-1.03-0.74+0.25)node[anchor=north]{$\mathbf{F}$};
\draw (-1.2-0.425,-0.74+0.25)node[anchor=north]{$\mathbf{F}$};
\draw (-1.2-0.425,0.74+0.25)node[anchor=north]{$\mathbf{F}$};
\draw (0.6+0.425,-1.03-0.74+0.25)node[anchor=north]{$\mathbf{F}$};
\draw (3-0.6,-1.2-0.6+0.25)node[anchor=north]{$\mathbf{F}$};
\draw (3-0.6,-1.2-0.6+0.25)node[anchor=north]{$\mathbf{F}$};
\draw (3+0.6,-1.2-0.6+0.25)node[anchor=north]{$\mathbf{F}$};
\draw (1.5-0.425,0.74+0.25)node[anchor=north]{$\mathbf{F}$};
\draw (1.5+0.425,0.74+0.25)node[anchor=north]{$\mathbf{F}$};
\end{tikzpicture}

Therefore, the composition is associative. Now, consider two $d$-pre-Calabi-Yau morphisms $(F_0,s_{d+1}\mathbf{F})$ and $(G_0,s_{d+1}\mathbf{G})$. 
Their composition is the pair formed by $G_0\circ F_0$ and the sum of diagrams of the form \eqref{fig:composition}.
The multinecklace composition of this sum of diagrams and the $d$-pre-Calabi-Yau structure $s_{d+1}M_{\MA}$ is a sum of diagrams of the form

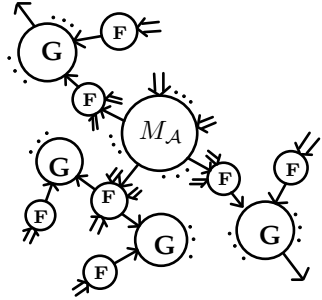
\begin{equation}
\begin{tikzpicture}[line cap=round,line join=round,x=1.0cm,y=1.0cm]
\clip(-7.7,-2.7) rectangle (11.428671538829317,2.8);
     \draw [line width=0.5pt] (0.,0.) circle (0.5cm);
     \shadedraw[rotate=45,shift={(0.5cm,0cm)}] \doublefleche;
     \shadedraw[rotate=135,shift={(0.5cm,0cm)}] \doublefleche;
     \shadedraw [shift={(0cm,1cm)},rotate=-90] \doubleflechescindeeleft;
     \shadedraw [shift={(0cm,1cm)},rotate=-90] \doubleflechescindeeright;
     \shadedraw [shift={(0cm,1cm)},rotate=-90] \fleche;
     \shadedraw[shift={(0.86cm,-0.5cm)},rotate=150] \doubleflechescindeeleft;
     \shadedraw[shift={(0.86cm,-0.5cm)},rotate=150] \doubleflechescindeeright;
     \shadedraw[shift={(0.86cm,-0.5cm)},rotate=150] \fleche;
     \shadedraw[shift={(-0.86cm,-0.5cm)},rotate=30] \doubleflechescindeeleft;
      \shadedraw[shift={(-0.86cm,-0.5cm)},rotate=30] \doubleflechescindeeright;
       \shadedraw[shift={(-0.86cm,-0.5cm)},rotate=30] \fleche;
     \draw [line width=0.5pt] (0,1.3) circle (0.3cm);
     \draw [rotate around={90:(0,1.3)}] [->,> = stealth] (0.3,1.3)--(0.6,1.3);
     \draw [line width=0.5pt] (0,2.2) circle (0.3cm);
     \draw [rotate around={90:(0,2.2)}] [->,> = stealth] (0.3,2.2)--(0.6,2.2);
     \draw [line width=0.5pt] (1.12,-0.65) circle (0.3cm);
     \draw [rotate around ={60:(1.12,-0.65)}] [->,> = stealth] (1.43,-0.65)--(1.73,-0.65);
     \draw [line width=0.5pt] (1.57,0.13) circle (0.3cm);
     \draw[rotate around={-90:(1.57,0.13)}] [->,> = stealth] (1.87,0.13)--(2.17,0.13);
     \draw[rotate around={90:(1.57,0.13)}] [->,> = stealth] (1.87,0.13)--(2.17,0.13);
     \draw[rotate around={120:(1.57,0.13)}] [<-,> = stealth] (1.87,0.13)--(2.17,0.13);
     \draw [line width=0.5pt] (1.12,0.91) circle (0.3cm);
     \shadedraw[rotate around={120:(1.12,0.91)},shift={(1.42cm,0.91cm)}] \doublefleche;
     \draw [rotate around ={-120:(1.12,-0.65)}] [->,> = stealth] (1.43,-0.65)--(1.73,-0.65);
     \draw [line width=0.5pt] (0.67,-1.43) circle (0.3cm);
      \draw[rotate around={90:(0.67,-1.43)}] [->,> = stealth] (0.97,-1.43)--(1.27,-1.43);
      \draw[rotate around={-90:(0.67,-1.43)}] [->,> = stealth] (0.97,-1.43)--(1.27,-1.43);
     \draw[rotate around={-60:(0.67,-1.43)}] [<-,> = stealth] (0.97,-1.43)--(1.27,-1.43);
     \draw [line width=0.5pt] (0.67+0.45,-1.43-0.78) circle (0.3cm);
     \shadedraw[rotate around={-60:(0.67+0.45,-1.43-0.78)},shift={(1.42cm,-2.21cm)}] \doublefleche;
     \draw [line width=0.5pt] (-1.12,-0.65) circle (0.3cm);
     \draw [rotate around ={-60:(-1.12,-0.65)}] [->,> = stealth] (-1.43,-0.65)--(-1.73,-0.65);
     \draw [line width=0.5pt] (-1.12+0.45,-0.65-0.78) circle (0.3cm);
     \draw[rotate around={-60:(-1.12+0.45,-0.65-0.78)}][->,> = stealth] (-1.12+0.75,-0.65-0.78)--(-1.12+1.05,-0.65-0.78);
     \draw [rotate around ={120:(-1.12,-0.65)}] [->,> = stealth] (-1.43,-0.65)--(-1.73,-0.65);
     \draw [line width=0.5pt] (-1.12-0.45,-0.65+0.78) circle (0.3cm);
     \draw[rotate around={0:(-1.12-0.45,-0.65+0.78)}][->,> = stealth] (-1.12-0.15,-0.65+0.78)--(-1.12+0.15,-0.65+0.78);
     \draw[rotate around={180:(-1.12-0.45,-0.65+0.78)}][->,> = stealth] (-1.12-0.15,-0.65+0.78)--(-1.12+0.15,-0.65+0.78);
     \draw[rotate around={-120:(-1.12-0.45,-0.65+0.78)}][<-,> = stealth] (-1.12-0.15,-0.65+0.78)--(-1.12+0.15,-0.65+0.78);
    \draw [line width=0.5pt] (-1.12-0.9,-0.65) circle (0.3cm);
    \shadedraw[rotate around={-120:(-1.12-0.9,-0.65)},shift={(-1.12-0.6,-0.65)}] \doublefleche;
\begin{scriptsize}
\draw [fill=black] (0,-0.6) circle (0.3pt);
\draw [fill=black] (0.2,-0.55) circle (0.3pt);
\draw [fill=black] (-0.2,-0.55) circle (0.3pt);
\draw [fill=black] (1.45,-0.85) circle (0.3pt);
\draw [fill=black] (1.5,-0.67) circle (0.3pt);
\draw [fill=black] (1.33,-0.97) circle (0.3pt);
\draw [fill=black] (-1.45,-0.85) circle (0.3pt);
\draw [fill=black] (-1.5,-0.67) circle (0.3pt);
\draw [fill=black] (-1.33,-0.97) circle (0.3pt);
\draw [fill=black] (0.4,2.2) circle (0.3pt);
\draw [fill=black] (0.35,2) circle (0.3pt);
\draw [fill=black] (0.35,2.4) circle (0.3pt);
\draw [rotate around={180:(0,2.2)}] [fill=black] (0.4,2.2) circle (0.3pt);
\draw [rotate around={180:(0,2.2)}][fill=black] (0.35,2) circle (0.3pt);
\draw [rotate around={180:(0,2.2)}][fill=black] (0.35,2.4) circle (0.3pt);
\draw [fill=black] (1.97,0.1) circle (0.3pt);
\draw [fill=black] (1.93,0.25) circle (0.3pt);
\draw [fill=black] (1.92,-0.07) circle (0.3pt);
\draw [rotate around={180:(1.57,0.13)}] [fill=black] (1.97,0.1) circle (0.3pt);
\draw [rotate around={180:(1.57,0.13)}] [fill=black] (1.93,0.25) circle (0.3pt);
\draw [rotate around={180:(1.57,0.13)}] [fill=black] (1.92,-0.05) circle (0.3pt);
\draw [fill=black] (1.05,-1.43) circle (0.3pt);
\draw [fill=black] (1,-1.23) circle (0.3pt);
\draw [fill=black] (1,-1.63) circle (0.3pt);
\draw [rotate around={180:(0.67,-1.43)}] [fill=black] (1.05,-1.43) circle (0.3pt);
\draw [rotate around={180:(0.67,-1.43)}] [fill=black] (1,-1.23) circle (0.3pt);
\draw [rotate around={180:(0.67,-1.43)}] [fill=black] (1,-1.63) circle (0.3pt);
\draw [fill=black] (-1.12-0.45,-0.65+0.78+0.4) circle (0.3pt);
\draw [fill=black] (-1.12-0.45-0.2,-0.65+0.78+0.34) circle (0.3pt);
\draw [fill=black] (-1.12-0.45+0.2,-0.65+0.78+0.34) circle (0.3pt);
\draw [rotate around={180:(-1.12-0.45,-0.65+0.78)}] [fill=black] (-1.12-0.45,-0.65+0.78+0.4) circle (0.3pt);
\draw [rotate around={180:(-1.12-0.45,-0.65+0.78)}][fill=black] (-1.12-0.45-0.1,-0.65+0.78+0.38) circle (0.3pt);
\draw [rotate around={180:(-1.12-0.45,-0.65+0.78)}][fill=black] (-1.12-0.45+0.1,-0.65+0.78+0.38) circle (0.3pt);
\draw [fill=black] (-1.45+0.65-0.18,-0.85-0.65-0.18) circle (0.3pt);
\draw [fill=black] (-1.5+0.65-0.2,-0.67-0.65-0.15) circle (0.3pt);
\draw [fill=black] (-1.33+0.65-0.15,-0.97-0.65-0.18) circle (0.3pt);
\draw [rotate around={180:(-1.12+0.45,-0.65-0.78)}][fill=black] (-1.45+0.65-0.18,-0.85-0.65-0.18) circle (0.3pt);
\draw [rotate around={180:(-1.12+0.45,-0.65-0.78)}][fill=black] (-1.5+0.65-0.2,-0.67-0.65-0.15) circle (0.3pt);
\draw [rotate around={180:(-1.12+0.45,-0.65-0.78)}][fill=black] (-1.33+0.65-0.15,-0.97-0.65-0.18) circle (0.3pt);
\end{scriptsize}
\draw (0,0.25)node[anchor=north]{$M_{\MA}$};
\draw (0,1.55)node[anchor=north]{$\mathbf{F}$};
\draw (1.12,-0.4)node[anchor=north]{$\mathbf{F}$};
\draw (-1.12,-0.4)node[anchor=north]{$\mathbf{F}$};
\draw (-1.12-0.9,-0.4)node[anchor=north]{$\mathbf{F}$};
\draw (1.12,1.16)node[anchor=north]{$\mathbf{F}$};
\draw (1.12,-1.95)node[anchor=north]{$\mathbf{F}$};

\draw (0,2.45)node[anchor=north]{$\mathbf{G}$};
\draw (0.67,-1.16)node[anchor=north]{$\mathbf{G}$};
\draw (-1.12-0.45,-0.65+0.78+0.25)node[anchor=north]{$\mathbf{G}$};
\draw (1.12+0.45,-0.65+0.78+0.25)node[anchor=north]{$\mathbf{G}$};
\draw (-0.67,-1.43+0.25)node[anchor=north]{$\mathbf{G}$};
\end{tikzpicture}
    \label{fig:proof-cp-1}
\end{equation}

Given $\doubar{x}\in\doubar{\MO}_{\MA}$, we have to sum over all the diagrams of such type and we have several possibilities for the type of the inner diagram, defined as the subdiagram consisting of the disc filled with $M_{\MA}$ together with all the discs sharing an arrow with it. Note that if we fix the type of the outer diagram given as the complement of the inner diagram, the type of the inner one is fixed. Moreover, changing the inner diagram for one of same type does not change the type of the whole diagram.
Therefore, taking the sum over all diagrams of type $\doubar{x}\in\doubar{\MO}_{\MA}$ is the same as taking the sum over all the possible types for the outer diagram and for each of those, taking the sum over all the inner diagrams of the suitable type. 
This second sum allows us to use that $(F_0,s_{d+1}\mathbf{F})$ is a pre-Calabi-Yau morphism to replace the inner diagram by one consisting of a disc filled with $M_{\MB}$ whose incoming arrows are connected with outgoing arrows of discs filled with \textbf{F}.
Then, the sum of all the diagrams of type $\doubar{x}$ of the form \eqref{fig:proof-cp-1} is equal to the sum of all the diagrams of type $\doubar{x}$ of the form
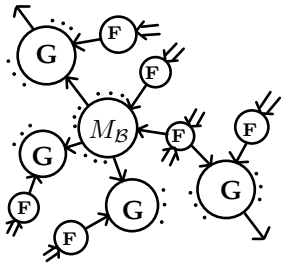
\begin{equation}
\label{fig:proof-cp-2}
\begin{tikzpicture}[line cap=round,line join=round,x=1.0cm,y=1.0cm]
\clip(-8,-1.4) rectangle (15.94,3);
     \draw [line width=0.5pt] (0.,0.) circle (0.5cm);
     \draw[rotate=-45][->,> = stealth] (0.5,0)--(1,0);
     \draw [line width=0.5pt] (0.92,-0.92) circle (0.3cm);
     \draw[rotate around={180:(0.92,-0.92)}][->,> = stealth] (1.22,-0.92)--(1.52,-0.92);
     \draw[rotate around={0:(0.92,-0.92)}][->,> = stealth] (1.22,-0.92)--(1.52,-0.92);
     \draw[rotate around={30:(0.92,-0.92)}][<-,> = stealth] (1.22,-0.92)--(1.52,-0.92);
     \draw [line width=0.5pt] (0.92+0.78,-0.92+0.45) circle (0.3cm);
     \shadedraw[rotate around={30:(0.92+0.78,-0.92+0.45)}, shift={(0.92+0.78+0.3,-0.92+0.45)}] \doublefleche;
    \draw[rotate=-135][->,> = stealth] (0.5,0)--(1,0);
    \draw [line width=0.5pt] (-0.92,-0.92) circle (0.3cm);
     \draw[rotate around={180:(-0.92,-0.92)}][->,> = stealth] (-1.22,-0.92)--(-1.52,-0.92);
     \draw[rotate around={0:(-0.92,-0.92)}][->,> = stealth] (-1.22,-0.92)--(-1.52,-0.92);
     \draw[rotate around={150:(-0.92,-0.92)}][<-,> = stealth] (-0.62,-0.92)--(-0.32,-0.92);
     \draw [line width=0.5pt] (-0.92-0.78,-0.92+0.45) circle (0.3cm);
     \shadedraw[rotate around={150:(-0.92-0.78,-0.92+0.45)}, shift={(-0.92-0.78+0.3,-0.92+0.45)}] \doublefleche;
     \draw[rotate=90][<-,> = stealth] (0.5,0)--(0.9,0);
     \draw [line width=0.5pt] (0,1.2) circle (0.3cm);
     \draw[rotate around={90:(0,1.2)}][->,> = stealth] (0.3,1.2)--(0.7,1.2);
     \draw [line width=0.5pt] (0,2.2) circle (0.3cm);
     \draw[rotate around={90:(0,2.2)}][->,> = stealth] (0.3,2.2)--(0.7,2.2);
     \shadedraw[rotate around={0:(0,1.2)}, shift={(0.3,1.2)}] \doublefleche;
     \shadedraw[rotate around={180:(0,1.2)}, shift={(0.3,1.2)}] \doublefleche;
     \draw[rotate=15][<-,> = stealth] (0.5,0)--(0.9,0);
     \draw [line width=0.5pt] (1.16,0.31) circle (0.3cm);
     \shadedraw[rotate around={15:(1.16,0.31)},shift={(1.46,0.31)}] \doublefleche;
    \draw[rotate=165][<-,> = stealth] (0.5,0)--(0.9,0);
     \draw [line width=0.5pt] (-1.16,0.31) circle (0.3cm);
     \shadedraw[rotate around={165:(-1.16,0.31)},shift={(-0.86,0.31)}] \doublefleche;
\begin{scriptsize}
\draw [fill=black] (0,-0.6) circle (0.3pt);
\draw [fill=black] (0.2,-0.55) circle (0.3pt);
\draw [fill=black] (-0.2,-0.55) circle (0.3pt);
\draw [rotate=150][fill=black] (-0.05,-0.6) circle (0.3pt);
\draw [rotate=150][fill=black] (0.15,-0.55) circle (0.3pt);
\draw [rotate=150][fill=black] (-0.22,-0.55) circle (0.3pt);
\draw [rotate=210][fill=black] (0.05,-0.6) circle (0.3pt);
\draw [rotate=210][fill=black] (0.2,-0.55) circle (0.3pt);
\draw [rotate=210][fill=black] (-0.12,-0.58) circle (0.3pt);
\draw [fill=black] (0.35,2) circle (0.3pt);
\draw [fill=black] (0.4,2.2) circle (0.3pt);
\draw [fill=black] (0.35,2.4) circle (0.3pt);
\draw [rotate around={180:(0,2.2)}][fill=black] (0.35,2) circle (0.3pt);
\draw [rotate around={180:(0,2.2)}][fill=black] (0.4,2.2) circle (0.3pt);
\draw [rotate around={180:(0,2.2)}][fill=black] (0.35,2.4) circle (0.3pt);
\draw [fill=black] (0.92,-1.32) circle (0.3pt);
\draw [fill=black] (0.72,-1.28) circle (0.3pt);
\draw [fill=black] (1.12,-1.28) circle (0.3pt);
\draw [rotate around={180:(0.92,-0.92)}][fill=black] (0.92,-1.32) circle (0.3pt);
\draw [rotate around={180:(0.92,-0.92)}][fill=black] (0.74,-1.28) circle (0.3pt);
\draw [rotate around={180:(0.92,-0.92)}][fill=black]  (1.1,-1.28) circle (0.3pt);
\draw [fill=black] (-0.92,-1.32) circle (0.3pt);
\draw [fill=black] (-0.72,-1.28) circle (0.3pt);
\draw [fill=black] (-1.12,-1.28) circle (0.3pt);
\draw [rotate around={180:(-0.92,-0.92)}][fill=black] (-0.92,-1.32) circle (0.3pt);
\draw [rotate around={180:(-0.92,-0.92)}][fill=black] (-0.74,-1.28) circle (0.3pt);
\draw [rotate around={180:(-0.92,-0.92)}][fill=black]  (-1.1,-1.28) circle (0.3pt);
\end{scriptsize}
\draw (0,0.25)node[anchor=north]{$M_{\MB}$};
\draw (0,1.45)node[anchor=north]{$\mathbf{F}$};
\draw (-1.16,0.56)node[anchor=north]{$\mathbf{F}$};
\draw (1.16,0.56)node[anchor=north]{$\mathbf{F}$};
\draw  (-0.92-0.78,-0.92+0.7)node[anchor=north]{$\mathbf{F}$};
\draw  (0.92+0.78,-0.92+0.7)node[anchor=north]{$\mathbf{F}$};
\draw (0,2.45)node[anchor=north]{$\mathbf{G}$};
\draw (-0.92,-0.67)node[anchor=north]{$\mathbf{G}$};
\draw (0.92,-0.67)node[anchor=north]{$\mathbf{G}$};
\end{tikzpicture}
\end{equation}
\noindent
and we now define the inner diagram as the filled diagram consisting of the disc filled with $M_{\MB}$ and of all the discs connected to it. The previous remarks on the types of the inner and outer diagrams still hold.
Thus, the sum over all possible types for the whole diagram is again the sum over all the possible types for the outer diagram and for each of those, taking the sum over all the suitable types for the inner one.
$(G_0,s_{d+1}\mathbf{G})$ being a pre-Calabi-Yau morphism, the sum of all the diagrams of type $\doubar{x}$ of the form \eqref{fig:proof-cp-2} is now equal to the sum of all the diagrams of type $\doubar{x}$ of the form 

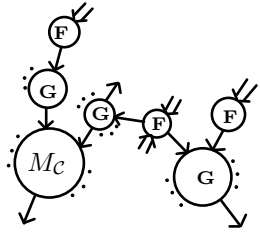
\begin{tikzpicture}[line cap=round,line join=round,x=1.0cm,y=1.0cm]
\clip(-7.5,-1.2) rectangle (9.613451596865602,4);
     \draw [line width=0.5pt] (0.,0.) circle (0.5cm);
     \draw[rotate=-45][->,> = stealth] (0.5,0)--(1,0);
    \draw[rotate=-135][->,> = stealth] (0.5,0)--(1,0);
     \draw[rotate=90][<-,> = stealth] (0.5,0)--(0.9,0);
     \draw [line width=0.5pt] (0,1.2) circle (0.3cm);
     \draw[rotate around={90:(0,1.2)}][<-,> = stealth] (0.3,1.2)--(0.6,1.2);
     \draw [line width=0.5pt] (0,2.1) circle (0.3cm);
     \draw[rotate around={90:(0,2.1)}][->,> = stealth] (0.3,2.1)--(0.6,2.1);
     \draw [line width=0.5pt] (0,3) circle (0.3cm);
     \draw[rotate around={90:(0,3)}][->,> = stealth] (0.3,3)--(0.7,3);
     \shadedraw[rotate around={0:(0,2.1)}, shift={(0.3,2.1)}] \doublefleche;
     \shadedraw[rotate around={180:(0,2.1)}, shift={(0.3,2.1)}] \doublefleche;
     \draw[rotate=15][<-,> = stealth] (0.5,0)--(0.9,0);
     \draw [line width=0.5pt] (1.16,0.31) circle (0.3cm);
     \draw[rotate around={45:(1.16,0.31)}][<-,> = stealth] (1.46,0.31)--(1.76,0.31);
     \draw[rotate around={-15:(1.16,0.31)}][<-,> = stealth] (1.46,0.31)--(1.76,0.31);
     \draw [line width=0.5pt] (1.8,0.95) circle (0.3cm);
     \shadedraw[rotate around={45:(1.8,0.95)}, shift={(2.1,0.95)}] \doublefleche;
     \draw [line width=0.5pt] (1.16+0.87,0.08) circle (0.3cm);
     \shadedraw[rotate around={-15:(1.16+0.87,0.08)}, shift={(1.46+0.87,0.08)}] \doublefleche;
    \draw[rotate=165][<-,> = stealth] (0.5,0)--(0.9,0);
     \draw [line width=0.5pt] (-1.16,0.31) circle (0.3cm);
    \draw[rotate around={135:(-1.16,0.31)}][<-,> = stealth] (-0.86,0.31)--(-0.56,0.31);
     \draw[rotate around={195:(-1.16,0.31)}][<-,> = stealth] (-0.86,0.31)--(-0.56,0.31);
     \draw [line width=0.5pt] (-1.8,0.95) circle (0.3cm);
     \shadedraw[rotate around={135:(-1.8,0.95)}, shift={(-1.5,0.95)}] \doublefleche;
     \draw [line width=0.5pt] (-1.16-0.87,0.08) circle (0.3cm);
     \shadedraw[rotate around={195:(-1.16-0.87,0.08)}, shift={(-1.16-0.57,0.08)}] \doublefleche;
\begin{scriptsize}
\draw [fill=black] (0,-0.6) circle (0.3pt);
\draw [fill=black] (0.2,-0.55) circle (0.3pt);
\draw [fill=black] (-0.2,-0.55) circle (0.3pt);
\draw [rotate=150][fill=black] (-0.05,-0.6) circle (0.3pt);
\draw [rotate=150][fill=black] (0.15,-0.55) circle (0.3pt);
\draw [rotate=150][fill=black] (-0.22,-0.55) circle (0.3pt);
\draw [rotate=210][fill=black] (0.05,-0.6) circle (0.3pt);
\draw [rotate=210][fill=black] (0.2,-0.55) circle (0.3pt);
\draw [rotate=210][fill=black] (-0.12,-0.58) circle (0.3pt);
\draw [fill=black] (0.35,1) circle (0.3pt);
\draw [fill=black] (0.4,1.2) circle (0.3pt);
\draw [fill=black] (0.35,1.4) circle (0.3pt);
\draw [rotate around={180:(0,1.2)}][fill=black] (0.35,1) circle (0.3pt);
\draw [rotate around={180:(0,1.2)}][fill=black] (0.4,1.2) circle (0.3pt);
\draw [rotate around={180:(0,1.2)}][fill=black] (0.35,1.4) circle (0.3pt);
\draw [fill=black] (0.4,3) circle (0.3pt);
\draw [fill=black] (0.34,2.8) circle (0.3pt);
\draw [fill=black] (0.35,3.2) circle (0.3pt);
\draw [rotate around={180:(0,3)}][fill=black] (0.4,3) circle (0.3pt);
\draw [rotate around={180:(0,3)}][fill=black] (0.34,2.8) circle (0.3pt);
\draw [rotate around={180:(0,3)}][fill=black] (0.35,3.2) circle (0.3pt);
\draw [fill=black] (1.51,0.5) circle (0.3pt);
\draw [fill=black] (1.55,0.42) circle (0.3pt);
\draw [fill=black] (1.54,0.33) circle (0.3pt);
\draw [rotate=150][fill=black] (1.51,0.5) circle (0.3pt);
\draw [rotate=150][fill=black] (1.55,0.42) circle (0.3pt);
\draw [rotate=150][fill=black] (1.54,0.33) circle (0.3pt);
(1.16,0.31)
\end{scriptsize}
\draw (0,0.25)node[anchor=north]{$M_{\mathcal{C}}$};
\draw (0,1.45)node[anchor=north]{$\mathbf{G}$};
\draw (0,3.25)node[anchor=north]{$\mathbf{G}$};
\draw (1.16,0.56)node[anchor=north]{$\mathbf{G}$};
\draw (-1.16,0.56)node[anchor=north]{$\mathbf{G}$};
\draw (0,2.35)node[anchor=north]{$\mathbf{F}$};
\draw (-1.16-0.87,0.08+0.25)node[anchor=north]{$\mathbf{F}$};
\draw (-1.8,0.95+0.25)node[anchor=north]{$\mathbf{F}$};
\draw (1.16+0.87,0.08+0.25)node[anchor=north]{$\mathbf{F}$};
\draw (1.8,0.95+0.25)node[anchor=north]{$\mathbf{F}$};
\end{tikzpicture}

Therefore, $(G_0\circ F_0,s_{d+1}\mathbf{G}\circ s_{d+1}\mathbf{F})$ is a $d$-pre-Calabi-Yau morphism.
\end{proof}
We end this subsection with the definitions of two classes of morphisms which will be useful in the next section.
\begin{definition}
\label{def:good-morphism}
Given $d$-pre-Calabi-Yau categories $(\MA,s_{d+1}M_{\MA})$ and $(\MB,s_{d+1}M_{\MB})$ with respective sets of objects $\MO_{\MA}$ and $\MO_{\MB}$ and a $d$-pre-Calabi-Yau morphism $F=(F_0,s_{d+1}\mathbf{F})$ between them, we say that $F$ is \textbf{\textcolor{ultramarine}{good}} if $\sum\mathcal{E}(\mathcalboondox{D})= \sum\mathcal{E}(\mathcalboondox{D'})$ where the sums are over all the filled diagrams $\mathcalboondox{D}$ and $\mathcalboondox{D'}$ of type $\doubar{x}$ of the form

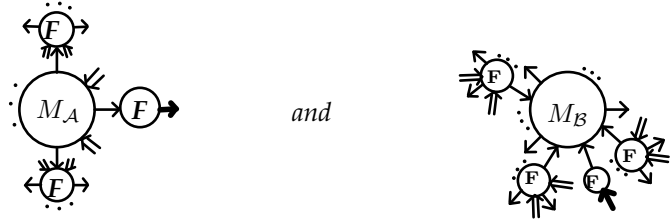
\begin{minipage}{21cm}
\begin{tikzpicture}[line cap=round,line join=round,x=1.0cm,y=1.0cm]
\clip(-4,-2) rectangle (4,2);
     \draw [line width=0.5pt] (0.,0.) circle (0.5cm);
     \shadedraw[rotate=45,shift={(0.5cm,0cm)}] \doublefleche;
     \shadedraw[rotate=135,shift={(0.5cm,0cm)}] \doublefleche;
     \draw [line width=0.5pt] (0,1.3) circle (0.3cm);
     \shadedraw [shift={(0cm,1cm)},rotate=-90] \fleche;
     \draw [line width=0.5pt] (1.12,-0.65) circle (0.3cm);
     \shadedraw[shift={(0.86cm,-0.5cm)},rotate=150] \doubleflechescindeeleft;
     \shadedraw[shift={(0.86cm,-0.5cm)},rotate=150] \doubleflechescindeeright;
     \shadedraw[shift={(0.86cm,-0.5cm)},rotate=150] \fleche;
     \draw [rotate around ={60:(1.12,-0.65)}] [->,> = stealth] (1.43,-0.65)--(1.73,-0.65);
     \draw [rotate around ={-120:(1.12,-0.65)}] [->,> = stealth] (1.43,-0.65)--(1.73,-0.65);
     \draw [line width=0.5pt] (-1.12,-0.65) circle (0.3cm);
     \shadedraw[shift={(-0.86cm,-0.5cm)},rotate=30] \doubleflechescindeeleft;
      \shadedraw[shift={(-0.86cm,-0.5cm)},rotate=30] \doubleflechescindeeright;
       \shadedraw[shift={(-0.86cm,-0.5cm)},rotate=30] \fleche;
     \draw[line width=1.1pt,rotate around={90:(0,1.3)}] [->,> = stealth] (0.3,1.3)--(0.6,1.3);
     \draw [line width=0.5pt] (1.12,-0.65) circle (0.3cm);
     \draw [rotate around ={60:(1.12,-0.65)}] [->,> = stealth] (1.43,-0.65)--(1.73,-0.65);
     \draw [rotate around ={-120:(1.12,-0.65)}] [->,> = stealth] (1.43,-0.65)--(1.73,-0.65);
     \draw [line width=0.5pt] (-1.12,-0.65) circle (0.3cm);
      \draw [rotate around ={-60:(-1.12,-0.65)}] [->,> = stealth] (-1.43,-0.65)--(-1.73,-0.65);
     \draw [rotate around ={120:(-1.12,-0.65)}] [->,> = stealth] (-1.43,-0.65)--(-1.73,-0.65);
\begin{scriptsize}
\draw [fill=black] (0,-0.6) circle (0.3pt);
\draw [fill=black] (0.2,-0.55) circle (0.3pt);
\draw [fill=black] (-0.2,-0.55) circle (0.3pt);
\draw [fill=black] (1.45,-0.85) circle (0.3pt);
\draw [fill=black] (1.5,-0.67) circle (0.3pt);
\draw [fill=black] (1.33,-0.97) circle (0.3pt);
\draw [fill=black] (-1.45,-0.85) circle (0.3pt);
\draw [fill=black] (-1.5,-0.67) circle (0.3pt);
\draw [fill=black] (-1.33,-0.97) circle (0.3pt);
\end{scriptsize}
\draw (0,0.25)node[anchor=north]{$M_{\MA}$};
\draw (0,1.55)node[anchor=north]{$\mathbf{F}$};
\draw (-1.12,-0.4)node[anchor=north]{$\mathbf{F}$};
\draw (1.12,-0.4)node[anchor=north]{$\mathbf{F}$};
\draw (3,0.25)node[anchor=north]{and};
\end{tikzpicture}
\label{fig:good-1}
\begin{tikzpicture}[line cap=round,line join=round,x=1.0cm,y=1.0cm]
\clip(-2,-2) rectangle (7.635492124039547,2);
     \draw [line width=0.5pt] (0.,0.) circle (0.5cm);
     \draw [rotate=90] [->,> = stealth] (0.5,0)--(0.9,0);
     \draw [rotate=-30] [->,> = stealth] (0.5,0)--(0.9,0);
     \draw [rotate=-150] [->,> = stealth] (0.5,0)--(0.9,0);
     \draw [rotate=0] [<-,> = stealth] (0.5,0)--(0.9,0);
     \draw [line width=0.5pt] (1.15,0) circle (0.25cm);
     \draw[rotate around={60:(1.15,0)}] [->,> = stealth] (1.4,0)--(1.7,0);
      \draw[rotate around={-60:(1.15,0)}] [->,> = stealth] (1.4,0)--(1.7,0);
      \shadedraw[rotate around={120:(1.15,0)}, shift={(1.4cm,0cm)}] \doublefleche;
       \shadedraw[shift={(1.4cm,0cm)}] \doublefleche;
     \draw [rotate=60] [<-,> = stealth] (0.5,0)--(0.9,0);
       \draw [rotate=185] [<-,> = stealth] (0.5,0)--(0.9,0);
      \draw [line width=0.5pt] (-1.15,-0.1) circle (0.25cm);
      \draw[line width=1.1pt,rotate around={120:(-1.15,-0.1)}] [->,> = stealth] (-0.9,-0.1)--(-0.6,-0.1);
      \draw[rotate around={-120:(-1.15,-0.1)}] [->,> = stealth] (-0.9,-0.1)--(-0.6,-0.1);
      \shadedraw[rotate around={60:(-1.15,-0.1)}, shift={(-0.9cm,-0.1cm)}] \doublefleche;
       \shadedraw[rotate around={180:(-1.15,-0.1)},shift={(-0.9cm,-0.1cm)}] \doublefleche;
        \draw [rotate=120] [<-,> = stealth] (0.5,0)--(0.9,0);
        \draw [line width=0.5pt] (0.575,1) circle (0.25cm);
        \draw[rotate around={120:(0.575,1)}] [->,> = stealth] (0.825,1)--(1.125,1);
      \draw[rotate around={0:(0.575,1)}] [->,> = stealth] (0.825,1)--(1.125,1);
      \shadedraw[rotate around={60:(0.575,1)}, shift={(0.825cm,1cm)}] \doublefleche;
       \shadedraw[rotate around={180:(0.575,1)},shift={(0.825cm,1cm)}] \doublefleche;
        \draw [line width=0.5pt] (-0.575,1) circle (0.25cm);
      \draw[rotate around={-120:(-0.575,1)}] [->,> = stealth] (-0.825,1)--(-1.125,1);
      \draw[rotate around={0:(-0.575,1)}] [->,> = stealth] (-0.825,1)--(-1.125,1);
      \shadedraw[rotate around={120:(-0.575,1)}, shift={(-0.325cm,1cm)}] \doublefleche;
       \shadedraw[rotate around={0:(-0.575,1)},shift={(-0.325cm,1cm)}] \doublefleche;
      \draw[rotate around={150:(0,0)}] [<-,> = stealth] (0.5,0)--(0.9,0);
    \draw [line width=0.5pt] (-1,0.575) circle (0.25cm);
    \draw[rotate around={150:(-1,0.575)}] [<-,> = stealth] (-0.75,0.575)--(-0.45,0.575);
\begin{scriptsize}
\draw [fill=black] (0.65,0.7) circle (0.3pt);
\draw [fill=black] (0.79,0.75) circle (0.3pt);
\draw [fill=black] (0.88,0.85) circle (0.3pt);
\draw [fill=black] (1,-0.3) circle (0.3pt);
\draw [fill=black] (1.15,-0.3) circle (0.3pt);
\draw [fill=black] (0.9,-0.2) circle (0.3pt);
\draw [fill=black] (0,-0.6) circle (0.3pt);
\draw [fill=black] (0.2,-0.55) circle (0.3pt);
\draw [fill=black] (-0.2,-0.55) circle (0.3pt);
\draw [rotate=120][fill=black] (0,-0.6) circle (0.3pt);
\draw [rotate=120][fill=black] (0.15,-0.55) circle (0.3pt);
\draw [rotate=120][fill=black] (-0.15,-0.55) circle (0.3pt);
\draw [rotate around={-120:(-0.575,1)}][fill=black] (-0.28,1) circle (0.3pt);
\draw [rotate around={-120:(-0.575,1)}][fill=black] (-0.3,0.85) circle (0.3pt);
\draw [rotate around={-120:(-0.575,1)}][fill=black] (-0.3,1.15) circle (0.3pt);
\draw [rotate=-32][fill=black] (-0.38,0.47) circle (0.3pt);
\draw [rotate=-32][fill=black] (-0.44,0.43) circle (0.3pt);
\draw [rotate=-32][fill=black] (-0.47,0.37) circle (0.3pt);
\draw [fill=black] (-0.59,-0.14) circle (0.3pt);
\draw [fill=black] (-0.58,-0.20) circle (0.3pt);
\draw [fill=black] (-0.56,-0.26) circle (0.3pt);
\draw [fill=black] (-0.95,-0.35) circle (0.3pt);
\draw [fill=black] (-0.85,-0.25) circle (0.3pt);
\draw [fill=black] (-1.1,-0.4) circle (0.3pt);
\end{scriptsize}
\draw (0,0.25)node[anchor=north]{$M_{\MB}$};
\draw (-1.15,0.15)node[anchor=north]{$\mathbf{F}$};
\draw (1.15,0.25)node[anchor=north]{$\mathbf{F}$};
\draw (0.575,1.25)node[anchor=north]{$\mathbf{F}$};
\draw (-0.575,1.25)node[anchor=north]{$\mathbf{F}$};
\draw (-0.98,0.82)node[anchor=north]{$\mathbf{F}$};
\end{tikzpicture}
\end{minipage}
respectively. The set of diagrams of type $\doubar{x}$ and of the left (resp. right) hand form is the subset of the set of diagrams of type $\doubar{x}$ and of the form \eqref{fig:morph-1} (resp. \eqref{fig:morph-2}) consisting of all diagrams whose last outgoing arrow is the outgoing arrow (resp. just before the incoming arrow) of a disc filled with $\mathbf{F}$ having only one incoming and one outgoing arrow.
\end{definition}
Since this condition is not closed under the pre-Calabi-Yau composition, we restrict it to the following.

\begin{definition}
\label{def:nice-subcat}
Given $d$-pre-Calabi-Yau categories $(\MA,s_{d+1}M_{\MA})$ and $(\MB,s_{d+1}M_{\MB})$ with respective sets of objects $\MO_{\MA}$ and $\MO_{\MB}$ and a $d$-pre-Calabi-Yau morphism $F=(F_0,s_{d+1}\mathbf{F})$ between them, we say that $F$ is \textbf{\textcolor{ultramarine}{nice}} if $\sum\mathcal{E}(\mathcalboondox{D})= \sum\mathcal{E}(\mathcalboondox{D'})$ where the sums are over all the filled diagrams $\mathcalboondox{D}$ and $\mathcalboondox{D'}$ of type $\doubar{x}$ of the form

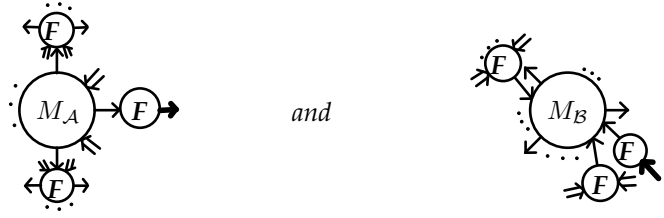
\begin{minipage}{21cm}
\begin{tikzpicture}[line cap=round,line join=round,x=1.0cm,y=1.0cm]
\clip(-4,-2) rectangle (4,2);
    \draw [line width=0.5pt] (0.,0.) circle (0.5cm);
     \shadedraw[rotate=45,shift={(0.5cm,0cm)}] \doublefleche;
     \shadedraw[rotate=135,shift={(0.5cm,0cm)}] \doublefleche;
     \draw [line width=0.5pt] (0,1.3) circle (0.3cm);
     \shadedraw [shift={(0cm,1cm)},rotate=-90] \fleche;
     \draw [line width=0.5pt] (1.12,-0.65) circle (0.3cm);
     \shadedraw[shift={(0.86cm,-0.5cm)},rotate=150] \doubleflechescindeeleft;
     \shadedraw[shift={(0.86cm,-0.5cm)},rotate=150] \doubleflechescindeeright;
     \shadedraw[shift={(0.86cm,-0.5cm)},rotate=150] \fleche;
     \draw [rotate around ={60:(1.12,-0.65)}] [->,> = stealth] (1.43,-0.65)--(1.73,-0.65);
     \draw [rotate around ={-120:(1.12,-0.65)}] [->,> = stealth] (1.43,-0.65)--(1.73,-0.65);
     \draw [line width=0.5pt] (-1.12,-0.65) circle (0.3cm);
     \shadedraw[shift={(-0.86cm,-0.5cm)},rotate=30] \doubleflechescindeeleft;
      \shadedraw[shift={(-0.86cm,-0.5cm)},rotate=30] \doubleflechescindeeright;
       \shadedraw[shift={(-0.86cm,-0.5cm)},rotate=30] \fleche;
     \draw[line width=1.1pt,rotate around={90:(0,1.3)}] [->,> = stealth] (0.3,1.3)--(0.6,1.3);
     \draw [line width=0.5pt] (1.12,-0.65) circle (0.3cm);
     \draw [rotate around ={60:(1.12,-0.65)}] [->,> = stealth] (1.43,-0.65)--(1.73,-0.65);
     \draw [rotate around ={-120:(1.12,-0.65)}] [->,> = stealth] (1.43,-0.65)--(1.73,-0.65);
     \draw [line width=0.5pt] (-1.12,-0.65) circle (0.3cm);
      \draw [rotate around ={-60:(-1.12,-0.65)}] [->,> = stealth] (-1.43,-0.65)--(-1.73,-0.65);
     \draw [rotate around ={120:(-1.12,-0.65)}] [->,> = stealth] (-1.43,-0.65)--(-1.73,-0.65);
\begin{scriptsize}
\draw [fill=black] (0,-0.6) circle (0.3pt);
\draw [fill=black] (0.2,-0.55) circle (0.3pt);
\draw [fill=black] (-0.2,-0.55) circle (0.3pt);
\draw [fill=black] (1.45,-0.85) circle (0.3pt);
\draw [fill=black] (1.5,-0.67) circle (0.3pt);
\draw [fill=black] (1.33,-0.97) circle (0.3pt);
\draw [fill=black] (-1.45,-0.85) circle (0.3pt);
\draw [fill=black] (-1.5,-0.67) circle (0.3pt);
\draw [fill=black] (-1.33,-0.97) circle (0.3pt);
\end{scriptsize}
\draw (0,0.25)node[anchor=north]{$M_{\MA}$};
\draw (0,1.55)node[anchor=north]{$\mathbf{F}$};
\draw (-1.12,-0.4)node[anchor=north]{$\mathbf{F}$};
\draw (1.12,-0.4)node[anchor=north]{$\mathbf{F}$};
\draw (3,0.25)node[anchor=north]{and};
\end{tikzpicture}
\begin{tikzpicture}[line cap=round,line join=round,x=1.0cm,y=1.0cm]
\clip(-2,-2) rectangle (5,2);
     \draw [line width=0.5pt] (0.,0.) circle (0.5cm);
     \draw [rotate=90] [->,> = stealth] (0.5,0)--(0.9,0);
     \draw [rotate=-30] [->,> = stealth] (0.5,0)--(0.9,0);
     \draw [line width=1.1pt,rotate=-150] [->,> = stealth] (0.5,0)--(0.9,0);
     \draw [rotate=120] [<-,> = stealth] (0.5,0)--(0.9,0);
     \draw [rotate=0] [<-,> = stealth] (0.5,0)--(0.9,0);
     \draw [line width=0.5pt] (1.15,0.) circle (0.25cm);
     \draw[rotate around={60:(1.15,0)}] [->,> = stealth] (1.4,0)--(1.7,0);
      \draw[rotate around={-60:(1.15,0)}] [->,> = stealth] (1.4,0)--(1.7,0);
      \shadedraw[rotate around={120:(1.15,0)}, shift={(1.4cm,0cm)}] \doublefleche;
       \shadedraw[shift={(1.4cm,0cm)}] \doublefleche;
     \draw [rotate=60] [<-,> = stealth] (0.5,0)--(0.9,0);
       \draw [rotate=180] [<-,> = stealth] (0.5,0)--(0.9,0);
      \draw [line width=0.5pt] (-1.15,0.) circle (0.25cm);
     \draw[rotate around={0:(-1.15,0)}] [<-,> = stealth] (-1.4,0)--(-1.7,0);
        \draw [line width=0.5pt] (0.575,1) circle (0.25cm);
        \draw[rotate around={120:(0.575,1)}] [->,> = stealth] (0.825,1)--(1.125,1);
      \draw[rotate around={0:(0.575,1)}] [->,> = stealth] (0.825,1)--(1.125,1);
      \shadedraw[rotate around={60:(0.575,1)}, shift={(0.825cm,1cm)}] \doublefleche;
       \shadedraw[rotate around={180:(0.575,1)},shift={(0.825cm,1cm)}] \doublefleche;
        \draw [line width=0.5pt] (-0.575,1) circle (0.25cm);
      \draw[rotate around={-120:(-0.575,1)}] [->,> = stealth] (-0.825,1)--(-1.125,1);
      \draw[rotate around={0:(-0.575,1)}] [->,> = stealth] (-0.825,1)--(-1.125,1);
      \shadedraw[rotate around={-120:(-0.575,1)}, shift={(-0.325cm,1cm)}] \doublefleche;
       \shadedraw[rotate around={120:(-0.575,1)},shift={(-0.325cm,1cm)}] \doublefleche;
\begin{scriptsize}
\draw [fill=black] (0.65,0.7) circle (0.3pt);
\draw [fill=black] (0.79,0.75) circle (0.3pt);
\draw [fill=black] (0.88,0.85) circle (0.3pt);
\draw [fill=black] (-0.25,1) circle (0.3pt);
\draw [fill=black] (-0.29,0.85) circle (0.3pt);
\draw [fill=black] (-0.29,1.15) circle (0.3pt);
\draw [fill=black] (1,-0.3) circle (0.3pt);
\draw [fill=black] (1.15,-0.3) circle (0.3pt);
\draw [fill=black] (0.9,-0.2) circle (0.3pt);
\draw [fill=black] (0,-0.6) circle (0.3pt);
\draw [fill=black] (0.2,-0.55) circle (0.3pt);
\draw [fill=black] (-0.2,-0.55) circle (0.3pt);
\draw [rotate=120][fill=black] (0,-0.6) circle (0.3pt);
\draw [rotate=120][fill=black] (0.15,-0.55) circle (0.3pt);
\draw [rotate=120][fill=black] (-0.15,-0.55) circle (0.3pt);
\draw [rotate=240][fill=black] (0,-0.6) circle (0.3pt);
\draw [rotate=240][fill=black] (0.15,-0.55) circle (0.3pt);
\draw [rotate=240][fill=black] (-0.15,-0.55) circle (0.3pt);
\end{scriptsize}
\draw (0,0.25)node[anchor=north]{$M_{\MB}$};
\draw (-1.15,0.25)node[anchor=north]{$\mathbf{F}$};
\draw (1.15,0.25)node[anchor=north]{$\mathbf{F}$};
\draw (0.575,1.25)node[anchor=north]{$\mathbf{F}$};
\draw (-0.575,1.25)node[anchor=north]{$\mathbf{F}$};
\end{tikzpicture}
\end{minipage}

\noindent respectively.  The set of diagrams of type $\doubar{x}$ and of the left hand form is the same as the one appearing in the definition of a good morphism. 
The set of diagrams of type $\doubar{x}$ and of the right hand form is the subset of the set of diagrams of type $\doubar{x}$ and of the form \eqref{fig:morph-2} consisting of all diagrams whose first incoming arrow is the incoming arrow of a disc filled with $\mathbf{F}$ place just after (in the clockwise order) an output of the disc filled with $M_{\MB}$ and having only one incoming and one outgoing arrow.
\end{definition}

\section{\texorpdfstring{The relation between pre-Calabi-Yau mor\-phisms and $A_{\infty}$-mor\-phisms}{Relation-pCY-morphisms-A-infinity-morphisms}} \label{main section}
In this section, we study the relation between $d$-pre-Calabi-Yau morphisms and $A_{\infty}$-morphisms.
\subsection{The mixed necklace graded Lie algebra}
\label{mixed necklace}
We first define a ``mixed" necklace bracket, which will be useful in the next subsection for the construction of the $A_{\infty}$-structure on $\MA\oplus\MB^*[d-1]$.
\begin{definition}
    Given two graded quivers $\MA$ and $\MB$ with respective sets of objects $\MO_{\MA}$ and $\MO_{\MB}$ and a map $\Phi : \MO_{\MA}\rightarrow\MO_{\MB}$,
    we define the graded vector spaces
    \[
    \mathcalboondox{B}^{\bullet}_{d,\Phi}(\MA,\MB)^{\MA}=\prod_{n\in\NN^*}\prod_{\doubar{x}\in\bar{\MO}_{\MA}^n}\Homgr_{\kk}\big(\bigotimes\limits_{i=1}^n\MA[1]^{\otimes \bar{x}^i},\bigotimes\limits_{i=1}^{n-1}{}_{\Phi(\llt(\bar{x}^i))}\MB_{\Phi(\rrt(\bar{x}^{i+1}))}[-d]\otimes {}_{\llt(\bar{x}^n)}\MA_{\rrt(\bar{x}^1)}[1]\big)
    \]
    and
    \begin{small}
    \begin{equation}
        \begin{split}
            \mathcalboondox{B}^{\bullet}_{d,\Phi}(\MA,\MB)^{\MB}&=\prod_{n\in\NN^*}\prod_{\doubar{x}\in\bar{\MO}_{\MA}^n}\Homgr_{\kk}\big(\MA[1]^{\otimes \bar{x}^1}\otimes {}_{\Phi(\rrt(\bar{x}^1))}\MB_{\Phi(\llt(\bar{x}^n))}[-d] \otimes \MA[1]^{\otimes \bar{x}^n}\otimes\dots\otimes \MA[1]^{\otimes \bar{x}^{n-1}},\\&\hskip8cm\bigotimes\limits_{i=1}^{n-1}{}_{\Phi(\llt(\bar{x}^i))}\MB_{\Phi(\rrt(\bar{x}^{i+1}))}[-d] \big).
        \end{split}
    \end{equation}
     \end{small}
We denote $\mathcalboondox{B}^{\bullet}_{d,\Phi}(\MA,\MB)=\mathcalboondox{B}^{\bullet}_{d,\Phi}(\MA,\MB)^{\MA}\oplus \mathcalboondox{B}^{\bullet}_{d,\Phi}(\MA,\MB)^{\MB}$.
\end{definition}

\begin{remark}
\label{remark:quiver-q-phi}
    We identify elements of $\mathcalboondox{B}_{d,\Phi}(\MA,\MB)^{\MA}$ and $\mathcalboondox{B}_{d,\Phi}(\MA,\MB)^{\MB}$ with elements of the graded vector space $\Multi_{d,\Phi}^{\bullet}(\mathcal{Q}_{\Phi})[d+1]$ where $\mathcal{Q}_{\Phi}$ is the graded quiver whose set of objects is $\MO_{\MA}$ and whose spaces of morphisms are defined as ${}_y(\mathcal{Q}_{\Phi})_x={}_{y}\MA_x\oplus{}_{\Phi(y)}\MB_{\Phi(x)}$. It suffices to use the inverse of the isomorphism \eqref{eq:iso-shift-tensor-prod} together with the inverse of either the isomorphism \eqref{eq:iso-shift-output} or the isomorphism \eqref{eq:iso-shift-input}. Moreover, we will depict elements of $\Multi_{d,\Phi}^{\bullet}(\mathcal{Q}_{\Phi})[d+1]$ by diagrams as we did for elements of $\Multi_{d}^{\bullet}(\MA)[d+1]$.
\end{remark}

\begin{definition}
\label{def:mixed-necklace}
Let $\MA$ and $\MB$ be graded quivers with respective sets of objects $\MO_{\MA}$ and $\MO_{\MB}$ and consider a map $\Phi : \MO_{\MA}\rightarrow\MO_{\MB}$.
Given $s_{d+1}\mathbf{F}, s_{d+1}\mathbf{G} \in \mathcalboondox{B}^{\bullet}_{d,\Phi}(\MA,\MB)$, we define their \textbf{\textcolor{ultramarine}{$\Phi$-mixed necklace product}} as the element $s_{d+1}\mathbf{F}\upperset{\Phi\nec}{\circ}s_{d+1}\mathbf{G}\in \mathcalboondox{B}^{\bullet}_{d,\Phi}(\MA,\MB)$ given by
\[
(s_{d+1}\mathbf{F}\upperset{\Phi\nec}{\circ}s_{d+1}\mathbf{G})^{\doubar{x}}=\sum\mathcal{E}(\mathcalboondox{D})\]
where the sum is over all the filled diagrams $\mathcalboondox{D}$ of type $\doubar{x}$ and of the form

 \begin{tikzpicture}[line cap=round,line join=round,x=1.0cm,y=1.0cm]
\clip(-7,-1.5) rectangle (4.5,1.5);
    \draw [line width=0.5pt] (0.,0.) circle (0.5cm);
    \draw [rotate=90] [->, >= stealth] (0.5,0) -- (0.9,0);
    \draw [rotate=-90] [->, >= stealth] (0.5,0) -- (0.9,0);
    \shadedraw [rotate=45, shift={(0.5cm,0cm)}] \doublefleche;
    \shadedraw [rotate=-45, shift={(0.5cm,0cm)}] \doublefleche;
    \draw [line width=0.5pt] (1.5,0.) circle (0.5cm);
    \shadedraw[shift={(1cm,0cm)},rotate=180] \doubleflechescindeeleft;
    \shadedraw[shift={(1cm,0cm)},rotate=180] \doubleflechescindeeright;
    \shadedraw[shift={(1cm,0cm)},rotate=180] \fleche;
    \draw [rotate around={0:(1.5,0)}] [->, >= stealth, >= stealth](2,0) -- (2.4,0);
    \draw [rotate around={-135:(1.5,0)}] [->, >= stealth, >= stealth](2,0) -- (2.4,0);
    \draw [rotate around={135:(1.5,0)}] [->, >= stealth, >= stealth](2,0) -- (2.4,0);
\begin{scriptsize}
\draw [fill=black] (1.5,0.6) circle (0.3pt);
\draw [fill=black] (1.7,0.55) circle (0.3pt);
\draw [fill=black] (1.3,0.55) circle (0.3pt);
\draw [fill=black] (1.5,-0.6) circle (0.3pt);
\draw [fill=black] (1.7,-0.55) circle (0.3pt);
\draw [fill=black] (1.3,-0.55) circle (0.3pt);
\draw [fill=black] (-0.6,0) circle (0.3pt);
\draw [fill=black] (-0.55,0.2) circle (0.3pt);
\draw [fill=black] (-0.55,-0.2) circle (0.3pt);
\end{scriptsize}
\draw(0,0.25)node[anchor=north]{$\mathbf{G}$};
\draw(1.5,0.25)node[anchor=north]{$\mathbf{F}$};
\end{tikzpicture}
\end{definition}

\begin{definition}
Consider two graded quivers $\MA$ and $\MB$ with respective sets of objects $\MO_{\MA}$ and $\MO_{\MB}$ as well as a map $\Phi : \MO_{\MA}\rightarrow\MO_{\MB}$. The \textbf{\textcolor{ultramarine}{$\Phi$-mixed necklace bracket}} is the bracket which is defined for elements  $s_{d+1}\mathbf{F},s_{d+1}\mathbf{G}\in \mathcalboondox{B}^{\bullet}_{d,\Phi}(\MA,\MB)$ by 
\[
[s_{d+1}\mathbf{F},s_{d+1}\mathbf{G}]_{\Phi\nec}=s_{d+1}\mathbf{F}\upperset{\Phi\nec}{\circ} s_{d+1}\mathbf{G}-(-1)^{(|\mathbf{F}|+d+1)(|\mathbf{G}|+d+1)} s_{d+1}\mathbf{G}\upperset{\Phi\nec}{\circ} s_{d+1}\mathbf{F}.
\]
\end{definition}

\begin{remark}
If $\MB=\MA$ and $\Phi=id$, we have isomorphisms $\mathcalboondox{B}_{d,\Phi}^{\bullet}(\MA,\MB)^{\MA}\simeq\Multi^{\bullet}_d(\MA)[d+1]$ and $\mathcalboondox{B}_{d,\Phi}^{\bullet}(\MA,\MB)^{\MB}\simeq\Multi^{\bullet}_d(\MA)[d+1]$ that are a combination of the isomorphism \eqref{eq:iso-shift-tensor-prod} with either the isomorphism \eqref{eq:iso-shift-output} or the isomorphism \eqref{eq:iso-shift-input}. 
In this case, we thus set $\mathcalboondox{B}^{\bullet}_{d,\Phi}(\MA,\MB)=\Multi^{\bullet}_d(\MA)[d+1]$ and the $\Phi$-mixed necklace bracket is in this case the necklace bracket as defined in Definition \ref{def:necklace-bracket}.
\end{remark}

As we did for the necklace bracket, we now investigate the relation between the $\Phi$-mixed necklace bracket and the Gerstenhaber bracket in order to show that the first one is a graded Lie bracket.

\begin{lemma}
\label{lemma:j-mixed}
Let $\MA$ and $\MB$ be graded quivers with respective sets of objects $\MO_{\MA}$ and $\MO_{\MB}$ and consider a map $\Phi : \MO_{\MA}\rightarrow\MO_{\MB}$.
Then, we have an injective map
\begin{equation}
    \label{eq:map-j-mixed} 
    \mathcalboondox{j}^{\Phi} : \mathcalboondox{B}^{\bullet}_{d,\Phi}(\MA,\MB) \rightarrow C(\MA\oplus\MB^*[d-1])[1]
\end{equation}
sending $s_{d+1}\phi\in \mathcalboondox{B}_{d,\Phi}^{\doubar{x}}(\MA,\MB)^{\MA}$ to $s\psi^{\MA}_{\doubar{x}}$ where 
\begin{small}
\begin{equation}
    \psi^{\MA}_{\doubar{x}} : \bigotimes\limits_{i=1}^{n-1}(\MA[1]^{\otimes \bar{x}^{n-i+1}}\otimes {}_{\Phi(\rrt(\bar{x}^{n-i+1}))}\MB^*_{\Phi(\llt(\bar{x}^{n-i}))})[d]\otimes\MA[1]^{\otimes \bar{x}^1} \rightarrow {}_{\llt(\bar{x}^n)}\MA_{\rrt(\bar{x}^1)}
\end{equation}
\end{small} is given by 
\begin{equation}
    \label{eq:map-j-x-mixed} 
    \psi_{\doubar{x}}^{\MA}(\bar{sa}^n,tf^{n-1},\bar{sa}^{n-2},tf^{n-2},...,\bar{sa}^{2},tf^{1},\bar{sa}^1)=(-1)^{\epsilon}\big( \bigotimes\limits_{i=1}^{n-1}(f^{i}\circ s_d)\otimes s_d\big)\big(\phi(\bar{sa}^1,\bar{sa}^{2},...,\bar{sa}^n)\big)
\end{equation}
with
\begin{small}
\begin{equation}
    \begin{split}
        \epsilon=\sum\limits_{i=1}^{n-1}|tf^i|\sum\limits_{j=i+1}^n|\Bar{sa}^i|+|\phi|\sum\limits_{i=1}^{n-1}|tf^i|+d(n-1)+\sum\limits_{1\leq i < j\leq n}|\bar{sa}^i||\bar{sa}^j|+\sum\limits_{1\leq i < j\leq n-1}|tf^i||tf^j|
    \end{split}
\end{equation}
\end{small}
and sending $s_{d+1}\phi\in \mathcalboondox{B}_{d,\Phi}^{\doubar{x}}(\MA,\MB)^{\MB}$ to $s\psi^{\MB}_{\doubar{x}}$ where 
\begin{small}
\begin{equation}
    \psi^{\MB}_{\doubar{x}} : \bigotimes\limits_{i=1}^{n-1}(\MA[1]^{\otimes \bar{x}^{n-i+1}}\otimes {}_{\Phi(\rrt(\bar{x}^{n-i+1}))}\MB^*_{\Phi(\llt(\bar{x}^{n-i}))}[d])\otimes\MA[1]^{\otimes \bar{x}^1} \rightarrow {}_{\Phi(\llt(\bar{x}^n))}\MB^*_{\Phi(\rrt(\bar{x}^1))}[d-1]
\end{equation}
\end{small} is given by 
\begin{equation}
    \label{eq:map-j-x-mixed2} 
    \begin{split}
    &\psi_{\doubar{x}}^{\MB}(\bar{sa}^n,tf^{n-1},\bar{sa}^{n-2},tf^{n-2},...,\bar{sa}^{2},tf^{1},\bar{sa}^1)(s_{1-d}b)
    \\& \hspace{4cm}=(-1)^{\delta}\big( \bigotimes\limits_{i=1}^{n-1}(f^{i}\circ s_d)\big)\big(\phi(\bar{sa}^1\otimes s_{-d}b \otimes\bar{sa}^n,\bar{sa}^{2},...,\bar{sa}^{n-1})\big)
    \end{split}
\end{equation}
for $n\in\NN^*$, $\doubar{x}=(\Bar{x}^1,\dots,\Bar{x}^n)$ and elements $\bar{sa}^i\in\MA[1]^{\otimes\bar{x}^i}$, $s_{1-d}b\in {}_{\Phi(\rrt(\bar{x}^1))}\MB^*_{\Phi(\llt(\bar{x}^n))}[d-1]$ and $tf^i\in{}_{\Phi(\rrt(\bar{x}^{i+1}))}\MB^{*}_{\Phi(\llt(\bar{x}^{i}))}[d]$
with
\begin{small}
\begin{equation}
    \begin{split}
        \delta=\epsilon+d(|\phi|+\sum\limits_{i=1}^n|\Bar{sa}^i|+\sum\limits_{i=1}^{n-1}|tf^i|)+(|\bar{sa}^n|+|sb|)\sum\limits_{i=2}^{n-1}|\bar{sa}^i|+|\bar{sa}^n||sb|.
    \end{split}
\end{equation}
\end{small}
for $n=\llg(\doubar{x})$, $\bar{sa}^i\in\MA[1]^{\otimes\bar{x}^i}$ and $tf^i\in{}_{\Phi(\rrt(\bar{x}^{i+1}))}\MB^{*}_{\Phi(\llt(\bar{x}^{i}))}[d]$.
\end{lemma}
\begin{proof}
    The proof is completely similar to the proof of Lemma \ref{lemma:j}.
\end{proof}

\begin{proposition}
\label{prop:Lie-bracket-mixed}
    Let $\MA$ and $\MB$ be graded quivers with respective sets of objects $\MO_{\MA}$ and $\MO_{\MB}$ and consider a map $\Phi : \MO_{\MA}\rightarrow\MO_{\MB}$.
Then, we have
\begin{equation}
    \mathcalboondox{j}^{\Phi}([s_{d+1}\mathbf{F},s_{d+1}\mathbf{G}]_{\Phi\nec})
    =[\mathcalboondox{j}^{\Phi}(s_{d+1}\mathbf{F}),\mathcalboondox{j}^{\Phi}(s_{d+1}\mathbf{G})]_G
\end{equation}
for $s_{d+1}\mathbf{F}, s_{d+1}\mathbf{G} \in\mathcalboondox{B}^{\bullet}_{d,\Phi}(\MA,\MB)$.
\end{proposition}
\begin{proof}
    The proof is completely similar to the proof of Proposition \ref{proposition:j-bracket}.
\end{proof}
\subsection{The case of strict morphisms} \label{strict case}
In this subsection, we study the relation between strict $d$-pre-Calabi-Yau morphisms and $A_{\infty}$-morphisms.
We will denote 
\[ 
\xoverline{\Phi_0(x)}=(\Phi_0(x_1),\dots,\Phi_0(x_n))
\] for $\bar{x}=(x_1,\dots,x_n)\in\MO_{\MA}^n$ and 
\[ 
\doubarl{\Phi_0(x)}=(\xoverline{\Phi_0(\bar{x}^1)},\dots,\xoverline{\Phi_0(\bar{x}^n)})
\]
for $\doubar{x}=(\bar{x}^1,\dots,\bar{x}^n)\in\bar{\MO}_{\MA}^n$.
Moreover, we will denote by $\pi_{\MA}$ (resp. $\pi_{\MB^*[d-1]}$) the canonical projection $\MA\oplus\MB^*[d-1]\rightarrow \MA$ (resp. $\MA\oplus\MB^*[d-1]\rightarrow \MB^*[d-1]$). We define $\pi_{\MA[1]}$ and $\pi_{\MB^*[d]}$ in a similar manner.
We first recall the definition of strict $d$-pre-Calabi-Yau morphism.
\begin{definition}
    Let $(\MA,s_{d+1}M_{\MA})$, $(\MB,s_{d+1}M_{\MB})$ be $d$-pre-Calabi-Yau categories with respective sets of objects $\MO_{\MA}$ and $\MO_{\MB}$. A $d$-pre-Calabi-Yau morphism $(\Phi_0,s_{d+1}\Phi) :(\MA,s_{d+1}M_{\MA})\rightarrow (\MB,s_{d+1}M_{\MB})$ is \textbf{\textcolor{ultramarine}{strict}} if $\Phi^{\doubar{x}}$ vanishes for each $\doubar{x}\in\bar{\MO}_{\MA}^n$ with $n>1$ and $\llg(\bar{x}^1)\neq 2$ if $n=1$.
    
    Equivalently, a strict $d$-pre-Calabi-Yau morphism between $d$-pre-Calabi-Yau categories $(\MA,s_{d+1}M_{\MA})$ and $(\MB,s_{d+1}M_{\MB})$ is the data of a map between their sets of objects $\Phi_0 : \MO_{\MA}\rightarrow\MO_{\MB}$ together with a collection $s_{d+1}\Phi=(\Phi^{x,y} : {}_{x}\MA_{y}[1] \rightarrow {}_{\Phi_0(x)}\MB_{\Phi_0(y)}[1])_{x,y\in\MO_{\MA}}$ of maps of degree 0 that satisfies  
\begin{equation}
    (\Phi^{\llt(\bar{x}^1),\rrt(\bar{x}^2)}\otimes\dots\otimes\Phi^{\llt(\bar{x}^n),\rrt(\bar{x}^1)})\circ s_{d+1}M_{\MA}^{\doubar{x}}=s_{d+1}M_{\MB}^{\scriptsize{\doubarl{\Phi_0(x)}}}\circ (\Phi^{\otimes \llg(\bar{x}^1)-1}\otimes\dots\otimes\Phi^{\otimes \llg(\bar{x}^n)-1})
\end{equation}
for each $n\in\NN^*$, $\doubar{x}=(\bar{x}^1,\dots,\bar{x}^n)\in\bar{\MO}_{\MA}^n$.
\end{definition}
For simplicity, we will omit the elements when writing the map $\Phi$.
\begin{definition}
    We denote by $SpCY_d$ the subcategory of $pCY_d$ whose objects are $d$-pre-Calabi-Yau categories and whose morphisms are strict $d$-pre-Calabi-Yau morphisms.
\end{definition}

Given $d$-pre-Calabi-Yau categories $(\MA,s_{d+1}M_{\MA})$ and $(\MB,s_{d+1}M_{\MB})$ with respective sets of objects $\MO_{\MA}$ and $\MO_{\MB}$ and a strict $d$-pre-Calabi-Yau morphism $(\Phi_0,s_{d+1}\Phi) :(\MA,s_{d+1}M_{\MA})\rightarrow (\MB,s_{d+1}M_{\MB})$, we now construct an $A_{\infty}$-structure on $\MA\oplus\MB^{*}[d-1]$.
For $\doubar{x}\in\doubar{\MO}_{\MA}$, we define 
\begin{equation}
    \begin{split}
        sm_{\MA\oplus\MA^*\to \MA}^{\doubar{x}}=\pi_{\MA[1]}\circ \mathcalboondox{j}_{\doubar{x}^{-1}}(s_{d+1}M_{\MA}^{\doubar{x}^{-1}})
        \text{ and }sm_{\MB\oplus\MB^*\to \MB^*}^{\scriptsize{\doubarl{\Phi_0(x)}}}&=\pi_{\MB^*[d]}\circ \mathcalboondox{j}_{\scriptsize{\doubarl{\Phi_0(x)}}^{-1}}(s_{d+1}M_{\MB}^{\scriptsize{\doubarl{\Phi_0(x)}}^{-1}}).
    \end{split}
\end{equation} 

\begin{definition}
\label{def:m-A-B}
    We define the element $sm_{\MA\oplus\MB^*}\in C(\MA\oplus\MB^*[d-1])[1]$ as the unique element such that $\pi_{\MA}\circ m_{\MA\oplus\MB^*}= m_{\MA\oplus\MB^*\to \MA}$ where
\begin{equation}
m_{\MA\oplus\MB^*\to \MA}^{\bar{x}^1,\dots,\bar{x}^n}= m_{\MA\oplus\MA^*\to \MA}^{\bar{x}^1,\dots,\bar{x}^n}\circ (\id^{\otimes \llg(\bar{x}^1)-1}\otimes \Phi^*\otimes\id^{\otimes \llg(\bar{x}^2)-1}\otimes\dots\otimes \id^{\otimes \llg(\bar{x}^{n-1})-1}\otimes\Phi^*\otimes\id^{\otimes \llg(\bar{x}^n)-1})
\end{equation}
and $\pi_{\MB^*[d-1]}\circ m_{\MA\oplus\MB^*}=m_{\MA\oplus\MB^*\to\MB^*}$ where
\begin{equation}
m_{\MA\oplus\MB^*\to \MB^*}^{\bar{x}^1,\dots,\bar{x}^n}= m_{\MB\oplus\MB^*\to \MB^*}^{\scriptsize{\xoverline{\Phi_0(\bar{x}^1)}},\dots,\scriptsize{\xoverline{\Phi_0(\bar{x}^n)}}}\circ (\Phi^{\otimes \llg(\bar{x}^1)-1}\otimes \id\otimes \Phi^{\otimes \llg(\bar{x}^{2})-1}\otimes\dots\otimes \id\otimes \Phi^{\otimes \llg(\bar{x}^n)-1})
\end{equation}
for every $n\in\NN^*$ and $\doubar{x}=(\bar{x}^1,\dots,\bar{x}^n)\in\bar{\MO}^n$.
\end{definition}

\begin{proposition}
\label{prop:struct-strict-case}
The element $sm_{\MA\oplus\MB^*}\in C(\MA\oplus\MB^*[d-1])[1]$ defines an $A_{\infty}$-structure on $\MA\oplus\MB^*[d-1]$ that is almost cyclic with respect to the $\Phi$-mixed bilinear form $\Gamma^{\Phi}$, defined in Example \ref{example:natural-mixed-bilinear-form}.
\end{proposition}
\begin{proof}
By Proposition \ref{prop:Lie-bracket-mixed},
the fact that $[sm_{\MA\oplus\MB^*},sm_{\MA\oplus\MB^*}]_G=0$ is tantamount to the fact that $[s_{d+1}M_{\MA\oplus\MB^*}, s_{d+1}M_{\MA\oplus\MB^*}]_{\Phi \nec}=0$ where 
$s_{d+1}M_{\MA\oplus\MB^*}\in \mathcalboondox{B}^{\bullet}_{d,\Phi_0}(\MA,\MB)$ is given by  
\[
p_{\MA}(s_{d+1}M_{\MA\oplus\MB^*}^{\doubar{x}})=(\Phi^{\otimes (n-1)}\otimes \id)\circ s_{d+1}M_{\MA}^{\doubar{x}}\in \mathcalboondox{B}^{\doubar{x}}_{d,\Phi_0}(\MA,\MB)^{\MA}\]
 and by $p_{\MB}(s_{d+1}M_{\MA\oplus\MB^*}^{\doubar{x}})$ equals to
 \begin{small}
     \begin{equation}
     \begin{split}
s_{d+1}M_{\MB}^{\scriptsize{\doubarl{\Phi_0(\doubar{y})}}}\circ(\Phi^{\otimes \llg(\bar{x}^1)-1}\otimes\id\otimes\Phi^{\otimes \llg(\bar{x}^n)-1}\otimes\Phi^{\otimes \llg(\bar{x}^2)-1}\otimes\dots\otimes \Phi^{\otimes \llg(\bar{x}^{n-1})-1})\in \mathcalboondox{B}^{\doubar{x}}_{d,\Phi_0}(\MA,\MB)^{\MB}
 \end{split}
     \end{equation}
 \end{small}
for $\doubar{x}=(\bar{x}^1,\dots,\bar{x}^n)\in\bar{\MO}_{\MA}^n$ and $\doubar{y}=(\bar{x}^1\sqcup\bar{x}^n,\bar{x}^2,\dots,\bar{x}^{n-1})$.
Moreover, 
\[
p_{\MA}(s_{d+1}M_{\MA\oplus\MB^*}\upperset{\Phi \nec}{\circ} s_{d+1}M_{\MA\oplus\MB^*})^{\doubar{x}}=\sum\mathcal{E}(\mathcalboondox{D})+\sum\mathcal{E}(\mathcalboondox{D'})\] where the sums are over all the filled diagrams $\mathcalboondox{D}$ and $\mathcalboondox{D'}$ of type $\doubar{x}$ that are of the form

\begin{minipage}{20cm}
\begin{tikzpicture}[line cap=round,line join=round,x=1.0cm,y=1.0cm]
\clip(-2.5,-1.5) rectangle (5.5,2.6);
     \draw [line width=0.5pt] (0.,0.) circle (0.5cm);
     \shadedraw[rotate=150,shift={(0.5cm,0cm)}] \doublefleche;
     \draw[line width=1.1pt,rotate=90][->, >= stealth, >= stealth](0.5,0)--(0.9,0);
     \shadedraw[rotate=30,shift={(0.5cm,0cm)}] \doubleflechescindeeleft;
     \shadedraw[rotate=30,shift={(0.5cm,0cm)}] \doubleflechescindeeright;
     \shadedraw[rotate=30,shift={(0.5cm,0cm)}] \fleche;
      \draw [line width=0.5pt] (1.3,0.75) circle (0.5cm);
     \draw [line width=0.5pt] (1.03,-0.6) circle (0.3cm);
    \draw[rotate=-30][->, >= stealth, >= stealth](0.5,0)--(0.9,0);
     \draw [rotate around ={-30:(1.03,-0.6)}] [->, >= stealth, >= stealth] (1.33,-0.6)--(1.63,-0.6);
    \draw[rotate=-150][->, >= stealth, >= stealth](0.5,0)--(0.9,0);
    \draw [line width=0.5pt] (-1.03,-0.6) circle (0.3cm);
     \draw [rotate around ={30:(-1.03,-0.6)}] [->, >= stealth, >= stealth] (-1.33,-0.6)--(-1.63,-0.6);
     \draw [rotate around ={-30:(1.3,0.75)}] [->, >= stealth, >= stealth] (1.8,0.75)--(2.2,0.75);
     \draw [rotate around ={90:(1.3,0.75)}] [->, >= stealth, >= stealth] (1.8,0.75)--(2.2,0.75);
     \draw [rotate around ={30:(1.3,0.75)}, shift={(1.8,0.75)}]\doublefleche;
     \draw [rotate around ={150:(1.3,0.75)}, shift={(1.8,0.75)}]\doublefleche;
     \draw [line width=0.5pt] (2.33,0.15) circle (0.3cm);
      \draw [rotate around ={-30:(2.33,0.15)}] [->, >= stealth, >= stealth] (2.63,0.15)--(2.93,0.15);
    \draw [line width=0.5pt] (1.3,1.95) circle (0.3cm);
      \draw [rotate around ={90:(1.3,1.95)}] [->, >= stealth, >= stealth] (1.6,1.95)--(1.9,1.95);  
\begin{scriptsize}
\draw [fill=black] (0,-0.6) circle (0.3pt);
\draw [fill=black] (0.2,-0.55) circle (0.3pt);
\draw [fill=black] (-0.2,-0.55) circle (0.3pt);
\draw [fill=black] (1.3,0.75-0.6) circle (0.3pt);
\draw [fill=black] (1.5,0.75-0.55) circle (0.3pt);
\draw [fill=black] (1.1,0.75-0.55) circle (0.3pt);
\end{scriptsize}
\draw(0,0.25)node[anchor=north]{$M_{\MA}$};
\draw(1.3,1)node[anchor=north]{$M_{\MA}$};
\draw(2.33,0.4)node[anchor=north]{$\Phi$};
\draw(1.3,2.2)node[anchor=north]{$\Phi$};
\draw(1.03,-0.35)node[anchor=north]{$\Phi$};
\draw(-1.03,-0.35)node[anchor=north]{$\Phi$};
\draw(4.5,0.25)node[anchor=north]{and};
\end{tikzpicture}
\begin{tikzpicture}[line cap=round,line join=round,x=1.0cm,y=1.0cm]
\clip(-4,-1.5) rectangle (5,2.6);
   \draw [line width=0.5pt] (0.,0.) circle (0.5cm);
     \draw [rotate=90] [->, >= stealth, >= stealth] (0.5,0)--(0.9,0);
     \draw [rotate=-30] [->, >= stealth, >= stealth] (0.5,0)--(0.9,0);
     \draw [rotate=-150] [->, >= stealth, >= stealth] (0.5,0)--(0.9,0);
     \draw [rotate=0] [<-, >= stealth, >= stealth] (0.5,0)--(0.9,0);
     \draw [line width=0.5pt] (1.15,0.) circle (0.25cm);
     \draw[rotate around={0:(1.15,0)}] [<-, >= stealth, >= stealth, >= stealth] (1.4,0)--(1.7,0);
     \draw [rotate=60] [<-, >= stealth, >= stealth] (0.5,0)--(0.9,0);
       \draw [rotate=180] [<-, >= stealth, >= stealth] (0.5,0)--(0.9,0);
      \draw [line width=0.5pt] (-1.15,0.) circle (0.25cm);
     \draw[rotate around={0:(-1.15,0)}] [<-, >= stealth, >= stealth] (-1.4,0)--(-1.7,0);
        \draw [rotate=120] [<-, >= stealth, >= stealth] (0.5,0)--(0.9,0);
        \draw [line width=0.5pt] (0.575,1) circle (0.25cm);
        \draw[rotate around={60:(0.575,1)}] [<-, >= stealth, >= stealth] (0.825,1)--(1.125,1);
        \draw [line width=0.5pt] (-0.575,1) circle (0.25cm);
      \draw[rotate around={300:(-0.575,1)}] [<-, >= stealth, >= stealth] (-0.825,1)--(-1.125,1);
      \draw[rotate around={150:(0,0)}] [<-, >= stealth, >= stealth] (0.5,0)--(0.9,0);
    \draw [line width=0.5pt] (-1,0.575) circle (0.25cm);
    \draw[rotate around={150:(-1,0.575)}] [<-, >= stealth, >= stealth] (-0.75,0.575)--(-0.35,0.575);
     \draw [line width=0.5pt] (-2,1.15) circle (0.5cm);
     \shadedraw[rotate around={30:(-2,1.15)},shift={(-1.5,1.15)}]\doublefleche;
     \shadedraw[rotate around={150:(-2,1.15)},shift={(-1.5,1.15)}]\doublefleche;
     \draw[rotate around={-150:(-2,1.15)}] [->, >= stealth, >= stealth] (-1.5,1.15)--(-1.1,1.15);
      \draw [line width=0.5pt] (-2-1.03,1.15-0.6) circle (0.3cm);
     \draw[rotate around={-150:(-2-1.03,1.15-0.6)}] [->, >= stealth, >= stealth] (-2-1.03+0.3,1.15-0.6)--(-2-1.03+0.6,1.15-0.6);
     \draw[line width=1.1pt,rotate around={90:(-2,1.15)}] [->, >= stealth, >= stealth] (-1.5,1.15)--(-1.1,1.15);
\begin{scriptsize}
\draw [fill=black] (0,-0.6) circle (0.3pt);
\draw [fill=black] (0.2,-0.55) circle (0.3pt);
\draw [fill=black] (-0.2,-0.55) circle (0.3pt);
\draw [rotate=120][fill=black] (0,-0.6) circle (0.3pt);
\draw [rotate=120][fill=black] (0.15,-0.55) circle (0.3pt);
\draw [rotate=120][fill=black] (-0.15,-0.55) circle (0.3pt);
\draw [fill=black] (-0.38,0.47) circle (0.3pt);
\draw [fill=black] (-0.44,0.43) circle (0.3pt);
\draw [fill=black] (-0.47,0.37) circle (0.3pt);
\draw [fill=black] (-0.56,0.22) circle (0.3pt);
\draw [fill=black] (-0.6,0.165) circle (0.3pt);
\draw [fill=black] (-0.6,0.1) circle (0.3pt);
\draw [fill=black] (-2,1.15-0.6) circle (0.3pt);
\draw [fill=black] (-2.2,1.15-0.55) circle (0.3pt);
\draw [fill=black] (-1.8,1.15-0.55) circle (0.3pt);
\end{scriptsize}
\draw(0,0.25)node[anchor=north]{$M_{\MB}$};
\draw(-2,1.4)node[anchor=north]{$M_{\MA}$};
\draw(-1,0.575+0.25)node[anchor=north]{$\Phi$};
\draw(0.575,1.25)node[anchor=north]{$\Phi$};
\draw(-0.575,1.25)node[anchor=north]{$\Phi$};
\draw(1.15,0.25)node[anchor=north]{$\Phi$};
\draw(-1.15,0.25)node[anchor=north]{$\Phi$};
\draw(-2-1.03,1.15-0.6+0.25)node[anchor=north]{$\Phi$};
\end{tikzpicture}
\end{minipage}

\noindent respectively.
Using that $\Phi$ is a strict $d$-pre-Calabi-Yau morphism, the diagram composed of the disc filled with $M_{\MB}$ and of the discs sharing an arrow with it on the right can be replaced by one with a disc filled with $M_{\MA}$ whose outgoing arrows are connected with the unique incoming arrow of discs of size $1$ filled with $\Phi$.
We thus obtain that 
\[
p_{\MA}(s_{d+1}M_{\MA\oplus\MB^*}\upperset{\Phi \nec}{\circ} s_{d+1}M_{\MA\oplus\MB^*})^{\doubar{x}}=\sum\mathcal{E}(\mathcalboondox{D})+\sum\mathcal{E}(\mathcalboondox{D}\dprime)\] where the sums are over all the filled diagrams $\mathcalboondox{D}$ and $\mathcalboondox{D}\dprime$ of type $\doubar{x}$ that are of the form

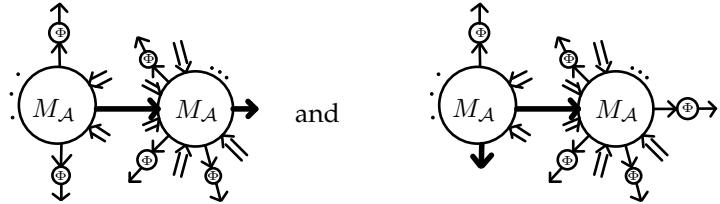
\begin{minipage}{20cm}
\begin{tikzpicture}[line cap=round,line join=round,x=1.0cm,y=1.0cm]
\clip(-2.5,-1.5) rectangle (5.5,2.6);
     \draw [line width=0.5pt] (0.,0.) circle (0.5cm);
     \shadedraw[rotate=150,shift={(0.5cm,0cm)}] \doublefleche;
     \draw[line width=1.1pt,rotate=90][->, >= stealth, >= stealth](0.5,0)--(0.9,0);
     \shadedraw[rotate=30,shift={(0.5cm,0cm)}] \doubleflechescindeeleft;
     \shadedraw[rotate=30,shift={(0.5cm,0cm)}] \doubleflechescindeeright;
     \shadedraw[rotate=30,shift={(0.5cm,0cm)}] \fleche;
      \draw [line width=0.5pt] (1.3,0.75) circle (0.5cm);
     \draw [line width=0.5pt] (1.03,-0.6) circle (0.3cm);
    \draw[rotate=-30][->, >= stealth, >= stealth](0.5,0)--(0.9,0);
     \draw [rotate around ={-30:(1.03,-0.6)}] [->, >= stealth, >= stealth] (1.33,-0.6)--(1.63,-0.6);
    \draw[rotate=-150][->, >= stealth, >= stealth](0.5,0)--(0.9,0);
    \draw [line width=0.5pt] (-1.03,-0.6) circle (0.3cm);
     \draw [rotate around ={30:(-1.03,-0.6)}] [->, >= stealth, >= stealth] (-1.33,-0.6)--(-1.63,-0.6);
     \draw [rotate around ={-30:(1.3,0.75)}] [->, >= stealth, >= stealth] (1.8,0.75)--(2.2,0.75);
     \draw [rotate around ={90:(1.3,0.75)}] [->, >= stealth, >= stealth] (1.8,0.75)--(2.2,0.75);
     \draw [rotate around ={30:(1.3,0.75)}, shift={(1.8,0.75)}]\doublefleche;
     \draw [rotate around ={150:(1.3,0.75)}, shift={(1.8,0.75)}]\doublefleche;
     \draw [line width=0.5pt] (2.33,0.15) circle (0.3cm);
      \draw [rotate around ={-30:(2.33,0.15)}] [->, >= stealth, >= stealth] (2.63,0.15)--(2.93,0.15);
    \draw [line width=0.5pt] (1.3,1.95) circle (0.3cm);
      \draw [rotate around ={90:(1.3,1.95)}] [->, >= stealth, >= stealth] (1.6,1.95)--(1.9,1.95);  
\begin{scriptsize}
\draw [fill=black] (0,-0.6) circle (0.3pt);
\draw [fill=black] (0.2,-0.55) circle (0.3pt);
\draw [fill=black] (-0.2,-0.55) circle (0.3pt);
\draw [fill=black] (1.3,0.75-0.6) circle (0.3pt);
\draw [fill=black] (1.5,0.75-0.55) circle (0.3pt);
\draw [fill=black] (1.1,0.75-0.55) circle (0.3pt);
\end{scriptsize}
\draw(0,0.25)node[anchor=north]{$M_{\MA}$};
\draw(1.3,1)node[anchor=north]{$M_{\MA}$};
\draw(2.33,0.4)node[anchor=north]{$\Phi$};
\draw(1.3,2.2)node[anchor=north]{$\Phi$};
\draw(1.03,-0.35)node[anchor=north]{$\Phi$};
\draw(-1.03,-0.35)node[anchor=north]{$\Phi$};
\draw(4.5,0.25)node[anchor=north]{and};
\end{tikzpicture}
\begin{tikzpicture}[line cap=round,line join=round,x=1.0cm,y=1.0cm]
\clip(-2,-1.5) rectangle (5,2.6);
 \draw [line width=0.5pt] (0.,0.) circle (0.5cm);
     \shadedraw[rotate=150,shift={(0.5cm,0cm)}] \doublefleche;
     \draw[rotate=90][->, >= stealth, >= stealth](0.5,0)--(0.9,0);
     \shadedraw[rotate=30,shift={(0.5cm,0cm)}] \doubleflechescindeeleft;
     \shadedraw[rotate=30,shift={(0.5cm,0cm)}] \doubleflechescindeeright;
     \shadedraw[rotate=30,shift={(0.5cm,0cm)}] \fleche;
      \draw [line width=0.5pt] (1.3,0.75) circle (0.5cm);
     \draw [line width=0.5pt] (1.03,-0.6) circle (0.3cm);
    \draw[rotate=-30][->, >= stealth, >= stealth](0.5,0)--(0.9,0);
     \draw [rotate around ={-30:(1.03,-0.6)}] [->, >= stealth, >= stealth] (1.33,-0.6)--(1.63,-0.6);
    \draw[rotate=-150][->, >= stealth, >= stealth](0.5,0)--(0.9,0);
    \draw [line width=0.5pt] (-1.03,-0.6) circle (0.3cm);
     \draw [rotate around ={30:(-1.03,-0.6)}] [->, >= stealth, >= stealth] (-1.33,-0.6)--(-1.63,-0.6);
     \draw [rotate around ={-30:(1.3,0.75)}] [->, >= stealth, >= stealth] (1.8,0.75)--(2.2,0.75);
     \draw [line width=1.1pt,rotate around ={90:(1.3,0.75)}] [->, >= stealth, >= stealth] (1.8,0.75)--(2.2,0.75);
     \draw [rotate around ={30:(1.3,0.75)}, shift={(1.8,0.75)}]\doublefleche;
     \draw [rotate around ={150:(1.3,0.75)}, shift={(1.8,0.75)}]\doublefleche;
     \draw [line width=0.5pt] (2.33,0.15) circle (0.3cm);
      \draw [rotate around ={-30:(2.33,0.15)}] [->, >= stealth, >= stealth] (2.63,0.15)--(2.93,0.15);
    \draw [line width=0.5pt] (0,1.2) circle (0.3cm);
      \draw [rotate around ={90:(0,1.2)}] [->, >= stealth, >= stealth] (0.3,1.2)--(0.6,1.2);  
\begin{scriptsize}
\draw [fill=black] (0,-0.6) circle (0.3pt);
\draw [fill=black] (0.2,-0.55) circle (0.3pt);
\draw [fill=black] (-0.2,-0.55) circle (0.3pt);
\draw [fill=black] (1.3,0.75-0.6) circle (0.3pt);
\draw [fill=black] (1.5,0.75-0.55) circle (0.3pt);
\draw [fill=black] (1.1,0.75-0.55) circle (0.3pt);
\end{scriptsize}
\draw(0,0.25)node[anchor=north]{$M_{\MA}$};
\draw(1.3,1)node[anchor=north]{$M_{\MA}$};
\draw(2.33,0.4)node[anchor=north]{$\Phi$};
\draw(0,1.45)node[anchor=north]{$\Phi$};
\draw(1.03,-0.35)node[anchor=north]{$\Phi$};
\draw(-1.03,-0.35)node[anchor=north]{$\Phi$};
\end{tikzpicture}
\end{minipage}
\noindent respectively.
Moreover, $\sum\mathcal{E}(\mathcalboondox{D})+\sum\mathcal{E}(\mathcalboondox{D}\dprime)=0$ since $s_{d+1}M_{\MA}$ is a $d$-pre-Calabi-Yau structure.
Thus, if we show that $sm_{\MA\oplus\MB^*}$ satisfies the cyclicity condition \eqref{eq:cyclicity-pcy} it will satisfy the Stasheff identities for every $\Bar{x}\in\Bar{\MO}_{\MA}$.

Using the definition of $\Gamma^{\Phi}$ and of $sm_{\MB\oplus\MB^*}$ together with the fact that the latter is cyclic with respect to $\Gamma^{\MB}$, we have that
\begin{equation}
\label{eq:stric-morph-1}
    \begin{split}
        &\Gamma^{\Phi}(sm_{\MA\oplus\MB^*\to \MB^*}^{\bar{x}^1,\dots,\bar{x}^n}(\bar{sa}^1,tf^1,\dots,\bar{sa}^{n-1},tf^{n-1},\bar{sa}^n),sb)\\&= \Gamma^{\MB}(sm_{\MB\oplus\MB^*\to \MB^*}^{\scriptsize{\xoverline{\Phi_0(\bar{x}^1)}},\dots,\scriptsize{\xoverline{\Phi_0(\bar{x}^n)}}}(\Phi^{\otimes \llg(\bar{x}^1)-1}(\bar{sa}^1),tf^{1},\dots,tf^{n-1},\Phi^{\otimes \llg(\bar{x}^n)-1}(\bar{sa}^n)),\Phi(sb))
        \\&=(-1)^{\epsilon}\Gamma^{\MB}(sm_{\MB\oplus\MB^*\to \MB}^{\scriptsize{\doubarl{\Phi_0(\doubar{y})}}}\big(\Phi^{\otimes \llg(\bar{x}^n)-1}(\bar{sa}^n)\otimes\Phi(sb)\otimes\Phi^{\otimes \llg(\bar{x}^1)-1}(\bar{sa}^1),tf^{1},\dots,\\&\hspace{8cm}tf^{n-2},\Phi^{\otimes \llg(\bar{x}^{n-1})-1}(\bar{sa}^{n-1})\big),tf^{n-1})
        \\&=(-1)^{\epsilon+\delta}
        \Gamma^{\MB}\big((f^{n-2}\circ s_d)\otimes\dots\otimes(f^{1}\circ s_d)\otimes s_d)
        \\&(M_{\MB}^{\scriptsize{\doubarl{\Phi_0(\doubar{y})}^{-1}}}(\Phi^{\otimes\llg(\bar{x}^{n-1})}(\bar{sa}^{n-1}),\dots,\Phi^{\otimes\llg(\bar{x}^n)}(\bar{sa}^n)\otimes\Phi(sb)\otimes\Phi^{\otimes\llg(\bar{x}^1)}(\bar{sa}^1)),tf^{n-1}\big)
    \end{split}
\end{equation}
for $\doubar{x}=(\Bar{x}^1,\dots,\Bar{x}^n)\in\Bar{\MO}_{\MA}^n$, $\doubar{y}=(\Bar{x}^n\sqcup\Bar{x}^1,\dots,\Bar{x}^{n-1})$ and elements $\bar{sa}^i\in \MA[1]^{\otimes \bar{x}^i}$, $sb\in{}_{\rrt(\bar{x}^n)}\MA_{\llt(\bar{x}^1)}[1]$, $tf^i\in {}_{\Phi_0(\rrt(\bar{x}^{i+1}))}\MB^*_{\Phi_0(\llt(\bar{x}^{i}))}[d]$
with 
\begin{small}
\begin{equation}
\begin{split}
\epsilon&=(|\bar{sa}^n|+|sb|)(\sum\limits_{i=1}^{n-1}(|\bar{sa}^i|+|tf^i|),
    \\\delta&=(|\bar{sa}^1|+|sb|+|\bar{sa}^n|)(\sum\limits_{i=2}^{n-1}|\bar{sa}^i|+\sum\limits_{i=1}^{n-2}|tf^i|)+\hskip-2mm\sum\limits_{2\leq i\leq j\leq n-2}\hskip-2mm|\bar{sa}^i||tf^j|\\&+dn
    +\hskip-2mm\sum\limits_{2\leq i<j\leq n-1}\hskip-2mm|\bar{sa}^i||\bar{sa}^j|+\hskip-2mm\sum\limits_{1\leq i<j\leq n-2}\hskip-2mm|tf^i||tf^j|.
    \end{split}
\end{equation}
\end{small}

Moreover, using that $\Phi$ is a $d$-pre-Calabi-Yau morphism, we have that the last member of \eqref{eq:stric-morph-1} is equal to
\begin{equation}
\label{eq:stric-morph-3}
    \begin{split}
    &(-1)^{\epsilon+\delta}\Gamma^{\MA}\big((\bigotimes\limits_{i=2}^{n-1}(f^{n-i}\circ \Tilde{\Phi})\otimes \id)\circ s_{d+1}M_{\MA}^{\doubar{y}^{-1}}
        (\bar{sa}^{n-1},\dots,\bar{sa}^2,\bar{sa}^n\otimes sb\otimes \bar{sa}^1),tf^{n-1}\circ\Tilde{\Phi}\big)
        \\&=(-1)^{\epsilon}\Gamma^{\MA}
        \big(sm_{\MA\oplus\MA^*\to \MA}^{\doubar{y}}(\bar{sa}^n\otimes sb\otimes \bar{sa}^1,tf^{1}\circ \Tilde{\Phi},\bar{sa}^{2},\dots,tf^{n-2}\circ \Tilde{\Phi},\bar{sa}^{n-1}),tf^{n-1}\circ\Tilde{\Phi}\big)
    \end{split}
\end{equation}
where $\Tilde{\Phi}$ denotes the morphism $\Phi[-d-1]$.
Thus, comparing \eqref{eq:stric-morph-1} and \eqref{eq:stric-morph-3} we get that 
\begin{equation}
    \label{eq:cyclicity-mA-to-mB}
    \begin{split}
    &\Gamma^{\MB}(sm_{\MB\oplus\MB^*\to \MB}^{\scriptsize{\doubarl{\Phi_0(\doubar{y})}^{-1}}}\big(\Phi^{\otimes \llg(\bar{x}^n)-1}(\bar{sa}^n)\otimes\Phi(sb)\otimes\Phi^{\otimes \llg(\bar{x}^1)-1}(\bar{sa}^1),tf^{1},\dots,\\&\hspace{9cm}tf^{n-2},\Phi^{\otimes \llg(\bar{x}^{n-1})-1}(\bar{sa}^{n-1})\big),tf^{n-1})
    \\&=\Gamma^{\MA}(sm_{\MA\oplus\MA^*\to \MA}^{\bar{x}^n\sqcup\bar{x}^1,\bar{x}^2,\dots,\bar{x}^{n-1}}(\bar{sa}^n\otimes sb\otimes \bar{sa}^1,tf^{1}\circ \Tilde{\Phi},\bar{sa}^{2},\dots,tf^{n-2}\circ \Tilde{\Phi},\bar{sa}^{n-1}),tf^{n-1}\circ \Tilde{\Phi}).
    \end{split}
\end{equation}
Finally, we have that 
\begin{equation}
\begin{split}
    &(-1)^{\epsilon}\Gamma^{\MA}(sm_{\MA\oplus\MA^*\to \MA}^{\bar{x}^n\sqcup\bar{x}^1,\bar{x}^2,\dots,\bar{x}^{n-1}}(\bar{sa}^n\otimes sb\otimes \bar{sa}^1,tf^{1}\circ \Tilde{\Phi},\bar{sa}^{2},\dots,tf^{n-2}\circ \Tilde{\Phi},\bar{sa}^{n-1}),tf^{n-1}\circ \Tilde{\Phi})
    \\&=(-1)^{\epsilon}\Gamma^{\Phi}(sm_{\MA\oplus\MB^*\to \MA}^{\bar{x}^n\sqcup\bar{x}^1,\bar{x}^2,\dots,\bar{x}^{n-1}}(\bar{sa}^n\otimes sb\otimes \bar{sa}^1,tf^{1},\bar{sa}^{2},\dots,tf^{n-2},\bar{sa}^{n-1}),tf^{n-1})
    \end{split}
\end{equation}
by definition of $sm_{\MA\oplus\MB^*\to\MA}$. Therefore, $\MA\oplus\MB^*[d-1]$ together with $sm_{\MA\oplus\MB^*}$ is an $A_{\infty}$-category that is almost cyclic with respect to $\Gamma^{\Phi}$.
\end{proof}

\begin{definition}
\label{def:maps-varphi}
 Let $(\MA,s_{d+1}M_{\MA})$, $(\MB,s_{d+1}M_{\MB})$ be $d$-pre-Calabi-Yau categories with respective sets of objects $\MO_{\MA}$ and $\MO_{\MB}$ and let $(\Phi_0,s_{d+1}\Phi) :(\MA,s_{d+1}M_{\MA})\rightarrow (\MB,s_{d+1}M_{\MB})$ be a strict $d$-pre-Calabi-Yau morphism.
   We define maps of graded vector spaces \[\varphi_{\MA}^{x,y} : {}_x\MA_y[1]\oplus{}_{\Phi_0(x)}\MB^{*}_{\Phi_0(y)}[d]\rightarrow {}_x(\MA[1]\oplus{}\MA^{*}[d])_y\]and \[\varphi^{x,y}_{\MB} : {}_x\MA_y[1]\oplus{}_{\Phi_0(x)}\MB^{*}_{\Phi_0(y)}[d] \rightarrow {}_{\Phi_0(x)}(\MB[1]\oplus\MB^{*}[d])_{\Phi_0(y)}\] given by
    $\varphi^{x,y}_{\MA}(sa)=sa$, $\varphi^{x,y}_{\MB}(sa)=\Phi^{x,y}(sa)$, $\varphi^{x,y}_{\MA}(tf)=tf\circ\Phi^{y,x}[-d-1]$ and $\varphi^{x,y}_{\MB}(tf)=tf$
    for $x,y\in\MO_{\MA}$, $sa\in {}_x\MA_y[1]$ and $tf\in {}_{\Phi_0(x)}\MB^*_{\Phi_0(y)}[d]$.
\end{definition}

\begin{proposition}
\label{prop:morphisms-strict-case}
The two maps $s\varphi_{\MA} : (\MA\oplus\MB^*[d-1],sm_{\MA\oplus\MB^*})\rightarrow (\MA\oplus\MA^*[d-1],sm_{\MA\oplus\MA^*})$ and $s\varphi_{\MB} : (\MA\oplus\MB^*[d-1],sm_{\MA\oplus\MB^*})\rightarrow (\MB\oplus\MB^*[d-1],sm_{\MB\oplus\MB^*})$ defined in Definition \ref{def:maps-varphi} are strict cyclic $A_{\infty}$-morphisms.
\end{proposition}
\begin{proof}
We only check that $s\varphi_{\MA}$ is a morphism since the case of $s\varphi_{\MB}$ is similar. We omit the objects when writing the map $s\varphi_{\MA}$ for simplicity.
We have to verify that 
\begin{equation}
    \label{eq:pf-strict-morphism-1}
    s\varphi_{\MA}\circ sm_{\MA\oplus\MB^*\to \MA}^{\bar{x}^1,\dots,\bar{x}^n}=sm_{\MA\oplus\MA^*\to \MA}^{\bar{x}^1,\dots,\bar{x}^n}\circ(s\varphi_{\MA}^{\otimes \llg(\bar{x}^1)-1}\otimes s\varphi_{\MA}\otimes s\varphi_{\MA}^{\otimes \llg(\bar{x}^2)-1}\otimes\dots\otimes 
 s\varphi_{\MA}^{\otimes \llg(\bar{x}^n)-1})
\end{equation}
and
\begin{equation}
    \label{eq:pf-strict-morphism-2}
    s\varphi_{\MA}\circ sm_{\MA\oplus\MB^*\to \MB^*}^{\bar{x}^1,\dots,\bar{x}^n}=sm_{\MA\oplus\MA^*\to \MA^*}^{\bar{x}^1,\dots,\bar{x}^n}\circ(s\varphi_{\MA}^{\otimes \llg(\bar{x}^1)-1}\otimes s\varphi_{\MA}\otimes s\varphi_{\MA}^{\otimes \llg(\bar{x}^2)-1}\otimes \dots\otimes 
 s\varphi_{\MA}^{\otimes \llg(\bar{x}^n)-1})
\end{equation}
 for $\doubar{x}=(\bar{x}^1,\dots,\bar{x}^n)\in\bar{\MO}_{\MA}^n$.
 
First, note that  
\begin{equation}
    \begin{split}
    s\varphi_{\MA}(sm_{\MA\oplus\MB^*\to \MA}^{\bar{x}^1,\dots,\bar{x}^n}&(\bar{sa}^1,tf^1,\bar{sa}^2,\dots,\bar{sa}^{n-1},tf^{n-1},\bar{sa}^n))
    \\ &=sm_{\MA\oplus\MB^*\to \MA}^{\bar{x}^1,\dots,\bar{x}^n}(\bar{sa}^1,tf^1,\bar{sa}^2,\dots,\bar{sa}^{n-1},tf^{n-1},\bar{sa}^n)
    \\&=sm_{\MA\oplus\MA^*\to \MA}^{\bar{x}^1,\dots,\bar{x}^n}(\bar{sa}^1,s\varphi_{\MA}(tf^1),\bar{sa}^2,\dots,\bar{sa}^{n-1},s\varphi_{\MA}(tf^{n-1}),\bar{sa}^n)
    \end{split}
\end{equation}
for $\bar{sa}^i\in\MA[1]^{\otimes \bar{x}^i}$ and $tf^i\in {}_{\Phi_0(\rrt(\bar{x}^{i+1}))}\MB^*_{\Phi_0(\llt(\bar{x}^{i}))}[d]$,
so that \eqref{eq:pf-strict-morphism-1} holds.

We also have 
\begin{equation}
    \begin{split}
        &s\varphi_{\MA}\circ sm_{\MA\oplus\MB^*\to \MB^*}^{\bar{x}^1,\dots,\bar{x}^n}(\bar{sa}^1,tf^1,\bar{sa}^2,\dots,\bar{sa}^{n-1},tf^{n-1},\bar{sa}^n)
        \\&\hskip3cm=sm_{\MA\oplus\MB^*\to \MB^*}^{\bar{x}^1,\dots,\bar{x}^n}(\bar{sa}^1,tf^1,\bar{sa}^2,\dots,\bar{sa}^{n-1},tf^{n-1},\bar{sa}^n)\circ \Phi[-d-1]
    \end{split}
\end{equation}
and given $s_{-d}b\in{}_{\llt(\Bar{x}^n)}\MA_{\rrt(\Bar{x}^1)}[-d]$, we have that
\begin{equation}
    \begin{split}
        &sm_{\MA\oplus\MB^*\to \MB^*}^{\doubar{x}}(\bar{sa}^1,tf^1,\bar{sa}^2,\dots,\bar{sa}^{n-1},tf^{n-1},\bar{sa}^n)(\Phi[-d-1](s_{-d}b))
        \\&=sm_{\MB\oplus\MB^*\to\MB^*}^{\scriptsize{\doubarl{\Phi_0(\doubar{x})}}}(\Phi^{\otimes \llg(\bar{x}^1)-1}(\bar{sa}^1),tf^1,\Phi^{\otimes \llg(\bar{x}^2)-1}(\bar{sa}^2),\dots,tf^{n-1},\Phi^{\otimes \llg(\bar{x}^n)-1}(\bar{sa}^n))(s_{-d}\Phi(sb))
        \\&=(-1)^{\epsilon}\Gamma^{\MB}(sm_{\MB\oplus\MB^*\to \MB}^{\scriptsize{\doubarl{\Phi_0(\doubar{y})}}}\big(\Phi^{\otimes \llg(\bar{x}^n)-1}(\bar{sa}^n)\otimes\Phi(sb)\otimes\Phi^{\otimes \llg(\bar{x}^1)-1}(\bar{sa}^1),tf^{1},\dots,\\&\hskip8cmtf^{n-2},\Phi^{\otimes \llg(\bar{x}^{n-1})-1}(\bar{sa}^{n-1})\big),tf^{n-1})
    \end{split}
\end{equation}
with $\epsilon=(|\bar{sa}^n|+|sb|)(\sum\limits_{i=1}^{n-1}(|\bar{sa}^i|+|tf^i|)+d(\sum\limits_{i=1}^{n}|\bar{sa}^i|+\sum\limits_{i=1}^{n-1}|tf^i|)$ and $\doubar{y}=(\Bar{x}^n\sqcup\Bar{x}^1,\dots,\Bar{x}^{n-1})$.

On the other hand, we have 
\begin{small}
\begin{equation}
    \begin{split}
        &sm_{\MA\oplus\MA^*\to \MA^*}^{\bar{x}^1,\dots,\bar{x}^n}(s\varphi_{\MA}^{\otimes \llg(\bar{x}^1)-1}(\bar{sa}^1)\otimes s\varphi_{\MA}(tf^1)\otimes s\varphi_{\MA}^{\otimes \llg(\bar{x}^2)-1}(\bar{sa}^2)\dots\otimes 
 s\varphi_{\MA}^{\otimes \llg(\bar{x}^n)-1}(\bar{sa}^n))(sb)
        \\&=(-1)^{d(\sum\limits_{i=1}^{n}|\bar{sa}^i|+\sum\limits_{i=1}^{n-1}|tf^i|)}
        \\&\phantom{=}\Gamma^{\MA}(sm_{\MA\oplus\MA^*\to \MA^*}^{\bar{x}^1,\dots,\bar{x}^n}(s\varphi_{\MA}^{\otimes \llg(\bar{x}^1)-1}(\bar{sa}^1)\otimes s\varphi_{\MA}(tf^1)\otimes s\varphi_{\MA}^{\otimes \llg(\bar{x}^2)-1}(\bar{sa}^2)\dots\otimes 
 s\varphi_{\MA}^{\otimes \llg(\bar{x}^n)-1}(\bar{sa}^n)),sb)
        \\&=(-1)^{\epsilon}\Gamma^{\MA}(sm_{\MA\oplus\MA^*\to \MA}^{\bar{x}^n\sqcup\bar{x}^1,\bar{x}^2,\dots,\bar{x}^{n-1}}(\bar{sa}^n\otimes sb\otimes \bar{sa}^1,tf^{1}\circ \Tilde{\Phi},\bar{sa}^{2},\dots,tf^{n-2}\circ \Tilde{\Phi},\bar{sa}^{n-1}),tf^{n-1}\circ \Tilde{\Phi}).
    \end{split}
\end{equation}
\end{small}
Using the identity \eqref{eq:cyclicity-mA-to-mB}, we thus get \eqref{eq:pf-strict-morphism-2}. 

It remains to show that the morphisms are cyclic. 
To prove it, we note that
\begin{equation}
    \begin{split}
    {}_{y}\Gamma^{\MA}_{x}(s\varphi_{\MA}^{y,x}(tf),s\varphi_{\MA}^{x,y}(sa))&={}_{y}\Gamma^{\MA}_{x}(t(f\circ\Phi^{x,y}[-1]),sa) =f(\Phi^{x,y}[-1](a)) ={}_{y}\Gamma^{\Phi[-1]}_{x}(tf,sa)
    \end{split}
\end{equation}
as well as  
\begin{equation}
    \begin{split}
    {}_{y}\Gamma^{\MB}_{x}(\varphi_{\MB}^{y,x}(tf),\varphi_{\MB}^{x,y}(sa))={}_{y}\Gamma^{\MB}_{x}(tf,\Phi^{x,y}[-1](a)) =f(\Phi^{x,y}[-1](a)) ={}_{y}\Gamma^{\Phi[-1]}_{x}(tf,sa)
    \end{split}
\end{equation}
for $sa\in{}_{x}\MA_{y}[1]$, $tf\in{}_{\Phi_0(y)}\MB^*_{\Phi_0(x)}[d]$. 
The second condition to be a cyclic morphism is obviously satisfied since $\varphi_{\MA}^{\bar{x}}$ and $\varphi_{\MB}^{\bar{x}}$ vanish for $\bar{x}\in\MO^n$ with $n>2$.
\end{proof}

\begin{definition}
    Let $(\MA\oplus\MA^*[d-1],sm_{\MA\oplus\MA^*})$, $(\MB\oplus\MB^*[d-1],sm_{\MB\oplus\MB^*})$ be $A_{\infty}$-categories and consider a map $\Phi : \MO_{\MA}\rightarrow \MO_{\MB}$.
    A \textbf{\textcolor{ultramarine}{hat morphism}} from $\MA\oplus\MA^*[d-1]$ to $\MB\oplus\MB^*[d-1]$ is a span $(sm_{\MA\oplus\MB^*},\varphi_{\MA},\varphi_{\MB})$ in the category of $A_{\infty}$-categories, \textit{i.e.} a triple  $(sm_{\MA\oplus\MB^*},s\varphi_{\MA},s\varphi_{\MB})$ where $sm_{\MA\oplus\MB^*}$ is an $A_{\infty}$-structure on the graded quiver  $\MA\oplus\MB^*[d-1]$ defined for $x,y\in\MO_{\MA}$ by ${}_x(\MA\oplus\MB^*[d-1])_y={}_x\MA_y\oplus {}_{\Phi(x)}\MB^*[d-1]_{\Phi(y)}$ and $(\id_{\MA},s\varphi_{\MA})$, $(\Phi,s\varphi_{\MB})$ are $A_{\infty}$-morphisms
     \begin{equation}
    \label{eq:hat-morphisms-2}
\begin{tikzcd}
&(\MA \oplus \MB^*[d-1],sm_{\MA\oplus\MB^*}) \arrow[swap,"s\varphi_{\MA}"]{dl} \arrow[swap,"s\varphi_{\MB}"]{dr}\\
(\MA \oplus \MA^*[d-1],sm_{\MA\oplus\MA^*})&& (\MB \oplus \MB^*[d-1], sm_{\MB\oplus\MB^*})
\end{tikzcd}
\end{equation}
\end{definition}

\begin{definition}
     Consider $A_{\infty}$-categories $(\MA\oplus\MA^*[d-1],sm_{\MA\oplus\MA^*})$, $(\MB\oplus\MB^*[d-1],sm_{\MB\oplus\MB^*})$ and $(\mathcal{C}\oplus\mathcal{C}^*[d-1],sm_{\mathcal{C}\oplus\mathcal{C}^*})$.
     We say that two hat morphisms 
     \begin{small}
     \begin{equation}
         (sm_{\MA\oplus\MB^*},s\varphi_{\MA},s\varphi_{\MB}) : \MA\oplus\MA^*[d-1]\rightarrow \MB\oplus\MB^*[d-1] \text{, }(sm_{\MB\oplus\mathcal{C}^*},s\psi_{\MB},s\psi_{\mathcal{C}}) : \MB\oplus\MB^*[d-1]\rightarrow \mathcal{C}\oplus\mathcal{C}^*[d-1]
     \end{equation}
     \end{small}  
     are \textbf{\textcolor{ultramarine}{composable}} if there is a triple $(sm_{\MA\oplus\mathcal{C}^*},s\chi_{\MA},s\chi_{\mathcal{C}})$ where $sm_{\MA\oplus\mathcal{C}^*}$ is an $A_{\infty}$-structure on $\MA\oplus\mathcal{C}^*[d-1]$ and where $ s\chi_{\MA}: \MA\oplus\mathcal{C}^*[d-1] \rightarrow \MA\oplus\MB^*[d-1]$ and $s\chi_{\mathcal{C}} : \MA\oplus\mathcal{C}^*[d-1] \rightarrow \MB\oplus\mathcal{C}^*[d-1]$ are $A_{\infty}$-morphisms such that $s\varphi_{\MB}\circ s\chi_{\MA}=s\psi_{\MB}\circ s\chi_{\mathcal{C}}$.
     The \textbf{\textcolor{ultramarine}{composition}} of $(sm_{\MA\oplus\MB^*},s\varphi_{\MA},s\varphi_{\MB})$ and $(sm_{\MB\oplus\mathcal{C}^*},s\psi_{\MB},s\psi_{\mathcal{C}})$ is then given by $(sm_{\MA\oplus\mathcal{C}^*},s\varphi_{\MA}\circ s\chi_{\MA},s\psi_{\mathcal{C}}\circ s\chi_{\mathcal{C}})$.
\end{definition}

\begin{definition}
    A \textbf{\textcolor{ultramarine}{partial category}} is an $A_{\infty}$-pre-category as defined in \cite{ks} where the multiplications $m_n$ vanish for $n>2$. Explicitly, a partial category $\mathcal{A}$ consists of a class of objects $\MO$, a subclass $\MO_2^{tr}$ of $\MO^2$ such that for every $(x,y)\in \MO_2^{tr}$ there exists a graded vector space ${}_{y}\MA_{x}$ and a subclass $\MO^{tr}_3$ of $\MO^3$ such that for every $(x,y,z)\in \MO_3^{tr}$ there exists an associative map $\circ : {}_{z}\MA_{y}\otimes {}_{y}\MA_{x}\rightarrow {}_{z}\MA_{x}$. It is required that if $(x,y,z)\in \MO^{tr}_3$, then $(x,y), (y,z) and (x,z)\in \MO^{tr}_2$. Elements of $\MO^{tr}_2$ and $\MO^{tr}_3$ are called \textbf{\textcolor{ultramarine}{transevrsal sequences}}. 
\end{definition}

\begin{definition}
    The \textbf{\textcolor{ultramarine}{$A_{\infty}$-hat category}} is the partial category $\Ahat$ whose objects are $A_{\infty}$-categories of the form $\MA\oplus\MA^*[d-1]$ and whose morphisms are hat morphisms.
\end{definition}

\begin{definition}
    A \textbf{\textcolor{ultramarine}{functor}} between partial categories $\MA$ and $\MB$ with respective sets of objects $\MO_{\MA}$ and $\MO_{\MB}$ is the data of a map $F_0 : \MO_{\MA}\rightarrow \MO_{\MB}$ together with a family $F=({}_yF_x)_{x,y\in\MO_{\MA}}$ sending a morphism $f:x\rightarrow y$ to a morphism ${}_yF_x(f) : F_0(x)\rightarrow F_0(y)$ such that if two morphisms $f:x\rightarrow y$ and $g : y\rightarrow z$ are composable, then ${}_yF_x(f)$ and ${}_zF_y(g)$ are composable and their composition is given by the morphism ${}_zF_y(g)\circ {}_yF_x(f)={}_zF_x(g\circ f)$.
\end{definition}

\begin{definition}
    We define the partial subcategory $\cyc$ of $\Ahat$ whose objects are cyclic $A_{\infty}$-categories of the form $\MA\oplus\MA^*[d-1]$ and whose morphisms $\MA\oplus\MA^*[d-1]\rightarrow \MB\oplus\MB^*[d-1]$ are the data of an almost cyclic $A_{\infty}$-structure on $\MA\oplus\MB^*[d-1]$ together with a diagram of the form \eqref{eq:hat-morphisms-2} where $s\varphi_{\MA}$ and $s\varphi_{\MB}$ are $A_{\infty}$-morphisms.
\end{definition}

\begin{definition}
    We define the partial subcategory $\Scyc$ of $\cyc$ whose objects are the ones of $\cyc$ and whose morphisms are morphisms $(sm_{\MA\oplus\MB^*},s\varphi_{\MA},s\varphi_{\MB})$ of $\cyc$ such that $s\varphi_{\MA}$ and $s\varphi_{\MB}$ are strict and cyclic.
\end{definition}
We now use the results of Propositions \ref{prop:struct-strict-case} and \ref{prop:morphisms-strict-case} to construct a functor between the category of strict $d$-pre-Calabi-Yau morphisms $\SpCY$ and the partial category $\Scyc$.

\begin{corollary}
\label{coro:strict-case}
There exists a functor $\mathcal{S} : \SpCY\rightarrow \Scyc$ which sends a $d$-pre-Calabi-Yau category $(\MA,s_{d+1}M_{\MA})$ to the cyclic $A_{\infty}$-category $(\MA\oplus\MA^*[d-1],sm_{\MA\oplus\MA^*})$ defined in Proposition \ref{prop:equiv-pCY-Ainf} and a strict $d$-pre-Calabi-Yau morphism $(\Phi_0,s_{d+1}\Phi) : (\MA,s_{d+1}M_{\MA})\rightarrow(\MB,s_{d+1}M_{\MB})$ to the almost cyclic $A_{\infty}$-structure $sm_{\MA\oplus\MB^*}$ given in Definition \ref{def:m-A-B} together with the strict cyclic $A_{\infty}$-morphisms 
 \begin{equation}
\begin{tikzcd}
&(\MA \oplus \MB^*[d-1],sm_{\MA\oplus\MB^*}) \arrow[swap,"s\varphi_{\MA}"]{dl} \arrow[swap,"s\varphi_{\MB}"]{dr}\\
(\MA \oplus \MA^*[d-1],sm_{\MA\oplus\MA^*})&& (\MB \oplus \MB^*[d-1], sm_{\MB\oplus\MB^*})
\end{tikzcd}
\end{equation}
given in Definition \ref{def:maps-varphi}.
\end{corollary}
\begin{proof}
    The fact that $sm_{\MA\oplus\MB^*}$ is an almost cyclic $A_{\infty}$-structure on $\MA\oplus\MB^*[d-1]$ is part of Proposition \ref{prop:struct-strict-case}. The fact that $s\varphi_{\MA}$ and $s\varphi_{\MB}$ are strict cyclic $A_{\infty}$-morphisms is part of Proposition \ref{prop:morphisms-strict-case}. We show that $\mathcal{S}$ is compatible with the partial products. 
    
    Consider $d$-pre-Calabi-Yau categories $(\MA,s_{d+1}M_{\MA})$, $(\MB,s_{d+1}M_{\MB})$ and $(\mathcal{C},s_{d+1}M_{\mathcal{C}})$ as well as two strict $d$-pre-Calabi-Yau morphisms that we denote $(F_0,s_{d+1}F) : (\MA,s_{d+1}M_{\MA})\rightarrow (\MB,s_{d+1}M_{\MB})$ and $(G_0,s_{d+1}G) : (\MB,s_{d+1}M_{\MB})\rightarrow (\mathcal{C},s_{d+1}M_{\mathcal{C}})$. Write  $\mathcal{S}(F_0,s_{d+1}F)=(sm_{\MA\oplus\MB^*},s\varphi_{\MA},s\varphi_{\MB})$ and $\mathcal{S}(G_0,s_{d+1}G)=(sm_{\MB\oplus\mathcal{C}^*},s\psi_{\MB},s\psi_{\mathcal{C}})$.
    To simplify, we denote $H_0 : \MO_{\MA}\rightarrow \MO_{\mathcal{C}}$ the composition of $F_0$ and $G_0$ and $H : \MA[1]\rightarrow \mathcal{C}[1]$ the composition of $s_{d+1}F$ and $s_{d+1}G$.
    By definition, $\mathcal{S}$ sends $H$ to the almost cyclic $A_{\infty}$-structure $sm_{\MA\oplus\mathcal{C}^*}$ whose composition of with the canonical projection on $\MA[1]$
    is given by the map $m_{\MA\oplus\mathcal{C}^*\to \MA}$ where $m_{\MA\oplus\mathcal{C}^*\to \MA}^{\doubar{x}}$ equals to
    \begin{small}
\begin{equation}
\begin{split}
m_{\MA\oplus\MA^*\to \MA}^{\doubar{x}}\circ (\id^{\otimes \llg(\bar{x}^1)-1}\otimes H^*[d+1]\otimes\id^{\otimes \llg(\bar{x}^2)-1}\otimes\dots\otimes \id^{\otimes \llg(\bar{x}^{n-1})-1}\otimes H^*[d+1]\otimes\id^{\otimes \llg(\bar{x}^n)-1})
\end{split}
\end{equation}
 \end{small}
for $\doubar{x}=(\bar{x}^1,\dots,\bar{x}^n)$ and whose composition with the canonical projection on $\mathcal{C}^*[d]$ is defined by 
\begin{small}
\begin{equation}
m_{\MA\oplus\mathcal{C}^*\to \mathcal{C}^*}^{\doubar{x}}= m_{\mathcal{C}\oplus\mathcal{C}^*\to \mathcal{C}^*}^{\scriptsize{\doubarl{H_0(\doubar{x})}}}\circ (H^{\otimes \llg(\bar{x}^1)-1}\otimes \id\otimes H^{\otimes \llg(\bar{x}^{2})-1}\otimes\dots\otimes \id\otimes H^{\otimes \llg(\bar{x}^n)-1})
\end{equation}
\end{small}
    together with the cyclic $A_{\infty}$-morphisms $s\xi_{\MA}$ and $s\xi_{\mathcal{C}}$ defined by 
    \begin{equation}
        \begin{split}
            s\xi_{\MA}^{x,y}(sa)=sa, \; s\xi_{\mathcal{C}}^{x,y}(sa)=s_{d+1}G^{x,y}\circ s_{d+1}F^{x,y}(sa)
        \end{split}
    \end{equation}
    for $sa\in {}_{x}\MA_{y}[1]$ and 
    \begin{equation}
        \begin{split}
            s\xi_{\MA}^{x,y}(tf)&=tf\circ (s_{d+1}G^{F_0(y),F_0(x)}\circ s_{d+1}F^{y,x})[-d-1], \;
            s\xi_{\mathcal{C}}^{x,y}(tf)=tf
        \end{split}
    \end{equation}
    for $tf\in {}_{H_0(x)}\mathcal{C}^*_{H_0(y)}[d]$. 

    Set $s\chi_{\MA} :\MA[1]\oplus\mathcal{C}^*[d]\rightarrow \MA[1]\oplus\MB^*[d]$ and $s\chi_{\mathcal{C}} :\MA[1]\oplus\mathcal{C}^*[d]\rightarrow \MB[1]\oplus\mathcal{C}^*[d]$ to be given by 
    \begin{equation}
        \begin{split}
            &\hskip2cm s\chi_{\MA}^{x,y}(sa)=sa\text{, }s\chi_{\mathcal{C}}^{x,y}(sa)=s_{d+1}F^{x,y}(sa)\text{ for }sa\in {}_{x}\MA_{y}[1]
            \\& \text{ and } s\chi_{\MA}^{x,y}(tf)=tf\circ s_{d+1}G^{F_0(y),F_0(x)}[-d-1]\text{, } s\chi_{\mathcal{C}}^{x,y}(tf)=tf\text{ for }tf\in {}_{H_0(x)}\mathcal{C}^*_{H_0(y)}[d].
        \end{split}
    \end{equation} 
    For simplicity, we will omit the elements when writing the maps $s\xi_{\MA}$ and $s\xi_{\mathcal{C}}$ and the maps $s_{d+1}G$ and $s_{d+1}F$.
    We have to show that $s\chi_{\MA}$ and $s\chi_{\mathcal{C}}$ are $A_{\infty}$-morphisms satisfying that $s\varphi_{\MA}\circ 
 s\chi_{\MA}=s\xi_{\MA}$ and $s\varphi_{\mathcal{C}}\circ 
 s\chi_{\mathcal{C}}=s\xi_{\mathcal{C}}$.
    We have that
    \begin{equation}
        \begin{split}
            &sm_{\MA\oplus\MB^*}^{\doubar{x}}\circ(s\chi_{\MA}^{\otimes \llg(\bar{x}^1)-1}\otimes\dots\otimes 
 s\chi_{\MA}^{\otimes\llg(\bar{x}^n)-1})(\bar{sa}^1,tf^1,\dots,\bar{sa}^{n-1},tf^{n-1},\bar{sa}^n)
            \\&=sm_{\MA\oplus\MB^*}^{\doubar{x}}(\bar{sa}^1,tf^1\circ s_{d+1}G[-d-1],\dots,\bar{sa}^{n-1},tf^{n-1}\circ s_{d+1}G[-d-1],\bar{sa}^n)
        \end{split}
    \end{equation}
    so we have 
    \begin{small}
    \begin{equation}
        \begin{split}
            &\pi_{\MA[1]}\circ sm_{\MA\oplus\MB^*}^{\doubar{x}}\circ(s\chi_{\MA}^{\otimes \llg(\bar{x}^1)-1}\otimes\dots\otimes 
 s\chi_{\MA}^{\otimes\llg(\bar{x}^n)-1})(\bar{sa}^1,tf^1,\dots,\bar{sa}^{n-1},tf^{n-1},\bar{sa}^n)
            \\&=sm_{\MA\oplus\MA^*\to \MA}^{\doubar{x}}(\bar{sa}^1,tf^1\circ (s_{d+1}G\circ s_{d+1}F)[-d-1],\dots,\bar{sa}^{n-1},tf^{n-1}\circ (s_{d+1}G\circ s_{d+1}F)[-d-1],\bar{sa}^n)
        \end{split}
    \end{equation}
    \end{small}
    and 
    \begin{equation}
        \begin{split}
            &\pi_{\MB^*[-d]}\circ sm_{\MA\oplus\MB^*}^{\doubar{x}}\circ(s\chi_{\MA}^{\otimes \llg(\bar{x}^1)-1}\otimes\dots\otimes s\chi_{\MA}^{\otimes\llg(\bar{x}^n)-1})(\bar{sa}^1,tf^1,\dots,\bar{sa}^{n-1},tf^{n-1},\bar{sa}^n)
            \\&=sm_{\MB\oplus\MB^*\to \MB^*}^{\scriptsize{\doubarl{F_0(\doubar{x})}}}(s_{d+1}F^{\otimes \llg(\bar{x}^1)-1}(\bar{sa}^1),tf^1\circ s_{d+1}G[-d-1],\dots,
            \\& \hspace{3cm} s_{d+1}F^{\otimes \llg(\bar{x}^{n-1})-1}(\bar{sa}^{n-1}),tf^{n-1}\circ s_{d+1}G[-d-1],s_{d+1}F^{\otimes \llg(\bar{x}^n)-1}(\bar{sa}^n))
        \end{split}
    \end{equation}
    for $n\in\NN^*$, $\doubar{x}=(\bar{x}^1,\dots,\bar{x}^n)\in\bar{\MO}_{\MA}^n$, $\bar{sa}^i\in\MA[1]^{\otimes\bar{x}^i}$ and $tf^i\in{}_{H_0(\rrt(\bar{x}^{i+1}))}\mathcal{C}^*_{H_0(\llt(\bar{x}^{i}))}[d]$.
    Thus, we have that 
    \begin{small}
        \begin{equation}
            \big(\pi_{\MB^*[-d]}\circ sm_{\MA\oplus\MB^*}^{\bar{x}^1,\dots,\bar{x}^n}\circ(s\chi_{\MA}^{\otimes \llg(\bar{x}^1)-1}\otimes s\chi_{\MA}\otimes\dots\otimes 
 s\chi_{\MA}^{\otimes\llg(\bar{x}^n)-1})(\bar{sa}^1,tf^1,\dots,\bar{sa}^{n-1},tf^{n-1},\bar{sa}^n)\big)
        \end{equation}
    \end{small}
    sends $(s_{-d}b)$ to
    \begin{small}
    \begin{equation}
        \begin{split}
            &(-1)^{\epsilon}(\bigotimes\limits_{i=1}^{n-1}(f^{n-i}\circ s_d \circ s_{d+1}G[-d-1]))
            \\&\phantom{=}s_{d+1}M_{\MB}^{\doubar{y}^{-1}}\big(s_{d+1}F^{\otimes\llg(\bar{x}^{n-1})-1}(\bar{sa}^{n-1})\otimes\dots\otimes s_{d+1}F^{\otimes\llg(\bar{x}^{n})-1}(\bar{sa}^n)\otimes sb\otimes s_{d+1}F^{\otimes\llg(\bar{x}^{1})-1}(\bar{sa}^1)\big)
            \\&=(-1)^{\epsilon}(\bigotimes\limits_{i=1}^{n-1}(f^{n-i}\circ s_d))
            \\&\phantom{=}s_{d+1}M_{\mathcal{C}}^{\scriptsize{\doubarl{G_0(\doubar{y}^{-1})}}}\big(H^{\otimes\llg(\bar{x}^{n-1})-1}(\bar{sa}^{n-1})\otimes\dots\otimes H^{\otimes\llg(\bar{x}^n)-1}(\bar{sa}^n)\otimes s_{d+1}G(sb)\otimes H^{\otimes\llg(\bar{x}^1)-1}(\bar{sa}^1)\big)
        \end{split}
    \end{equation}
    \end{small}
    with $\doubar{y}=\doubarl{F_0(\bar{x}^n\sqcup\bar{x}^1,\dots,\bar{x}^{n-1})}$ and
    \begin{small}
    \begin{equation}
        \begin{split}
            \epsilon&=\sum\limits_{i=1}^{n-1}|tf^i|(\sum\limits_{k=1}^i|\bar{sa}^k|+d+1)+d(n-1)
    +\hskip -3mm\sum\limits_{1\leq i<j\leq n-1}\hskip -3mm|\bar{sa}^i||\bar{sa}^j|+\hskip -3mm\sum\limits_{1\leq i<j\leq n-1}\hskip -3mm|tf^i||tf^j|
    \\&\phantom{=}+(d+1)\sum\limits_{i=1}^n|\bar{sa}^i|+(|\bar{sa}^n|+|sb|)|\bar{sa}^1|
        \end{split}
    \end{equation}
    \end{small}
    since $(G_0,s_{d+1}G)$ is a strict $d$-pre-Calabi-Yau morphism.

    On the other hand, we have 
    \begin{equation}
        \begin{split}
            &\pi_{\MA[1]}\circ s\chi_{\MA}\circ sm_{\MA\oplus\mathcal{C}^*}^{\bar{x}^1,\dots,\bar{x}^n}(\bar{sa}^1,tf^1,\dots,\bar{sa}^{n-1},tf^{n-1},\bar{sa}^n)
            \\&=sm_{\MA\oplus\MA^*\to \MA}^{\bar{x}^1,\dots,\bar{x}^n}(\bar{sa}^1,tf^1\circ H[-d-1],\dots,\bar{sa}^{n-1},tf^{n-1}\circ H[-d-1],\bar{sa}^n)
        \end{split}
    \end{equation}
    and
\begin{equation}
    \begin{split}
        &\pi_{\MB^*[d]}\circ s\chi_{\MA}\circ sm_{\MA\oplus\mathcal{C}^*}^{\bar{x}^1,\dots,\bar{x}^n}(\bar{sa}^1,tf^1,\dots,\bar{sa}^{n-1},tf^{n-1},\bar{sa}^n)
        \\&=sm_{\mathcal{C}\oplus\mathcal{C}^*\to \mathcal{C}^*}^{\scriptsize{\doubarl{H_0(\doubar{x})}}}(\bar{sa}^1,tf^1\circ H[-d-1],\dots,\bar{sa}^{n-1},tf^{n-1}\circ H[-d-1],\bar{sa}^n)\circ s_{d+1}G[-d-1]
    \end{split}
\end{equation}
    Thus, we have that
    \begin{small}
    \begin{equation}
    \begin{split}
        &\big(\pi_{\MB^*[d]}\circ s\chi_{\MA}\circ sm_{\MA\oplus\mathcal{C}^*}^{\bar{x}^1,\dots,\bar{x}^n}(\bar{sa}^1,tf^1,\dots,\bar{sa}^{n-1},tf^{n-1},\bar{sa}^n)\big)(s_{-d}b)\\&=sm_{\mathcal{C}\oplus\mathcal{C}^*\to \mathcal{C}^*}^{\scriptsize{\doubarl{H_0(\doubar{x})}}}\big(H^{\otimes\llg(\bar{x}^1)-1}(\bar{sa}^1),tf^1,\dots,H^{\otimes\llg(\bar{x}^{n-1})-1}(\bar{sa}^{n-1}),tf^{n-1},H^{\otimes\llg(\bar{x}^n)-1}(\bar{sa}^n)\big)\big(s_{d+1}G(sb)\big)
        \\&=(-1)^{\epsilon} \big(\bigotimes\limits_{i=1}^{n-1}(f^i\circ s_d)\big)
        \\&\phantom{=}s_{d+1}M_{\mathcal{C}}^{\scriptsize{\doubarl{G_0(\doubar{y}^{-1})}}}\big(H^{\otimes \llg(\bar{x}^{n-1})-1}(\bar{sa}^{n-1})\otimes\dots\otimes H^{\otimes\llg(\bar{x}^n)-1}(\bar{sa}^n)\otimes s_{d+1}G(sb)\otimes H^{\otimes\llg(\bar{x}^1)-1}(\bar{sa}^1)\big)
    \end{split}
\end{equation}
\end{small}
Therefore, $s\chi_{\MA}$ is an $A_{\infty}$-morphism. The fact that $s\chi_{\mathcal{C}}$ is an $A_{\infty}$-morphism can be proved in a similar manner.
    
    Moreover, we have that 
    $(s\varphi_{\MB}\circ s\chi_{\MA})(sa)=s_{d+1}F^{x,y}(sa)=(s\psi_{\MB}\circ 
 s\chi_{\mathcal{C}})(sa)$
    for an element $sa\in{}_{x}\MA_{y}[1]$ and 
    \[
    (s\varphi_{\MB}\circ s\chi_{\MA})(tf)=tf\circ s_{d+1}G^{F_0(y),F_0(x)}[-d-1]=(s\psi_{\MB}\circ s\chi_{\mathcal{C}})(tf)
    \]
    for $tf\in{}_{H_0(x)}\mathcal{C}^*_{H_0(y)}[d]$ meaning that $(sm_{\MA\oplus \mathcal{C}^*},s\xi_{\MA},s\xi_{\mathcal{C}})$ is the composition of the morphisms $\mathcal{S}(F_0,s_{d+1}F)$ and $\mathcal{S}(G_0,s_{d+1}G)$.
    Therefore, $\mathcal{S}$ is compatible with partial products.
    \end{proof}
\subsection{General case} \label{general case}
We now present the relation between non necessarily strict $d$-pre-Calabi-Yau morphisms and $A_{\infty}$-morphisms.
Consider $d$-pre-Calabi-Yau categories $(\MA,s_{d+1}M_{\MA})$, $(\MB,s_{d+1}M_{\MB})$ as well as a $d$-pre-Calabi-Yau morphism $(\Phi_0,s_{d+1}\Phi) :(\MA,s_{d+1}M_{\MA})\rightarrow (\MB,s_{d+1}M_{\MB})$ as defined in Definition \ref{def:pcY-morphism}.
We first construct an $A_{\infty}$-structure on $\MA\oplus\MB^{*}[d-1]$. 

\begin{definition}
\label{def:structure-pCY-case}
We define $s_{d+1}M_{\MA\oplus\MB^*\to\MA}\in\mathcalboondox{B}^{\bullet}_{d,\Phi_0}(\MA,\MB)^{\MA}$ by $s_{d+1}M_{\MA\oplus\MB^*\to\MA}^{\doubar{x}}=\sum\mathcal{E}(\mathcalboondox{D})$ where the sum is over all the filled diagrams $\mathcalboondox{D}$ of type $\doubar{x}$ and of the form 

    \begin{equation}
    \label{eq:maps-to-A}
       \begin{tikzpicture}[line cap=round,line join=round,x=1.0cm,y=1.0cm]
\clip(-7.5,-1.3) rectangle (10.573959328125309,1);
      \draw [line width=0.5pt] (0.,0.) circle (0.5cm);
     \shadedraw[rotate=30,shift={(0.5cm,0cm)}] \doublefleche;
     \shadedraw[rotate=150,shift={(0.5cm,0cm)}] \doublefleche;
     \draw[line width=1.1pt,rotate=90][->, >= stealth, >= stealth](0.5,0)--(0.9,0);
     \draw [line width=0.5pt] (1.12,-0.65) circle (0.3cm);
     \shadedraw[shift={(0.86cm,-0.5cm)},rotate=150] \doubleflechescindeeleft;
     \shadedraw[shift={(0.86cm,-0.5cm)},rotate=150] \doubleflechescindeeright;
     \shadedraw[shift={(0.86cm,-0.5cm)},rotate=150] \fleche;
     \draw [rotate around ={60:(1.12,-0.65)}] [->, >= stealth, >= stealth] (1.43,-0.65)--(1.73,-0.65);
     \draw [rotate around ={-120:(1.12,-0.65)}] [->, >= stealth, >= stealth] (1.43,-0.65)--(1.73,-0.65);
     \draw [line width=0.5pt] (-1.12,-0.65) circle (0.3cm);
     \shadedraw[shift={(-0.86cm,-0.5cm)},rotate=30] \doubleflechescindeeleft;
      \shadedraw[shift={(-0.86cm,-0.5cm)},rotate=30] \doubleflechescindeeright;
       \shadedraw[shift={(-0.86cm,-0.5cm)},rotate=30] \fleche;
      \draw [rotate around ={-60:(-1.12,-0.65)}] [->, >= stealth, >= stealth] (-1.43,-0.65)--(-1.73,-0.65);
     \draw [rotate around ={120:(-1.12,-0.65)}] [->, >= stealth, >= stealth] (-1.43,-0.65)--(-1.73,-0.65);
\begin{scriptsize}
\draw [fill=black] (0,-0.6) circle (0.3pt);
\draw [fill=black] (0.2,-0.55) circle (0.3pt);
\draw [fill=black] (-0.2,-0.55) circle (0.3pt);
\draw [fill=black] (1.45,-0.85) circle (0.3pt);
\draw [fill=black] (1.5,-0.67) circle (0.3pt);
\draw [fill=black] (1.33,-0.97) circle (0.3pt);
\draw [fill=black] (-1.45,-0.85) circle (0.3pt);
\draw [fill=black] (-1.5,-0.67) circle (0.3pt);
\draw [fill=black] (-1.33,-0.97) circle (0.3pt);
\end{scriptsize}
\draw (0,0.25) node[anchor=north ] {$M_{\mathcal{A}}$};
\draw (-1.12,-0.4) node[anchor=north ] {$\Phi$};
\draw (1.12,-0.4) node[anchor=north ] {$\Phi$};
\end{tikzpicture}
    \end{equation}
    and $s_{d+1}M_{\MA\oplus\MB^*\to\MB^*}\in\mathcalboondox{B}^{\bullet}_{d,\Phi_0}(\MA,\MB)^{\MB}$ by $s_{d+1}M_{\MA\oplus\MB^*\to\MB^*}^{\doubar{x}}=\sum\mathcal{E}(\mathcalboondox{D'})$ where the sum is over all the filled diagrams $\mathcalboondox{D'}$ of type $\doubar{x}$ and of the form 
\begin{equation}
    \label{eq:maps-to-B}
\begin{tikzpicture}[line cap=round,line join=round,x=1.0cm,y=1.0cm]
\clip(-5,-0.7) rectangle (5.104458484699738,1.45);
  \draw [line width=0.5pt] (0.,0.) circle (0.5cm);
     \draw [rotate=90] [->, >= stealth, >= stealth] (0.5,0)--(0.9,0);
     \draw [rotate=-30] [->, >= stealth, >= stealth] (0.5,0)--(0.9,0);
     \draw [rotate=-150] [->, >= stealth, >= stealth] (0.5,0)--(0.9,0);
     \draw [rotate=0] [<-, >= stealth, >= stealth] (0.5,0)--(0.9,0);
     \draw [line width=0.5pt] (1.15,0.) circle (0.25cm);
     \draw[rotate around={60:(1.15,0)}] [->, >= stealth, >= stealth] (1.4,0)--(1.7,0);
      \draw[rotate around={-60:(1.15,0)}] [->, >= stealth, >= stealth] (1.4,0)--(1.7,0);
      \shadedraw[rotate around={120:(1.15,0)}, shift={(1.4cm,0cm)}] \doublefleche;
       \shadedraw[shift={(1.4cm,0cm)}] \doublefleche;
     \draw [rotate=60] [<-, >= stealth, >= stealth] (0.5,0)--(0.9,0);
       \draw [rotate=180] [<-, >= stealth, >= stealth] (0.5,0)--(0.9,0);
      \draw [line width=0.5pt] (-1.15,0.) circle (0.25cm);
     \draw[rotate around={-60:(-1.15,0)}] [->, >= stealth, >= stealth] (-1.4,0)--(-1.7,0);
      \draw[rotate around={60:(-1.15,0)}] [->, >= stealth, >= stealth] (-1.4,0)--(-1.7,0);
      \shadedraw[rotate around={-60:(-1.15,0)}, shift={(-0.9cm,0cm)}] \doublefleche;
       \shadedraw[rotate around={180:(-1.15,0)},shift={(-0.9cm,0cm)}] \doublefleche;
        \draw [rotate=120] [<-, >= stealth, >= stealth] (0.5,0)--(0.9,0);
        \draw [line width=0.5pt] (0.575,1) circle (0.25cm);
        \draw[rotate around={120:(0.575,1)}] [->, >= stealth, >= stealth] (0.825,1)--(1.125,1);
      \draw[rotate around={0:(0.575,1)}] [->, >= stealth, >= stealth] (0.825,1)--(1.125,1);
      \shadedraw[rotate around={60:(0.575,1)}, shift={(0.825cm,1cm)}] \doublefleche;
       \shadedraw[rotate around={180:(0.575,1)},shift={(0.825cm,1cm)}] \doublefleche;
        \draw [line width=0.5pt] (-0.575,1) circle (0.25cm);
      \draw[rotate around={-120:(-0.575,1)}] [->, >= stealth, >= stealth] (-0.825,1)--(-1.125,1);
      \draw[rotate around={0:(-0.575,1)}] [->, >= stealth, >= stealth] (-0.825,1)--(-1.125,1);
      \shadedraw[rotate around={120:(-0.575,1)}, shift={(-0.325cm,1cm)}] \doublefleche;
       \shadedraw[rotate around={0:(-0.575,1)},shift={(-0.325cm,1cm)}] \doublefleche;
      \draw[line width=1.1pt,rotate around={150:(0,0)}] [<-, >= stealth, >= stealth] (0.5,0)--(0.9,0);
\begin{scriptsize}
\draw [fill=black] (0.65,0.7) circle (0.3pt);
\draw [fill=black] (0.79,0.75) circle (0.3pt);
\draw [fill=black] (0.88,0.85) circle (0.3pt);
\draw [fill=black] (-1,0.27) circle (0.3pt);
\draw [fill=black] (-0.9,0.2) circle (0.3pt);
\draw [fill=black] (-1.15,0.3) circle (0.3pt);
\draw [fill=black] (1,-0.3) circle (0.3pt);
\draw [fill=black] (1.15,-0.3) circle (0.3pt);
\draw [fill=black] (0.9,-0.2) circle (0.3pt);
\draw [fill=black] (0,-0.6) circle (0.3pt);
\draw [fill=black] (0.2,-0.55) circle (0.3pt);
\draw [fill=black] (-0.2,-0.55) circle (0.3pt);
\draw [rotate=120][fill=black] (0,-0.6) circle (0.3pt);
\draw [rotate=120][fill=black] (0.15,-0.55) circle (0.3pt);
\draw [rotate=120][fill=black] (-0.15,-0.55) circle (0.3pt);
\draw [rotate around={-120:(-0.575,1)}][fill=black] (-0.28,1) circle (0.3pt);
\draw [rotate around={-120:(-0.575,1)}][fill=black] (-0.3,0.85) circle (0.3pt);
\draw [rotate around={-120:(-0.575,1)}][fill=black] (-0.3,1.15) circle (0.3pt);
\draw [fill=black] (-0.38,0.47) circle (0.3pt);
\draw [fill=black] (-0.44,0.43) circle (0.3pt);
\draw [fill=black] (-0.47,0.37) circle (0.3pt);
\draw [fill=black] (-0.56,0.22) circle (0.3pt);
\draw [fill=black] (-0.6,0.165) circle (0.3pt);
\draw [fill=black] (-0.6,0.1) circle (0.3pt);
\end{scriptsize}
\draw (0,0.25) node[anchor=north ] {$M_{\mathcal{B}}$};
\draw (-1.15,0.25) node[anchor=north ] {$\Phi$};
\draw (1.15,0.25) node[anchor=north ] {$\Phi$};
\draw (0.575,1.25) node[anchor=north ] {$\Phi$};
\draw (-0.575,1.25) node[anchor=north ] {$\Phi$};
\end{tikzpicture}
\end{equation}
This defines an element $s_{d+1}M_{\MA\oplus\MB^*}\in\mathcalboondox{B}^{\bullet}_{d,\Phi_0}(\MA,\MB)$ and we thus define $sm_{\MA\oplus\MB^*}$ of $ C(\MA\oplus\MB^*[d-1])[1]$ as $sm_{\MA\oplus\MB^*}^{\doubar{x}}=\mathcalboondox{j}_{\doubar{x}^{-1}}^{\Phi}(s_{d+1}M_{\MA\oplus\MB^*})\in C(\MA\oplus\MB^*[d-1])[1]$. We will denote by $m_{\MA\oplus\MB^*\to\MA}$ (resp. $m_{\MA\oplus\MB^*\to\MB^*}$)
    the composition of $m_{\MA\oplus\MB^*}$ with the canonical projection on $\MA$ (resp. on $\MB^*[d-1]$).
\end{definition}

\begin{proposition}
\label{prop:struct-general-case}
The element $sm_{\MA\oplus\MB^*}$ defines an $A_{\infty}$-structure on $\MA\oplus\MB^*[d-1]$.
Moreover, if the morphism $(\Phi_0,s_{d+1}\Phi)$ is good (see Definition \ref{def:good-morphism}), $sm_{\MA\oplus\MB^*}$ satisfies the cyclicity condition \eqref{eq:cyclicity-pcy}. 
\end{proposition}
\begin{proof}
Using Proposition \ref{prop:Lie-bracket-mixed}, it suffices to show that $s_{d+1}M_{\MA\oplus\MB^*}\upperset{\Phi \nec}{\circ}s_{d+1}M_{\MA\oplus\MB^*}=0$.
Moreover, $\pi_{\MA[1]}\circ (s_{d+1}M_{\MA\oplus\MB^*}\upperset{\Phi \nec}{\circ}s_{d+1}M_{\MA\oplus\MB^*})=0$
is tantamount to $\sum\mathcal{E}(\mathcalboondox{D})-\sum\mathcal{E}(\mathcalboondox{D'})+\sum\mathcal{E}(\mathcalboondox{D}\dprime)=0$ where the sums are over all the filled diagrams $\mathcalboondox{D}$, $\mathcalboondox{D'}$ and $\mathcalboondox{D}\dprime$ of type $\doubar{x}$ of the form

\begin{minipage}{14cm}
\begin{tikzpicture}[line cap=round,line join=round,x=1.0cm,y=1.0cm]
\clip(-1.45,-2.5) rectangle (3.3,3);
    \draw [line width=0.5pt] (0.,0.) circle (0.5cm);
     \shadedraw[rotate=150,shift={(0.5cm,0cm)}] \doublefleche;
     \draw[line width=1.1pt,rotate=90][->, >= stealth, >= stealth](0.5,0)--(0.9,0);
     \shadedraw[rotate=30,shift={(0.5cm,0cm)}] \doubleflechescindeeleft;
     \shadedraw[rotate=30,shift={(0.5cm,0cm)}] \doubleflechescindeeright;
     \shadedraw[rotate=30,shift={(0.5cm,0cm)}] \fleche;
      \draw [line width=0.5pt] (1.3,0.75) circle (0.5cm);
     \draw [line width=0.5pt] (1.03,-0.6) circle (0.3cm);
    \draw[rotate=-30][->, >= stealth, >= stealth](0.5,0)--(0.9,0);
     \draw [rotate around ={60:(1.03,-0.6)}] [->, >= stealth, >= stealth] (1.33,-0.6)--(1.63,-0.6);
     \draw [rotate around ={-120:(1.03,-0.6)}] [->, >= stealth, >= stealth] (1.33,-0.6)--(1.63,-0.6);
    \draw[rotate=-150][->, >= stealth, >= stealth](0.5,0)--(0.9,0);
    \draw [line width=0.5pt] (-1.03,-0.6) circle (0.3cm);
     \draw [rotate around ={-60:(-1.03,-0.6)}] [->, >= stealth, >= stealth] (-1.33,-0.6)--(-1.63,-0.6);
     \draw [rotate around ={120:(-1.03,-0.6)}] [->, >= stealth, >= stealth] (-1.33,-0.6)--(-1.63,-0.6);
     \draw [rotate around ={-30:(1.3,0.75)}] [->, >= stealth, >= stealth] (1.8,0.75)--(2.2,0.75);
     \draw [rotate around ={90:(1.3,0.75)}] [->, >= stealth, >= stealth] (1.8,0.75)--(2.2,0.75);
     \draw [rotate around ={30:(1.3,0.75)}, shift={(1.8,0.75)}]\doublefleche;
     \draw [rotate around ={150:(1.3,0.75)}, shift={(1.8,0.75)}]\doublefleche;
     \draw [line width=0.5pt] (2.33,0.15) circle (0.3cm);
      \draw [rotate around ={60:(2.33,0.15)}] [->, >= stealth, >= stealth] (2.63,0.15)--(2.93,0.15);
       \draw [rotate around ={-120:(2.33,0.15)}] [->, >= stealth, >= stealth] (2.63,0.15)--(2.93,0.15);
    \draw [line width=0.5pt] (1.3,1.95) circle (0.3cm);
      \draw [rotate around ={180:(1.3,1.95)}] [->, >= stealth, >= stealth] (1.6,1.95)--(1.9,1.95);
      \draw [rotate around ={0:(1.3,1.95)}] [->, >= stealth, >= stealth] (1.6,1.95)--(1.9,1.95);
     \shadedraw[rotate around={150:(2.33,0.15)},shift={(2.63cm,0.15cm)}] \doubleflechescindeeleft;
     \shadedraw[rotate around={150:(2.33,0.15)},shift={(2.63cm,0.15cm)}] \doubleflechescindeeright;
    \shadedraw[rotate around={-90:(1.3,1.95)},shift={(1.6,1.95)}] \doubleflechescindeeleft;
     \shadedraw[rotate around={-90:(1.3,1.95)},shift={(1.6,1.95)}] \doubleflechescindeeright;
    \shadedraw[rotate around={30:(-1.03,-0.6)},shift={(-0.73,-0.6)}] \doubleflechescindeeleft;
     \shadedraw[rotate around={30:(-1.03,-0.6)},shift={(-0.73,-0.6)}] \doubleflechescindeeright;
     \shadedraw[rotate around={150:(1.03,-0.6)},shift={(1.33,-0.6)}] \doubleflechescindeeleft;
     \shadedraw[rotate around={150:(1.03,-0.6)},shift={(1.33,-0.6)}] \doubleflechescindeeright;
\begin{scriptsize}
\draw [fill=black] (0,-0.6) circle (0.3pt);
\draw [fill=black] (0.2,-0.55) circle (0.3pt);
\draw [fill=black] (-0.2,-0.55) circle (0.3pt);
\draw [fill=black] (1.3,0.75-0.6) circle (0.3pt);
\draw [fill=black] (1.5,0.75-0.55) circle (0.3pt);
\draw [fill=black] (1.1,0.75-0.55) circle (0.3pt);
\draw [fill=black] (1.3,2.35) circle (0.3pt);
\draw [fill=black] (1.5,2.3) circle (0.3pt);
\draw [fill=black] (1.1,2.3) circle (0.3pt);
\draw [fill=black] (2.7,-0.02) circle (0.3pt);
\draw [fill=black] (2.73,0.18) circle (0.3pt);
\draw [fill=black] (2.55,-0.18) circle (0.3pt);
\draw [fill=black] (1.4,-0.6) circle (0.3pt);
\draw [fill=black] (1.35,-0.81) circle (0.3pt);
\draw [fill=black] (1.2,-0.95) circle (0.3pt);
\draw [fill=black] (-1.4,-0.6) circle (0.3pt);
\draw [fill=black] (-1.35,-0.81) circle (0.3pt);
\draw [fill=black] (-1.2,-0.95) circle (0.3pt);
\end{scriptsize}
\draw(0,0.25)node[anchor=north]{$M_{\MA}$};
\draw(1.3,1)node[anchor=north]{$M_{\MA}$};
\draw(2.33,0.4)node[anchor=north]{$\Phi$};
\draw(1.3,2.2)node[anchor=north]{$\Phi$};
\draw(1.03,-0.35)node[anchor=north]{$\Phi$};
\draw(-1.03,-0.35)node[anchor=north]{$\Phi$};
\draw(3,0.25)node[anchor=north]{,};
\end{tikzpicture}
\begin{tikzpicture}[line cap=round,line join=round,x=1.0cm,y=1.0cm]
\clip(-1.45,-3) rectangle (2.3,3);
  \draw [line width=0.5pt] (0.,0.) circle (0.3cm);
     \shadedraw[rotate=90,shift={(0.3cm,0cm)}] \doubleflechescindeeleft;
     \shadedraw[rotate=90,shift={(0.3cm,0cm)}] \doubleflechescindeeright;
     \shadedraw[rotate=90,shift={(0.3cm,0cm)}] \fleche;
     \shadedraw[rotate=-90,shift={(0.3cm,0cm)}] \doubleflechescindeeleft;
     \shadedraw[rotate=-90,shift={(0.3cm,0cm)}] \doubleflechescindeeright;
     \shadedraw[rotate=-90,shift={(0.3cm,0cm)}] \fleche;
     \draw[->, >= stealth,>=stealth](0.3,0)--(0.6,0);
     \draw [line width=0.5pt] (0,1.3) circle (0.5cm);
    \draw [rotate around ={-30:(0,1.3)}, shift={(0.5,1.3)}]\doublefleche;
     \draw [rotate around ={-150:(0,1.3)}, shift={(0.5,1.3)}]\doublefleche;
      \draw [line width=0.5pt] (-1.03,1.9) circle (0.3cm);
     \draw[rotate around={-30:(-1.03,1.9)}][<-, >= stealth,>=stealth](-0.73,1.9)--(-0.33,1.9);
    \shadedraw[rotate around={-30:(-1.03,1.9)},shift={(-0.73cm,1.9cm)}] \doubleflechescindeeleft;
     \shadedraw[rotate around={-30:(-1.03,1.9)},shift={(-0.73cm,1.9cm)}] \doubleflechescindeeright;
     \draw[rotate around={60:(-1.03,1.9)}][->, >= stealth,>=stealth](-0.73,1.9)--(-0.43,1.9);
     \draw[rotate around={-120:(-1.03,1.9)}][->, >= stealth,>=stealth](-0.73,1.9)--(-0.43,1.9);
     \draw [line width=0.5pt] (1.03,1.9) circle (0.3cm);
     \draw[rotate around={-150:(1.03,1.9)}][<-, >= stealth,>=stealth](1.33,1.9)--(1.73,1.9);
    \shadedraw[rotate around={-150:(1.03,1.9)},shift={(1.33cm,1.9cm)}] \doubleflechescindeeleft;
     \shadedraw[rotate around={-150:(1.03,1.9)},shift={(1.33cm,1.9cm)}] \doubleflechescindeeright;
     \draw[rotate around={-60:(1.03,1.9)}][->, >= stealth,>=stealth](1.33,1.9)--(1.63,1.9);
     \draw[rotate around={120:(1.03,1.9)}][->, >= stealth,>=stealth](1.33,1.9)--(1.63,1.9);
     \draw [line width=0.5pt] (0,-1.3) circle (0.5cm);
      \draw [rotate around ={30:(0,-1.3)}, shift={(0.5,-1.3)}]\doublefleche;
     \draw [rotate around ={150:(0,-1.3)}, shift={(0.5,-1.3)}]\doublefleche;
     \draw[line width=1.1pt,rotate around={-150:(0,-1.3)}][->, >= stealth,>=stealth](0.5,-1.3)--(0.9,-1.3);
      \draw [line width=0.5pt] (1.03,-1.9) circle (0.3cm);
     \draw[rotate around={150:(1.03,-1.9)}][<-, >= stealth,>=stealth](1.33,-1.9)--(1.73,-1.9);
    \shadedraw[rotate around={150:(1.03,-1.9)},shift={(1.33cm,-1.9cm)}] \doubleflechescindeeleft;
     \shadedraw[rotate around={150:(1.03,-1.9)},shift={(1.33cm,-1.9cm)}] \doubleflechescindeeright;
     \draw[rotate around={60:(1.03,-1.9)}][->, >= stealth,>=stealth](1.33,-1.9)--(1.63,-1.9);
     \draw[rotate around={-120:(1.03,-1.9)}][->, >= stealth,>=stealth](1.33,-1.9)--(1.63,-1.9);
    \begin{scriptsize}
\draw [fill=black] (0,1.9) circle (0.3pt);
\draw [fill=black] (0.2,1.85) circle (0.3pt);
\draw [fill=black] (-0.2,1.85) circle (0.3pt);
\draw [fill=black] (0,-1.9) circle (0.3pt);
\draw [fill=black] (0.2,-1.85) circle (0.3pt);
\draw [fill=black] (-0.2,-1.85) circle (0.3pt);
\draw [fill=black] (-0.4,0) circle (0.3pt);
\draw [fill=black] (-0.35,-0.2) circle (0.3pt);
\draw [fill=black] (-0.35,0.2) circle (0.3pt);
\draw [fill=black] (1.4,-2.1) circle (0.3pt);
\draw [fill=black] (1.42,-1.9) circle (0.3pt);
\draw [fill=black] (1.25,-2.25) circle (0.3pt);
\draw [fill=black] (1.4,2.1) circle (0.3pt);
\draw [fill=black] (1.42,1.9) circle (0.3pt);
\draw [fill=black] (1.25,2.25) circle (0.3pt);
\draw [fill=black] (-1.4,2.1) circle (0.3pt);
\draw [fill=black] (-1.42,1.9) circle (0.3pt);
\draw [fill=black] (-1.25,2.25) circle (0.3pt);
\end{scriptsize}
\draw(0,1.55)node[anchor=north]{$M_{\MA}$};
\draw(0,-1.05)node[anchor=north]{$M_{\MA}$};
\draw(0,0.25)node[anchor=north]{$\Phi$};
\draw(1.03,-1.65)node[anchor=north]{$\Phi$};
\draw(1.03,2.15)node[anchor=north]{$\Phi$};
\draw(-1.03,2.15)node[anchor=north]{$\Phi$};
\draw(1.6,0.25)node[anchor=north]{and};
\end{tikzpicture}
\begin{tikzpicture}[line cap=round,line join=round,x=1.0cm,y=1.0cm]
\clip(-4,-2.8) rectangle (4,3);
 \draw [line width=0.5pt] (0.,0.) circle (0.5cm);
     \draw [rotate=90] [->, >= stealth, >= stealth] (0.5,0)--(0.9,0);
     \draw [rotate=-30] [->, >= stealth, >= stealth] (0.5,0)--(0.9,0);
     \draw [rotate=-150] [->, >= stealth, >= stealth] (0.5,0)--(0.9,0);
     \draw [rotate=0] [<-, >= stealth, >= stealth] (0.5,0)--(0.9,0);
     \draw [line width=0.5pt] (1.15,0.) circle (0.25cm);
     \draw[rotate around={0:(1.15,0)}] [->, >= stealth, >= stealth] (1.4,0)--(1.8,0);
     \draw[rotate around={90:(1.15,0)},shift={(1.4,0)}] \doublefleche;
     \draw [rotate=60] [<-, >= stealth, >= stealth] (0.5,0)--(0.9,0);
       \draw [rotate=180] [<-, >= stealth, >= stealth] (0.5,0)--(0.9,0);
      \draw [line width=0.5pt] (-1.15,0.) circle (0.25cm);
     \draw[rotate around={0:(-1.15,0)}] [->, >= stealth, >= stealth] (-1.4,0)--(-1.8,0);
     \draw[rotate around={-90:(-1.15,0)},shift={(-0.9,0)}]\doublefleche;
        \draw [rotate=120] [<-, >= stealth, >= stealth] (0.5,0)--(0.9,0);
        
        \draw [line width=0.5pt] (0.575,1) circle (0.25cm);
        \draw[rotate around={120:(0.575,1)}] [->, >= stealth, >= stealth] (0.825,1)--(1.225,1);
        \draw[rotate around={0:(0.575,1)}] [->, >= stealth, >= stealth] (0.825,1)--(1.225,1);
        \draw[rotate around={60:(0.575,1)},shift={(0.825,1)}] \doublefleche;
        \draw[rotate around={-60:(0.575,1)},shift={(0.825,1)}] \doublefleche;
        \draw [line width=0.5pt] (-0.575,1) circle (0.25cm);
      \draw[rotate around={300:(-0.575,1)}] [->, >= stealth, >= stealth] (-0.825,1)--(-1.225,1);
      \draw[rotate around={30:(-0.575,1)},shift={(-0.325,1)}] \doublefleche;
      \draw[rotate around={150:(0,0)}] [<-, >= stealth, >= stealth] (0.5,0)--(1.3,0);
    \draw [line width=0.5pt] (-1.34,0.775) circle (0.25cm);
    \draw[rotate around={150:(-1.34,0.775)}] [<-, >= stealth, >= stealth] (-1.09,0.775)--(-0.69,0.775);
    \draw[rotate around={90:(-1.34,0.775)}] [->, >= stealth, >= stealth] (-1.09,0.775)--(-0.71,0.775);
    \draw [rotate around={-150:(-1.34,0.775)}] [->, >= stealth, >= stealth] (-1.09,0.775)--(-0.71,0.775);
    \draw[rotate around={30:(-1.34,0.775)},shift={(-1.09,0.775)}]\doublefleche;
    \draw[rotate around={150:(-1.34,0.775)},shift={(-1.09,0.775)}]\doubleflechescindeeleft;
    \draw[rotate around={150:(-1.34,0.775)},shift={(-1.09,0.775)}]\doubleflechescindeeright;
     \draw [line width=0.5pt] (-2.34,1.35) circle (0.5cm);
     \shadedraw[rotate around={30:(-2.34,1.35) },shift={(-1.84,1.35)}]\doublefleche;
     \shadedraw[rotate around={150:(-2.34,1.35) },shift={(-1.84,1.35) }]\doublefleche;
     \draw[rotate around={-150:(-2.34,1.35) }] [->, >= stealth, >= stealth] (-1.84,1.35) --(-1.44,1.35);
      \draw [line width=0.5pt] (-2.34-1.03,1.35-0.6) circle (0.3cm);
     \draw[rotate around={90:(-2.34-1.03,1.35-0.6)}] [->, >= stealth, >= stealth] (-2.34-1.03+0.3,1.35-0.6)--(-2.34-1.03+0.6,1.35-0.6);
     \draw[rotate around={-30:(-2.34-1.03,1.35-0.6)}] [->, >= stealth, >= stealth] (-2.34-1.03+0.3,1.35-0.6)--(-2.34-1.03+0.6,1.35-0.6);
     \draw[rotate around={-150:(-2.34-1.03,1.35-0.6)}] [->, >= stealth, >= stealth] (-2.34-1.03+0.3,1.35-0.6)--(-2.34-1.03+0.6,1.35-0.6);
     \draw[rotate around={-90:(-2.34-1.03,1.35-0.6)},shift={(-2.34-1.03+0.3,1.35-0.6)}] \doublefleche;
      \draw[rotate around={30:(-2.34-1.03,1.35-0.6)},shift={(-2.34-1.03+0.3,1.35-0.6)}] \doubleflechescindeeleft;
      \draw[rotate around={30:(-2.34-1.03,1.35-0.6)},shift={(-2.34-1.03+0.3,1.35-0.6)}] \doubleflechescindeeright;
     \draw[line width=1.1pt,rotate around={90:(-2.34,1.35) }][->, >= stealth,>=stealth](-1.84,1.35) --(-1.44,1.35);
\begin{scriptsize}
\draw [fill=black] (0,-0.6) circle (0.3pt);
\draw [fill=black] (0.2,-0.55) circle (0.3pt);
\draw [fill=black] (-0.2,-0.55) circle (0.3pt);
\draw [rotate=120][fill=black] (0,-0.6) circle (0.3pt);
\draw [rotate=120][fill=black] (0.15,-0.55) circle (0.3pt);
\draw [rotate=120][fill=black] (-0.15,-0.55) circle (0.3pt);
\draw [fill=black] (-0.38,0.47) circle (0.3pt);
\draw [fill=black] (-0.44,0.43) circle (0.3pt);
\draw [fill=black] (-0.47,0.37) circle (0.3pt);
\draw [fill=black] (-0.56,0.22) circle (0.3pt);
\draw [fill=black] (-0.6,0.165) circle (0.3pt);
\draw [fill=black] (-0.6,0.1) circle (0.3pt);
\draw [fill=black] (-2.34,0.75) circle (0.3pt);
\draw [fill=black] (-2.14,0.8) circle (0.3pt);
\draw [fill=black] (-2.54,0.8) circle (0.3pt);
\draw [rotate around={180:(-1.15,0)}][fill=black] (-1.15,-0.35) circle (0.3pt);
\draw [rotate around={180:(-1.15,0)}][fill=black] (-1,-0.3) circle (0.3pt);
\draw [rotate around={180:(-1.15,0)}][fill=black] (-1.3,-0.3) circle (0.3pt);
\draw (1.15,-0.35) circle (0.3pt);
\draw (1,-0.3) circle (0.3pt);
\draw (1.3,-0.3) circle (0.3pt);
\draw [fill=black] (-3.6,1.1) circle (0.3pt);
\draw [fill=black] (-3.72,0.97) circle (0.3pt);
\draw [fill=black] (-3.77,0.8) circle (0.3pt);
\draw [fill=black] (-1.44,0.47) circle (0.3pt);
\draw [fill=black] (-1.32,0.45) circle (0.3pt);
\draw [fill=black] (-1.18,0.47) circle (0.3pt);
\draw [fill=black] (-0.85,0.85) circle (0.3pt);
\draw [fill=black] (-0.88,1.05) circle (0.3pt);
\draw [fill=black] (-0.7,0.72) circle (0.3pt);
\draw [rotate=-60][fill=black] (-0.85,0.85) circle (0.3pt);
\draw [rotate=-60] [fill=black] (-0.6,0.7) circle (0.3pt);
\draw [rotate=-60] [fill=black] (-0.75,0.75) circle (0.3pt);
\end{scriptsize}
\draw(0,0.25)node[anchor=north]{$M_{\MB}$};
\draw(-2.34,1.6)node[anchor=north]{$M_{\MA}$};
\draw(-1.34,0.775+0.25)node[anchor=north]{$\Phi$};
\draw(0.575,1.25)node[anchor=north]{$\Phi$};
\draw(-0.575,1.25)node[anchor=north]{$\Phi$};
\draw(1.15,0.25)node[anchor=north]{$\Phi$};
\draw(-1.15,0.25)node[anchor=north]{$\Phi$};
\draw(-2.34-1.03,1.35-0.6+0.25)node[anchor=north]{$\Phi$};
\end{tikzpicture}
\end{minipage}

\noindent
respectively.
Using that $(\Phi_0,s_{d+1}\Phi)$ is a pre-Calabi-Yau morphism, we have that 
\begin{equation}
\label{eq:diagram-D-dprime}
\sum\mathcal{E}(\mathcalboondox{D}\dprime)=\sum\mathcal{E}(\mathcalboondox{D_1})+\sum\mathcal{E}(\mathcalboondox{D_2})
\end{equation}
where $\mathcalboondox{D_1}$ and $\mathcalboondox{D_2}$ are filled diagrams of the form 

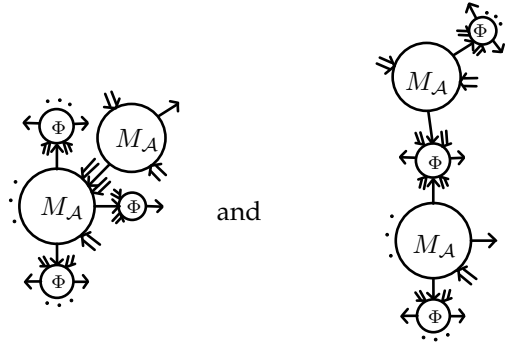
\begin{minipage}{10cm}
    \begin{equation}
        \label{eq:third-stasheff}
\begin{tikzpicture}[line cap=round,line join=round,x=1.0cm,y=1.0cm]
\clip(-2.7,-2) rectangle (5,1.8);
   \draw [line width=0.5pt] (0.,0.) circle (0.5cm);
     \shadedraw[rotate=150,shift={(0.5cm,0cm)}] \doublefleche;
      \draw [line width=0.5pt] (0.,1.3) circle (0.3cm);
     \shadedraw[rotate around={-90:(0,1.3)},shift={(0.3cm,1.3cm)}] \doubleflechescindeeleft;
     \shadedraw[rotate around={-90:(0,1.3)},shift={(0.3cm,1.3cm)}] \doubleflechescindeeright;
     \shadedraw[rotate around={-90:(0,1.3)},shift={(0.3cm,1.3cm)}] \fleche;
      \draw[rotate around={180:(0,1.3)}][->, >= stealth, >= stealth](0.3,1.3)--(0.6,1.3);
      \draw[rotate around={0:(0,1.3)}][->, >= stealth, >= stealth](0.3,1.3)--(0.6,1.3);
     \shadedraw[rotate=30,shift={(0.5cm,0cm)}] \doubleflechescindeeleft;
     \shadedraw[rotate=30,shift={(0.5cm,0cm)}] \doubleflechescindeeright;
     \shadedraw[rotate=30,shift={(0.5cm,0cm)}] \fleche;
      \draw [line width=0.5pt] (1.3,0.75) circle (0.5cm);
     \draw [line width=0.5pt] (1.03,-0.6) circle (0.3cm);
    \draw[rotate=-30][->, >= stealth, >= stealth](0.5,0)--(0.9,0);
     \draw [rotate around ={60:(1.03,-0.6)}] [->, >= stealth, >= stealth] (1.33,-0.6)--(1.63,-0.6);
     \draw [rotate around ={-120:(1.03,-0.6)}] [->, >= stealth, >= stealth] (1.33,-0.6)--(1.63,-0.6);
    \draw[rotate=-150][->, >= stealth, >= stealth](0.5,0)--(0.9,0);
    \draw [line width=0.5pt] (-1.03,-0.6) circle (0.3cm);
     \draw [rotate around ={-60:(-1.03,-0.6)}] [->, >= stealth, >= stealth] (-1.33,-0.6)--(-1.63,-0.6);
     \draw [rotate around ={120:(-1.03,-0.6)}] [->, >= stealth, >= stealth] (-1.33,-0.6)--(-1.63,-0.6);
     \draw [rotate around ={-30:(1.3,0.75)}] [->, >= stealth, >= stealth] (1.8,0.75)--(2.2,0.75);
     \draw [line width=1.1pt,rotate around ={90:(1.3,0.75)}] [->, >= stealth, >= stealth] (1.8,0.75)--(2.2,0.75);
     \draw [rotate around ={30:(1.3,0.75)}, shift={(1.8,0.75)}]\doublefleche;
     \draw [rotate around ={150:(1.3,0.75)}, shift={(1.8,0.75)}]\doublefleche;
     \draw [line width=0.5pt] (2.33,0.15) circle (0.3cm);
      \draw [rotate around ={60:(2.33,0.15)}] [->, >= stealth, >= stealth] (2.63,0.15)--(2.93,0.15);
       \draw [rotate around ={-120:(2.33,0.15)}] [->, >= stealth, >= stealth] (2.63,0.15)--(2.93,0.15);
     \shadedraw[rotate around={150:(2.33,0.15)},shift={(2.63cm,0.15cm)}] \doubleflechescindeeleft;
     \shadedraw[rotate around={150:(2.33,0.15)},shift={(2.63cm,0.15cm)}] \doubleflechescindeeright;
    \shadedraw[rotate around={30:(-1.03,-0.6)},shift={(-0.73,-0.6)}] \doubleflechescindeeleft;
     \shadedraw[rotate around={30:(-1.03,-0.6)},shift={(-0.73,-0.6)}] \doubleflechescindeeright;
     \shadedraw[rotate around={150:(1.03,-0.6)},shift={(1.33,-0.6)}] \doubleflechescindeeleft;
     \shadedraw[rotate around={150:(1.03,-0.6)},shift={(1.33,-0.6)}] \doubleflechescindeeright;
\begin{scriptsize}
\draw [fill=black] (0,-0.6) circle (0.3pt);
\draw [fill=black] (0.2,-0.55) circle (0.3pt);
\draw [fill=black] (-0.2,-0.55) circle (0.3pt);
\draw [fill=black] (1.3,0.75-0.6) circle (0.3pt);
\draw [fill=black] (1.5,0.75-0.55) circle (0.3pt);
\draw [fill=black] (1.1,0.75-0.55) circle (0.3pt);
\draw [fill=black] (0,1.7) circle (0.3pt);
\draw [fill=black] (0.2,1.65) circle (0.3pt);
\draw [fill=black] (-0.2,1.65) circle (0.3pt);
\draw [fill=black] (2.7,-0.02) circle (0.3pt);
\draw [fill=black] (2.73,0.18) circle (0.3pt);
\draw [fill=black] (2.55,-0.18) circle (0.3pt);
\draw [fill=black] (1.4,-0.6) circle (0.3pt);
\draw [fill=black] (1.35,-0.81) circle (0.3pt);
\draw [fill=black] (1.2,-0.95) circle (0.3pt);
\draw [fill=black] (-1.4,-0.6) circle (0.3pt);
\draw [fill=black] (-1.35,-0.81) circle (0.3pt);
\draw [fill=black] (-1.2,-0.95) circle (0.3pt);
\end{scriptsize}
\draw(4,0.25)node[anchor=north]{and};
\draw(0,0.25)node[anchor=north]{$M_{\MA}$};
\draw(1.3,1)node[anchor=north]{$M_{\MA}$};
\draw(2.33,0.4)node[anchor=north]{$\Phi$};
\draw(0,1.55)node[anchor=north]{$\Phi$};
\draw(1.03,-0.35)node[anchor=north]{$\Phi$};
\draw(-1.03,-0.35)node[anchor=north]{$\Phi$};
\end{tikzpicture}
\begin{tikzpicture}[line cap=round,line join=round,x=1.0cm,y=1.0cm]
\clip(-1.6,-2.4) rectangle (2.5,2.6);
  \draw [line width=0.5pt] (0.,0.) circle (0.3cm);
     \shadedraw[rotate=90,shift={(0.3cm,0cm)}] \doubleflechescindeeleft;
     \shadedraw[rotate=90,shift={(0.3cm,0cm)}] \doubleflechescindeeright;
     \shadedraw[rotate=90,shift={(0.3cm,0cm)}] \fleche;
     \shadedraw[rotate=-90,shift={(0.3cm,0cm)}] \doubleflechescindeeleft;
     \shadedraw[rotate=-90,shift={(0.3cm,0cm)}] \doubleflechescindeeright;
     \shadedraw[rotate=-90,shift={(0.3cm,0cm)}] \fleche;
     \draw[->, >= stealth,>=stealth](0.3,0)--(0.6,0);
     \draw [line width=0.5pt] (0,1.3) circle (0.5cm);
    \draw [rotate around ={-30:(0,1.3)}, shift={(0.5,1.3)}]\doublefleche;
     \draw [rotate around ={-150:(0,1.3)}, shift={(0.5,1.3)}]\doublefleche;
      \draw [line width=0.5pt] (-1.03,1.9) circle (0.3cm);
     \draw[rotate around={-30:(-1.03,1.9)}][<-, >= stealth,>=stealth](-0.73,1.9)--(-0.33,1.9);
    \shadedraw[rotate around={-30:(-1.03,1.9)},shift={(-0.73cm,1.9cm)}] \doubleflechescindeeleft;
     \shadedraw[rotate around={-30:(-1.03,1.9)},shift={(-0.73cm,1.9cm)}] \doubleflechescindeeright;
     \draw[rotate around={60:(-1.03,1.9)}][->, >= stealth,>=stealth](-0.73,1.9)--(-0.43,1.9);
     \draw[rotate around={-120:(-1.03,1.9)}][->, >= stealth,>=stealth](-0.73,1.9)--(-0.43,1.9);
     \draw [line width=0.5pt] (1.03,1.9) circle (0.3cm);
     \draw[rotate around={-150:(1.03,1.9)}][<-, >= stealth,>=stealth](1.33,1.9)--(1.73,1.9);
    \shadedraw[rotate around={-150:(1.03,1.9)},shift={(1.33cm,1.9cm)}] \doubleflechescindeeleft;
     \shadedraw[rotate around={-150:(1.03,1.9)},shift={(1.33cm,1.9cm)}] \doubleflechescindeeright;
     \draw[rotate around={-60:(1.03,1.9)}][->, >= stealth,>=stealth](1.33,1.9)--(1.63,1.9);
     \draw[rotate around={120:(1.03,1.9)}][->, >= stealth,>=stealth](1.33,1.9)--(1.63,1.9);
     \draw [line width=0.5pt] (0,-1.3) circle (0.5cm);
      \draw [rotate around ={30:(0,-1.3)}, shift={(0.5,-1.3)}]\doublefleche;
     \draw [rotate around ={150:(0,-1.3)}, shift={(0.5,-1.3)}]\doublefleche;
     \draw[line width=1.1pt,rotate around={-150:(0,-1.3)}][->, >= stealth,>=stealth](0.5,-1.3)--(0.9,-1.3);
      \draw [line width=0.5pt] (1.03,-1.9) circle (0.3cm);
     \draw[rotate around={150:(1.03,-1.9)}][<-, >= stealth,>=stealth](1.33,-1.9)--(1.73,-1.9);
    \shadedraw[rotate around={150:(1.03,-1.9)},shift={(1.33cm,-1.9cm)}] \doubleflechescindeeleft;
     \shadedraw[rotate around={150:(1.03,-1.9)},shift={(1.33cm,-1.9cm)}] \doubleflechescindeeright;
     \draw[rotate around={60:(1.03,-1.9)}][->, >= stealth,>=stealth](1.33,-1.9)--(1.63,-1.9);
     \draw[rotate around={-120:(1.03,-1.9)}][->, >= stealth,>=stealth](1.33,-1.9)--(1.63,-1.9);
    \begin{scriptsize}
\draw [fill=black] (0,1.9) circle (0.3pt);
\draw [fill=black] (0.2,1.85) circle (0.3pt);
\draw [fill=black] (-0.2,1.85) circle (0.3pt);
\draw [fill=black] (0,-1.9) circle (0.3pt);
\draw [fill=black] (0.2,-1.85) circle (0.3pt);
\draw [fill=black] (-0.2,-1.85) circle (0.3pt);
\draw [fill=black] (-0.4,0) circle (0.3pt);
\draw [fill=black] (-0.35,-0.2) circle (0.3pt);
\draw [fill=black] (-0.35,0.2) circle (0.3pt);
\draw [fill=black] (1.4,-2.1) circle (0.3pt);
\draw [fill=black] (1.42,-1.9) circle (0.3pt);
\draw [fill=black] (1.25,-2.25) circle (0.3pt);
\draw [fill=black] (1.4,2.1) circle (0.3pt);
\draw [fill=black] (1.42,1.9) circle (0.3pt);
\draw [fill=black] (1.25,2.25) circle (0.3pt);
\draw [fill=black] (-1.4,2.1) circle (0.3pt);
\draw [fill=black] (-1.42,1.9) circle (0.3pt);
\draw [fill=black] (-1.25,2.25) circle (0.3pt);
\end{scriptsize}
\draw(0,1.55)node[anchor=north]{$M_{\MA}$};
\draw(0,-1.05)node[anchor=north]{$M_{\MA}$};
\draw(0,0.25)node[anchor=north]{$\Phi$};
\draw(1.03,-1.65)node[anchor=north]{$\Phi$};
\draw(1.03,2.15)node[anchor=north]{$\Phi$};
\draw(-1.03,2.15)node[anchor=north]{$\Phi$};
\end{tikzpicture}
    \end{equation}
\end{minipage}

The minus in the identity \eqref{eq:diagram-D-dprime} comes from the fact that the discs filled with $M_{\MA}$ change their place, in the sense that the order of the labeling of their first outgoing arrow changes. Since $s_{d+1}M_{\MA}$ is an element
of degree $1$, this creates a minus sign.
Moreover, $\sum\mathcal{E}(\mathcalboondox{D'})=\sum\mathcal{E}(\mathcalboondox{D_2})$ so that it remains to show that $\sum\mathcal{E}(\mathcalboondox{D})+\sum\mathcal{E}(\mathcalboondox{D_1})=0$.
This is the case since $s_{d+1}M_{\MA}$ is a $d$-pre-Calabi-Yau structure. Indeed, the sum of these evaluations of diagrams is the composition of \[s_{d+1}M_{\MA}\upperset{\nec}{\circ}s_{d+1}M_{\MA}\] with a tensor product composed of maps of the collection $\Phi$ and of the identity map $\id$ in the last tensor factor. We can check in a similar manner that \[\pi_{\MB^*[d-1]}\circ(s_{d+1}M_{\MA\oplus\MB^*}\upperset{\Phi \nec}{\circ}s_{d+1}M_{\MA\oplus\MB^*})=0. \]
Therefore, the element $sm_{\MA\oplus\MB^*}$ satisfies the Stasheff identities for every $\Bar{x}\in\Bar{\MO}_{\MA}$.

It is clear that if the morphism $(\Phi_0,s_{d+1}\Phi)$ is good, then the $A_{\infty}$-structure on $\MA\oplus\MB^*[d-1]$ is almost cyclic with respect to $\Gamma^{\Phi}$.
Indeed, $\Gamma^{\Phi}\circ (sm_{\MA\oplus\MB^*\to \MA}\otimes \id_{\MA^*})=\sum\mathcal{E}(\mathcalboondox{D})$ where the sum is over all the filled diagrams $\mathcal{E}(\mathcalboondox{D})$ of the form

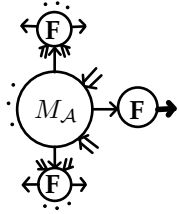
\begin{tikzpicture}[line cap=round,line join=round,x=1.0cm,y=1.0cm]
\clip(-7,-2) rectangle (4,2);
   \draw [line width=0.5pt] (0.,0.) circle (0.5cm);
     \shadedraw[rotate=30,shift={(0.5cm,0cm)}] \doublefleche;
     \shadedraw[rotate=150,shift={(0.5cm,0cm)}] \doublefleche;
     \draw[rotate=90][->, >= stealth, >= stealth](0.5,0)--(0.9,0);
     \draw [line width=0.5pt] (1.12,-0.65) circle (0.3cm);
     \shadedraw[shift={(0.86cm,-0.5cm)},rotate=150] \doubleflechescindeeleft;
     \shadedraw[shift={(0.86cm,-0.5cm)},rotate=150] \doubleflechescindeeright;
     \shadedraw[shift={(0.86cm,-0.5cm)},rotate=150] \fleche;
     \draw [rotate around ={60:(1.12,-0.65)}] [->, >= stealth, >= stealth] (1.43,-0.65)--(1.73,-0.65);
     \draw [rotate around ={-120:(1.12,-0.65)}] [->, >= stealth, >= stealth] (1.43,-0.65)--(1.73,-0.65);
     \draw [line width=0.5pt] (-1.12,-0.65) circle (0.3cm);
     \shadedraw[shift={(-0.86cm,-0.5cm)},rotate=30] \doubleflechescindeeleft;
      \shadedraw[shift={(-0.86cm,-0.5cm)},rotate=30] \doubleflechescindeeright;
       \shadedraw[shift={(-0.86cm,-0.5cm)},rotate=30] \fleche;
      \draw [rotate around ={-60:(-1.12,-0.65)}] [->, >= stealth, >= stealth] (-1.43,-0.65)--(-1.73,-0.65);
     \draw [rotate around ={120:(-1.12,-0.65)}] [->, >= stealth, >= stealth] (-1.43,-0.65)--(-1.73,-0.65);
     \draw [line width=0.5pt] (0.,1.2) circle (0.3cm);
     \draw [line width=1.1pt,rotate around={90:(0.,1.2)}][->, >= stealth, >=stealth](0.3,1.2)--(0.6,1.2);
\begin{scriptsize}
\draw [fill=black] (0,-0.6) circle (0.3pt);
\draw [fill=black] (0.2,-0.55) circle (0.3pt);
\draw [fill=black] (-0.2,-0.55) circle (0.3pt);
\draw [fill=black] (1.45,-0.85) circle (0.3pt);
\draw [fill=black] (1.5,-0.67) circle (0.3pt);
\draw [fill=black] (1.33,-0.97) circle (0.3pt);
\draw [fill=black] (-1.45,-0.85) circle (0.3pt);
\draw [fill=black] (-1.5,-0.67) circle (0.3pt);
\draw [fill=black] (-1.33,-0.97) circle (0.3pt);
\end{scriptsize}
\draw(0,0.25)node[anchor=north]{$M_{\MA}$};
\draw(1.12,-0.4)node[anchor=north]{$\Phi$};
\draw(-1.12,-0.4)node[anchor=north]{$\Phi$};
\draw(0,1.45)node[anchor=north]{$\Phi$};
\end{tikzpicture}

\noindent
On the other hand, $\Gamma^{\Phi}\circ (sm_{\MA\oplus\MB^*\to \MB^*}\otimes \id_{\MB})=\sum\mathcal{E}(\mathcalboondox{D'})$ where the sum is over all the filled diagrams $\mathcal{E}(\mathcalboondox{D'})$ of the form 

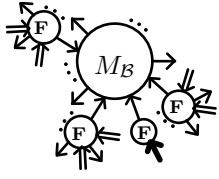
\begin{tikzpicture}[line cap=round,line join=round,x=1.0cm,y=1.0cm]
\clip(-7,-1) rectangle (7.635492124039547,2);
  \draw [line width=0.5pt] (0.,0.) circle (0.5cm);
     \draw [rotate=90] [->, >= stealth, >= stealth] (0.5,0)--(0.9,0);
     \draw [rotate=-30] [->, >= stealth, >= stealth] (0.5,0)--(0.9,0);
     \draw [rotate=-150] [->, >= stealth, >= stealth] (0.5,0)--(0.9,0);
     \draw [rotate=0] [<-, >= stealth, >= stealth] (0.5,0)--(0.9,0);
     \draw [line width=0.5pt] (1.15,0.) circle (0.25cm);
     \draw[rotate around={60:(1.15,0)}] [->, >= stealth, >= stealth] (1.4,0)--(1.7,0);
      \draw[rotate around={-60:(1.15,0)}] [->, >= stealth, >= stealth] (1.4,0)--(1.7,0);
      \shadedraw[rotate around={120:(1.15,0)}, shift={(1.4cm,0cm)}] \doublefleche;
       \shadedraw[shift={(1.4cm,0cm)}] \doublefleche;
     \draw [rotate=60] [<-, >= stealth, >= stealth] (0.5,0)--(0.9,0);
       \draw [rotate=180] [<-, >= stealth, >= stealth] (0.5,0)--(0.9,0);
      \draw [line width=0.5pt] (-1.15,0.) circle (0.25cm);
     \draw[rotate around={-60:(-1.15,0)}] [->, >= stealth, >= stealth] (-1.4,0)--(-1.7,0);
      \draw[rotate around={60:(-1.15,0)}] [->, >= stealth, >= stealth] (-1.4,0)--(-1.7,0);
      \shadedraw[rotate around={-60:(-1.15,0)}, shift={(-0.9cm,0cm)}] \doublefleche;
       \shadedraw[rotate around={180:(-1.15,0)},shift={(-0.9cm,0cm)}] \doublefleche;
        \draw [rotate=120] [<-, >= stealth, >= stealth] (0.5,0)--(0.9,0);
        \draw [line width=0.5pt] (0.575,1) circle (0.25cm);
        \draw[rotate around={120:(0.575,1)}] [->, >= stealth, >= stealth] (0.825,1)--(1.125,1);
      \draw[rotate around={0:(0.575,1)}] [->, >= stealth, >= stealth] (0.825,1)--(1.125,1);
      \shadedraw[rotate around={60:(0.575,1)}, shift={(0.825cm,1cm)}] \doublefleche;
       \shadedraw[rotate around={180:(0.575,1)},shift={(0.825cm,1cm)}] \doublefleche;
        \draw [line width=0.5pt] (-0.575,1) circle (0.25cm);
      \draw[rotate around={-120:(-0.575,1)}] [->, >= stealth, >= stealth] (-0.825,1)--(-1.125,1);
      \draw[rotate around={0:(-0.575,1)}] [->, >= stealth, >= stealth] (-0.825,1)--(-1.125,1);
      \shadedraw[rotate around={120:(-0.575,1)}, shift={(-0.325cm,1cm)}] \doublefleche;
       \shadedraw[rotate around={0:(-0.575,1)},shift={(-0.325cm,1cm)}] \doublefleche;
      \draw[rotate around={150:(0,0)}] [<-, >= stealth, >= stealth] (0.5,0)--(0.9,0);
      \draw [line width=0.5pt] (-1,0.575) circle (0.25cm);
      \draw[line width=1.1pt,rotate around={150:(-1,0.575)}] [<-, >= stealth, >= stealth] (-0.75,0.575)--(-0.45,0.575);
\begin{scriptsize}
\draw [fill=black] (0.65,0.7) circle (0.3pt);
\draw [fill=black] (0.79,0.75) circle (0.3pt);
\draw [fill=black] (0.88,0.85) circle (0.3pt);
\draw [fill=black] (-1,0.27) circle (0.3pt);
\draw [fill=black] (-0.9,0.2) circle (0.3pt);
\draw [fill=black] (-1.15,0.3) circle (0.3pt);
\draw [fill=black] (1,-0.3) circle (0.3pt);
\draw [fill=black] (1.15,-0.3) circle (0.3pt);
\draw [fill=black] (0.9,-0.2) circle (0.3pt);
\draw [fill=black] (0,-0.6) circle (0.3pt);
\draw [fill=black] (0.2,-0.55) circle (0.3pt);
\draw [fill=black] (-0.2,-0.55) circle (0.3pt);
\draw [rotate=120][fill=black] (0,-0.6) circle (0.3pt);
\draw [rotate=120][fill=black] (0.15,-0.55) circle (0.3pt);
\draw [rotate=120][fill=black] (-0.15,-0.55) circle (0.3pt);
\draw [rotate around={-120:(-0.575,1)}][fill=black] (-0.28,1) circle (0.3pt);
\draw [rotate around={-120:(-0.575,1)}][fill=black] (-0.3,0.85) circle (0.3pt);
\draw [rotate around={-120:(-0.575,1)}][fill=black] (-0.3,1.15) circle (0.3pt);
\draw [fill=black] (-0.38,0.47) circle (0.3pt);
\draw [fill=black] (-0.44,0.43) circle (0.3pt);
\draw [fill=black] (-0.47,0.37) circle (0.3pt);
\draw [fill=black] (-0.56,0.22) circle (0.3pt);
\draw [fill=black] (-0.6,0.165) circle (0.3pt);
\draw [fill=black] (-0.6,0.1) circle (0.3pt);
\end{scriptsize}
\draw(0,0.25)node[anchor=north]{$M_{\MB}$};
\draw(-0.575,1.25)node[anchor=north]{$\Phi$};
\draw(0.575,1.25)node[anchor=north]{$\Phi$};
\draw(-1,0.575+0.25)node[anchor=north]{$\Phi$};
\draw(-1.15,0.25)node[anchor=north]{$\Phi$};
\draw(1.15,0.25)node[anchor=north]{$\Phi$};
\end{tikzpicture}
\noindent
which shows that $sm_{\MA\oplus\MB^*}$ is an almost cyclic $A_{\infty}$-structure.
\end{proof}

Now, the aim is to define two morphisms of $A_{\infty}$-categories 
\begin{equation}
    s\varphi_{\MA} : \MA[1]\oplus\MB^{*}[d]\rightarrow \MA[1]\oplus\MA^{*}[d] \text{ and }s\varphi_{\MB} : \MA[1]\oplus\MB^{*}[d] \rightarrow \MB[1]\oplus\MB^{*}[d]
\end{equation}
using the $d$-pre-Calabi-Yau morphism $(\Phi_0,s_{d+1}\Phi)$. 

\begin{definition}
    Consider $d$-pre-Calabi-Yau categories $(\MA,s_{d+1}M_{\MA})$ and $(\MB,s_{d+1}M_{\MB})$ as well as a $d$-pre-Calabi-Yau morphism $(\Phi_0,s_{d+1}\Phi) :(\MA,s_{d+1}M_{\MA})\rightarrow(\MB,s_{d+1}M_{\MB})$.
    Then, $s_{d+1}\Phi$ induces maps
    \begin{equation}
\label{eq:phi-MA*}
\begin{split}
    \phi^{\doubar{x}}_{\MA^*} : \bigotimes\limits_{i=1}^{n-1}(\MA[1]^{\bar{x}^i}\otimes {}_{\Phi_0(\rrt(\bar{x}^{i+1}))}\MB^*_{\Phi_0(\llt(\bar{x}^{i}))}[d])\otimes \MA[1]^{\bar{x}^n} \rightarrow {}_{\llt(\bar{x}^1)}\MA^*_{\rrt(\bar{x}^n)}[d]
    \end{split}
\end{equation} 
defined by 
\begin{small}
\begin{equation}
    \begin{split}
        \phi^{\doubar{x}}_{\MA^*}&(\bar{sa}^1,tf^1,\dots,\bar{sa}^{n-1},tf^{n-1},\bar{sa}^n)(s_{-d}b)=\\&(-1)^{\epsilon}\big((f^{n-1}\circ s_d)\otimes\dots\otimes(f^{1}\circ s_d)\big)\big(\Phi^{\bar{x}^n\sqcup\bar{x}^1,\bar{x}^{n-1},\dots,\bar{x}^2}(\bar{sa}^n\otimes sb\otimes \bar{sa}^1,\bar{sa}^{n-1},\dots,\bar{sa}^2)\big)
    \end{split}
\end{equation}
\end{small}
for $\bar{sa}^i\in \MA[1]^{\otimes \bar{x}^i}$, $tf^i\in {}_{\Phi_0(\rrt(\bar{x}^{i+1}))}\MB^*_{\Phi_0(\llt(\bar{x}^{i}))}[d]$ and $sb\in {}_{\rrt(\bar{x}^n)}\MA_{\llt(\bar{x}^1)}[1]$ where 
\begin{equation}
    \begin{split}
        \epsilon&=\sum\limits_{i=1}^{n-1}|tf^i|\sum\limits_{j=i+1}^n|\bar{sa}^i|+\sum\limits_{1\leq i<j\leq n}|\bar{sa}^i||\bar{sa}^j|\sum\limits_{1\leq i<j\leq n-1}|tf^i||tf^j|+d(n-1)\\&+(|\bar{sa}^1|+|sb|)\sum\limits_{i=1}^{n-1}|\bar{sa}^i|+|\bar{sa}^1||sb|+d\sum\limits_{i=1}^{n-1}|tf^i|+(d+1)\sum\limits_{i=1}^n |\bar{sa}^i|
    \end{split}
\end{equation}
and
\begin{equation}
\label{eq:phi-MB}
    \phi^{\doubar{x}}_{\MB} : \bigotimes\limits_{i=1}^{n-1}(\MA[1]^{\bar{x}^i}\otimes {}_{\Phi_0(\rrt(\bar{x}^{i+1}))}\MB^*_{\Phi_0(\llt(\bar{x}^{i}))}[d])\otimes \MA[1]^{\bar{x}^n}\rightarrow {}_{\Phi_0(\llt(\bar{x}^1))}\MB_{\Phi_0(\rrt(\bar{x}^n))}[1]
\end{equation}
defined by 
\begin{small}
\begin{equation}
    \begin{split}
        \phi^{\doubar{x}}_{\MB}&(\bar{sa}^1,tf^1,\dots,\bar{sa}^{n-1},tf^{n-1},\bar{sa}^n)\\&=(-1)^{\delta}((f^{n-1}\circ s_d)\otimes \dots\otimes(f^1\circ s_d) \otimes\id)(s_{d+1}\Phi^{\doubar{x}^{-1}}(\bar{sa}^n,\dots,\bar{sa}^1))
    \end{split}
\end{equation}
\end{small}
for $\bar{sa}^i\in \MA[1]^{\otimes \bar{x}^i}$, $tf^i\in {}_{\Phi_0(\rrt(\bar{x}^{i+1}))}\MB^*_{\Phi_0(\llt(\bar{x}^{i}))}[d]$ with 
\begin{equation}
    \delta=\sum\limits_{i=1}^{n-1}|tf^i|\sum\limits_{j=i+1}^n|\bar{sa}^i|+\sum\limits_{1\leq i<j\leq n}|\bar{sa}^i||\bar{sa}^j|\sum\limits_{1\leq i<j\leq n-1}|tf^i||tf^j|+d(n-1)+d\sum\limits_{i=1}^{n-1}|tf^i|
\end{equation}
and for each $n\in\NN^*$, $\doubar{x}=(\bar{x}^1,\dots,\bar{x}^n)\in\bar{\MO}_{\MA}^n$.
\end{definition}

\begin{definition}
\label{def:morphisms-pCY-case}
 Consider two $d$-pre-Calabi-Yau categories $(\MA,s_{d+1}M_{\MA})$ and $(\MB,s_{d+1}M_{\MB})$ as well as a $d$-pre-Calabi-Yau morphism $(\Phi_0,s_{d+1}\Phi) :(\MA,s_{d+1}M_{\MA})\rightarrow(\MB,s_{d+1}M_{\MB})$.
    We define $s\varphi_{\MA}^{x,y}(sa)=sa$, $s\varphi_{\MB}^{x,y}(sa)=s_{d+1}\Phi^{x,y}(sa)$ for $sa\in {}_{x}\MA_{y}[1]$ and $s\varphi_{\MA}^{x,y}(tf)=tf\circ s_{d+1}\Phi^{y,x}[1-d]$, 
    $s\varphi_{\MB}^{x,y}(tf)=tf$, for $tf\in {}_{\Phi_0(x)}\MB^*_{\Phi_0(y)}[d]$, $x,y\in\MO_{\MA}$, as well as  
\begin{equation}
    \pi_{\MA^*[d]}\circ 
 s\varphi_{\MA}^{\doubar{x}}=s_{d+1}\phi^{\doubar{x}}_{\MA^*},  \pi_{\MA[1]}\circ 
 s\varphi_{\MA}^{\doubar{x}}=0
\end{equation}
and 
\begin{equation}
    \pi_{\MB[1]}\circ s\varphi_{\MB}^{\doubar{x}}= s_{d+1}\phi^{\doubar{x}}_{\MB},  \pi_{\MB^*[d]}\circ s\varphi_{\MB}^{\doubar{x}}=0
\end{equation}
where $\doubar{x}=(\bar{x}^1,\dots,\bar{x}^n)\in\bar{\MO}^n_{\MA}$ for $n>1$ or $\llg(\bar{x}^1)>2$ if $n=1$.
\end{definition}

\begin{proposition}
\label{prop:morphisms-general-case}
The maps $(\Phi_0,s\varphi_{\MA})$ and $(\Phi_0,s\varphi_{\MB})$ are morphisms of $A_{\infty}$-categories. 
\end{proposition}
\begin{proof}
The part of the identity \eqref{eq:stasheff-morphisms} for $\varphi_{\MA}$ that takes place in 
\[
\Homgr_{\kk}(\bigotimes\limits_{i=1}^{n-1}(\MA[1]^{\bar{x}^i}\otimes {}_{\Phi_0(\rrt(\bar{x}^{i+1}))}\MB^*_{\Phi_0(\llt(\bar{x}^{i}))}[d])\otimes \MA[1]^{\bar{x}^n},{}_{\llt(\bar{x}^1)}\MA_{\rrt(\bar{x}^n)}[1])
\] is clearly satisfied.
Moreover, by definition of the $A_{\infty}$-structure $sm_{\MA\oplus\MB^*}$, the part of the identity \eqref{eq:stasheff-morphisms} that takes place in \[
\Homgr_{\kk}(\bigotimes\limits_{i=1}^{n-1}(\MA[1]^{\bar{x}^i}\otimes {}_{\Phi_0(\rrt(\bar{x}^{i+1}))}\MB^*_{\Phi_0(\llt(\bar{x}^{i}))}[d])\otimes \MA[1]^{\bar{x}^n},{}_{\llt(\bar{x}^1)}\MA^*_{\rrt(\bar{x}^n)}[d])
\] is tantamount to
$\sum\mathcal{E}(\mathcalboondox{D_1})-\sum\mathcal{E}(\mathcalboondox{D_3})=-\sum\mathcal{E}(\mathcalboondox{D_2})$ where the sums are over all the filled diagrams $\mathcalboondox{D_1},\mathcalboondox{D_2}$ and $\mathcalboondox{D_3}$ of type $\doubar{x}$ of the form

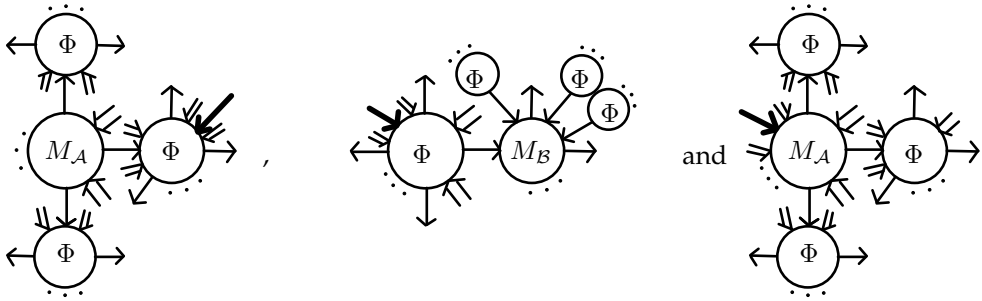
\begin{minipage}{21cm}
\begin{tikzpicture}[line cap=round,line join=round,x=1.0cm,y=1.0cm]
\clip(-2,-2) rectangle (3,3);
  \draw [line width=0.5pt] (0.,0.) circle (0.5cm);
     \shadedraw[rotate=30,shift={(0.5cm,0cm)}] \doublefleche;
     \shadedraw[rotate=150,shift={(0.5cm,0cm)}] \doublefleche;
     \draw[rotate=90][->, >= stealth, >= stealth](0.5,0)--(0.9,0);
     \draw [line width=0.5pt] (1.12,-0.65) circle (0.3cm);
     \shadedraw[shift={(0.86cm,-0.5cm)},rotate=150] \doubleflechescindeeleft;
     \shadedraw[shift={(0.86cm,-0.5cm)},rotate=150] \doubleflechescindeeright;
     \shadedraw[shift={(0.86cm,-0.5cm)},rotate=150] \fleche;
     \draw [rotate around ={60:(1.12,-0.65)}] [->, >= stealth, >= stealth] (1.43,-0.65)--(1.73,-0.65);
     \draw [rotate around ={-120:(1.12,-0.65)}] [->, >= stealth, >= stealth] (1.43,-0.65)--(1.73,-0.65);
     \draw [line width=0.5pt] (-1.12,-0.65) circle (0.3cm);
     \shadedraw[shift={(-0.86cm,-0.5cm)},rotate=30] \doubleflechescindeeleft;
      \shadedraw[shift={(-0.86cm,-0.5cm)},rotate=30] \doubleflechescindeeright;
       \shadedraw[shift={(-0.86cm,-0.5cm)},rotate=30] \fleche;
      \draw [rotate around ={-60:(-1.12,-0.65)}] [->, >= stealth, >= stealth] (-1.43,-0.65)--(-1.73,-0.65);
     \draw [rotate around ={120:(-1.12,-0.65)}] [->, >= stealth, >= stealth] (-1.43,-0.65)--(-1.73,-0.65);
     \draw [line width=0.5pt] (0.,1.2) circle (0.3cm);
     \draw[rotate around={-90:(0,1.2)},shift={(0.3,1.2)}]\doubleflechescindeeleft;
     \draw[rotate around={-90:(0,1.2)},shift={(0.3,1.2)}]\doubleflechescindeeright;
     \draw[rotate around={30:(0,1.2)},shift={(0.3,1.2)}]\doubleflechescindeeleft;
     \draw[rotate around={30:(0,1.2)},shift={(0.3,1.2)}]\doubleflechescindeeright;
     \draw [line width=1.1pt,rotate around={30:(0.,1.2)}][<-, >= stealth, >=stealth](0.3,1.2)--(0.7,1.2);
     \draw [rotate around={90:(0.,1.2)}][->, >= stealth, >=stealth](0.3,1.2)--(0.6,1.2);
     \draw [rotate around={-30:(0.,1.2)}][->, >= stealth, >=stealth](0.3,1.2)--(0.6,1.2);
     \draw [rotate around={-150:(0.,1.2)}][->, >= stealth, >=stealth](0.3,1.2)--(0.6,1.2);
\begin{scriptsize}
\draw [fill=black] (0,-0.6) circle (0.3pt);
\draw [fill=black] (0.2,-0.55) circle (0.3pt);
\draw [fill=black] (-0.2,-0.55) circle (0.3pt);
\draw [fill=black] (1.45,-0.85) circle (0.3pt);
\draw [fill=black] (1.5,-0.67) circle (0.3pt);
\draw [fill=black] (1.33,-0.97) circle (0.3pt);
\draw [fill=black] (-1.45,-0.85) circle (0.3pt);
\draw [fill=black] (-1.5,-0.67) circle (0.3pt);
\draw [fill=black] (-1.33,-0.97) circle (0.3pt);
\draw [fill=black] (-0.15,1.56) circle (0.3pt);
\draw [fill=black] (-0.3,1.42) circle (0.3pt);
\draw [fill=black] (-0.4,1.25) circle (0.3pt);
\end{scriptsize}
\draw(0,0.25)node[anchor=north]{$M_{\MA}$};
\draw(1.12,-0.4)node[anchor=north]{$\Phi$};
\draw(-1.12,-0.4)node[anchor=north]{$\Phi$};
\draw(0,1.45)node[anchor=north]{$\Phi$};
\draw(2,0.25)node[anchor=north]{,};
\end{tikzpicture}
\begin{tikzpicture}[line cap=round,line join=round,x=1.0cm,y=1.0cm]
\clip(-2,-2) rectangle (3,3);
 \draw [line width=0.5pt] (0.,0.) circle (0.5cm);
     \shadedraw[rotate around={30:(0,1.2)},shift={(0.3cm,1.2cm)}] \doublefleche;
     \shadedraw[rotate=150,shift={(0.5cm,0cm)}] \doublefleche;
     \draw[rotate=90][->, >= stealth, >= stealth](0.5,0)--(0.9,0);
     \draw [line width=0.5pt] (1.12,-0.65) circle (0.3cm);
     \shadedraw[shift={(0.86cm,-0.5cm)},rotate=150] \doubleflechescindeeleft;
     \shadedraw[shift={(0.86cm,-0.5cm)},rotate=150] \doubleflechescindeeright;
     \shadedraw[shift={(0.86cm,-0.5cm)},rotate=150] \fleche;
     \draw [rotate around ={60:(1.12,-0.65)}] [->, >= stealth, >= stealth] (1.43,-0.65)--(1.73,-0.65);
     \draw [rotate around ={-120:(1.12,-0.65)}] [->, >= stealth, >= stealth] (1.43,-0.65)--(1.73,-0.65);
     \draw [line width=0.5pt] (-1.12,-0.65) circle (0.3cm);
     \shadedraw[shift={(-0.86cm,-0.5cm)},rotate=30] \doubleflechescindeeleft;
      \shadedraw[shift={(-0.86cm,-0.5cm)},rotate=30] \doubleflechescindeeright;
       \shadedraw[shift={(-0.86cm,-0.5cm)},rotate=30] \fleche;
      \draw [rotate around ={-60:(-1.12,-0.65)}] [->, >= stealth, >= stealth] (-1.43,-0.65)--(-1.73,-0.65);
     \draw [rotate around ={120:(-1.12,-0.65)}] [->, >= stealth, >= stealth] (-1.43,-0.65)--(-1.73,-0.65);
     \draw [line width=0.5pt] (0.,1.2) circle (0.3cm);
     \draw[rotate around={-90:(0,1.2)},shift={(0.3,1.2)}]\doubleflechescindeeleft;
     \draw[rotate around={-90:(0,1.2)},shift={(0.3,1.2)}]\doubleflechescindeeright;
     \draw[rotate around={30:(0,0)},shift={(0.5,0)}]\doubleflechescindeeleft;
     \draw[rotate around={30:(0,0)},shift={(0.5,0)}]\doubleflechescindeeright;
     \draw [line width=1.1pt,rotate around={30:(0.,0)}][<-, >= stealth, >=stealth](0.5,0)--(0.9,0);
     \draw [rotate around={90:(0.,1.2)}][->, >= stealth, >=stealth](0.3,1.2)--(0.6,1.2);
     \draw [rotate around={-30:(0.,1.2)}][->, >= stealth, >=stealth](0.3,1.2)--(0.6,1.2);
     \draw [rotate around={-150:(0.,1.2)}][->, >= stealth, >=stealth](0.3,1.2)--(0.6,1.2);
\begin{scriptsize}
\draw [fill=black] (0,-0.6) circle (0.3pt);
\draw [fill=black] (0.2,-0.55) circle (0.3pt);
\draw [fill=black] (-0.2,-0.55) circle (0.3pt);
\draw [fill=black] (1.45,-0.85) circle (0.3pt);
\draw [fill=black] (1.5,-0.67) circle (0.3pt);
\draw [fill=black] (1.33,-0.97) circle (0.3pt);
\draw [fill=black] (-1.45,-0.85) circle (0.3pt);
\draw [fill=black] (-1.5,-0.67) circle (0.3pt);
\draw [fill=black] (-1.33,-0.97) circle (0.3pt);
\draw [fill=black] (-0.15,1.56) circle (0.3pt);
\draw [fill=black] (-0.3,1.42) circle (0.3pt);
\draw [fill=black] (-0.4,1.25) circle (0.3pt);
\end{scriptsize}
\draw(0,0.25)node[anchor=north]{$M_{\MA}$};
\draw(1.12,-0.4)node[anchor=north]{$\Phi$};
\draw(-1.12,-0.4)node[anchor=north]{$\Phi$};
\draw(0,1.45)node[anchor=north]{$\Phi$};
\draw(2.5,0.25)node[anchor=north]{and};
\end{tikzpicture}
\begin{tikzpicture}[line cap=round,line join=round,x=1.0cm,y=1.0cm]
\clip(-2.3,-2) rectangle (3,3);
 \draw [line width=0.5pt] (0.,0.) circle (0.5cm);
     \draw [rotate=90] [->, >= stealth, >= stealth] (0.5,0)--(0.9,0);
     \draw [rotate=-30] [->, >= stealth, >= stealth] (0.5,0)--(0.9,0);
     \draw [rotate=-150] [->, >= stealth, >= stealth] (0.5,0)--(0.9,0);
     \draw [rotate=0] [<-, >= stealth, >= stealth] (0.5,0)--(0.9,0);
     \draw [line width=0.5pt] (1.15,0.) circle (0.25cm);
     \draw[rotate around={0:(1.15,0)}] [->, >= stealth, >= stealth] (1.4,0)--(1.8,0);
     \draw[rotate around={90:(1.15,0)},shift={(1.4,0)}] \doublefleche;
     \draw [rotate=60] [<-, >= stealth, >= stealth] (0.5,0)--(0.9,0);
       \draw [rotate=180] [<-, >= stealth, >= stealth] (0.5,0)--(0.9,0);
      \draw [line width=0.5pt] (-1.15,0.) circle (0.25cm);
     \draw[rotate around={0:(-1.15,0)}] [->, >= stealth, >= stealth] (-1.4,0)--(-1.8,0);
     \draw[rotate around={-90:(-1.15,0)},shift={(-0.9,0)}]\doublefleche;
        \draw [rotate=120] [<-, >= stealth, >= stealth] (0.5,0)--(0.9,0);
        
        \draw [line width=0.5pt] (0.575,1) circle (0.25cm);
        \draw[rotate around={120:(0.575,1)}] [->, >= stealth, >= stealth] (0.825,1)--(1.225,1);
        \draw[rotate around={0:(0.575,1)}] [->, >= stealth, >= stealth] (0.825,1)--(1.225,1);
        \draw[rotate around={60:(0.575,1)},shift={(0.825,1)}] \doublefleche;
        \draw[rotate around={-60:(0.575,1)},shift={(0.825,1)}] \doublefleche;
        \draw [line width=0.5pt] (-0.575,1) circle (0.25cm);
      \draw[rotate around={300:(-0.575,1)}] [->, >= stealth, >= stealth] (-0.825,1)--(-1.225,1);
      \draw[rotate around={30:(-0.575,1)},shift={(-0.325,1)}] \doublefleche;
      \draw[rotate around={150:(0,0)}] [<-, >= stealth, >= stealth] (0.5,0)--(1.3,0);
    \draw [line width=0.5pt] (-1.34,0.775) circle (0.25cm);
    \draw[line width=1.1pt,rotate around={150:(-1.34,0.775)}] [<-, >= stealth, >= stealth] (-1.09,0.775)--(-0.69,0.775);
    \draw[rotate around={90:(-1.34,0.775)}] [->, >= stealth, >= stealth] (-1.09,0.775)--(-0.71,0.775);
    \draw [rotate around={-150:(-1.34,0.775)}] [->, >= stealth, >= stealth] (-1.09,0.775)--(-0.71,0.775);
    \draw[rotate around={30:(-1.34,0.775)},shift={(-1.09,0.775)}]\doublefleche;
    \draw[rotate around={150:(-1.34,0.775)},shift={(-1.09,0.775)}]\doubleflechescindeeleft;
    \draw[rotate around={150:(-1.34,0.775)},shift={(-1.09,0.775)}]\doubleflechescindeeright;
\begin{scriptsize}
\draw [fill=black] (0,-0.6) circle (0.3pt);
\draw [fill=black] (0.2,-0.55) circle (0.3pt);
\draw [fill=black] (-0.2,-0.55) circle (0.3pt);
\draw [rotate=120][fill=black] (0,-0.6) circle (0.3pt);
\draw [rotate=120][fill=black] (0.15,-0.55) circle (0.3pt);
\draw [rotate=120][fill=black] (-0.15,-0.55) circle (0.3pt);
\draw [fill=black] (-0.38,0.47) circle (0.3pt);
\draw [fill=black] (-0.44,0.43) circle (0.3pt);
\draw [fill=black] (-0.47,0.37) circle (0.3pt);
\draw [fill=black] (-0.56,0.22) circle (0.3pt);
\draw [fill=black] (-0.6,0.165) circle (0.3pt);
\draw [fill=black] (-0.6,0.1) circle (0.3pt);
\draw [fill=black] (0,-0.6) circle (0.3pt);
\draw [fill=black] (0.2,-0.55) circle (0.3pt);
\draw [fill=black] (-0.2,-0.55) circle (0.3pt);
\draw [rotate=120][fill=black] (0,-0.6) circle (0.3pt);
\draw [rotate=120][fill=black] (0.15,-0.55) circle (0.3pt);
\draw [rotate=120][fill=black] (-0.15,-0.55) circle (0.3pt);
\draw [fill=black] (-0.38,0.47) circle (0.3pt);
\draw [fill=black] (-0.44,0.43) circle (0.3pt);
\draw [fill=black] (-0.47,0.37) circle (0.3pt);
\draw [fill=black] (-0.56,0.22) circle (0.3pt);
\draw [fill=black] (-0.6,0.165) circle (0.3pt);
\draw [fill=black] (-0.6,0.1) circle (0.3pt);
\draw [rotate around={180:(-1.15,0)}][fill=black] (-1.15,-0.35) circle (0.3pt);
\draw [rotate around={180:(-1.15,0)}][fill=black] (-1,-0.3) circle (0.3pt);
\draw [rotate around={180:(-1.15,0)}][fill=black] (-1.3,-0.3) circle (0.3pt);
\draw (1.15,-0.35) circle (0.3pt);
\draw (1,-0.3) circle (0.3pt);
\draw (1.3,-0.3) circle (0.3pt);
\draw [fill=black] (-1.44,0.47) circle (0.3pt);
\draw [fill=black] (-1.32,0.45) circle (0.3pt);
\draw [fill=black] (-1.18,0.47) circle (0.3pt);
\draw [fill=black] (-0.85,0.85) circle (0.3pt);
\draw [fill=black] (-0.88,1.05) circle (0.3pt);
\draw [fill=black] (-0.7,0.72) circle (0.3pt);
\draw [rotate=-60][fill=black] (-0.85,0.85) circle (0.3pt);
\draw [rotate=-60] [fill=black] (-0.6,0.7) circle (0.3pt);
\draw [rotate=-60] [fill=black] (-0.75,0.75) circle (0.3pt);
\end{scriptsize}
\draw(0,0.25)node[anchor=north]{$M_{\MB}$};
\draw(-0.575,1.25)node[anchor=north]{$\Phi$};
\draw(0.575,1.25)node[anchor=north]{$\Phi$};
\draw(-1.34,0.775+0.25)node[anchor=north]{$\Phi$};
\draw(-1.15,0.25)node[anchor=north]{$\Phi$};
\draw(1.15,0.25)node[anchor=north]{$\Phi$};
\end{tikzpicture}
\end{minipage}

\noindent
respectively. 
Using that $(\Phi_0,s_{d+1}\Phi)$ is a $d$-pre-Calabi-Yau morphism, we thus have that $s\varphi_{\MA}$ is an $A_{\infty}$-morphism since $\sum\mathcal{E}(\mathcalboondox{D_3})=\sum\mathcal{E}(\mathcalboondox{D_1})+\sum\mathcal{E}(\mathcalboondox{D_2})$
The case of $s\varphi_{\MB}$ is similar.
\end{proof}

We now use the results of Propositions \ref{prop:struct-general-case} and \ref{prop:morphisms-general-case} to construct a functor between the category $\pCY$ and the partial category $\Ahat$.

\begin{corollary}
\label{corollary:functors}
There is a functor $\mathcal{P} : \pCY\rightarrow \Ahat$ sending a $d$-pre-Calabi-Yau category $(\MA,s_{d+1}M_{\MA})$ to the $A_{\infty}$-category $(\MA\oplus\MA^*[d-1],sm_{\MA\oplus\MA^*})$ defined in Proposition \ref{prop:equiv-pCY-Ainf} and a $d$-pre-Calabi-Yau morphism $(\Phi_0,s_{d+1}\Phi) : (\MA,s_{d+1}M_{\MA})\rightarrow (\MB,s_{d+1}M_{\MB})$ to the $A_{\infty}$-structure $sm_{\MA\oplus\MB^*}$ defined in Definition \ref{def:structure-pCY-case} together with the $A_{\infty}$-morphisms 
 \begin{equation}
\begin{tikzcd}
&(\MA \oplus \MB^*[d-1],sm_{\MA\oplus\MB^*}) \arrow[swap,"s\varphi_{\MA}"]{dl} \arrow[swap,"s\varphi_{\MB}"]{dr}\\
(\MA \oplus \MA^*[d-1],sm_{\MA\oplus\MA^*})&& (\MB \oplus \MB^*[d-1], sm_{\MB\oplus\MB^*})
\end{tikzcd}
\end{equation}
defined in Definition \ref{def:morphisms-pCY-case}.
Moreover, this functor restricts to a functor $\NpCY\rightarrow \cyc$.
\end{corollary}
\begin{proof}
    We only have to show that $\mathcal{P}$ is compatible with partial products. 
    Consider $d$-pre-Calabi-Yau categories $(\MA,s_{d+1}M_{\MA})$, $(\MB,s_{d+1}M_{\MB})$ and $(\mathcal{C},s_{d+1}M_{\mathcal{C}})$ as well as $d$-pre-Calabi-Yau morphisms 
    \begin{equation}
        (F_0,s_{d+1}\mathbf{F}) :(\MA,s_{d+1}M_{\MA})\rightarrow(\MB,s_{d+1}M_{\MB}) \text{ and }(G_0,s_{d+1}\mathbf{G}) : (\MB,s_{d+1}M_{\MB}) \rightarrow (\mathcal{C},s_{d+1}M_{\mathcal{C}}).
    \end{equation} 
    
    We will denote $\mathcal{P}(F_0,s_{d+1}\mathbf{F})=(sm_{\MA\oplus\MB^*},s\varphi_{\MA},s\varphi_{\MB})$, $\mathcal{P}(G_0,s_{d+1}\mathbf{G})=(sm_{\MB\oplus\mathcal{C}^*},s\psi_{\MB},s\psi_{\mathcal{C}})$ and $\mathcal{P}(G_0\circ F_0,s_{d+1}\mathbf{G}\upperset{}{\circ}s_{d+1}\mathbf{F})=(sm_{\MA\oplus\mathcal{C}^*},s\chi_{\MA},s\chi_{\mathcal{C}})$. To simplify, we set $H_0=G_0\circ F_0$.
    
    We define a map $s\xi_{\MA} : \MA[1]\oplus\mathcal{C}^*[d]\rightarrow \MA[1]\oplus\mathcal{B}^*[d]$ of degree $0$  given by $s\xi_{\MA}^{x,y}(sa)=sa$, for $sa\in {}_{x}\MA_{y}[1]$ and $s\xi_{\MA}^{x,y}(tf)=tf\circ s_{d+1}G[-d-1]$ for $tf\in{}_{(G_0\circ F_0)(x)}\mathcal{C}^*_{(G_0\circ F_0)(y)}[d]$ and by
$\pi_{\MA[1]}\circ s\xi_{\MA}^{\doubar{x}}=0$ and 
\begin{equation}
    \begin{split}
        &\big(\pi_{\MB^*[d]}\circ 
 s\xi_{\MA}^{\doubar{x}}(\bar{sa}^1,tf^1,\bar{sa}^2,\dots,\bar{sa}^n)\big)(s_{-d}b)
        \\&=(-1)^{\epsilon}\big((f^{n-1}\circ s_d)\otimes\dots\otimes(f^{1}\circ s_d)\big)\big(H_1^{\bar{x}^n\sqcup\bar{x}^1,\bar{x}^{n-1},\bar{x}^{n-2},\dots,\bar{x}^2}(\bar{sa}^n\otimes sb\otimes \bar{sa}^1,\bar{sa}^{n-1},\dots,\bar{sa}^2)\big)
    \end{split}
\end{equation}
for every $\doubar{x}=(\Bar{x}^1,\dots,\Bar{x}^n)\in\bar{\MO}_{\MA}^n$ and $\doubar{y}=(\Bar{x}^n\sqcup\Bar{x}^1,\dots,\Bar{x}^{n-1})$
where 
\begin{small}
    \begin{equation}
        \begin{split}
             \epsilon&=\sum\limits_{i=1}^{n-1}|tf^i|\sum\limits_{j=i+1}^n|\bar{sa}^i|+\hskip-2mm\sum\limits_{1\leq i<j\leq n}\hskip-2mm|\bar{sa}^i||\bar{sa}^j|+\hskip-2mm\sum\limits_{1\leq i<j\leq n-1}\hskip-2mm|tf^i||tf^j|+d(n-1)\\&+(|\bar{sa}^1|+|sb|)\sum\limits_{i=1}^{n-1}|\bar{sa}^i|+|\bar{sa}^1||sb|+d\sum\limits_{i=1}^{n-1}|tf^i|+(d+1)\sum\limits_{i=1}^n |\bar{sa}^i|
        \end{split}
    \end{equation}
\end{small}
and $s_{d+1}H_1^{\doubar{x}}=\sum\mathcal{E}(\mathcalboondox{D_1})$ where the sum is over all the filled diagrams $\mathcalboondox{D_1}$ of type $\doubar{x}$ and of the form

\begin{tikzpicture}[line cap=round,line join=round,x=1.0cm,y=1.0cm]
\clip(-7,-3.2) rectangle (11.420226760467184,1.7);
   \draw [line width=0.5pt] (0.,0.) circle (0.5cm);
     \draw [rotate=30] [->, >= stealth] (0.5,0)--(0.9,0);
     \draw [rotate=-90] [->, >= stealth] (0.5,0)--(0.9,0);
     \draw [rotate=150] [->, >= stealth] (0.5,0)--(0.9,0);
     \draw [rotate=0] [<-, >= stealth] (0.5,0)--(0.9,0);
     \draw [line width=0.5pt] (1.2,0.) circle (0.3cm);
     \draw [->, >= stealth] (1.5,0)--(1.9,0);
     \shadedraw[rotate around={90:(1.2,0)},shift={(1.5cm,0cm)}] \doublefleche;
     \shadedraw[rotate around={-90:(1.2,0)},shift={(1.5cm,0cm)}] \doublefleche;
      \draw [line width=0.5pt] (2.4,0.) circle (0.5cm);
      \draw [rotate around={-90:(2.4,0)}] [<-, >= stealth] (2.9,0)--(3.3,0);
      \draw [rotate around={-45:(2.4,0)}] [->, >= stealth] (2.9,0)--(3.3,0);
    \draw [rotate around={135:(2.4,0)}] [->, >= stealth] (2.9,0)--(3.3,0);
     \draw [line width=0.5pt] (2.4,-1.2) circle (0.3cm);
     \shadedraw[rotate around={-90:(2.4,-1.2)},shift={(2.7cm,-1.2cm)}] \doublefleche;
     \draw [rotate=-60] [<-, >= stealth] (0.5,0)--(0.9,0);
      \draw [line width=0.5pt] (0.6,-1.03) circle (0.3cm);
     \shadedraw[rotate around={-60:(0.6,-1.03)},shift={(0.9cm,-1.03cm)}] \doublefleche;
     \draw [line width=0.5pt] (-1.2,0.) circle (0.3cm);
     \draw [<-, >= stealth] (-0.5,0)--(-0.9,0);
     \shadedraw[rotate around={180:(-1.2,0)},shift={(-0.9cm,0cm)}] \doublefleche;
     \draw [rotate=-120] [<-, >= stealth] (0.5,0)--(0.9,0);
    \draw [line width=0.5pt] (-0.6,-1.03) circle (0.3cm);
     \shadedraw[rotate around={-120:(-0.6,-1.03)},shift={(-0.3cm,-1.03cm)}] \doublefleche;
     \shadedraw[rotate around={-240:(-0.6,-1.03)},shift={(-0.3cm,-1.03cm)}] \doublefleche;
     \draw[rotate around={180:(-0.6,-1.03)}][->, >= stealth] (-0.3,-1.03)--(0.1,-1.03);
     \draw[rotate around={-60:(-0.6,-1.03)}][->, >= stealth] (-0.3,-1.03)--(0.1,-1.03);
      \draw [line width=0.5pt] (-1.8,-1.03) circle (0.5cm);
      \draw [rotate around={120:(-1.8,-1.03)}] [<-, >= stealth] (-1.3,-1.03)--(-0.9,-1.03);
       \draw [rotate around={-60:(-1.8,-1.03)}] [->, >= stealth] (-1.3,-1.03)--(-0.9,-1.03);
        \draw [rotate around={180:(-1.8,-1.03)}] [->, >= stealth] (-1.3,-1.03)--(-0.9,-1.03);
     \draw [line width=0.5pt] (-2.4,0.) circle (0.3cm);
     \shadedraw[rotate around={120:(-2.4,0)},shift={(-2.1cm,0cm)}] \doublefleche;
     \draw [line width=0.5pt] (0,-2.06) circle (0.5cm);
      \draw [rotate around={-60:(0,-2.06)}] [->, >= stealth] (0.5,-2.06)--(0.9,-2.06);
    \draw [line width=1.1pt,rotate around={-135:(2.4,0)}] [<-, >= stealth,>=stealth] (2.9,0)--(3.3,0);
\begin{scriptsize}
\draw [fill=black] (0,0.6) circle (0.3pt);
\draw [fill=black] (0.2,0.55) circle (0.3pt);
\draw [fill=black] (-0.2,0.55) circle (0.3pt);
\draw [fill=black] (0.5,-0.3) circle (0.3pt);
\draw [fill=black] (0.4,-0.4) circle (0.3pt);
\draw [fill=black] (0.55,-0.18) circle (0.3pt);
\draw [fill=black] (-0.5,-0.3) circle (0.3pt);
\draw [fill=black] (-0.4,-0.4) circle (0.3pt);
\draw [fill=black] (-0.55,-0.18) circle (0.3pt);
\draw [fill=black] (1.8,-0.12) circle (0.3pt);
\draw [fill=black] (1.82,-0.2) circle (0.3pt);
\draw [fill=black] (1.88,-0.28) circle (0.3pt);
\draw [fill=black] (2.1,-0.5) circle (0.3pt);
\draw [fill=black] (2.18,-0.55) circle (0.3pt);
\draw [fill=black] (2.27,-0.57) circle (0.3pt);

\draw [fill=black] (2.8,0.4) circle (0.3pt);
\draw [fill=black] (2.95,0.2) circle (0.3pt);
\draw [fill=black] (2.6,0.55) circle (0.3pt);
\draw [fill=black] (-0.2,-1.03) circle (0.3pt);
\draw [fill=black] (-0.25,-0.9) circle (0.3pt);
\draw [fill=black] (-0.25,-1.16) circle (0.3pt);
\draw [fill=black] (0.3,-1.55) circle (0.3pt);
\draw [fill=black] (0.5,-1.7) circle (0.3pt);
\draw [fill=black] (0.6,-1.9) circle (0.3pt);
\draw [rotate around={180:(0,-2.06)}][fill=black] (0.6,-1.9) circle (0.3pt);
\draw [rotate around={180:(0,-2.06)}][fill=black] (0.5,-1.7) circle (0.3pt);
\draw [rotate around={180:(0,-2.06)}][fill=black] (0.3,-1.55) circle (0.3pt);
\draw [fill=black] (-1.3,-0.7) circle (0.3pt);
\draw [fill=black] (-1.45,-0.55) circle (0.3pt);
\draw [fill=black] (-1.67,-0.45) circle (0.3pt);
\draw [rotate around={180:(-1.8,-1.03)}][fill=black] (-1.3,-0.7) circle (0.3pt);
\draw [rotate around={180:(-1.8,-1.03)}][fill=black] (-1.45,-0.55) circle (0.3pt); circle (0.3pt);
\draw [rotate around={180:(-1.8,-1.03)}][fill=black] (-1.67,-0.45) circle (0.3pt);
\end{scriptsize}
\draw (0,0.25)node[anchor=north]{$\mathbf{G}$};
\draw (0,-1.8)node[anchor=north]{$\mathbf{G}$};
\draw (2.4,0.25)node[anchor=north]{$\mathbf{G}$};
\draw (-1.8,-0.78)node[anchor=north]{$\mathbf{G}$};
\draw (1.2,0.25)node[anchor=north]{$\mathbf{F}$};
\draw (-1.2,0.25)node[anchor=north]{$\mathbf{F}$};
\draw (2.4,-0.95)node[anchor=north]{$\mathbf{F}$};
\draw (-0.6,-0.78)node[anchor=north]{$\mathbf{F}$};
\draw (-2.4,0.25)node[anchor=north]{$\mathbf{F}$};
\draw (0.6,-0.78)node[anchor=north]{$\mathbf{F}$};
\end{tikzpicture}

\noindent if $\doubar{x}=(\bar{x}^1,\dots,\bar{x}^n)$ with $n>1$ or $n=1$ and $\llg(\bar{x}^1)>2$ where $\pi_{\MA[1]}$ (resp. $\pi_{\MB^*[d]})$ denotes the canonical projection $\MA[1]\oplus\mathcal{B}^*[d]\rightarrow \MA[1]$ (resp. $\MA[1]\oplus\mathcal{B}^*[d]\rightarrow\mathcal{B}^*[d]$).

    Similarly, we define $s\xi_{\mathcal{C}} : \MA[1]\oplus\mathcal{C}^*[d]\rightarrow \MB[1]\oplus\mathcal{C}^*[d]$ given by $s\xi_{\mathcal{C}}^{x,y}(sa)=s_{d+1}F^{x,y}(sa)$ for $sa\in {}_{x}\MA_{y}[1]$ and $s\xi_{\mathcal{C}}(tf)=tf$ for $tf\in{}_{H_0(x)}\mathcal{C}^*_{H_0(y)}[d]$ and by
\begin{equation}
    \begin{split}
        \pi_{\MB[1]}\circ s\xi_{\MA}^{\doubar{x}}=(-1)^{\delta} \big((f^{n-1}\circ s_d)\otimes\dots\otimes(f^{1}\circ s_d)\otimes \id\big)\big(H_2^{\doubar{x}^{-1}}(\bar{sa}^{n},\dots,\bar{sa}^1)\big)
    \end{split}
\end{equation}
where 
\begin{small}
    \begin{equation}
        \delta=\sum\limits_{i=1}^{n-1}|tf^i|\sum\limits_{j=i+1}^n|\bar{sa}^i|+\hskip-2mm\sum\limits_{1\leq i<j\leq n}\hskip-2mm|\bar{sa}^i||\bar{sa}^j|+\hskip-2mm\sum\limits_{1\leq i<j\leq n-1}\hskip-2mm|tf^i||tf^j|+d(n-1)+d\sum\limits_{i=1}^{n-1}|tf^i|
    \end{equation}
\end{small}
and $s_{d+1}H_2^{\doubar{x}}=\sum\mathcal{E}(\mathcalboondox{D_2})$ where the sum is over all the filled diagrams $\mathcalboondox{D_2}$ of type $\doubar{x}$ and of the form

\begin{tikzpicture}[line cap=round,line join=round,x=1.0cm,y=1.0cm]
\clip(-7,-2.5) rectangle (11.094475148786122,1.5);
 \draw [line width=0.5pt] (0.,0.) circle (0.5cm);
     \draw [rotate=30] [->, >= stealth] (0.5,0)--(0.9,0);
     \draw [rotate=-90] [->, >= stealth] (0.5,0)--(0.9,0);
     \draw [rotate=150] [->, >= stealth] (0.5,0)--(0.9,0);
     \draw [rotate=0] [<-, >= stealth] (0.5,0)--(0.9,0);
     \draw [line width=0.5pt] (1.2,0.) circle (0.3cm);
     \draw [->, >= stealth] (1.5,0)--(1.9,0);
     \shadedraw[rotate around={90:(1.2,0)},shift={(1.5cm,0cm)}] \doublefleche;
     \shadedraw[rotate around={-90:(1.2,0)},shift={(1.5cm,0cm)}] \doublefleche;
      \draw [line width=0.5pt] (2.4,0.) circle (0.5cm);
      \draw [rotate around={-90:(2.4,0)}] [<-, >= stealth] (2.9,0)--(3.3,0);
      \draw [rotate around={-45:(2.4,0)}] [->, >= stealth] (2.9,0)--(3.3,0);
    \draw [rotate around={135:(2.4,0)}] [->, >= stealth] (2.9,0)--(3.3,0);
     \draw [line width=0.5pt] (2.4,-1.2) circle (0.3cm);
     \shadedraw[rotate around={-90:(2.4,-1.2)},shift={(2.7cm,-1.2cm)}] \doublefleche;
     \draw [rotate=-60] [<-, >= stealth] (0.5,0)--(0.9,0);
      \draw [line width=0.5pt] (0.6,-1.03) circle (0.3cm);
     \shadedraw[rotate around={-60:(0.6,-1.03)},shift={(0.9cm,-1.03cm)}] \doublefleche;
     \draw [line width=0.5pt] (-1.2,0.) circle (0.3cm);
     \draw [<-, >= stealth] (-0.5,0)--(-0.9,0);
     \shadedraw[rotate around={180:(-1.2,0)},shift={(-0.9cm,0cm)}] \doublefleche;
     \draw [rotate=-120] [<-, >= stealth] (0.5,0)--(0.9,0);
    \draw [line width=0.5pt] (-0.6,-1.03) circle (0.3cm);
     \shadedraw[rotate around={-120:(-0.6,-1.03)},shift={(-0.3cm,-1.03cm)}] \doublefleche;
     \shadedraw[rotate around={-240:(-0.6,-1.03)},shift={(-0.3cm,-1.03cm)}] \doublefleche;
     \draw[rotate around={180:(-0.6,-1.03)}][->, >= stealth] (-0.3,-1.03)--(0.1,-1.03);
     \draw[line width=1.1pt,rotate around={-60:(-0.6,-1.03)}][->, >= stealth] (-0.3,-1.03)--(0.1,-1.03);
      \draw [line width=0.5pt] (-1.8,-1.03) circle (0.5cm);
      \draw [rotate around={120:(-1.8,-1.03)}] [<-, >= stealth] (-1.3,-1.03)--(-0.9,-1.03);
       \draw [rotate around={-60:(-1.8,-1.03)}] [->, >= stealth] (-1.3,-1.03)--(-0.9,-1.03);
        \draw [rotate around={180:(-1.8,-1.03)}] [->, >= stealth] (-1.3,-1.03)--(-0.9,-1.03);
     \draw [line width=0.5pt] (-2.4,0.) circle (0.3cm);
     \shadedraw[rotate around={120:(-2.4,0)},shift={(-2.1cm,0cm)}] \doublefleche;
\begin{scriptsize}
\draw [fill=black] (0,0.6) circle (0.3pt);
\draw [fill=black] (0.2,0.55) circle (0.3pt);
\draw [fill=black] (-0.2,0.55) circle (0.3pt);
\draw [fill=black] (0.5,-0.3) circle (0.3pt);
\draw [fill=black] (0.4,-0.4) circle (0.3pt);
\draw [fill=black] (0.55,-0.18) circle (0.3pt);
\draw [fill=black] (-0.5,-0.3) circle (0.3pt);
\draw [fill=black] (-0.4,-0.4) circle (0.3pt);
\draw [fill=black] (-0.55,-0.18) circle (0.3pt);
\draw [fill=black] (1.98,-0.4) circle (0.3pt);
\draw [fill=black] (1.88,-0.26) circle (0.3pt);
\draw [fill=black] (2.12,-0.5) circle (0.3pt);
\draw [fill=black] (2.8,0.4) circle (0.3pt);
\draw [fill=black] (2.95,0.2) circle (0.3pt);
\draw [fill=black] (2.6,0.55) circle (0.3pt);
\draw [fill=black] (-0.2,-1.03) circle (0.3pt);
\draw [fill=black] (-0.25,-0.9) circle (0.3pt);
\draw [fill=black] (-0.25,-1.16) circle (0.3pt);
\draw [fill=black] (-1.3,-0.7) circle (0.3pt);
\draw [fill=black] (-1.45,-0.55) circle (0.3pt);
\draw [fill=black] (-1.67,-0.45) circle (0.3pt);
\draw [rotate around={180:(-1.8,-1.03)}][fill=black] (-1.3,-0.7) circle (0.3pt);
\draw [rotate around={180:(-1.8,-1.03)}][fill=black] (-1.45,-0.55) circle (0.3pt); circle (0.3pt);
\draw [rotate around={180:(-1.8,-1.03)}][fill=black] (-1.67,-0.45) circle (0.3pt);
\end{scriptsize}
\draw (0,0.25)node[anchor=north]{$\mathbf{G}$};
\draw (2.4,0.25)node[anchor=north]{$\mathbf{G}$};
\draw (-1.8,-0.78)node[anchor=north]{$\mathbf{G}$};
\draw (1.2,0.25)node[anchor=north]{$\mathbf{F}$};
\draw (-1.2,0.25)node[anchor=north]{$\mathbf{F}$};
\draw (2.4,-0.95)node[anchor=north]{$\mathbf{F}$};
\draw (-0.6,-0.78)node[anchor=north]{$\mathbf{F}$};
\draw (-2.4,0.25)node[anchor=north]{$\mathbf{F}$};
\draw (0.6,-0.78)node[anchor=north]{$\mathbf{F}$};
\end{tikzpicture}

\noindent and $\pi_{\mathcal{C}^*[d]}\circ 
 s\xi_{\mathcal{C}}^{\doubar{x}}=0$
if $\doubar{x}=(\bar{x}^1,\dots,\bar{x}^n)$ with $n>1$ or $n=1$ and $\llg(\bar{x}^1)>2$ 
where $\pi_{\MB[1]}$ (resp. $\pi_{\mathcal{C}^*[d]}$) denotes the canonical projection $\MB[1]\oplus\mathcal{C}^*[d]\rightarrow \MB[1]$ (resp. $\MB[1]\oplus\mathcal{C}^*[d]\rightarrow \mathcal{C}^*[d]$).

    We only show that $s\xi_{\MA}$ is an $A_{\infty}$-morphism since the case of $s\xi_{\mathcal{C}^*}$ is similar. We have that 
    \[
    \pi_{\MA[1]}(s\xi_{\mathcal{A}}\upperset{G}{\circ}sm_{\MA\oplus\mathcal{C}^*})^{\doubar{x}}=\sum\mathcal{E}(\mathcalboondox{D})
    \]
    where the sum is over all the filled diagrams $\mathcalboondox{D}$ of type $\doubar{x}$ and of the form

\begin{tikzpicture}[line cap=round,line join=round,x=1.0cm,y=1.0cm]
\clip(-7,-2.7) rectangle (13.154435817002396,1.7);
   \draw [line width=0.5pt] (0.,0.) circle (0.5cm);
     \shadedraw[rotate=30,shift={(0.5cm,0cm)}] \doublefleche;
     \shadedraw[rotate=150,shift={(0.5cm,0cm)}] \doublefleche;
     \shadedraw [line width=1.1pt,shift={(0cm,1cm)},rotate=-90] \fleche;
     \shadedraw[shift={(0.86cm,-0.5cm)},rotate=150] \doubleflechescindeeleft;
     \shadedraw[shift={(0.86cm,-0.5cm)},rotate=150] \doubleflechescindeeright;
     \shadedraw[shift={(0.86cm,-0.5cm)},rotate=150] \fleche;
     \shadedraw[shift={(-0.86cm,-0.5cm)},rotate=30] \doubleflechescindeeleft;
      \shadedraw[shift={(-0.86cm,-0.5cm)},rotate=30] \doubleflechescindeeright;
       \shadedraw[shift={(-0.86cm,-0.5cm)},rotate=30] \fleche;
     \draw [line width=0.5pt] (1.12,-0.65) circle (0.3cm);
     \draw [rotate around ={60:(1.12,-0.65)}] [->, >= stealth] (1.43,-0.65)--(1.73,-0.65);
     \draw [line width=0.5pt] (1.57,0.13) circle (0.3cm);
     \draw[rotate around={-90:(1.57,0.13)}] [->, >= stealth] (1.87,0.13)--(2.17,0.13);
     \draw[rotate around={90:(1.57,0.13)}] [->, >= stealth] (1.87,0.13)--(2.17,0.13);
     \draw[rotate around={120:(1.57,0.13)}] [<-, >= stealth] (1.87,0.13)--(2.17,0.13);
     \draw [line width=0.5pt] (1.12,0.91) circle (0.3cm);
     \shadedraw[rotate around={120:(1.12,0.91)},shift={(1.42cm,0.91cm)}] \doublefleche;
     \draw [rotate around ={-120:(1.12,-0.65)}] [->, >= stealth] (1.43,-0.65)--(1.73,-0.65);
     \draw [line width=0.5pt] (0.67,-1.43) circle (0.3cm);
      \draw[rotate around={90:(0.67,-1.43)}] [->, >= stealth] (0.97,-1.43)--(1.27,-1.43);
      \draw[rotate around={-90:(0.67,-1.43)}] [->, >= stealth] (0.97,-1.43)--(1.27,-1.43);
     \draw[rotate around={-60:(0.67,-1.43)}] [<-, >= stealth] (0.97,-1.43)--(1.27,-1.43);
     \draw [line width=0.5pt] (0.67+0.45,-1.43-0.78) circle (0.3cm);
     \shadedraw[rotate around={-60:(0.67+0.45,-1.43-0.78)},shift={(1.42cm,-2.21cm)}] \doublefleche;
     \draw [line width=0.5pt] (-1.12,-0.65) circle (0.3cm);
     \draw [rotate around ={-60:(-1.12,-0.65)}] [->, >= stealth] (-1.43,-0.65)--(-1.73,-0.65);
     \draw [line width=0.5pt] (-1.12+0.45,-0.65-0.78) circle (0.3cm);
     \draw[rotate around={-60:(-1.12+0.45,-0.65-0.78)}][->, >= stealth] (-1.12+0.75,-0.65-0.78)--(-1.12+1.05,-0.65-0.78);
     \draw [rotate around ={120:(-1.12,-0.65)}] [->, >= stealth] (-1.43,-0.65)--(-1.73,-0.65);
     \draw [line width=0.5pt] (-1.12-0.45,-0.65+0.78) circle (0.3cm);
     \draw[rotate around={0:(-1.12-0.45,-0.65+0.78)}][->, >= stealth] (-1.12-0.15,-0.65+0.78)--(-1.12+0.15,-0.65+0.78);
     \draw[rotate around={180:(-1.12-0.45,-0.65+0.78)}][->, >= stealth] (-1.12-0.15,-0.65+0.78)--(-1.12+0.15,-0.65+0.78);
     \draw[rotate around={-120:(-1.12-0.45,-0.65+0.78)}][<-, >= stealth] (-1.12-0.15,-0.65+0.78)--(-1.12+0.15,-0.65+0.78);
    \draw [line width=0.5pt] (-1.12-0.9,-0.65) circle (0.3cm);
    \shadedraw[rotate around={-120:(-1.12-0.9,-0.65)},shift={(-1.12-0.6,-0.65)}] \doublefleche;
\begin{scriptsize}
\draw [fill=black] (0,-0.6) circle (0.3pt);
\draw [fill=black] (0.2,-0.55) circle (0.3pt);
\draw [fill=black] (-0.2,-0.55) circle (0.3pt);
\draw [fill=black] (1.45,-0.85) circle (0.3pt);
\draw [fill=black] (1.5,-0.67) circle (0.3pt);
\draw [fill=black] (1.33,-0.97) circle (0.3pt);
\draw [fill=black] (-1.45,-0.85) circle (0.3pt);
\draw [fill=black] (-1.5,-0.67) circle (0.3pt);
\draw [fill=black] (-1.33,-0.97) circle (0.3pt);
\draw [fill=black] (1.97,0.1) circle (0.3pt);
\draw [fill=black] (1.93,0.25) circle (0.3pt);
\draw [fill=black] (1.92,-0.07) circle (0.3pt);
\draw [rotate around={180:(1.57,0.13)}] [fill=black] (1.97,0.1) circle (0.3pt);
\draw [rotate around={180:(1.57,0.13)}] [fill=black] (1.93,0.25) circle (0.3pt);
\draw [rotate around={180:(1.57,0.13)}] [fill=black] (1.92,-0.05) circle (0.3pt);
\draw [fill=black] (1.05,-1.43) circle (0.3pt);
\draw [fill=black] (1,-1.23) circle (0.3pt);
\draw [fill=black] (1,-1.63) circle (0.3pt);
\draw [rotate around={180:(0.67,-1.43)}] [fill=black] (1.05,-1.43) circle (0.3pt);
\draw [rotate around={180:(0.67,-1.43)}] [fill=black] (1,-1.23) circle (0.3pt);
\draw [rotate around={180:(0.67,-1.43)}] [fill=black] (1,-1.63) circle (0.3pt);
\draw [fill=black] (-1.12-0.45,-0.65+0.78+0.4) circle (0.3pt);
\draw [fill=black] (-1.12-0.45-0.2,-0.65+0.78+0.34) circle (0.3pt);
\draw [fill=black] (-1.12-0.45+0.2,-0.65+0.78+0.34) circle (0.3pt);
\draw [rotate around={180:(-1.12-0.45,-0.65+0.78)}] [fill=black] (-1.12-0.45,-0.65+0.78+0.4) circle (0.3pt);
\draw [rotate around={180:(-1.12-0.45,-0.65+0.78)}][fill=black] (-1.12-0.45-0.1,-0.65+0.78+0.38) circle (0.3pt);
\draw [rotate around={180:(-1.12-0.45,-0.65+0.78)}][fill=black] (-1.12-0.45+0.1,-0.65+0.78+0.38) circle (0.3pt);
\draw [fill=black] (-1.45+0.65-0.18,-0.85-0.65-0.18) circle (0.3pt);
\draw [fill=black] (-1.5+0.65-0.2,-0.67-0.65-0.15) circle (0.3pt);
\draw [fill=black] (-1.33+0.65-0.15,-0.97-0.65-0.18) circle (0.3pt);
\draw [rotate around={180:(-1.12+0.45,-0.65-0.78)}][fill=black] (-1.45+0.65-0.18,-0.85-0.65-0.18) circle (0.3pt);
\draw [rotate around={180:(-1.12+0.45,-0.65-0.78)}][fill=black] (-1.5+0.65-0.2,-0.67-0.65-0.15) circle (0.3pt);
\draw [rotate around={180:(-1.12+0.45,-0.65-0.78)}][fill=black] (-1.33+0.65-0.15,-0.97-0.65-0.18) circle (0.3pt);
\end{scriptsize}
\draw (0,0.25)node[anchor=north]{$M_{\mathcal{A}}$};
\draw (1.12,-0.4)node[anchor=north]{$\mathbf{F}$};
\draw (-1.12,-0.4)node[anchor=north]{$\mathbf{F}$};
\draw (-1.12-0.9,-0.4)node[anchor=north]{$\mathbf{F}$};
\draw (1.12,1.16)node[anchor=north]{$\mathbf{F}$};
\draw (1.12,-1.95)node[anchor=north]{$\mathbf{F}$};
\draw (0.67,-1.16)node[anchor=north]{$\mathbf{G}$};
\draw (-1.12-0.45,-0.65+0.78+0.25)node[anchor=north]{$\mathbf{G}$};
\draw (1.12+0.45,-0.65+0.78+0.25)node[anchor=north]{$\mathbf{G}$};
\draw (-0.67,-1.43+0.25)node[anchor=north]{$\mathbf{G}$};
\end{tikzpicture}

It is clear that $\pi_{\MA[1]}(sm_{\MA\oplus\mathcal{B}^*}\upperset{M}{\circ}s\xi_{\mathcal{A}})^{\doubar{x}}=\sum\mathcal{E}(\mathcalboondox{D})$ as well.

Moreover, we have that $\pi_{\MB^*[d]}(sm_{\MA\oplus\mathcal{B}^*}\upperset{M}{\circ} s\xi_{\mathcal{A}})^{\doubar{x}}=-\sum\mathcal{E}(\mathcalboondox{D_1'})$
where the sum is over all the filled diagrams $\mathcalboondox{D_1'}$ of type $\doubar{x}$ and of the form

\begin{tikzpicture}[line cap=round,line join=round,x=1.0cm,y=1.0cm]
\clip(-7,-2) rectangle (15.541535560357394,2.5);
   \draw [line width=0.5pt] (0.,0.) circle (0.5cm);
     \draw[rotate=-45][->, >= stealth] (0.5,0)--(1,0);
     \draw [line width=0.5pt] (0.92,-0.92) circle (0.3cm);
     \draw[rotate around={180:(0.92,-0.92)}][->, >= stealth] (1.22,-0.92)--(1.52,-0.92);
     \draw[rotate around={0:(0.92,-0.92)}][->, >= stealth] (1.22,-0.92)--(1.52,-0.92);
     \draw[rotate around={30:(0.92,-0.92)}][<-, >= stealth] (1.22,-0.92)--(1.52,-0.92);
     \draw [line width=0.5pt] (0.92+0.78,-0.92+0.45) circle (0.3cm);
     \shadedraw[rotate around={30:(0.92+0.78,-0.92+0.45)}, shift={(0.92+0.78+0.3,-0.92+0.45)}] \doublefleche;
    \draw[rotate=-135][->, >= stealth] (0.5,0)--(1,0);
    \draw [line width=0.5pt] (-0.92,-0.92) circle (0.3cm);
     \draw[rotate around={180:(-0.92,-0.92)}][->, >= stealth] (-1.22,-0.92)--(-1.52,-0.92);
     \draw[rotate around={0:(-0.92,-0.92)}][->, >= stealth] (-1.22,-0.92)--(-1.52,-0.92);
     \draw[rotate around={150:(-0.92,-0.92)}][<-, >= stealth] (-0.62,-0.92)--(-0.32,-0.92);
     \draw [line width=0.5pt] (-0.92-0.78,-0.92+0.45) circle (0.3cm);
     \shadedraw[rotate around={150:(-0.92-0.78,-0.92+0.45)}, shift={(-0.92-0.78+0.3,-0.92+0.45)}] \doublefleche;
     \draw[line width=1.1pt,rotate=90][<-, >= stealth] (0.5,0)--(0.9,0);
     \draw[rotate=15][<-, >= stealth] (0.5,0)--(0.9,0);
     \draw [line width=0.5pt] (1.16,0.31) circle (0.3cm);
     \shadedraw[rotate around={105:(1.16,0.31)},shift={(1.46,0.31)}] \doublefleche;
     \shadedraw[rotate around={15:(1.16,0.31)}] [->, >= stealth,>=stealth](1.46,0.31)--(1.86,0.31);
     \draw [line width=0.5pt] (2.13,0.57) circle (0.3cm);
     \shadedraw[rotate around={15:(2.13,0.57)}] [->, >= stealth](2.43,0.57)--(2.83,0.57);
    \draw[rotate=165][<-, >= stealth] (0.5,0)--(0.9,0);
     \draw [line width=0.5pt] (-1.16,0.31) circle (0.3cm);
     \shadedraw[rotate around={75:(-1.16,0.31)},shift={(-0.86,0.31)}] \doublefleche;
     \shadedraw[rotate around={165:(-1.16,0.31)}] [->, >= stealth](-0.86,0.31)--(-0.46,0.31);
     \draw [line width=0.5pt] (-2.13,0.57) circle (0.3cm);
     \shadedraw[rotate around={205:(-2.13,0.57)}] [->, >= stealth](-1.83,0.57)--(-1.43,0.57);
     \shadedraw[rotate around={85:(-2.13,0.57)}] [<-, >= stealth](-1.83,0.57)--(-1.43,0.57);
     \draw [line width=0.5pt] (-2.02,1.57) circle (0.3cm);
     \shadedraw[rotate around={85:(-2.02,1.57)},shift={(-1.72,1.57)}] \doublefleche;
\begin{scriptsize}
\draw [fill=black] (0,-0.6) circle (0.3pt);
\draw [fill=black] (0.2,-0.55) circle (0.3pt);
\draw [fill=black] (-0.2,-0.55) circle (0.3pt);
\draw [rotate=150][fill=black] (-0.05,-0.6) circle (0.3pt);
\draw [rotate=150][fill=black] (0.15,-0.55) circle (0.3pt);
\draw [rotate=150][fill=black] (-0.22,-0.55) circle (0.3pt);
\draw [rotate=210][fill=black] (0.05,-0.6) circle (0.3pt);
\draw [rotate=210][fill=black] (0.2,-0.55) circle (0.3pt);
\draw [rotate=210][fill=black] (-0.12,-0.58) circle (0.3pt);
\draw [fill=black] (0.92,-1.32) circle (0.3pt);
\draw [fill=black] (0.72,-1.28) circle (0.3pt);
\draw [fill=black] (1.12,-1.28) circle (0.3pt);
\draw [rotate around={180:(0.92,-0.92)}][fill=black] (0.92,-1.32) circle (0.3pt);
\draw [rotate around={180:(0.92,-0.92)}][fill=black] (0.74,-1.28) circle (0.3pt);
\draw [rotate around={180:(0.92,-0.92)}][fill=black]  (1.1,-1.28) circle (0.3pt);
\draw [fill=black] (-0.92,-1.32) circle (0.3pt);
\draw [fill=black] (-0.72,-1.28) circle (0.3pt);
\draw [fill=black] (-1.12,-1.28) circle (0.3pt);
\draw [rotate around={180:(-0.92,-0.92)}][fill=black] (-0.92,-1.32) circle (0.3pt);
\draw [rotate around={180:(-0.92,-0.92)}][fill=black] (-0.74,-1.28) circle (0.3pt);
\draw [rotate around={180:(-0.92,-0.92)}][fill=black]  (-1.1,-1.28) circle (0.3pt);
\draw [fill=black] (-1.75,0.6) circle (0.3pt);
\draw [fill=black] (-1.78,0.77) circle (0.3pt);
\draw [fill=black] (-1.9,0.9) circle (0.3pt);
\draw [fill=black] (-1.5,0.1) circle (0.3pt);
\draw [fill=black] (-1.3,-0.05) circle (0.3pt);
\draw [fill=black] (-1.05,-0.05) circle (0.3pt);
\draw [fill=black] (1.5,0.1) circle (0.3pt);
\draw [fill=black] (1.3,-0.05) circle (0.3pt);
\draw [fill=black] (1.05,-0.05) circle (0.3pt);
\draw [fill=black] (2.2,0.93) circle (0.3pt);
\draw [fill=black] (1.78,0.75) circle (0.3pt);
\draw [fill=black] (1.95,0.9) circle (0.3pt);
\draw[rotate around={180:(2.13,0.57)}] [fill=black] (2.2,0.93) circle (0.3pt);
\draw [rotate around={180:(2.13,0.57)}][fill=black] (1.78,0.75) circle (0.3pt);
\draw [rotate around={180:(2.13,0.57)}][fill=black] (1.95,0.9) circle (0.3pt);
\end{scriptsize}
\draw (0,0.25)node[anchor=north]{$M_{\mathcal{B}}$};
\draw (-1.16,0.56)node[anchor=north]{$\mathbf{F}$};
\draw (1.16,0.56)node[anchor=north]{$\mathbf{F}$};
\draw  (-0.92-0.78,-0.92+0.7)node[anchor=north]{$\mathbf{F}$};
\draw  (0.92+0.78,-0.92+0.7)node[anchor=north]{$\mathbf{F}$};
\draw (-0.92,-0.67)node[anchor=north]{$\mathbf{G}$};
\draw (0.92,-0.67)node[anchor=north]{$\mathbf{G}$};
\draw (-2.13,0.57+0.25)node[anchor=north]{$\mathbf{G}$};
\draw  (-2.02,1.57+0.25)node[anchor=north]{$\mathbf{F}$};
\draw (2.13,0.57+0.25)node[anchor=north]{$\mathbf{G}$};
\end{tikzpicture}

On the other hand, $\pi_{\MB^*[d]}(s\xi_{\mathcal{A}}\upperset{G}{\circ}sm_{\MA\oplus\mathcal{C}^*})^{\doubar{x}}=-\sum\mathcal{E}(\mathcalboondox{D_2'})+\sum\mathcal{E}(\mathcalboondox{D_2\dprimeind})$ where the sum is over all the filled diagrams $\mathcalboondox{D_2'}$ and $\mathcalboondox{D_2\dprimeind}$ of type $\doubar{x}$ and of the form

\begin{minipage}{21cm}
\begin{tikzpicture}[line cap=round,line join=round,x=1.0cm,y=1.0cm]
\clip(-3,-3) rectangle (4.5,2.8);
 \draw [line width=0.5pt] (0.,0.) circle (0.5cm);
     \draw[rotate=-45][->, >= stealth] (0.5,0)--(1,0);
    \draw[rotate=-135][->, >= stealth] (0.5,0)--(1,0);
     \draw[rotate=90][<-, >= stealth] (0.5,0)--(0.9,0);
     \draw [line width=0.5pt] (0,1.2) circle (0.3cm);
     \draw[rotate around={150:(0,1.2)}][<-, >= stealth] (0.3,1.2)--(0.6,1.2);
     \draw[rotate around={30:(0,1.2)}][<-, >= stealth] (0.3,1.2)--(0.6,1.2);
     \draw[line width=1.1pt,rotate around={90:(0,1.2)}][<-, >= stealth] (0.3,1.2)--(0.7,1.2);
     \draw [line width=0.5pt] (0.78,1.65) circle (0.3cm);
     \shadedraw[rotate around={30:(0.78,1.65)}, shift={(1.08,1.65)}] \doublefleche;
     \draw [line width=0.5pt] (-0.78,1.65) circle (0.3cm);
      \shadedraw[rotate around={150:(-0.78,1.65)}, shift={(-0.48,1.65)}] \doublefleche;
     \draw[rotate=15][<-, >= stealth] (0.5,0)--(0.9,0);
     \draw [line width=0.5pt] (1.16,0.31) circle (0.3cm);
     \draw[rotate around={45:(1.16,0.31)}][<-, >= stealth] (1.46,0.31)--(1.76,0.31);
     \draw[rotate around={-15:(1.16,0.31)}][<-, >= stealth] (1.46,0.31)--(1.76,0.31);
     \draw [line width=0.5pt] (1.8,0.95) circle (0.3cm);
     \shadedraw[rotate around={135:(1.8,0.95)}, shift={(2.1,0.95)}] \doublefleche;
     \shadedraw[rotate around={45:(1.8,0.95)}][->,>=stealth] (2.1,0.95)--(2.4,0.95);
      \draw [line width=0.5pt] (2.44,1.59) circle (0.3cm);
      \shadedraw[rotate around={45:(2.44,1.59)}][->,>=stealth] (2.74,1.59)--(3.04,1.59);
     \draw [line width=0.5pt] (1.16+0.87,0.08) circle (0.3cm);
     \shadedraw[rotate around={-15:(1.16+0.87,0.08)}, shift={(1.46+0.87,0.08)}] \doublefleche;
    \draw[rotate=165][<-, >= stealth] (0.5,0)--(0.9,0);
     \draw [line width=0.5pt] (-1.16,0.31) circle (0.3cm);
    \draw[rotate around={135:(-1.16,0.31)}][<-, >= stealth] (-0.86,0.31)--(-0.56,0.31);
     \draw[rotate around={195:(-1.16,0.31)}][<-, >= stealth] (-0.86,0.31)--(-0.56,0.31);
     \draw [line width=0.5pt] (-1.8,0.95) circle (0.3cm);
     \shadedraw[rotate around={135:(-1.8,0.95)}, shift={(-1.5,0.95)}] \doublefleche;
     \draw [line width=0.5pt] (-1.16-0.87,0.08) circle (0.3cm);
     \shadedraw[rotate around={195:(-1.16-0.87,0.08)}, shift={(-1.16-0.57,0.08)}] \doublefleche;
\begin{scriptsize}
\draw [fill=black] (0,-0.6) circle (0.3pt);
\draw [fill=black] (0.2,-0.55) circle (0.3pt);
\draw [fill=black] (-0.2,-0.55) circle (0.3pt);
\draw [rotate=150][fill=black] (-0.05,-0.6) circle (0.3pt);
\draw [rotate=150][fill=black] (0.15,-0.55) circle (0.3pt);
\draw [rotate=150][fill=black] (-0.22,-0.55) circle (0.3pt);
\draw [rotate=210][fill=black] (0.05,-0.6) circle (0.3pt);
\draw [rotate=210][fill=black] (0.2,-0.55) circle (0.3pt);
\draw [rotate=210][fill=black] (-0.12,-0.58) circle (0.3pt);
\draw [fill=black] (0.16,1.63) circle (0.3pt);
\draw [fill=black] (0.25,1.6) circle (0.3pt);
\draw [fill=black] (0.32,1.53) circle (0.3pt);
\draw [rotate around={60:(0,1.2)}][fill=black] (0.16,1.63) circle (0.3pt);
\draw [rotate around={60:(0,1.2)}][fill=black] (0.25,1.6) circle (0.3pt);
\draw [rotate around={60:(0,1.2)}][fill=black] (0.32,1.53) circle (0.3pt);
\draw [fill=black] (1.51,0.5) circle (0.3pt);
\draw [fill=black] (1.55,0.42) circle (0.3pt);
\draw [fill=black] (1.54,0.33) circle (0.3pt);
\draw [rotate=150][fill=black] (1.51,0.5) circle (0.3pt);
\draw [rotate=150][fill=black] (1.55,0.42) circle (0.3pt);
\draw [rotate=150][fill=black] (1.54,0.33) circle (0.3pt);
(1.16,0.31)
\draw [fill=black] (2.2,1.9) circle (0.3pt);
\draw [fill=black] (2.08,1.7) circle (0.3pt);
\draw [fill=black] (2.4,1.97) circle (0.3pt);
\draw [rotate around={180:(2.44,1.59)}][fill=black] (2.2,1.9) circle (0.3pt);
\draw [rotate around={180:(2.44,1.59)}][fill=black] (2.08,1.7) circle (0.3pt);
\draw [rotate around={180:(2.44,1.59)}][fill=black] (2.4,1.97) circle (0.3pt);
\draw [fill=black] (1.9,0.58) circle (0.3pt);
\draw [fill=black] (2.1,0.68) circle (0.3pt);
\draw [fill=black] (2.2,0.88) circle (0.3pt);
\end{scriptsize}
\draw (0,0.25)node[anchor=north]{$M_{\mathcal{C}}$};
\draw (0,1.45)node[anchor=north]{$\mathbf{G}$};
\draw (1.16,0.56)node[anchor=north]{$\mathbf{G}$};
\draw (-1.16,0.56)node[anchor=north]{$\mathbf{G}$};
\draw (-1.16-0.87,0.08+0.25)node[anchor=north]{$\mathbf{F}$};
\draw (-1.8,0.95+0.25)node[anchor=north]{$\mathbf{F}$};
\draw (1.16+0.87,0.08+0.25)node[anchor=north]{$\mathbf{F}$};
\draw (1.8,0.95+0.25)node[anchor=north]{$\mathbf{F}$};
\draw (0.78,1.9)node[anchor=north]{$\mathbf{F}$};
\draw (-0.78,1.9)node[anchor=north]{$\mathbf{F}$};
\draw (2.44,1.59+0.25)node[anchor=north]{$\mathbf{G}$};
\draw (4,0.2)node[anchor=north]{and};
\end{tikzpicture}
\begin{tikzpicture}[line cap=round,line join=round,x=1.0cm,y=1.0cm]
\clip(-3.5,-3) rectangle (13.154435817002396,2.5);
   \draw [line width=0.5pt] (0.,0.) circle (0.5cm);
     \shadedraw[rotate=30,shift={(0.5cm,0cm)}] \doublefleche;
     \shadedraw[rotate=150,shift={(0.5cm,0cm)}] \doublefleche;
     \shadedraw[shift={(0.cm,1cm)},rotate=-90] \doubleflechescindeeleft;
     \shadedraw[shift={(0.cm,1cm)},rotate=-90] \doubleflechescindeeright;
     \shadedraw[shift={(0.cm,1cm)},rotate=-90] \fleche;
     \shadedraw[shift={(0.86cm,-0.5cm)},rotate=150] \doubleflechescindeeleft;
     \shadedraw[shift={(0.86cm,-0.5cm)},rotate=150] \doubleflechescindeeright;
     \shadedraw[shift={(0.86cm,-0.5cm)},rotate=150] \fleche;
     \shadedraw[shift={(-0.86cm,-0.5cm)},rotate=30] \doubleflechescindeeleft;
      \shadedraw[shift={(-0.86cm,-0.5cm)},rotate=30] \doubleflechescindeeright;
       \shadedraw[shift={(-0.86cm,-0.5cm)},rotate=30] \fleche;
       \draw (0,1.3) circle (0.3cm);
       \draw[rotate around={0:(0,1.3)}][->,>=stealth] (0.3,1.3)--(0.6,1.3);
       \draw (0.9,1.3) circle (0.3cm);
       \draw[rotate around={-150:(0.9,1.3)}][->,>=stealth](1.2,1.3)--(1.5,1.3);
       \draw[rotate around={30:(0.9,1.3)}][->,>=stealth](1.2,1.3)--(1.5,1.3);
       \draw[rotate around={60:(0.9,1.3)}][<-,>=stealth](1.2,1.3)--(1.5,1.3);
       \draw (1.35,2.08) circle (0.3cm);
       \draw[rotate around={60:(1.35,2.08)}, shift={(1.65,2.08)}]\doublefleche;
       \draw[rotate around={180:(0,1.3)}][->,>=stealth] (0.3,1.3)--(0.6,1.3);
       \draw (-0.9,1.3) circle (0.3cm);
       \draw[rotate around={150:(-0.9,1.3)}][->,>=stealth](-0.6,1.3)--(-0.3,1.3);
       \draw[rotate around={-30:(-0.9,1.3)}][->,>=stealth](-0.6,1.3)--(-0.3,1.3);
       \draw[rotate around={120:(-0.9,1.3)}][<-,>=stealth](-0.6,1.3)--(-0.3,1.3);
       \draw (-1.35,2.08) circle (0.3cm);
       \draw[rotate around={120:(-1.35,2.08)}, shift={(-1.05,2.08)}]\doublefleche;
       \draw[line width=1.1pt, rotate around={60:(-0.9,1.3)}][<-,>=stealth](-0.6,1.3)--(-0.3,1.3);
     \draw [line width=0.5pt] (1.12,-0.65) circle (0.3cm);
     \draw [rotate around={60:(1.12,-0.65)}] [->, >= stealth] (1.43,-0.65)--(1.73,-0.65);
     \draw [line width=0.5pt] (1.57,0.13) circle (0.3cm);
     \draw[rotate around={180:(1.57,0.13)}] [->, >= stealth] (1.87,0.13)--(2.17,0.13);
     \draw[rotate around={0:(1.57,0.13)}] [->, >= stealth] (1.87,0.13)--(2.17,0.13);
     \draw[rotate around={-60:(1.57,0.13)}] [<-, >= stealth] (1.87,0.13)--(2.17,0.13);
     \draw [line width=0.5pt] (1.12+0.9,-0.65) circle (0.3cm);
    \shadedraw[rotate around={-60:(1.12+0.9,-0.65)},shift={(1.12+0.6+0.6,-0.65)}] \doublefleche;
     \draw [rotate around ={-120:(1.12,-0.65)}] [->, >= stealth] (1.43,-0.65)--(1.73,-0.65);
     \draw [line width=0.5pt] (0.67,-1.43) circle (0.3cm);
      \draw[rotate around={90:(0.67,-1.43)}] [->, >= stealth] (0.97,-1.43)--(1.27,-1.43);
      \draw[rotate around={-90:(0.67,-1.43)}] [->, >= stealth] (0.97,-1.43)--(1.27,-1.43);
     \draw[rotate around={-60:(0.67,-1.43)}] [<-, >= stealth] (0.97,-1.43)--(1.27,-1.43);
     \draw [line width=0.5pt] (0.67+0.45,-1.43-0.78) circle (0.3cm);
     \shadedraw[rotate around={-60:(0.67+0.45,-1.43-0.78)},shift={(1.42cm,-2.21cm)}] \doublefleche;
     \draw [line width=0.5pt] (-1.12,-0.65) circle (0.3cm);
     \draw [rotate around ={-60:(-1.12,-0.65)}] [->, >= stealth] (-1.43,-0.65)--(-1.73,-0.65);
     \draw [line width=0.5pt] (-1.12+0.45,-0.65-0.78) circle (0.3cm);
     \draw[rotate around={-60:(-1.12+0.45,-0.65-0.78)}][->, >= stealth] (-1.12+0.75,-0.65-0.78)--(-1.12+1.05,-0.65-0.78);
     \draw [rotate around ={120:(-1.12,-0.65)}] [->, >= stealth] (-1.43,-0.65)--(-1.73,-0.65);
     \draw [line width=0.5pt] (-1.12-0.45,-0.65+0.78) circle (0.3cm);
     \draw[rotate around={0:(-1.12-0.45,-0.65+0.78)}][->, >= stealth] (-1.12-0.15,-0.65+0.78)--(-1.12+0.15,-0.65+0.78);
     \draw[rotate around={180:(-1.12-0.45,-0.65+0.78)}][->, >= stealth] (-1.12-0.15,-0.65+0.78)--(-1.12+0.15,-0.65+0.78);
     \draw[rotate around={-120:(-1.12-0.45,-0.65+0.78)}][<-, >= stealth] (-1.12-0.15,-0.65+0.78)--(-1.12+0.15,-0.65+0.78);
    \draw [line width=0.5pt] (-1.12-0.9,-0.65) circle (0.3cm);
    \shadedraw[rotate around={-120:(-1.12-0.9,-0.65)},shift={(-1.12-0.6,-0.65)}] \doublefleche;
\begin{scriptsize}
\draw [fill=black] (0,-0.6) circle (0.3pt);
\draw [fill=black] (0.2,-0.55) circle (0.3pt);
\draw [fill=black] (-0.2,-0.55) circle (0.3pt);
\draw [fill=black] (1.45,-0.85) circle (0.3pt);
\draw [fill=black] (1.5,-0.67) circle (0.3pt);
\draw [fill=black] (1.33,-0.97) circle (0.3pt);
\draw [fill=black] (-1.45,-0.85) circle (0.3pt);
\draw [fill=black] (-1.5,-0.67) circle (0.3pt);
\draw [fill=black] (-1.33,-0.97) circle (0.3pt);
\draw [fill=black] (+1.12+0.45,-0.65+0.78+0.4) circle (0.3pt);
\draw [fill=black] (+1.12+0.45-0.2,-0.65+0.78+0.34) circle (0.3pt);
\draw [fill=black] (+1.12+0.45+0.2,-0.65+0.78+0.34) circle (0.3pt);
\draw [rotate around={180:(+1.12+0.45,-0.65+0.78)}] [fill=black] (1.12+0.45,-0.65+0.78+0.4) circle (0.3pt);
\draw [rotate around={180:(1.12+0.45,-0.65+0.78)}][fill=black] (+1.12+0.45-0.1,-0.65+0.78+0.38) circle (0.3pt);
\draw [rotate around={180:(1.12+0.45,-0.65+0.78)}][fill=black] (1.12+0.45+0.1,-0.65+0.78+0.38) circle (0.3pt);
\draw [fill=black] (1.05,-1.43) circle (0.3pt);
\draw [fill=black] (1,-1.23) circle (0.3pt);
\draw [fill=black] (1,-1.63) circle (0.3pt);
\draw [rotate around={180:(0.67,-1.43)}] [fill=black] (1.05,-1.43) circle (0.3pt);
\draw [rotate around={180:(0.67,-1.43)}] [fill=black] (1,-1.23) circle (0.3pt);
\draw [rotate around={180:(0.67,-1.43)}] [fill=black] (1,-1.63) circle (0.3pt);
\draw [fill=black] (-1.12-0.45,-0.65+0.78+0.4) circle (0.3pt);
\draw [fill=black] (-1.12-0.45-0.2,-0.65+0.78+0.34) circle (0.3pt);
\draw [fill=black] (-1.12-0.45+0.2,-0.65+0.78+0.34) circle (0.3pt);
\draw [rotate around={180:(-1.12-0.45,-0.65+0.78)}] [fill=black] (-1.12-0.45,-0.65+0.78+0.4) circle (0.3pt);
\draw [rotate around={180:(-1.12-0.45,-0.65+0.78)}][fill=black] (-1.12-0.45-0.1,-0.65+0.78+0.38) circle (0.3pt);
\draw [rotate around={180:(-1.12-0.45,-0.65+0.78)}][fill=black] (-1.12-0.45+0.1,-0.65+0.78+0.38) circle (0.3pt);
\draw [fill=black] (-1.45+0.65-0.18,-0.85-0.65-0.18) circle (0.3pt);
\draw [fill=black] (-1.5+0.65-0.2,-0.67-0.65-0.15) circle (0.3pt);
\draw [fill=black] (-1.33+0.65-0.15,-0.97-0.65-0.18) circle (0.3pt);
\draw [rotate around={180:(-1.12+0.45,-0.65-0.78)}][fill=black] (-1.45+0.65-0.18,-0.85-0.65-0.18) circle (0.3pt);
\draw [rotate around={180:(-1.12+0.45,-0.65-0.78)}][fill=black] (-1.5+0.65-0.2,-0.67-0.65-0.15) circle (0.3pt);
\draw [rotate around={180:(-1.12+0.45,-0.65-0.78)}][fill=black] (-1.33+0.65-0.15,-0.97-0.65-0.18) circle (0.3pt);
\draw[fill=black] (0,1.7) circle (0.3pt);
\draw[fill=black] (-0.15,1.65) circle (0.3pt);
\draw[fill=black] (0.15,1.65) circle (0.3pt);
\draw[fill=black] (0.6,1.5) circle (0.3pt);
\draw[fill=black] (0.72,1.62) circle (0.3pt);
\draw[fill=black] (0.9,1.67) circle (0.3pt);
\draw[rotate around={180:(0.9,1.3)}][fill=black] (0.6,1.5) circle (0.3pt);
\draw[rotate around={180:(0.9,1.3)}][fill=black] (0.72,1.62) circle (0.3pt);
\draw[rotate around={180:(0.9,1.3)}][fill=black] (0.9,1.67) circle (0.3pt);
\draw[rotate around={180:(-0.9,1.3)}][fill=black] (-0.6,1.5) circle (0.3pt);
\draw[rotate around={180:(-0.9,1.3)}][fill=black] (-0.72,1.62) circle (0.3pt);
\draw[rotate around={180:(-0.9,1.3)}][fill=black] (-0.9,1.67) circle (0.3pt);
\draw[fill=black] (-0.8,1.67) circle (0.3pt);
\draw[fill=black] (-0.9,1.7) circle (0.3pt);
\draw[fill=black] (-1,1.67) circle (0.3pt);
\draw[rotate around={-60:(-0.9,1.3)}][fill=black] (-0.8,1.67) circle (0.3pt);
\draw[rotate around={-60:(-0.9,1.3)}][fill=black] (-0.9,1.7) circle (0.3pt);
\draw[rotate around={-60:(-0.9,1.3)}][fill=black] (-1,1.67) circle (0.3pt);
\end{scriptsize}
\draw (0,0.25)node[anchor=north]{$M_{\mathcal{A}}$};
\draw (1.12,-0.4)node[anchor=north]{$\mathbf{F}$};
\draw (-1.12,-0.4)node[anchor=north]{$\mathbf{F}$};
\draw (-1.12-0.9,-0.4)node[anchor=north]{$\mathbf{F}$};
\draw (1.12+0.9,-0.4)node[anchor=north]{$\mathbf{F}$};
\draw (1.12,-1.95)node[anchor=north]{$\mathbf{F}$};
\draw (0,1.55)node[anchor=north]{$\mathbf{F}$};
\draw (1.35,2.33)node[anchor=north]{$\mathbf{F}$};
\draw (-1.35,2.33)node[anchor=north]{$\mathbf{F}$};
\draw (0.67,-1.16)node[anchor=north]{$\mathbf{G}$};
\draw (-1.12-0.45,-0.65+0.78+0.25)node[anchor=north]{$\mathbf{G}$};
\draw (1.12+0.45,-0.65+0.78+0.25)node[anchor=north]{$\mathbf{G}$};
\draw (-0.67,-1.43+0.25)node[anchor=north]{$\mathbf{G}$};
\draw (0.9,1.55)node[anchor=north]{$\mathbf{G}$};
\draw (-0.9,1.55)node[anchor=north]{$\mathbf{G}$};
\end{tikzpicture}
\end{minipage}

Therefore, using that $(F_0,s_{d+1}\mathbf{F})$ is a $d$-pre-Calabi-Yau morphism, we finally obtain the equality $\sum\mathcal{E}(\mathcalboondox{D_2'})=\sum\mathcal{E}(\mathcalboondox{D_1'})+\sum\mathcal{E}(\mathcalboondox{D_2\dprimeind})$ which shows that $s\xi_{\MA}$ is an $A_{\infty}$-morphism. Moreover, $s\xi_{\MA}$ and $s\xi_{\mathcal{C}^*}$ satisfy that $s\varphi_{\MB}\circ s\xi_{\MA}=s\psi_{\MB}\circ 
 s\xi_{\mathcal{C}}$ so the triple $(sm_{\MA\oplus\mathcal{C}^*}, s\varphi_{\MA}\circ  s\xi_{\MA},s\psi_{\mathcal{C}}\circ 
 s\xi_{\mathcal{C}})$ is the composition of $\mathcal{P}(F_0,s_{d+1}\mathbf{F})$ and $\mathcal{P}(G_0,s_{d+1}\mathbf{G})$.    
Since $s\varphi_{\MA}\circ s\xi_{\MA}=s\chi_{\MA}$ and $s\psi_{\mathcal{C}}\circ 
 s\xi_{\mathcal{C}}=s\chi_{\mathcal{C}}$ the functor $\mathcal{P}$ is compatible with partial products.
\end{proof}

\bibliographystyle{alpha}
\bibliography{mybiblio}

\vspace{1cm}

MARION BOUCROT: Univ. Grenoble Alpes, CNRS, IF, 38000 Grenoble, France

\textit{E-mail adress :} \href{mailto:marion.boucrot@univ-grenoble-alpes.fr}{\texttt{marion.boucrot@univ-grenoble-alpes.fr}}

\end{document}